\documentclass{amsart}
\usepackage[margin=1in]{geometry}
\usepackage{amsmath,amssymb,amsthm,mathtools,microtype,mathrsfs,aliascnt}
\usepackage{enumitem}
\usepackage[colorlinks=true]{hyperref}
\usepackage[nameinlink,capitalize,noabbrev]{cleveref}

\numberwithin{equation}{section}
\allowdisplaybreaks

\newtheorem{theorem}{Theorem}[section]
\newaliascnt{proposition}{theorem}
\newtheorem{proposition}[proposition]{Proposition}
\aliascntresetthe{proposition}
\newaliascnt{lemma}{theorem}
\newtheorem{lemma}[lemma]{Lemma}
\aliascntresetthe{lemma}
\newaliascnt{corollary}{theorem}
\newtheorem{corollary}[corollary]{Corollary}
\aliascntresetthe{corollary}
\theoremstyle{definition}
\newaliascnt{definition}{theorem}
\newtheorem{definition}[definition]{Definition}
\aliascntresetthe{definition}
\theoremstyle{remark}
\newaliascnt{remark}{theorem}
\newtheorem{remark}[remark]{Remark}
\aliascntresetthe{remark}

\crefname{theorem}{Theorem}{Theorems}
\Crefname{theorem}{Theorem}{Theorems}
\crefname{proposition}{Proposition}{Propositions}
\Crefname{proposition}{Proposition}{Propositions}
\crefname{lemma}{Lemma}{Lemmas}
\Crefname{lemma}{Lemma}{Lemmas}
\crefname{corollary}{Corollary}{Corollaries}
\Crefname{corollary}{Corollary}{Corollaries}
\crefname{definition}{Definition}{Definitions}
\Crefname{definition}{Definition}{Definitions}
\crefname{remark}{Remark}{Remarks}
\Crefname{remark}{Remark}{Remarks}

\newcommand{\T}{\mathbb T}
\newcommand{\hc}{\mathrm{hc}}
\newcommand{\ind}{\mathrm{ind}}
\newcommand{\ord}{\mathrm{ord}}
\newcommand{\one}{\mathbf 1}
\newcommand{\dd}{\,\mathrm d}
\newcommand{\norm}[1]{\left\lVert#1\right\rVert}
\newcommand{\localheading}[1]{%
  \par\addvspace{0.7\baselineskip}%
  \noindent\hspace*{\parindent}%
  \textit{#1}\hspace{0.5em}\ignorespaces
}
\newcommand{\Hdeg}{\mathcal H}
\newcommand{\RW}{\mathrm{RW}}
\newcommand{\SEP}{\mathrm{SEP}}
\newcommand{\PhiGauss}{\Phi}
\newcommand{\TV}{\mathrm{TV}}
\newcommand{\Ran}{\operatorname{Ran}}
\newcommand{\Marg}{\mathsf{Marg}}
\newcommand{\E}{\mathbb E}
\newcommand{\Pp}{\mathbb P}
\newcommand{\Cum}{\operatorname{Cum}}
\newcommand{\abs}[1]{\left|#1\right|}
\DeclareMathOperator{\Spec}{Spec}
\DeclareMathOperator{\Sym}{Sym}

\hypersetup{
  pdftitle={Canonical Local Equilibrium and Cutoff Profiles for the Symmetric Exclusion Process on Discrete Tori},
  pdfauthor={Joe P. Chen}
}

\begin{document}
\title[Cutoff for exclusion on $D$-dimensional tori, positive density regime]{Canonical Local Equilibrium and Cutoff Profiles for the Symmetric Exclusion Process on Discrete Tori}
\author{Joe P.\@ Chen}
\address{Department of Mathematics, Colgate University, Hamilton, NY 13346, USA}
\email{jpchen@colgate.edu}
\date{\today}
\thanks{The author thanks the Research Council of Colgate University for partial financial support.}
\begin{abstract}
We prove a canonical (or fixed-population) local equilibrium theorem for the
symmetric simple exclusion process on the discrete torus \(\T_N^D\), \(D\ge2\),
at particle densities bounded away from $0$ and $1$, uniformly over all
deterministic initial configurations with the prescribed particle number.
At times
\[
 t_N(s)=\frac{\log(N^D)+s}{2\gamma_N},
 \qquad \gamma_N=2-2\cos\left(\frac{2\pi}{N}\right),
\]
the Radon--Nikodym density of the process relative to equilibrium converges
in \(L^2\) to a canonical exponential tilt generated by the unique small
mean-zero calibration field whose one-site marginals match the evolving
one-particle heat profile.  Consequently, whenever the profile coordinate
converges, the corresponding total variation profile is a Gaussian shift.
More precisely, for a deterministic initial sequence $(S_N)$, if the
covariance-normalized squared amplitude $\mathfrak q_N^{S_N}(s)$ converges to a
scalar $\mathfrak q$ at a fixed $s$, then the distance to stationarity converges to
$
 2\PhiGauss\!\left({\sqrt{\mathfrak q}}/2\right)-1
$,
where $\PhiGauss$ is the standard normal distribution function.  The profile
coordinate is asymptotically determined by the one-particle eigenspace
corresponding to the smallest nonzero eigenvalue.

The proof combines a calibrated canonical comparison, a fixed-degree
comparison between independent and exclusion dynamics, and all-degree control
obtained from pair energy estimates and preservation of the Strong--Rayleigh
property.
\end{abstract}

\maketitle
\setcounter{tocdepth}{2}
\enlargethispage{4pt}
\tableofcontents

\section{Model, normalization, and main results}

Fix $D\ge2$. For each $N\geq 3$, let $V_N=\T_N^D=(\mathbb Z/N\mathbb Z)^D$ and $n_N=|V_N|=N^D$.
Unless
otherwise specified, the \emph{spatial representative} of a site
$x\in V_N$ means its unique representative in $\{0,\ldots,N-1\}^D$;
when an arbitrary integer lift is intended, we say so explicitly. 
We endow $V_N$ with the nearest-neighbor edge set $E_N:=\bigl\{\{x,x+\mathbf e_j\}:x\in V_N,\ 1\le j\le D\bigr\}$, where $\mathbf e_j$ is the $j$th coordinate vector.
The positive semidefinite one-particle Laplacian is
\begin{equation*}
 (L_N^{\RW}f)(x)=\sum_{y\sim x}\left[f(x)-f(y)\right].
\end{equation*}
Its Fourier eigenvalues are
\begin{equation}
 \lambda_N(p)=\sum_{j=1}^D\left(2-2\cos\frac{2\pi p_j}{N}\right),
 \qquad p=(p_1,\ldots, p_D)\in\T_N^D.
 \label{prof:eq:torus-spectrum}
\end{equation}
For $p\in\T_N^D$, set
\begin{equation}
 e_p(x):=e^{2\pi i p\cdot x/N},
 \qquad
 \chi_p:=n_N^{-1/2}e_p.
 \label{prof:eq:probability-fourier-basis}
\end{equation}
Thus $(e_p)_p$ is orthonormal for the uniform probability measure on $V_N$,
whereas $(\chi_p)_p$ is orthonormal in counting measure.
Fix the half-open representative set
\begin{equation}
 \mathcal Q_N
 :=\left\{-\left\lfloor\frac{N-1}{2}\right\rfloor,
 \ldots,\left\lfloor\frac N2\right\rfloor\right\}^{D}
 \subset\mathbb Z^D,
 \label{prof:eq:momentum-representatives}
\end{equation}
and write \(\widetilde p\in\mathcal Q_N\) for the unique representative of
\(p\in\T_N^D\).  Equivalently, \eqref{prof:eq:torus-spectrum} is the exact
identity
\begin{equation}
 \lambda_N(p)
 =4\sum_{j=1}^D\sin^2\!\left(\frac{\pi\widetilde p_j}{N}\right).
 \label{prof:eq:torus-sine-spectrum}
\end{equation}
The smallest nonzero eigenvalue of $L_N^{\RW}$, called the \emph{spectral gap}, is
\begin{equation*}
 \gamma_N=2-2\cos\frac{2\pi}{N}\sim\frac{4\pi^2}{N^2}.
\end{equation*}
The momenta at the spectral gap and their Fourier eigenspace are
\begin{equation}
 \mathcal S_N^{(1)}
 :=\{p\in\T_N^D:\lambda_N(p)=\gamma_N\},
 \qquad
 \mathscr E_N^{(1)}
 :=\operatorname{span}\{\chi_p:p\in\mathcal S_N^{(1)}\}.
 \label{prof:eq:first-nonzero-eigenspace}
\end{equation}

Let $k_N\in\{0,\ldots,n_N\}$ particles evolve on $\T_N^D$ by symmetric exclusion.  We assume throughout
this paper that there exists a fixed $\rho_0\in(0,1/2]$ such that for all $N\ge 3$,
\begin{equation}
 \rho_N:=\frac{k_N}{n_N}\in[\rho_0,1-\rho_0].
 \label{prof:eq:density-window}
\end{equation}
We refer to $[\rho_0, 1-\rho_0]$ as the \emph{density window} below.
Set
\begin{equation*}
 \ell_N=\min\{k_N,n_N-k_N\},
\end{equation*}
the maximal nontrivial degree in the slice decomposition introduced below.
The state space is $\Omega_{N,k_N}=\left\{\eta\in\{0,1\}^{V_N}:\sum_x\eta_x=k_N\right\}$, and the stationary distribution is $\pi_N=\operatorname{Unif}(\Omega_{N,k_N})$.
The positive semidefinite exclusion generator is
\begin{equation*}
 (L_N^{\SEP}F)(\eta)
 =\sum_{\{x,y\}\in E_N} [F(\eta)-F(\eta^{xy})],
\end{equation*}
where $\eta^{xy}$ denotes the configuration obtained from swapping the occupations $\eta_x$ and $\eta_y$ in $\eta$:
\[
(\eta^{xy})_z =
\left\{
\begin{array}{ll}
\eta_y, & z=x,\\
\eta_x, & z=y,\\
\eta_z, & \text{else}.
\end{array}
\right.
\]
For a deterministic initial occupied set $S_N$, $|S_N|=k_N$, let
$\mu_{N,t}^{S_N}$ be the law of the exclusion process at time $t$, and write
\begin{equation*}
 h_{N,t}^{S_N}:=\frac{\dd\mu_{N,t}^{S_N}}{\dd\pi_N}
\end{equation*}
for its Radon--Nikodym derivative relative to $\pi_N$.

In this paper a \emph{one-particle field} refers to a function $f:V_N\to\mathbb R$.
We write
\[
 \mathbb R_0^{V_N}
 :=\left\{f:V_N\to\mathbb R:\sum_{x\in V_N}f(x)=0\right\}
\]
for the \emph{mean-zero subspace}.  When complex-valued functions are used, we
work in its complexification \(\mathbb C_0^{V_N}\), and extend the counting
and averaged inner products sesquilinearly.
The standing conventions are
\[
 \langle f,g\rangle_{\mathrm{cnt}}=\sum_xf(x)\overline{g(x)},
 \qquad
 \langle f,g\rangle_{\mathrm{av}}=n_N^{-1}\sum_xf(x)\overline{g(x)},
 \qquad
 \|f\|_{2,\mathrm{cnt}}^2=n_N\|f\|_{2,\mathrm{av}}^2.
\]

Put
\begin{equation*}
 \beta_N=\rho_N(1-\rho_N)\frac{n_N}{n_N-1}.
\end{equation*}
For
\(f,g\in\mathbb R_0^{V_N}\), a direct computation gives
\begin{equation}
 \left\langle\sum_{x\in V_N} f(x)(\eta_x-\rho_N),
 \sum_{y\in V_N} g(y)(\eta_y-\rho_N)\right\rangle_{\pi_N}
 =\beta_N\sum_{x\in V_N}f(x)g(x).
 \label{prof:eq:first-chaos-isometry}
\end{equation}
Let
\begin{equation}
 F_N^{S_N}:=\one_{S_N}-\rho_N,
 \qquad m^{S_N}_{N,t}=e^{-tL_N^{\RW}}F_N^{S_N}.
 \label{prof:eq:profile}
\end{equation}
The exact one-site marginal identity is $\mathbb E_{\mu_{N,t}^{S_N}}[\eta_x]-\rho_N=m^{S_N}_{N,t}(x)$.
For $s\in\mathbb R$, define
\begin{equation}
 t_N(s)=\frac{\log n_N+s}{2\gamma_N},
 \qquad
 \mathfrak q_N^{S_N}(s):=\beta_N^{-1}
 \|m_{N,t_N(s)}^{S_N}\|_{2,\mathrm{cnt}}^2.
 \label{prof:eq:profile-coordinate}
\end{equation}
Whenever $t_N(s)$ is used as an actual time for the exclusion process, $N$ is understood to be
large enough that $t_N(s)\ge0$.  For every nonempty compact $J\subset\mathbb R$, this
condition holds uniformly for $s\in J$ once $N$ is sufficiently large, so it
does not affect any limit below.
When a single deterministic source $S_N$ is fixed, we suppress this dependence
and write $F_N$, $m_{N,t}$, $\mu_{N,t}$, and $\mathfrak q_N(s)$ for the corresponding
source-dependent objects.
\begin{definition}[Calibrated canonical tilt]
\label{prof:def:tilt}
For $v\in\mathbb R_0^{V_N}$, put
\begin{equation}
 g_{N,v}(\eta)=
 \frac{\exp\{\sum_xv(x)(\eta_x-\rho_N)\}}
 {\mathbb E_{\pi_N}\exp[\{\sum_xv(x)(\eta_x-\rho_N)\}]},
 \qquad
 \nu_{N,v}:=g_{N,v}\pi_N.
 \label{prof:eq:tilt}
\end{equation}
Thus $g_{N,v}$ is the density of the tilted law $\nu_{N,v}$ relative to
$\pi_N$.  In particular, $\nu_{N,v}$ is a canonical tilt: it remains supported
on the fixed-population slice $\Omega_{N,k_N}$.  Since
\(
 \sum_x(\eta_x-\rho_N)=0
\)
on $\Omega_{N,k_N}$, adding a constant to $v$ does not change $g_{N,v}$.
The restriction $v\in\mathbb R_0^{V_N}$ fixes this additive gauge.

Whenever $m_{N,t}$ lies in the local calibration regime of
\cref{prof:lem:calibration}, let $v_{N,t}$ be the unique field in the
specified small $L^\infty$ neighborhood of the origin such that the one-site
marginals of $\nu_{N,v_{N,t}}$ equal $\rho_N+m_{N,t}$, and set
$g_{N,t}:=g_{N,v_{N,t}}$.  In source-explicit notation, these objects are
written $v_{N,t}^{S_N}$ and
$g_{N,t}^{S_N}:=g_{N,v_{N,t}^{S_N}}$.
\end{definition}
Under \eqref{prof:eq:density-window}, \cref{prof:lem:calibration} gives a
unique $v_{N,t}\in\mathbb R_0^{V_N}$ with
$\|v_{N,t}\|_\infty\le c_{\rho_0}$ whenever
$m_{N,t}\in\mathbb R_0^{V_N}$ and
$\|m_{N,t}\|_\infty\le\delta_{\rho_0}$, together with the quantitative approximation to
$\beta_N^{-1}m_{N,t}$ stated in \eqref{prof:eq:calibration-estimate}.

\begin{remark}
The adjective \emph{canonical} is inherited from the \emph{canonical ensemble} in statistical mechanics, referring to the setting where the total number of particles is fixed.
For the exclusion process considered here, this means fixed particle
number \(k_N\).
Thus \(\pi_N\) and the tilted laws \(\nu_{N,v}\) are supported on
\(\Omega_{N,k_N}\), and all partition functions, cumulants, and response
maps below are computed with respect to measures on this same fixed-particle
slice.
\end{remark}

We now state the main result.  On the gap-centered scale $t_N(s)$,
the canonical tilt $g_{N,t_N(s)}^{S_N}$ is calibrated so that its one-site
marginals agree exactly with those of the evolving exclusion process.  
The next theorem upgrades this first-moment matching to an approximation of the full
Radon--Nikodym density in $L^2(\pi_N)$, uniformly over deterministic initial
sets and compact profile windows.  Thus ``canonical local equilibrium'' here
means full-density approximation by a law on the same fixed-particle slice,
not merely convergence of local marginals.  The resulting total-variation
profile is then governed by the scalar coordinate $\mathfrak q_N^{S_N}(s)$.

\begin{theorem}[Calibrated local equilibrium and cutoff profile]
\label{prof:thm:main}
Assume $D\ge2$ and \eqref{prof:eq:density-window}.  For every nonempty compact
$J\subset\mathbb R$, the calibrated fields
$v_{N,t_N(s)}^{S_N}$ and densities $g_{N,t_N(s)}^{S_N}$ are defined for all
sufficiently large $N$, uniformly over
$S_N\subset V_N$ with $|S_N|=k_N$ and $s\in J$.  For those $N$,
\begin{equation}
 \sup_{\substack{S_N\subset V_N,\ |S_N|=k_N\\ s\in J}}
 \|h_{N,t_N(s)}^{S_N}-g_{N,t_N(s)}^{S_N}\|_{L^2(\pi_N)}
 \longrightarrow0.
 \label{prof:eq:L2-local-equilibrium}
\end{equation}
Consequently,
\begin{equation}
 \sup_{\substack{S_N\subset V_N,\ |S_N|=k_N\\ s\in J}}\left|
 \|\mu_{N,t_N(s)}^{S_N}-\pi_N\|_{\TV}
 -\frac12\mathbb E_{\pi_N}|g_{N,t_N(s)}^{S_N}-1|
 \right|\longrightarrow0.
 \label{prof:eq:TV-tilt-reduction}
\end{equation}
Moreover, let $(S_N)$ be any deterministic source sequence with
$|S_N|=k_N$.  If, at some fixed $s\in\mathbb R$,
$\mathfrak q_N^{S_N}(s)\to\mathfrak q$ for a scalar $\mathfrak q\ge0$, then
\begin{equation}
 \|\mu_{N,t_N(s)}^{S_N}-\pi_N\|_{\TV}
 \longrightarrow
 2\PhiGauss\!\left(\frac{\sqrt{\mathfrak q}}{2}\right)-1.
 \label{prof:eq:explicit-profile}
\end{equation}
Here \(\PhiGauss\) denotes the standard normal distribution function,
$\PhiGauss(z)=\frac{1}{\sqrt{2\pi}}\int_{-\infty}^z e^{-x^2/2}\,\dd x$.
\end{theorem}

\Cref{prof:thm:main} reduces the cutoff profile to the asymptotics of
the scalar coordinate $\mathfrak q_N^{S_N}(s)$.  The next two corollaries
sharpen this reduction in complementary directions.  The first upgrades the
pointwise profile statement to compact-uniform convergence whenever
$\mathfrak q_N^{S_N}$ converges uniformly on the profile window.  The second
identifies the leading asymptotics of $\mathfrak q_N^{S_N}$ itself: on the
gap-centered scale $t_N(s)$, only the projection of the initial discrepancy
onto the first nonzero one-particle eigenspace $\mathscr E_N^{(1)}$ from
\eqref{prof:eq:first-nonzero-eigenspace} survives.  Their proofs are deferred
to \cref{prof:sec:profile-corollary-proofs}, after the proof of
\cref{prof:thm:main} is completed.

\begin{corollary}[Compact-uniform profile convergence]
\label{prof:cor:uniform-profile}
Let $S_N\subset V_N$ be a deterministic source sequence with $|S_N|=k_N$,
and let $J\subset\mathbb R$ be nonempty and compact.  If a
function $\mathfrak q:J\to[0,\infty)$ satisfies
\begin{equation}
 \sup_{s\in J}|\mathfrak q_N(s)-\mathfrak q(s)|\longrightarrow0,
 \label{prof:eq:uniform-q-convergence}
\end{equation}
then
\begin{equation}
 \sup_{s\in J}\left|
 \|\mu_{N,t_N(s)}^{S_N}-\pi_N\|_{\TV}
 -\left(2\PhiGauss\!\left(\frac{\sqrt{\mathfrak q(s)}}2\right)-1\right)
 \right|\longrightarrow0.
 \label{prof:eq:uniform-explicit-profile}
\end{equation}
\end{corollary}

\begin{corollary}[Reduction of the profile coordinate to the first nonzero eigenspace]
\label{prof:cor:first-shell-profile}
For a deterministic source $S_N\subset V_N$ with $|S_N|=k_N$, and with
$\mathcal S_N^{(1)}$ and $\mathscr E_N^{(1)}$ as in
\eqref{prof:eq:first-nonzero-eigenspace}, set
\begin{equation}
 \widehat F_N(p):=\langle F_N,\chi_p\rangle_{\mathrm{cnt}}.
 \label{prof:eq:source-fourier-coefficients}
\end{equation}
Define the covariance-normalized amplitude of the projection onto
$\mathscr E_N^{(1)}$ by
\begin{equation}
 a_N(S_N):=\frac1{\beta_Nn_N}
 \sum_{p\in\mathcal S_N^{(1)}}|\widehat F_N(p)|^2.
 \label{prof:eq:first-shell-amplitude}
\end{equation}
For every nonempty compact $J\subset\mathbb R$,
\begin{equation}
 \sup_{s\in J}|\mathfrak q_N(s)-e^{-s}a_N(S_N)|
 \le \frac{C_{D,\rho_0,J}}{n_N}
 \label{prof:eq:first-shell-reduction}
\end{equation}
uniformly over all $S_N\subset V_N$ with $|S_N|=k_N$.  In
particular, if $a_N(S_N)\to a$, then $\mathfrak q_N(s)\to ae^{-s}$ uniformly on
compact $s$-intervals, and
\begin{equation*}
 \|\mu_{N,t_N(s)}^{S_N}-\pi_N\|_{\TV}
 \longrightarrow
 2\PhiGauss\!\left(\frac{\sqrt a}{2}e^{-s/2}\right)-1
\end{equation*}
uniformly for $s$ in compact subsets of $\mathbb R$.
\end{corollary}

The amplitude $a_N(S_N)$ distinguishes two qualitatively different source
geometries.  A macroscopic slab has a nonzero projection onto the slowest
Fourier modes in $\mathscr E_N^{(1)}$ and therefore exhibits a nontrivial
profile on the cutoff window $t_N(s)$.  More
precisely, suppose $w_N/N\to\rho\in(0,1)$, the slab sequence satisfies
the standing density window \eqref{prof:eq:density-window} (as happens
eventually whenever $\rho\in(\rho_0,1-\rho_0)$), and
\[
 S_N=\{x\in\T_N^D:0\le x_1<w_N\}.
\]
Then
\begin{equation}
 a_N(S_N)\longrightarrow
 a_\rho:=\frac{2\sin^2(\pi\rho)}{\pi^2\rho(1-\rho)};
 \label{prof:eq:slab-amplitude}
\end{equation}
in particular, a half torus has amplitude $8/\pi^2$.  By contrast, if
$\widehat F_N$ vanishes on $\mathcal S_N^{(1)}$, then
$\mathfrak q_N(s)=O_J(n_N^{-1})$.  Applying \cref{prof:cor:uniform-profile} with the
zero limiting profile gives
\[
 \sup_{s\in J}
 \|\mu_{N,t_N(s)}^{S_N}-\pi_N\|_{\TV}\longrightarrow0
\]
for every nonempty compact $J\subset\mathbb R$.  The even-torus checkerboard is a
concrete example.  These calculations are given in
\cref{prof:sec:profile-corollary-proofs}.

\begin{remark}[Source dependence of the profile]

Parseval's identity gives
\[
 0\le a_N(S_N)\le\frac{n_N-1}{n_N}<1.
\]
By \cref{prof:cor:first-shell-profile}, every subsequential limit of
$a_N(S_N)$ determines the corresponding Gaussian shift profile.  Hence the
profile on the gap-centered scale $t_N(s)$ converges along the full sequence precisely when
$a_N(S_N)$ converges; distinct subsequential limits of the amplitude produce
distinct subsequential profiles on that scale.  The source-uniform local equilibrium
conclusion of \cref{prof:thm:main} remains valid in either case.
\end{remark}

\subsection{Chaos and norm conventions}

The \emph{Hoeffding decomposition} of the fixed-particle slice is
\begin{align}
 L^2(\Omega_{N,k_N},\pi_N)=\bigoplus_{r=0}^{\ell_N}\Hdeg_{N,r}.
 \label{prof:eq:Hoeffding}
\end{align}
We call $r$ the \emph{degree} and $\Hdeg_{N,r}$ the degree-$r$ \emph{chaos
layer}.  It represents the genuine $r$-site interaction after all lower-degree
interactions have been removed.  \Cref{prof:sec:chaos-interface}
constructs its normalized ordered distinct-tuple coordinate.

All slice, product, and ordered
distinct-tuple spaces carry their uniform probability measures.  On those
spaces $\|\cdot\|_2$ is probability normalized, and
$\langle\cdot,\cdot\rangle_2$ denotes the corresponding
probability-normalized inner product whenever no more specific subscript is
shown.

Denote $[r]=\{1,\ldots,r\}$ and
$(a)_{\underline r}=a(a-1)\cdots(a-r+1)$.
We use $\mathbb N=\{1,2,\ldots\}$ and
$\mathbb N_0=\{0,1,2,\ldots\}$.
We write
$A\lesssim_\Theta B$ if $A\le C_\Theta B$, where $C_\Theta$ depends only on
$\Theta$ and not on $N$.  We write $A\asymp_\Theta B$ if both
$A\lesssim_\Theta B$ and $B\lesssim_\Theta A$.  The uppercase letter $D$
always denotes the spatial dimension.

\section{Sketch of the proof}

The proof of \cref{prof:thm:main} is divided into three branches: a static canonical branch, a dynamic
fixed-degree branch, and an all-degree warm-start branch for the exclusion
density.  
The first two branches meet in the fixed-degree matching
\cref{prof:thm:fixed-degree-matching}.
The last branch combines analytic pair control with a Strong--Rayleigh
closure mechanism that lifts pair information to all chaos degrees. 
All three branches are assembled in \cref{prof:sec:completion-proof}.

\localheading{Static canonical and dynamic fixed-degree branches
(\cref{prof:sec:calibration,prof:sec:chaos-interface,prof:sec:fixed-degree}).}
The relevant one-particle field is
$
 m_{N,t}=e^{-tL_N^{\RW}}(\one_{S_N}-\rho_N)$,
whose slow modes remain visible on the cutoff scale.  In
\cref{prof:sec:calibration} we construct the unique local calibration field
and the corresponding canonical tilt, whose marginals are
$\rho_N+m_{N,t}$, and establish the Gaussianization input used later to
evaluate its total variation distance from equilibrium.
\Cref{prof:sec:chaos-interface} places the calibrated tilt and the exclusion
density in the same finite-population chaos coordinates.  At each fixed
degree, it identifies the leading chaos coordinate of the calibrated tilt
with the corresponding coordinate generated by $m_{N,t}$, and it
separately controls the high-chaos tail of the tilt.

For the dynamic comparison, \cref{prof:sec:fixed-degree} compares
independent walkers with their exclusion-constrained, or hard-core,
counterparts.  A Schur complement estimate identifies a finite-dimensional
low-energy comparison range and gives an asymptotically sharp hard-core
spectral bottom there.  The tensor source is then split according to a
finite one-particle low-energy truncation.  The contribution generated
entirely by the truncated one-particle source is compared directly by
Duhamel's formula.  For the remaining source contribution, a one-sided
leakage estimate controls its projection onto the low hard-core spectral
subspace, while its high spectral component decays rapidly under the
hard-core semigroup.  The static and dynamic fixed-degree branches meet at
\cref{prof:thm:fixed-degree-matching}.

\localheading{Exclusion high degrees: pair analysis and Strong--Rayleigh closure
(\cref{prof:sec:warm-start,prof:sec:pair-energy,prof:sec:chaos-weight,prof:sec:row,prof:sec:strong-rayleigh}).}
Before the target profile window, the genuinely degree-two part of the pair
discrepancy, after its constant and one-particle components have been
removed, is controlled by the cutoff window pair-energy estimate and the
resulting terminal row bound.  Strong--Rayleigh preservation then supplies
the structural closure: the number of exclusion particles in the source
set has a Poisson--binomial law, and the resulting centered occupation
moments of every degree can be bounded in terms of the one-particle profile
and the controlled pair contribution.  These moment bounds give a positive
exponential weight in the chaos degree at an earlier profile time.  The
linear-in-degree spectral gap then propagates this weighted $L^2$ control
forward until the weight reaches one, yielding the required warm start.

\localheading{Assembly (\cref{prof:sec:completion-proof}).}
From $t_N(s_0)$ to $t_N(s)$, the degree-$r$ gap is linear in $r$, yielding
\cref{prof:prop:degree-damping}.  Fixed-degree matching, the resulting
high-degree bound for the exclusion density, and the canonical tilt tail
\cref{prof:lem:tilt-tail} imply $L^2$ local equilibrium
\eqref{prof:eq:L2-local-equilibrium}.  The total variation comparison and
Gaussianization established in \cref{prof:sec:calibration} then evaluate
the total variation distance of the canonical tilt from equilibrium, and
yield the profile statement in \cref{prof:thm:main}.
The profile corollaries and examples follow in
\cref{prof:sec:profile-corollary-proofs}.

\subsection{Connections with prior literature}

On the $D=1$ torus (circle) the cutoff profile for symmetric exclusion was proved by Lacoin
\cite{LacoinCircleCutoff}.
His earlier work \cite{LacoinCircleDiffusive}
identified the diffusive cutoff window on the circle, and also gave on the $D\ge 2$ torus
mixing time lower bounds based on the slow one-particle mode.  These
works already isolate the gap-centered time scale, the Gaussian shape of the
profile, and the dominant role of the first nonzero eigenspace.  In
\cite{LacoinCircleCutoff} an exponential tilt of equilibrium by the principal
Fourier mode is central to the profile analysis.  
The calibration introduced in our \cref{prof:def:tilt} is
more flexible: the source-dependent canonical field is determined by exact
matching of the one-site marginals, and is not restricted \emph{a priori} to the
principal mode.  Lacoin's proof also exploits one-dimensional coupling and
order structure that are unavailable on the $D\ge2$ torus.

Comparisons between exclusion and the corresponding one-particle random walk, at both spectral and mixing-time levels, have a substantial history.  Caputo, Liggett, and Richthammer
\cite{CaputoLiggettRichthammer} proved Aldous' spectral-gap conjecture for the
interchange process, thereby identifying its gap with the single-particle
random-walk gap.  Diaconis and Saloff-Coste \cite{DiaconisSaloffCoste}
developed geometric comparison inequalities for reversible chains and
applied them to exclusion, while Oliveira \cite{OliveiraExclusionMixing}
obtained general exclusion mixing bounds in terms of the corresponding
single-particle random walk.  The fixed-degree analysis in our \cref{prof:sec:fixed-degree}
is related in spirit but more local: it compares independent and hard-core
operators on a prescribed chaos sector, and supplements quadratic form comparison by
source-specific resolvent estimates.

The finite-population harmonic analysis used here is likewise connected to
the geometry of the Boolean slice and the Johnson scheme.  Filmus
\cite{FilmusSlice} constructs an explicit orthogonal basis for functions on a
slice and relates it to the Johnson-scheme eigenspaces.  For complete-graph exclusion, the Bernoulli--Laplace analysis of Diaconis and
Shahshahani \cite{DiaconisShahshahani} provides a classical spectral model for
fixed-population mixing.  These structures underlie, at a different level of
resolution, the Hoeffding/harmonic decomposition and the complete-graph
all-degree gap used in our \cref{prof:sec:chaos-interface,prof:sec:chaos-weight}.

Finally, the all-degree closure uses negative dependence in an essential way.
Borcea, Br\"and\'en, and Liggett \cite{BBL09} introduced Strong--Rayleigh
measures and proved that the property is preserved by symmetric exclusion;
\cite[Proposition~5.1]{BBL09} is the external preservation theorem invoked in
\cref{prof:sec:strong-rayleigh}.  Negative dependence has also been used to
obtain quantitative mixing results: Salez \cite{SalezReservoirs} used it to
reduce the mixing analysis of exclusion with reservoirs to single-site
marginals, while Hermon and Salez \cite{HermonSalezSR} derived modified
log-Sobolev and mixing-time bounds from the stochastic covering property, a
condition weaker than Strong--Rayleigh.  
In this paper, we combine preservation
of the Strong--Rayleigh property for the evolving conservative exclusion law
with controlled one- and two-particle chaos layers to produce an
all-degree moment envelope and, ultimately, the $L^2$ warm start.

\section{One-particle profile and the calibrated canonical tilt}
\label{prof:sec:calibration}

\subsection{One-particle spectral and cutoff window estimates}

\begin{lemma}[Separation above the one-particle gap and spectral-sum bounds]
\label{prof:lem:one-particle-spectral-sums}
There are constants
\[
 c_*:=\frac{16/\pi^2-1}{3} \in \left(0, \frac{4}{\pi^2}\right),
 \qquad C_D<\infty,
\]
such that the following statements hold for every $N\ge 3$.
For each nonzero $p\in\T_N^D$, with representative
$\widetilde p\in\mathcal Q_N$ from
\eqref{prof:eq:momentum-representatives},
\begin{align}
 \frac4{\pi^2}\gamma_N|\widetilde p|^2
 &\le\lambda_N(p)
 \le\frac{\pi^2}{4}\gamma_N|\widetilde p|^2,
 \label{prof:eq:dispersion-comparison}\\
 \lambda_N(p)
 &\ge\gamma_N\bigl[1+c_*(|\widetilde p|^2-1)\bigr].
 \label{prof:eq:first-shell-separation}
\end{align}
Moreover, uniformly for $a\ge c_*$,
\begin{equation}
 \sum_{\ell\in\mathbb Z^D\setminus\{0\}}
 e^{-a(|\ell|^2-1)}
 +
 \sum_{\ell\in\mathbb Z^D\setminus\{0\}}
 |\ell|^2e^{-a(|\ell|^2-1)}
 \le C_D.
 \label{prof:eq:lattice-gaussian-sums}
\end{equation}
Consequently, for every $t\ge\gamma_N^{-1}$,
\begin{align}
 \sum_{p\ne0}e^{-2t\lambda_N(p)}
 &\le C_De^{-2\gamma_Nt},
 \label{prof:eq:long-heat-trace}\\
 \sum_{p\ne0}\lambda_N(p)e^{-2t\lambda_N(p)}
 &\le C_D\gamma_Ne^{-2\gamma_Nt}.
 \label{prof:eq:long-energy-trace}
\end{align}
\end{lemma}

\begin{proof}
Dividing the exact sine representation
\eqref{prof:eq:torus-sine-spectrum} by
$\gamma_N=4\sin^2(\pi/N)$ gives
\[
 \frac{\lambda_N(p)}{\gamma_N}
 =\sum_{j=1}^D
 \frac{\sin^2(\pi\widetilde p_j/N)}{\sin^2(\pi/N)}.
\]
For every integer $\ell$ with $|\ell|\le N/2$,
\[
 \frac{4\ell^2}{\pi^2}
 \le
 \frac{\sin^2(\pi \ell/N)}{\sin^2(\pi/N)}
 \le
 \frac{\pi^2\ell^2}{4}.
\]
Indeed, the lower bound follows from
$\sin u\ge2u/\pi$ for $0\le u\le\pi/2$ and
$\sin(\pi/N)\le\pi/N$, while the upper bound follows from
$|\sin u|\le |u|$ and $\sin(\pi/N)\ge2/N$.
Summing over the coordinates proves
\eqref{prof:eq:dispersion-comparison}.

For $\ell=\pm1$, the corresponding coordinate ratio equals one.  For
$|\ell|\ge2$,
\[
 \frac{\sin^2(\pi \ell/N)}{\sin^2(\pi/N)}
 \ge\frac{4\ell^2}{\pi^2}
 \ge1+c_*(\ell^2-1),
\]
because
\[
 \frac{4\ell^2}{\pi^2}-1-c_*(\ell^2-1)
 =\left(\frac4{\pi^2}-c_*\right)(\ell^2-4)\ge0.
\]
Let
$\operatorname{nz}(p):=\#\{j:\widetilde p_j\ne0\}$.
Adding the coordinatewise inequalities gives
\[
 \frac{\lambda_N(p)}{\gamma_N}
 \ge
 \operatorname{nz}(p)
 +c_*\bigl(|\widetilde p|^2-\operatorname{nz}(p)\bigr).
\]
The right side equals
\[
 1+c_*(|\widetilde p|^2-1)
 +(\operatorname{nz}(p)-1)(1-c_*).
\]
Since $0<c_*<1$, and $p\ne0$ implies
$\operatorname{nz}(p)\ge1$, this proves
\eqref{prof:eq:first-shell-separation}.

Turning to the left-hand side of  \eqref{prof:eq:lattice-gaussian-sums}, observe that the summands from $|\ell|=1$ produce a finite number,
while for $|\ell|\ge 2$ the summands are bounded by their values at $a=c_*$,
and the resulting Gaussian series converge.  This proves
\eqref{prof:eq:lattice-gaussian-sums}.

Finally, if $t\ge\gamma_N^{-1}$, then
$2c_*\gamma_Nt\ge2c_*$.
Extending the finite momentum sum from
$\mathcal Q_N$ to $\mathbb Z^D$, then using
\eqref{prof:eq:first-shell-separation} and \eqref{prof:eq:lattice-gaussian-sums}, yields
\[
 \sum_{p\ne0}e^{-2t\lambda_N(p)}
 \le
 e^{-2\gamma_Nt}
 \sum_{\ell\in\mathbb Z^D\setminus\{0\}}
 e^{-2c_*\gamma_Nt(|\ell|^2-1)}
 \le C_De^{-2\gamma_Nt}.
\]
For the energy trace, combine
\eqref{prof:eq:dispersion-comparison} with
\eqref{prof:eq:first-shell-separation}, and then use \eqref{prof:eq:lattice-gaussian-sums} to obtain
\[
 \sum_{p\ne0}\lambda_N(p)e^{-2t\lambda_N(p)}
 \le
 C_D\gamma_Ne^{-2\gamma_Nt}
 \sum_{\ell\in\mathbb Z^D\setminus\{0\}}
 |\ell|^2e^{-2c_*\gamma_Nt(|\ell|^2-1)}
 \le C_D\gamma_Ne^{-2\gamma_Nt}.
\qedhere
\]
\end{proof}

To state the spatial smoothing estimates below, define the averaged one-particle Dirichlet form and the nearest-neighbor oscillation by
\begin{equation}
 \mathcal E_N^{\RW}(f)
 :=\langle f,L_N^{\RW}f\rangle_{\mathrm{av}}
 =\frac1{2n_N}\sum_{x\sim y}|f(x)-f(y)|^2,
 \qquad
 \omega_N(f):=\max_{x\sim y}|f(x)-f(y)|,
 \label{prof:eq:one-particle-energy-oscillation}
\end{equation}
where the edge sum is over ordered neighboring pairs.

Our next result collects both long-time smoothing estimates and their cutoff window consequences for the profile $m_{N,t}$ from \eqref{prof:eq:profile}.

\begin{lemma}[One-particle smoothing and cutoff window bounds]
\label{prof:lem:one-particle-bounds}
Uniformly over all $S_N\subset V_N$ with $|S_N|=k_N$,
\begin{equation}
 \omega_N(m_{N,t})^2+\mathcal E_N^{\RW}(m_{N,t})
 \le C_D\gamma_Ne^{-2\gamma_Nt},
 \qquad t\ge\gamma_N^{-1}.
 \label{prof:eq:long-one-particle}
\end{equation}
Moreover, whenever $t_N(s)\ge\gamma_N^{-1}$,
\begin{equation}
 \|m_{N,t_N(s)}\|_{2,\mathrm{cnt}}^2
 +n_N\|m_{N,t_N(s)}\|_\infty^2
 \le C_De^{-s}.
 \label{prof:eq:profile-one-particle-explicit}
\end{equation}
Finally, for every nonempty compact $J\subset\mathbb R$ there is $C_J<\infty$ such that
\begin{equation}
 \sup_{s\in J}\left\{
 \|m_{N,t_N(s)}\|_{2,\mathrm{cnt}}
 + \sqrt{n_N}\|m_{N,t_N(s)}\|_\infty
 +\gamma_N^{-1}\|L_N^{\RW}m_{N,t_N(s)}\|_{2,\mathrm{cnt}}\right\}
 \le C_J.
 \label{prof:eq:one-particle-bounds}
\end{equation}
\end{lemma}

\begin{proof}
Recall the orthonormal Fourier characters from
\eqref{prof:eq:probability-fourier-basis} and the source coefficients from
\eqref{prof:eq:source-fourier-coefficients}.  Since $F_N$ has mean zero, its Fourier decomposition sums over all $p\neq 0$, in which case $\lambda_N(p)\ge \gamma_N$. By Parseval,
\begin{equation}
 \|m_{N,t}\|_{2,\mathrm{cnt}}^2
 =\sum_{p\ne0}e^{-2t\lambda_N(p)}|\widehat F_N(p)|^2
 \le e^{-2\gamma_Nt}\|F_N\|_{2,\mathrm{cnt}}^2
 \le n_Ne^{-2\gamma_Nt}.
 \label{prof:eq:one-particle-parseval-decay}
\end{equation}
At \(t=t_N(s)\), the final quantity equals \(e^{-s}\).

We first record the long-time estimates in \eqref{prof:eq:long-one-particle}.  Since
$\|F_N\|_{2,\mathrm{cnt}}^2=n_N\rho_N(1-\rho_N)\le n_N/4$, Parseval gives
\[
 \mathcal E_N^{\RW}(m_{N,t})
 =\frac1{n_N}\sum_{p\ne0}\lambda_N(p)e^{-2t\lambda_N(p)}
   |\widehat F_N(p)|^2.
\]
When $t\ge\gamma_N^{-1}$, the function
$\lambda\mapsto\lambda e^{-2t\lambda}$ is decreasing on
$[\gamma_N,\infty)$, and therefore
\[
 \mathcal E_N^{\RW}(m_{N,t})
 \le \frac14\gamma_Ne^{-2\gamma_Nt}.
\]
If $x\sim y$, then the character formula and
\eqref{prof:eq:torus-sine-spectrum} give
$|e_p(x)-e_p(y)|^2\le\lambda_N(p)$.  Cauchy--Schwarz in the counting-normalized Fourier expansion thus yields
\[
 |m_{N,t}(x)-m_{N,t}(y)|^2
 \le \frac{\|F_N\|_{2,\mathrm{cnt}}^2}{n_N}
       \sum_{p\ne0}|e_p(x)-e_p(y)|^2e^{-2t\lambda_N(p)}
 \le \frac14\sum_{p\ne0}\lambda_N(p)e^{-2t\lambda_N(p)}.
\]
The long-energy trace \eqref{prof:eq:long-energy-trace} proves the oscillation bound and hence \eqref{prof:eq:long-one-particle}.

We next control the pointwise norm on the profile scale.  Suppose first that $t_N(s)\ge\gamma_N^{-1}$.  The separation estimate
\eqref{prof:eq:first-shell-separation} extracts the exact factor
$e^{-2\gamma_Nt_N(s)}$ associated with $\mathscr E_N^{(1)}$, while the
lattice sum estimate \eqref{prof:eq:lattice-gaussian-sums} controls all
remaining momenta.  The condition $t_N(s)\ge\gamma_N^{-1}$ gives
$2c_*\gamma_Nt_N(s)=c_*(\log n_N+s)\ge2c_*$.  Hence
\eqref{prof:eq:first-shell-separation} and \eqref{prof:eq:lattice-gaussian-sums} give
\begin{align*}
 \sum_{p\ne0}e^{-2t_N(s)\lambda_N(p)}
 &\le e^{-2\gamma_Nt_N(s)}
 \sum_{\ell\in\mathbb Z^D\setminus\{0\}}
 e^{-2c_*\gamma_Nt_N(s)(|\ell|^2-1)}
 \le C_De^{-2\gamma_Nt_N(s)}.
\end{align*}
Cauchy--Schwarz in the Fourier expansion, followed by
\(\|F_N\|_{2,\mathrm{cnt}}\le\sqrt{n_N}\), therefore gives
\[
 \|m_{N,t_N(s)}\|_\infty
 \le C_De^{-\gamma_Nt_N(s)}
 = C_D n_N^{-1/2} e^{-s/2}.
\]
Together with the cutoff-time specialization of
\eqref{prof:eq:one-particle-parseval-decay}, namely
$\|m_{N,t_N(s)}\|_{2,\mathrm{cnt}}^2\le e^{-s}$, this proves
\eqref{prof:eq:profile-one-particle-explicit}.
For $s$ in a fixed compact set $J$, the condition
$t_N(s)\ge\gamma_N^{-1}$ holds uniformly for all sufficiently large $N$;
after enlarging $C_J$, the first two terms in
\eqref{prof:eq:one-particle-bounds} also cover the finitely many remaining values of $N$.

Finally put \(t=t_N(s)\).  Uniformly for \(s\in J\),
\(t\gamma_N=(\log n_N+s)/2\to\infty\).  Hence for all sufficiently large
\(N\),
\[
 \sup_{\lambda\ge\gamma_N}
 \frac{\lambda^2e^{-2t\lambda}}
      {\gamma_N^2e^{-2t\gamma_N}}
 =\sup_{y\ge1}y^2e^{-2t\gamma_N(y-1)}
 \le C_J.
\]
For the finitely many remaining \(N\), the same bound holds after enlarging
\(C_J\), because the nonzero spectrum of \(L_N^{\RW}\) is bounded, and the
ratio is continuous in \(s\in J\).  Thus for every \(N\),
\begin{equation}
 \sup_{\lambda\in\operatorname{Spec}(L_N^{\RW})\setminus\{0\}}
 \frac{\lambda^2e^{-2t_N(s)\lambda}}
      {\gamma_N^2e^{-2t_N(s)\gamma_N}}
 \le C_J,
 \qquad s\in J.
 \label{prof:eq:cutoff-spectral-multiplier}
\end{equation}
Substituting \eqref{prof:eq:cutoff-spectral-multiplier} into the Fourier
expansion and using \(\|F_N\|_{2,\mathrm{cnt}}^2\le n_N\) gives
\[
 \|L_N^{\RW}m_{N,t_N(s)}\|_{2,\mathrm{cnt}}^2
 \le C_J\gamma_N^2n_Ne^{-2\gamma_Nt_N(s)}
 =C_J \gamma_N^2 e^{-s}.
\]
Dividing by \(\gamma_N^2\) gives the bound on the remaining term in
\eqref{prof:eq:one-particle-bounds}.
\end{proof}

\subsection{Canonical response and calibration}
\label{prof:sec:canonical-response-calibration}

For a field \(v\in\mathbb R_0^{V_N}\), define the centered linear
statistic
\begin{equation}
 X_{N,v}(\eta)
 :=\sum_{x\in V_N}v(x)(\eta_x-\rho_N).
 \label{prof:eq:centered-linear-statistic}
\end{equation}
Because \(v\) has zero spatial mean, \(X_{N,v}(\eta)=\sum_xv(x)\eta_x\).
Hence the calibrated canonical tilt defined in \eqref{prof:eq:tilt} can be
written as
\begin{align}
 g_{N,v}
 =\frac{e^{X_{N,v}}}{\mathbb E_{\pi_N}[e^{X_{N,v}}]}
 =\exp\{X_{N,v}-\Lambda_N(v)\},
 \label{prof:eq:gNv}
\end{align}
where
\begin{equation}
 \Lambda_N(v):=\log\mathbb E_{\pi_N}[e^{X_{N,v}}].
 \label{prof:eq:canonical-coefficient}
\end{equation}
Recall from \eqref{prof:eq:tilt} that \(\nu_{N,v}=g_{N,v}\pi_N\).

For \(u\in\mathbb R_0^{V_N}\) and random variables \(Y_1,\ldots,Y_j\), define
their joint cumulant under \(\nu_{N,u}\) by
\begin{equation}
 \operatorname{Cum}_u(Y_1,\ldots,Y_j)
 :=\left.
 \partial_{t_1}\cdots\partial_{t_j}
 \log\mathbb E_{\nu_{N,u}}\left[
 \exp\left\{\sum_{\ell=1}^jt_\ell Y_\ell\right\}
 \right]\right|_{t_1=\cdots=t_j=0}.
 \label{prof:eq:canonical-cumulant-definition}
\end{equation}
For \(j\ge2\), adding a constant to any one argument leaves the cumulant
unchanged.  Thus cumulants of order at least two may use \(\eta_x\) or
\(\eta_x-\rho_N\) interchangeably.

\localheading{Partition chain rule and moment--cumulant identity.}
For \(r\ge1\), let \(\operatorname{Part}([r])\) denote the set of set partitions of
\([r]=\{1,\ldots,r\}\).  If \(\partial_1,\ldots,\partial_r\) are commuting first-order derivative
operators, \(F\) is a smooth scalar function, and \(G\) is a smooth scalar
field, write \(\partial_B:=\prod_{i\in B}\partial_i\).  Then
\begin{equation}
 \partial_1\cdots \partial_r F(G)
 =\sum_{\mathfrak p\in\operatorname{Part}([r])}
   F^{(|\mathfrak p|)}(G)
   \prod_{B\in\mathfrak p}\partial_BG.
 \label{prof:eq:partition-chain-rule}
\end{equation}
This is the partition form of the multivariate Fa\`a di Bruno identity.
Indeed, the identity is immediate for \(r=1\), as it reproduces the usual chain rule.
If it holds for \(r\), then
applying \(\partial_{r+1}\) to a term indexed by \(\mathfrak p\) either
differentiates \(F^{(|\mathfrak p|)}(G)\), which creates the singleton block
\(\{r+1\}\), or differentiates one factor \(\partial_BG\), which adjoins \(r+1\)
to that block.  These alternatives produce every partition of \([r+1]\)
exactly once.

Apply \eqref{prof:eq:partition-chain-rule} with \(\partial_i=\partial_{t_i}\),
\(F(s)=e^s\), and
\[
 G(t_1,\ldots,t_r)
 :=\log\mathbb E_{\nu_{N,u}}\left[
    \exp\left\{\sum_{i=1}^rt_iY_i\right\}\right].
\]
Since \(G(0)=0\) and
\(\partial_BG(0)=\operatorname{Cum}_u((Y_i)_{i\in B})\), evaluation at the origin
gives the \emph{moment--cumulant identity}
\begin{equation}
 \mathbb E_{\nu_{N,u}}\left[\prod_{i=1}^rY_i\right]
 =\sum_{\mathfrak p\in\operatorname{Part}([r])}
   \prod_{B\in\mathfrak p}
   \operatorname{Cum}_u((Y_i)_{i\in B}).
 \label{prof:eq:moment-cumulant-identity}
\end{equation}

\localheading{Cumulant and response.}
We write \(\mathrm d^j\Lambda_N(u)\) for the \(j\)th Fr\'echet derivative,
viewed as a symmetric \(j\)-linear form on \((\mathbb R_0^{V_N})^j\).
For field directions \(a_1,\ldots,a_j\in\mathbb R_0^{V_N}\), linearity of
\(u\mapsto X_{N,u}\) gives
\[
 \Lambda_N\left(u+\sum_{\ell=1}^jt_\ell a_\ell\right)-\Lambda_N(u)
 =\log\mathbb E_{\nu_{N,u}}\left[
 \exp\left\{\sum_{\ell=1}^jt_\ell X_{N,a_\ell}\right\}\right].
\]
Taking the mixed derivative at \(t_1=\cdots=t_j=0\) yields
\begin{equation}
 \mathrm d^j\Lambda_N(u)[a_1,\ldots,a_j]
 =\operatorname{Cum}_u(X_{N,a_1},\ldots,X_{N,a_j}).
 \label{prof:eq:Frechet-cumulant-identity}
\end{equation}
In particular,
\[
 \mathrm d\Lambda_N(u)[a]=\mathbb E_{\nu_{N,u}}[X_{N,a}],
 \qquad
 \mathrm d^2\Lambda_N(u)[a,b]
 =\operatorname{Cov}_{\nu_{N,u}}(X_{N,a},X_{N,b}).
\]

The next lemma collects all finite-population response estimates used in
\cref{prof:sec:calibration,prof:sec:chaos-interface}.
Its proof is given in \cref{prof:app:canonical-response-consequences}.  The
sitewise estimate \eqref{prof:eq:canonical-distinct-cumulant-general} comes
from the scalar deleted-sum saddle and connected-hypergraph expansion in
\cref{prof:app:conditional-bernoulli-cumulants,prof:app:canonical-cumulants}.
The multilinear response \eqref{prof:eq:canonical-multilinear-response} then
follows from \eqref{prof:eq:Frechet-cumulant-identity} by sorting coordinate
tuples by their coincidence partition, and the covariance response
\eqref{prof:eq:canonical-covariance-response} follows from the exact
covariance at \(u=0\) and the fundamental theorem of calculus.

\begin{lemma}[Canonical cumulant and response package]
\label{prof:lem:canonical-response}
There is \(c>0\), depending only on \(\rho_0\), such that the following
hold whenever \(u\in\mathbb R_0^{V_N}\) and
\(\|u\|_\infty\le c\).

\emph{(Sitewise cumulants.)}
For every fixed \(j\ge1\), if \(x_1,\ldots,x_j\) involve \(d\) distinct
sites, then
\begin{equation}
 \bigl|\operatorname{Cum}_u(\eta_{x_1},\ldots,\eta_{x_j})\bigr|
 \le C_{j,\rho_0}n_N^{1-d}.
 \label{prof:eq:canonical-distinct-cumulant-general}
\end{equation}

\emph{(Multilinear response.)}
For every \(j\ge3\) and all field directions
\(a_1,\ldots,a_j\in\mathbb R_0^{V_N}\),
\begin{equation}
 |\mathrm d^j\Lambda_N(u)[a_1,\ldots,a_j]|
 \le C_{j,\rho_0}
 \min_{1\le i\le j}
 \left\{
   \|a_i\|_\infty
   \prod_{\ell\ne i}\|a_\ell\|_{2,\mathrm{cnt}}
 \right\}.
 \label{prof:eq:canonical-multilinear-response}
\end{equation}

\emph{(Covariance response.)}
As operators on \(\mathbb R_0^{V_N}\),
\begin{equation}
 \|\mathrm d^2\Lambda_N(u)-\beta_N I\|_{2\to2}
 +\|\mathrm d^2\Lambda_N(u)-\beta_N I\|_{\infty\to\infty}
 \le C_{\rho_0}\|u\|_\infty.
 \label{prof:eq:canonical-covariance-response}
\end{equation}
All bounds are uniform in \(N\), in the density window, and in the
locations and coincidence pattern of the sites.
\end{lemma}

Define the centered canonical marginal map
\[
 \Marg_N:\mathbb R_0^{V_N}\to\mathbb R_0^{V_N},
 \qquad
 \Marg_N(v):=\mathbb E_{\nu_{N,v}}[\eta-\rho_N].
\]
By \eqref{prof:eq:Frechet-cumulant-identity}, its Fr\'echet derivative is the
covariance operator \(\mathrm d\Marg_N(v)=\mathrm d^2\Lambda_N(v)\) on
\(\mathbb R_0^{V_N}\), under the counting inner product.  The next lemma
quantifies the resulting local inverse and constructs the calibration field
announced in \cref{prof:def:tilt}.

\begin{lemma}[Local canonical calibration]
\label{prof:lem:calibration}
There exist constants \(\delta_{\rho_0},c_{\rho_0}>0\) and
\(C_{\rho_0}<\infty\) such that, uniformly under
\eqref{prof:eq:density-window}, every
\(m\in\mathbb R_0^{V_N}\) with \(\|m\|_\infty\le\delta_{\rho_0}\) has a
unique solution \(v=v_N(m)\in\mathbb R_0^{V_N}\) with
\(\|v\|_\infty\le c_{\rho_0}\) to
\(\Marg_N(v)=m\).  Moreover,
\begin{equation}
 \|v-\beta_N^{-1}m\|_{2,\mathrm{cnt}}
 \le C_{\rho_0}\|m\|_\infty\|m\|_{2,\mathrm{cnt}},
 \qquad
 \|v\|_\infty\le C_{\rho_0}\|m\|_\infty.
 \label{prof:eq:calibration-estimate}
\end{equation}
\end{lemma}

\begin{proof}
The first chaos covariance identity \eqref{prof:eq:first-chaos-isometry} gives
$\mathrm d\Marg_N(0)=\beta_NI$.  The covariance response estimate
\eqref{prof:eq:canonical-covariance-response} gives, for $v$ in its stated
neighborhood,
\begin{equation}
 \|\mathrm d\Marg_N(v)-\beta_NI\|_{2\to2}
 +\|\mathrm d\Marg_N(v)-\beta_NI\|_{\infty\to\infty}
 \le C_{\rho_0}\|v\|_\infty.
 \label{prof:eq:response-linearization}
\end{equation}
Introduce the derivative remainder and the corresponding nonlinear
remainder of the marginal map,
\[
 \mathsf{Err}_N(v):=\mathrm d\Marg_N(v)-\beta_NI,
 \qquad
 \mathsf{Rem}^{\mathrm{resp}}_N(v):=\Marg_N(v)-\beta_Nv.
\]
Since \(\Marg_N(0)=0\), the fundamental theorem of calculus gives
\begin{align}
 \mathsf{Rem}^{\mathrm{resp}}_N(v)
 &=\int_0^1 \mathsf{Err}_N(\xi v)v\,\dd\xi,
 \label{prof:eq:calibration-remainder-origin}\\
 \mathsf{Rem}^{\mathrm{resp}}_N(v)-\mathsf{Rem}^{\mathrm{resp}}_N(w)
 &=\int_0^1 \mathsf{Err}_N\bigl(w+\xi(v-w)\bigr)(v-w)\,\dd\xi.
 \label{prof:eq:calibration-remainder-segment}
\end{align}

Fix once and for all a contraction parameter \(\varepsilon\in(0,1)\); for definiteness one may
take \(\varepsilon=\frac12\).  The density assumption
\eqref{prof:eq:density-window} yields a constant
\(\beta_*=\beta_*(\rho_0)>0\) such that \(\beta_N\ge\beta_*\) uniformly.
Choose \(c_{\rho_0}\) no larger than the response radius in
\cref{prof:lem:canonical-response} and small enough that
$
 C_{\rho_0}c_{\rho_0}\le \varepsilon\beta_*$.
Then whenever \(\|u\|_\infty\le c_{\rho_0}\),
\begin{equation}
 \|\mathsf{Err}_N(u)\|_{2\to2}
 +\|\mathsf{Err}_N(u)\|_{\infty\to\infty}
 \le \varepsilon\beta_N.
 \label{prof:eq:calibration-epsilon-bound}
\end{equation}
Thus the same parameter \(\varepsilon\) controls both the contraction
constant and the later \(L^2\) absorption.  Since \(\varepsilon\) is fixed
universally, all constants below may still be indexed only by \(\rho_0\).

Choose \(\delta_{\rho_0}>0\) so that
$
 \delta_{\rho_0}\le (1-\varepsilon)\beta_*c_{\rho_0}$.
Fix \(m\in\mathbb R_0^{V_N}\) with
\(\|m\|_\infty\le\delta_{\rho_0}\), and set
\[
 r_m:=\beta_N^{-1}\|m\|_\infty,
 \qquad
 B_m:=\left\{v\in\mathbb R_0^{V_N}:
 \|v\|_\infty\le\frac{r_m}{1-\varepsilon}\right\}.
\]
The ball \(B_m\) is convex.
Moreover, since
$
 \frac{r_m}{1-\varepsilon}
 \le \frac{\beta_*^{-1}\delta_{\rho_0}}{1-\varepsilon}
 \le c_{\rho_0}$,
\(B_m\) lies inside the response neighborhood on which
\eqref{prof:eq:calibration-epsilon-bound} holds.

Define
\[
 \mathsf{Fix}_m(v)
 :=\beta_N^{-1}m-\beta_N^{-1}\mathsf{Rem}^{\mathrm{resp}}_N(v)
 =\beta_N^{-1}m
 -\beta_N^{-1}(\Marg_N(v)-\beta_Nv).
\]
Because \(m\), \(v\), and \(\Marg_N(v)\) have zero spatial mean,
\(\mathsf{Fix}_m(v)\in\mathbb R_0^{V_N}\).  For \(v,w\in B_m\), convexity
of \(B_m\), \eqref{prof:eq:calibration-remainder-origin}--%
\eqref{prof:eq:calibration-remainder-segment}, and
\eqref{prof:eq:calibration-epsilon-bound} give
\[
 \|\mathsf{Rem}^{\mathrm{resp}}_N(v)\|_\infty
 \le\varepsilon\beta_N\|v\|_\infty,
 \qquad
 \|\mathsf{Fix}_m(v)-\mathsf{Fix}_m(w)\|_\infty
 \le\varepsilon\|v-w\|_\infty.
\]
Consequently,
\[
 \|\mathsf{Fix}_m(v)\|_\infty
 \le r_m+\varepsilon\|v\|_\infty
 \le\frac{r_m}{1-\varepsilon},
\]
so \(\mathsf{Fix}_m\) maps \(B_m\) into itself and is an
\(\varepsilon\)-contraction there.  Since \(B_m\) is complete, Banach's
fixed-point theorem gives a unique \(v\in B_m\) satisfying
\(\mathsf{Fix}_m(v)=v\), equivalently \(\Marg_N(v)=m\).  In particular,
\begin{equation}
 \|v\|_\infty
 \le\frac{\beta_N^{-1}}{1-\varepsilon}\|m\|_\infty
 \le C_{\rho_0}\|m\|_\infty.
 \label{prof:eq:calibration-linfty}
\end{equation}

This fixed point is the only solution in the full response neighborhood
\(\{\|w\|_\infty\le c_{\rho_0}\}\).  Indeed, if
\(\Marg_N(w)=m\) and \(\|w\|_\infty\le c_{\rho_0}\), then
\eqref{prof:eq:calibration-remainder-origin} and
\eqref{prof:eq:calibration-epsilon-bound} yield
$
 \|\mathsf{Rem}^{\mathrm{resp}}_N(w)\|_\infty\le\varepsilon\beta_N\|w\|_\infty
 $.
Since \(m=\beta_Nw+\mathsf{Rem}^{\mathrm{resp}}_N(w)\),
$
 (1-\varepsilon)\beta_N\|w\|_\infty
 \le\|m\|_\infty
$,
and hence \(\|w\|_\infty\le r_m/(1-\varepsilon)\).  Thus \(w\in B_m\),
where fixed-point uniqueness applies.

It remains to prove the \(L^2\) estimate in \eqref{prof:eq:calibration-estimate}.
From
\eqref{prof:eq:calibration-remainder-origin} and the \(2\to2\) part of
\eqref{prof:eq:response-linearization},
\begin{equation}
 \|\mathsf{Rem}^{\mathrm{resp}}_N(v)\|_{2,\mathrm{cnt}}
 \le C_{\rho_0}\|v\|_\infty\|v\|_{2,\mathrm{cnt}}.
 \label{prof:eq:calibration-remainder-L2}
\end{equation}
Because \(v\in B_m\subseteq\{\|u\|_\infty\le c_{\rho_0}\}\), the sharper
bound \eqref{prof:eq:calibration-epsilon-bound} also gives
\[
 \|m-\beta_Nv\|_{2,\mathrm{cnt}}
 =\|\mathsf{Rem}^{\mathrm{resp}}_N(v)\|_{2,\mathrm{cnt}}
 \le\varepsilon\beta_N\|v\|_{2,\mathrm{cnt}}.
\]
Therefore
\begin{equation}
 \|v\|_{2,\mathrm{cnt}}
 \le\frac{1}{(1-\varepsilon)\beta_N}
       \|m\|_{2,\mathrm{cnt}}
 \le C_{\rho_0}\|m\|_{2,\mathrm{cnt}}.
 \label{prof:eq:calibration-v-L2}
\end{equation}
Finally, \(m=\beta_Nv+\mathsf{Rem}^{\mathrm{resp}}_N(v)\).  Applying
\eqref{prof:eq:calibration-remainder-L2} together with
\eqref{prof:eq:calibration-v-L2}, \eqref{prof:eq:calibration-linfty}, and the
density window lower bound on \(\beta_N\) yields
\[
 \|v-\beta_N^{-1}m\|_{2,\mathrm{cnt}}
 =\beta_N^{-1}\|\mathsf{Rem}^{\mathrm{resp}}_N(v)\|_{2,\mathrm{cnt}}
 \le C_{\rho_0}\|m\|_\infty\|m\|_{2,\mathrm{cnt}}. \qedhere
\]
\end{proof}

\subsection{Gaussianization of the canonical tilt}

We now state the key Gaussian asymptotic that produces the total variation profile.
Below ``$\overset{d}{\rightarrow}$'' denotes convergence in distribution, and $\mathcal{N}(0,\varsigma^2)$ denotes the centered normal distribution with variance $\varsigma^2$.

\begin{lemma}[Finite-population Gaussianization]
\label{prof:lem:gaussianization}
Let $v_N\in\mathbb R_0^{V_N}$ satisfy $\|v_N\|_\infty\to0$ and
$\|v_N\|_{2,\mathrm{cnt}}=O(1)$ as $N\to\infty$.
With \(X_{N,v}\) as in \eqref{prof:eq:centered-linear-statistic}, put
$X_N:=X_{N,v_N}$ under $\pi_N$, and
let $\varsigma_N\ge0$ be defined by $\varsigma_N^2=\beta_N\|v_N\|_{2,\mathrm{cnt}}^2$.  
Then uniformly for $z$ in each compact subset of $\mathbb R$,
\begin{equation}
 \log\mathbb E_{\pi_N}\left[e^{zX_N}\right]
 =\frac{z^2\varsigma_N^2}{2}+o(1).
 \label{prof:eq:mgf-gaussian}
\end{equation}
If $\varsigma_N\to\varsigma$ for some $\varsigma\ge0$, then $X_N\overset{d}{\rightarrow} \mathcal N(0,\varsigma^2)$ and
\begin{equation}
 \frac12\mathbb E_{\pi_N}|g_{N,v_N}-1|
 \longrightarrow
 2\PhiGauss\!\left(\frac{\varsigma}{2}\right)-1.
 \label{prof:eq:tilt-TV}
\end{equation}
More generally, let $\mathcal I_N$ be nonempty index sets, and let
$v_{N,\alpha}\in\mathbb R_0^{V_N}$ satisfy
\[
 \sup_{\alpha\in \mathcal I_N}\|v_{N,\alpha}\|_\infty\longrightarrow0,
 \qquad
 \sup_N\sup_{\alpha\in \mathcal I_N}\|v_{N,\alpha}\|_{2,\mathrm{cnt}}<\infty.
\]
Writing
\(X_{N,\alpha}:=X_{N,v_{N,\alpha}}\), and
letting \(\varsigma_{N,\alpha}\ge0\) be defined by \(\varsigma_{N,\alpha}^2=\beta_N\|v_{N,\alpha}\|_{2,\mathrm{cnt}}^2\), we have for every fixed
\(Z\ge0\),
\begin{equation}
 \sup_{\alpha\in \mathcal I_N}\sup_{|z|\le Z}
 \left|
 \log\mathbb E_{\pi_N}\left[e^{zX_{N,\alpha}}\right]-\frac{z^2\varsigma_{N,\alpha}^2}2
 \right|\longrightarrow0.
 \label{prof:eq:uniform-mgf-gaussian}
\end{equation}
Moreover,
\begin{equation}
 \sup_{\alpha\in \mathcal I_N}\left|
 \frac12\mathbb E_{\pi_N}|g_{N,v_{N,\alpha}}-1|
 -\left(2\PhiGauss\!\left(\frac{\varsigma_{N,\alpha}}{2}\right)-1\right)
 \right|\longrightarrow0.
 \label{prof:eq:uniform-tilt-TV}
\end{equation}
\end{lemma}

\begin{proof}
Recall that $u\mapsto X_{N,u}$ is linear and, by \eqref{prof:eq:canonical-coefficient}, $\Lambda_N(u)=\log \mathbb{E}_{\pi_N}\left[e^{X_{N,u}}\right]$.
Thus the log moment generating function (MGF) in \eqref{prof:eq:mgf-gaussian} is $\Lambda_N(zv_N)$, and its counterpart in \eqref{prof:eq:uniform-mgf-gaussian} is $\Lambda_N(z v_{N,\alpha})$.

Fix $Z\in (0,\infty)$. By \cref{prof:lem:canonical-response}, uniformly for \(|z|\le Z\),
\[
 \bigl|\mathrm d^3\Lambda_N(zu)[u,u,u]\bigr|
 \le C_{Z,\rho_0}\|u\|_\infty\|u\|_{2,\mathrm{cnt}}^2
\]
whenever \(\|u\|_\infty\) is sufficiently small.  Taylor's formula with
integral remainder, together with
\(\mathrm d\Lambda_N(0)=0\) and
\(\mathrm d^2\Lambda_N(0)[u,u]=\beta_N\|u\|_{2,\mathrm{cnt}}^2\), therefore gives
\[
 \Lambda_N(zu)
 =\frac{z^2}{2}\beta_N\|u\|_{2,\mathrm{cnt}}^2
 +O_{Z,\rho_0}(\|u\|_\infty\|u\|_{2,\mathrm{cnt}}^2).
\]
This proves \eqref{prof:eq:mgf-gaussian} uniformly for \(z\) in each compact subset of \(\mathbb R\).  The same deterministic remainder bound, applied uniformly to the
family \(v_{N,\alpha}\), proves \eqref{prof:eq:uniform-mgf-gaussian}.  If
\(\varsigma_N\to\varsigma\), then the MGFs converge on a
fixed neighborhood of the origin to \(z\mapsto e^{\varsigma^2z^2/2}\), the moment
generating function of \(\mathcal N(0,\varsigma^2)\).  The continuity theorem for
MGFs thus implies that the law of $X_N$ under $\pi_N$
converges in distribution to $\mathcal N(0,\varsigma^2)$.

At \(z=1\), \eqref{prof:eq:mgf-gaussian} also gives
\[
 \Lambda_N(v_N)=\frac{\varsigma_N^2}{2}+o(1)\longrightarrow\frac{\varsigma^2}{2}.
\]
Thus, if \(X\sim\mathcal N(0,\varsigma^2)\), Slutsky's theorem and then the continuous
mapping theorem yield
\begin{align}
 \log g_{N,v_N}=X_N-\Lambda_N(v_N)
 \overset{d}{\longrightarrow} X-\frac{\varsigma^2}{2},
 \qquad
 g_{N,v_N}\overset{d}{\longrightarrow} e^{X-\varsigma^2/2}.
 \label{prof:eq:conv-in-dist}
\end{align}
Choose a fixed \(\vartheta>0\) and put \(p=2+\vartheta\).  By
\eqref{prof:eq:gNv},
\[
 \mathbb E_{\pi_N}\left[g_{N,v_N}^{p}\right]
 =\exp\!\left\{\Lambda_N(pv_N)-p\Lambda_N(v_N)\right\}.
\]
Applying \eqref{prof:eq:mgf-gaussian} at \(z=p\) and \(z=1\) gives
\[
 \Lambda_N(pv_N)-p\Lambda_N(v_N)
 = \left(\frac{p^2 \varsigma_N^2}{2} +o(1)\right) - p\left(\frac{\varsigma_N^2}{2}+o(1)\right)
 =\frac{p^2-p}{2}\,\varsigma_N^2+o(1).
\]
Since \((\varsigma_N^2)_N\) is bounded, it follows that
\[
 \sup_N\mathbb E_{\pi_N}\left[g_{N,v_N}^{2+\vartheta}\right]<\infty.
\]
Hence \((|g_{N,v_N}-1|)_N\) is uniformly integrable, and we can upgrade the convergence in distribution \eqref{prof:eq:conv-in-dist} to convergence in mean:
\[
 \frac12\mathbb E_{\pi_N}|g_{N,v_N}-1|
 \longrightarrow
 \frac12\mathbb E_{\mathcal N(0,\varsigma^2)}|e^{X-\varsigma^2/2}-1|.
\]
For \(\varsigma=0\) the last expectation is zero.  Suppose \(\varsigma>0\), and set
\[
 A:=\{X>\varsigma^2/2\}
   =\{e^{X-\varsigma^2/2}>1\}.
\]
Since \(X\sim\mathcal N(0,\varsigma^2)\),
\(\mathbb E_{\mathcal N(0,\varsigma^2)}[e^{X-\varsigma^2/2}]=1\).  Thus the random variable
\(Y:=e^{X-\varsigma^2/2}-1\) has mean zero, and therefore its positive and
negative parts have the same expectation.  As \(A=\{Y>0\}\),
\begin{equation*}
 \frac12\mathbb E_{\mathcal N(0,\varsigma^2)}|e^{X-\varsigma^2/2}-1|
 =\mathbb E_{\mathcal N(0,\varsigma^2)}\!\left[(e^{X-\varsigma^2/2}-1)\mathbf 1_A\right]
 =
 \mathbb E_{\mathcal N(0,\varsigma^2)}\!\left[e^{X-\varsigma^2/2}\mathbf 1_A\right]
   -\mathbb P_{\mathcal N(0,\varsigma^2)}(A).
\end{equation*}
To simplify the first expectation on the right-hand side, we apply a Gaussian exponential change-of-measure,
\[
 \frac{\mathrm d\mathcal N(\varsigma^2,\varsigma^2)}
      {\mathrm d\mathcal N(0,\varsigma^2)}(x)
 =e^{x-\varsigma^2/2}.
\]
thereby yielding
\[
\mathbb E_{\mathcal N(0,\varsigma^2)}\!\left[e^{X-\varsigma^2/2}\mathbf 1_A\right]
=
\mathbb P_{\mathcal N(\varsigma^2,\varsigma^2)}(A) = \PhiGauss\left(\frac{\varsigma}{2}\right).
\]
On the other hand, $\mathbb P_{\mathcal N(0,\varsigma^2)}(A) = 1- \PhiGauss(\varsigma/2)$.
Conclude then that
\begin{align*}
 \frac12\mathbb E|e^{X-\varsigma^2/2}-1|
 =\PhiGauss\!\left(\frac{\varsigma}{2}\right)
   -\left(1-\PhiGauss\!\left(\frac{\varsigma}{2}\right)\right)
 =2\PhiGauss\!\left(\frac{\varsigma}{2}\right)-1.
\end{align*}
This proves \eqref{prof:eq:tilt-TV}.

It remains to prove the uniform family assertion \eqref{prof:eq:uniform-tilt-TV}.
If \eqref{prof:eq:uniform-tilt-TV} failed, there would be a subsequence $N_j$
and indices $\alpha_j\in\mathcal I_{N_j}$ for which the absolute difference
in \eqref{prof:eq:uniform-tilt-TV} stayed bounded away from zero.  The
numbers $\varsigma_{N_j,\alpha_j}$ are uniformly bounded, so
a further subsequence converges to some $\varsigma\ge0$.  The fields
$v_{N_j,\alpha_j}$ satisfy the hypotheses of the first part of the lemma.  Hence
their tilt distances converge to
$2\PhiGauss(\varsigma/2)-1$, while continuity gives
$2\PhiGauss(\varsigma_{N_j,\alpha_j}/2)-1\longrightarrow
2\PhiGauss(\varsigma/2)-1$.  The selected difference therefore converges to
zero, a contradiction.
\end{proof}

Fix a nonempty compact $J\subset\mathbb R$.  For all sufficiently large $N$,
for an admissible source $S_N\subset V_N$ with $|S_N|=k_N$ and $s\in J$,
abbreviate
\[
 v:=v_{N,t_N(s)}^{S_N},
 \qquad
 m:=m_{N,t_N(s)}^{S_N},
 \qquad
 q:=\mathfrak q_N^{S_N}(s).
\]
All estimates below are uniform over these choices.  Since
$q=\beta_N^{-1}\|m\|_{2,\mathrm{cnt}}^2$, polarization followed by
Cauchy--Schwarz and triangle inequalities gives
\begin{align*}
 \left|\beta_N\|v\|_{2,\mathrm{cnt}}^2-q\right|
 &=\beta_N\left|\left\langle
     v-\beta_N^{-1}m,
     v+\beta_N^{-1}m
   \right\rangle_{\mathrm{cnt}}\right|\\
 &\le \beta_N\|v-\beta_N^{-1}m\|_{2,\mathrm{cnt}}
       \left(\|v\|_{2,\mathrm{cnt}}
             +\beta_N^{-1}\|m\|_{2,\mathrm{cnt}}\right).
\end{align*}
By \cref{prof:lem:calibration} and \eqref{prof:eq:one-particle-bounds},
\[
 \|v-\beta_N^{-1}m\|_{2,\mathrm{cnt}}
 \le C_{\rho_0}\|m\|_\infty\|m\|_{2,\mathrm{cnt}}
 \le C_{J,\rho_0}n_N^{-1/2}.
\]
The calibration estimate \eqref{prof:eq:calibration-estimate}, the one-particle
bound \eqref{prof:eq:one-particle-bounds}, and the density window  \eqref{prof:eq:density-window} also give
\[
 \|v\|_{2,\mathrm{cnt}}+\beta_N^{-1}\|m\|_{2,\mathrm{cnt}}
 \le C_{J,\rho_0}.
\]
Consequently,
\begin{equation}
 \sup_{\substack{S_N\subset V_N,\ |S_N|=k_N\\ s\in J}}
 \left|
 \beta_N\|v_{N,t_N(s)}^{S_N}\|_{2,\mathrm{cnt}}^2
 -\mathfrak q_N^{S_N}(s)
 \right|
 \le C_{J,\rho_0}n_N^{-1/2}.
 \label{prof:eq:calibrated-q}
\end{equation}
Thus the calibrated squared norm and the profile coordinate agree uniformly
on every compact profile window, identifying the variance parameter that
governs the limiting profile.

\section{Finite-population chaos coordinates: dynamic and static interfaces}
\label{prof:sec:chaos-interface}

This section constructs a concrete degree-$r$ coordinate space for the
finite-population Hoeffding decomposition.
Recall the orthogonal Hoeffding decomposition \eqref{prof:eq:Hoeffding}.
Let $P_{N,r}$ be the orthogonal projection onto the degree-$r$ chaos layer $\Hdeg_{N,r}$.
For $f\in L^2(\Omega_{N,k_N},\pi_N)$,
we will identify $P_{N,r}f$ with a coordinate in the ordered degree-$r$ space $\mathcal K_{N,r}^{\ord}$; express that coordinate through the normalized $r$-point moment of $f$; and place the hard-core and independent $r$-particle evolutions in the same coordinate space for later comparison.
The central identification is the normalized harmonic lift
$
 U_{N,r}:\mathcal K_{N,r}^{\ord}\to \Hdeg_{N,r},
$
defined below.
\Cref{prof:lem:harmonic-lift-unitary} proves that this map is unitary.

Throughout the section, and
subsequently whenever these spaces reappear, we use $\mathsf X_{N,j}^{\bullet}$
for ambient $L^2$ spaces, and $\mathcal K_{N,j}^{\bullet}$ for their
distinguished harmonic or top subspaces.
The superscripts ``$\mathrm{set}$'' and ``$\ord$'' record the coordinate realization:
unordered sets or ordered tuples.  The labels ``$\ind$'' and ``$\hc$'', used below,
instead identify the independent and hard-core ambient/dynamical realizations.

\subsection{Finite-population Hilbert spaces and the up--down calculus}

For $0\le j\le n_N$, set
\[
 \Omega_{N,j}^{\mathrm{set}}
 :=\{A\subseteq V_N:|A|=j\},
 \qquad
 \upsilon_{N,j}^{\mathrm{set}}
 :=\operatorname{Unif}(\Omega_{N,j}^{\mathrm{set}}),
\]
and define
\[
 \mathsf X_{N,j}^{\mathrm{set}}
 :=L^2(\Omega_{N,j}^{\mathrm{set}},\upsilon_{N,j}^{\mathrm{set}}).
\]
Thus
\[
 \langle f,g\rangle_{N,j}^{\mathrm{set}}
 =\binom{n_N}{j}^{-1}
  \sum_{A\in\Omega_{N,j}^{\mathrm{set}}}
  f(A)\overline{g(A)}.
\]
For $1\le j\le n_N$, let
\[
 \Omega_{N,j}^{\ord}
 :=\{(x_1,\ldots,x_j)\in V_N^j:x_a\ne x_b\text{ for }a\ne b\},
 \qquad
 \upsilon_{N,j}^{\ord}
 :=\operatorname{Unif}(\Omega_{N,j}^{\ord}),
\]
and set $\Omega_{N,0}^{\ord}:=\{()\}$ and
$\upsilon_{N,0}^{\ord}:=\delta_{()}$.  Let
\[
 \mathsf X_{N,j}^{\ord}
 :=L^2(\Omega_{N,j}^{\ord},\upsilon_{N,j}^{\ord}),
\]
and $\mathsf X_{N,j}^{\ord,\mathrm{sym}}$ be the
subspace of symmetric functions in $\mathsf X_{N,j}^{\ord}$.
For $j\ge1$, let $\Sym([j])$ be the symmetric group on $[j]$.  A
particle-label permutation $\varpi\in\Sym([j])$ acts on ordered tuples by
\[
 (\mathsf P_\varpi\boldsymbol x)_i:=x_{\varpi^{-1}(i)},
 \qquad 1\le i\le j.
\]
Thus $\varpi$ and $\mathsf P_\varpi$ are reserved for permutations of
particle labels and their coordinate action, respectively.
For a site permutation $\sigma\in\Sym(V_N)$ and $0\le q\le n_N$, let
$\mathsf T_{\sigma}^{(q)}$ denote its induced action on level-$q$ functions.
On set functions and ordered-tuple functions, respectively,
\[
 (\mathsf T_{\sigma}^{(q)}f)(A):=f(\sigma A),
 \qquad
 (\mathsf T_{\sigma}^{(q)}f)(x_1,\ldots,x_q)
 :=f(\sigma x_1,\ldots,\sigma x_q).
\]
For $j\geq 1$, we define the unitary map
\[
\iota_{N,j}: \mathsf X_{N,j}^{\mathrm{set}} \rightarrow \mathsf X_{N,j}^{\ord,\mathrm{sym}},
\qquad
 (\iota_{N,j}f)(x_1,\ldots,x_j)
 :=f(\{x_1,\ldots,x_j\}).
\]
If $j=0$, both $\mathsf X_{N,0}^{\mathrm{set}}$ and $\mathsf X_{N,0}^{\ord, \mathrm{sym}}$ consist of constants, so we set $\iota_{N,0}=I$.
We use this identification
without further comment, but retain the words ``set'' and ``ordered'' when
the domain of an operator matters.

\localheading{Normalized up and down maps.}
For $0\le j<n_N$, define
\begin{align}
 (\mathsf{Up}_{N,j\to j+1}f)(B)
 &:=\frac1{j+1}\sum_{x\in B}f(B\setminus\{x\}),
 &&B\in\Omega_{N,j+1}^{\mathrm{set}},
 \notag\\
 (\mathsf{Down}_{N,j+1\to j}g)(A)
 &:=\frac1{n_N-j}\sum_{x\notin A}g(A\cup\{x\}),
 &&A\in\Omega_{N,j}^{\mathrm{set}}.
 \label{prof:eq:up-down-normalization}
\end{align}
A direct counting calculation gives, for every
$f\in\mathsf X_{N,j}^{\mathrm{set}}$ and
$g\in\mathsf X_{N,j+1}^{\mathrm{set}}$,
\[
\langle f, \mathsf{Down}_{N,j+1\to j} g\rangle^{\rm set}_{N,j}
= \langle \mathsf{Up}_{N,j\to j+1} f, g\rangle^{\rm set}_{N,j+1}.
\]
In other words,
$
 \mathsf{Down}_{N,j+1\to j}
 =\mathsf{Up}_{N,j\to j+1}^*
$
as Hilbert space adjoints.
Under the ordered realization, they read
\begin{align}
 (\mathsf{Up}_{N,j\to j+1}g)(x_1,\ldots,x_{j+1})
 &=\frac1{j+1}\sum_{a=1}^{j+1}
   g(x_1,\ldots,\widehat{x_a},\ldots,x_{j+1}),
 \notag\\
 (\mathsf{Down}_{N,j+1\to j}f)(x_1,\ldots,x_j)
 &=\frac1{n_N-j}
   \sum_{z\notin\{x_1,\ldots,x_j\}}
   f(x_1,\ldots,x_j,z).
 \label{prof:eq:ordered-up-down-formulas}
\end{align}

For computation purposes we also consider the unnormalized, or counting measure, operators
\begin{align*}
 (\mathsf{Add}_jf)(B)
 &:=\sum_{x\in B}f(B\setminus\{x\}),
 \notag\\
 (\mathsf{Del}_{j+1}g)(A)
 &:=\sum_{x\notin A}g(A\cup\{x\}).
\end{align*}
Then
\[
 \mathsf{Up}_{N,j\to j+1}=\frac1{j+1}\mathsf{Add}_j,
 \qquad
 \mathsf{Down}_{N,j+1\to j}=\frac1{n_N-j}\mathsf{Del}_{j+1},
\]
$\mathsf{Del}_{j+1}=\mathsf{Add}_j^*$ in counting measure, and direct
counting gives
\begin{equation}
 \mathsf{Del}_{j+1}\mathsf{Add}_j
 -\mathsf{Add}_{j-1}\mathsf{Del}_j
 =(n_N-2j)I.
 \label{prof:eq:add-del-commutator}
\end{equation}
We use the conventions $\mathsf{Add}_{-1}=0$ and $\mathsf{Del}_0=0$.
The normalized maps govern the Hilbert space geometry; the unnormalized
maps will be used only to prove exact coefficient identities.

\localheading{Harmonic kernels and the top projection.}
Let $1\le r \le \ell_N$.
Observe from \eqref{prof:eq:ordered-up-down-formulas} that for $\kappa\in \mathsf X_{N,r}^{\mathrm{set}}$,
\[
\mathsf{Down}_{N,r\to r-1} \kappa =0
~~
\Longleftrightarrow
~~
 \sum_{x\notin C}\kappa(C\cup\{x\})=0
 \quad\text{for every }C\in\Omega_{N,r-1}^{\mathrm{set}}.
\]
When this holds we say that $\kappa$ is degree-$r$ \emph{harmonic}.
We define the \emph{harmonic kernel} by
\[
 \mathcal K_{N,r}^{\mathrm{set}}
 :=\ker\mathsf{Down}_{N,r\to r-1}
 =\ker\mathsf{Del}_r
\subseteq\mathsf X_{N,r}^{\mathrm{set}},
\qquad 1\le r\le\ell_N.
\]
At degree zero we set
$
\mathcal K_{N,0}^{\mathrm{set}}:=\mathsf X_{N,0}^{\mathrm{set}}$.
For $1\le r\le\ell_N$, its ordered realization is
\[
 \mathcal K_{N,r}^{\ord}
 :=\iota_{N,r}\mathcal K_{N,r}^{\mathrm{set}}
 =\ker\mathsf{Down}_{N,r\to r-1}
 \subseteq\mathsf X_{N,r}^{\ord,\mathrm{sym}}.
\]
At degree zero we likewise set
$
\mathcal K_{N,0}^{\ord}:=\mathsf X_{N,0}^{\ord,\mathrm{sym}}$.

For $0\le j< q \leq n_N$, write
$
 \mathsf{Up}_{N,j\to q}
 :=\mathsf{Up}_{N,q-1\to q}\cdots
   \mathsf{Up}_{N,j\to j+1}$, and set
$ \mathsf{Up}_{N,q\to q}:=I$.
This iterated up map is the uniform average over retained $j$-subsets: for
$f\in\mathsf X_{N,j}^{\mathrm{set}}$ and symmetric
$g\in\mathsf X_{N,j}^{\ord}$,
\begin{align}
 (\mathsf{Up}_{N,j\to q}f)(A)
 &=\binom qj^{-1}
   \sum_{\substack{B\subseteq A\\|B|=j}}f(B),
 &&A\in\Omega_{N,q}^{\mathrm{set}},
 \notag\\
 (\mathsf{Up}_{N,j\to q}g)(x_1,\ldots,x_q)
 &=\binom qj^{-1}
   \sum_{\substack{I\subseteq[q]\\|I|=j}}
   g((x_i)_{i\in I}),
 &&(x_1,\ldots,x_q)\in\Omega_{N,q}^{\ord}.
 \label{prof:eq:iterated-up-average}
\end{align}
Indeed, each retained $j$-subset occurs along exactly $(q-j)!$ deletion
chains, while the product of the one-step normalizing factors is $j!/q!$,
so its total coefficient is $\binom qj^{-1}$.

Whenever $0\le j\le q\le\ell_N$, the degree-$j$ Hoeffding subspace at level $q$
is $\mathsf{Up}_{N,j\to q}\mathcal K_{N,j}^{\ord}$.
The following lemma proves two facts: that these subspaces give the
finite-population Hoeffding decomposition, and that the normalized down--up
operator is diagonal on them.

\begin{lemma}[Finite-population Hoeffding decomposition and down--up spectrum]
\label{prof:lem:fd-down-up-spectrum}
For every $0\le q\le\ell_N$,
\begin{equation}
 \mathsf X_{N,q}^{\ord,\mathrm{sym}}
 =\bigoplus_{j=0}^{q}
  \mathsf{Up}_{N,j\to q}\mathcal K_{N,j}^{\ord}.
 \label{prof:eq:finitepopHoeffding}
\end{equation}
At the physical slice level,
\begin{equation}
 \mathsf X_{N,k_N}^{\ord,\mathrm{sym}}
 =\bigoplus_{j=0}^{\ell_N}
  \mathsf{Up}_{N,j\to k_N}\mathcal K_{N,j}^{\ord}.
 \label{prof:eq:finitepopHoeffding-physical}
\end{equation}
Moreover, if $1\le q\le\ell_N$, then on the degree-$j$ summand of
$\mathsf X_{N,q-1}^{\ord,\mathrm{sym}}$, $0\le j\le q-1$, the positive
operator
$
 \mathsf{Down}_{N,q\to q-1}\mathsf{Up}_{N,q-1\to q}
$
acts as the scalar
\begin{equation}
 a_{q,j,N}
 =\frac{q-j}{q}\left(1-\frac{j}{n_N-q+1}\right).
 \label{prof:eq:fd-down-up-eigenvalue}
\end{equation}
All these scalars are strictly positive.  Hence
$\mathsf{Up}_{N,q-1\to q}$ is injective, and
$\mathsf{Down}_{N,q\to q-1}\mathsf{Up}_{N,q-1\to q}$ is invertible on
$\mathsf X_{N,q-1}^{\ord,\mathrm{sym}}$.
\end{lemma}

\begin{proof}
We identify symmetric ordered functions with set functions, and argue by
induction on the ambient level.  The decomposition at level $q=0$ is immediate.
Fix $1\le q\le\ell_N$ and assume that \eqref{prof:eq:finitepopHoeffding}
holds at level $q-1$.

First compute the down--up action on each inherited degree.  If
$h\in\mathcal K_{N,j}^{\mathrm{set}}$, then $\mathsf{Del}_jh=0$, and a
repeated use of the commutator identity
\eqref{prof:eq:add-del-commutator} gives
\[
 \mathsf{Del}_q\mathsf{Add}_{q-1}\cdots\mathsf{Add}_jh
 =(q-j)(n_N-q+1-j)
  \mathsf{Add}_{q-2}\cdots\mathsf{Add}_jh.
\]
After inserting the normalizing factors from
\eqref{prof:eq:up-down-normalization}, this is exactly
\eqref{prof:eq:fd-down-up-eigenvalue} on
$\mathsf{Up}_{N,j\to q-1}\mathcal K_{N,j}^{\ord}$.
Since $q\le\ell_N\le n_N/2$ and $0\le j\le q-1$, both factors in
$a_{q,j,N}$ are positive.  By the induction hypothesis, the down--up
operator is therefore positive and invertible on all of
$\mathsf X_{N,q-1}^{\ord,\mathrm{sym}}$.  In particular,
\[
 \|\mathsf{Up}_{N,q-1\to q}g\|_2^2
 =\left\langle g,
   \mathsf{Down}_{N,q\to q-1}\mathsf{Up}_{N,q-1\to q}g
  \right\rangle_2
\]
shows that $\mathsf{Up}_{N,q-1\to q}$ is injective.

The inherited degree summands remain mutually orthogonal after applying the
up map.  Indeed, if $f$ and $g$ belong to two distinct degree summands at
level $q-1$, then
\[
 \langle \mathsf{Up}_{N,q-1\to q}f,
          \mathsf{Up}_{N,q-1\to q}g\rangle_2
 =\langle f,
   \mathsf{Down}_{N,q\to q-1}\mathsf{Up}_{N,q-1\to q}g\rangle_2=0,
\]
because the second factor is a scalar multiple of $g$.  Consequently,
\[
 \operatorname{Ran}\mathsf{Up}_{N,q-1\to q}
 =\bigoplus_{j=0}^{q-1}
   \mathsf{Up}_{N,j\to q}\mathcal K_{N,j}^{\ord}.
\]
Finally,
$\mathsf{Down}_{N,q\to q-1}=\mathsf{Up}_{N,q-1\to q}^*$, so
\[
 \mathsf X_{N,q}^{\ord,\mathrm{sym}}
 =\operatorname{Ran}\mathsf{Up}_{N,q-1\to q}
   \oplus\ker\mathsf{Down}_{N,q\to q-1}
 =\bigoplus_{j=0}^{q}
   \mathsf{Up}_{N,j\to q}\mathcal K_{N,j}^{\ord}.
\]
This closes the induction through level $\ell_N$ and proves the down--up
spectral assertion.

If $k_N\le n_N/2$, then $k_N=\ell_N$ and
\eqref{prof:eq:finitepopHoeffding-physical} is already included in the
induction.  Suppose instead that $k_N>n_N/2$, so
$\ell_N=n_N-k_N$.  For $0\le j\le\ell_N$ and
$h\in\mathcal K_{N,j}^{\mathrm{set}}$, set $G_j=h$ and
$G_a=\mathsf{Add}_{a-1}\cdots\mathsf{Add}_jh$ for $j<a\le k_N$.
The same commutator induction gives
\[
 \mathsf{Del}_aG_a
 =(a-j)(n_N-a+1-j)G_{a-1}.
\]
Because $j\le\ell_N=n_N-k_N$, every scalar on the right is positive for
$j<a\le k_N$.  Since $\mathsf{Del}_a=\mathsf{Add}_{a-1}^*$, the lift
$h\mapsto G_{k_N}$ is injective.

The lifted harmonic degrees are mutually orthogonal.  If
$0\le j<j'\le\ell_N$, $h\in\mathcal K_{N,j}^{\mathrm{set}}$, and
$h'\in\mathcal K_{N,j'}^{\mathrm{set}}$, move the adjoints of
$\mathsf{Add}_{k_N-1}\cdots\mathsf{Add}_{j'}$ onto $G_{k_N}$.  Repeated use
of this recursion reduces the resulting vector
to a scalar multiple of
$\mathsf{Add}_{j'-1}\cdots\mathsf{Add}_jh$, which lies in
$\operatorname{Ran}\mathsf{Add}_{j'-1}$ and is therefore orthogonal to
$h'\in\ker\mathsf{Del}_{j'}$.

Finally, for $1\le j\le\ell_N$, the decomposition already proved at
level $j$ and injectivity of $\mathsf{Up}_{N,j-1\to j}$ give
$
 \dim\mathcal K_{N,j}^{\mathrm{set}}
 =\binom{n_N}{j}-\binom{n_N}{j-1}$.
The same formula holds for $j=0$ with $\binom{n_N}{-1}:=0$.  Hence
\[
 \sum_{j=0}^{\ell_N}\dim\mathcal K_{N,j}^{\mathrm{set}}
 =\binom{n_N}{\ell_N}
 =\binom{n_N}{k_N}
 =\dim\mathsf X_{N,k_N}^{\ord,\mathrm{sym}}.
\]
The mutually orthogonal lifted subspaces therefore exhaust the physical
slice, proving \eqref{prof:eq:finitepopHoeffding-physical}.
\end{proof}

Under the ordered/set identification, the physical decomposition
\eqref{prof:eq:finitepopHoeffding-physical} is the concrete realization of
the abstract slice decomposition recalled at the beginning of the section:
\[
 \iota_{N,k_N}\Hdeg_{N,r}
 =\mathsf{Up}_{N,r\to k_N}\mathcal K_{N,r}^{\ord}.
\]
Thus the degree-$r$ slice layer is the lift of the harmonic kernel at level
$r$.  \Cref{prof:lem:harmonic-lift-unitary} determines the normalization
that makes this lift unitary.

The positivity and invertibility in
\cref{prof:lem:fd-down-up-spectrum} give an explicit orthogonal projection.
Set $\mathsf{Top}_{N,0}^{\ord}:=I$.  For $1\le r\le\ell_N$, define
\begin{equation}
 \mathsf{Top}_{N,r}^{\ord}
 =I-\mathsf{Up}_{N,r-1\to r}
 \bigl(\mathsf{Down}_{N,r\to r-1}
       \mathsf{Up}_{N,r-1\to r}\bigr)^{-1}
 \mathsf{Down}_{N,r\to r-1}.
 \label{prof:eq:fd-top-projection-formula}
\end{equation}
Indeed, $\mathsf{Down}_{N,r\to r-1}
=\mathsf{Up}_{N,r-1\to r}^*$, so the second term on the right-hand side is the orthogonal
projection onto $\Ran\mathsf{Up}_{N,r-1\to r}$.
Hence $\mathsf{Top}_{N,r}^{\ord}$ is the
orthogonal projection onto the harmonic kernel
\[
 (\Ran\mathsf{Up}_{N,r-1\to r})^\perp
 =\ker\mathsf{Down}_{N,r\to r-1}
 =\mathcal K_{N,r}^{\ord}.
\]

\localheading{The normalized harmonic lift.}
For $0\le r\le k\le n_N$, define the unnormalized subset lift by
\begin{equation}
 \mathsf L_{k,r}
 :=\binom{k}{r}\mathsf{Up}_{N,r\to k}.
 \label{prof:eq:unnormalized-subset-lift}
\end{equation}
By \eqref{prof:eq:iterated-up-average}, it satisfies
\[
 (\mathsf L_{k,r}\kappa)(A)
 =\sum_{\substack{B\subseteq A\\|B|=r}}\kappa(B),
 \qquad A\in\Omega_{N,k}^{\mathrm{set}},
\]
so $\mathsf L_{k,r}$ is exactly the iterated up map with its averaging
factor removed.

For $0\le r\le\ell_N$ and
$\kappa\in\mathcal K_{N,r}^{\mathrm{set}}$, define
\[
 \mathsf{Harm}_{r\to k_N}\kappa
 :=c_{n_N,k_N,r}^{-1/2}\mathsf L_{k_N,r}\kappa,
\]
where, for $0\le r\le\min\{k,n-k\}$,
\[
 c_{n,k,r}
 :=\frac{\binom nr\binom{n-2r}{k-r}}{\binom nk}.
\]
After the unitary identification $\iota_{N,r}$, write
\begin{align*}
 U_{N,r}
 :=\mathsf{Harm}_{r\to k_N}\iota_{N,r}^{-1}:
 \mathcal K_{N,r}^{\ord}\to L^2(\Omega_{N,k_N},\pi_N).
\end{align*}

\begin{lemma}[Unitary harmonic lift and projection identities]
\label{prof:lem:harmonic-lift-unitary}
The map $U_{N,r}$ is unitary from $\mathcal K_{N,r}^{\ord}$ onto the
slice Hoeffding layer $\Hdeg_{N,r}$.  In particular,
\begin{equation}
 U_{N,r}^*U_{N,r}=I_{\mathcal K_{N,r}^{\ord}},
 \qquad
 U_{N,r}U_{N,r}^*=P_{N,r}.
 \label{prof:eq:harmonic-lift-projection-identities}
\end{equation}
\end{lemma}

\begin{proof}
Put $n=n_N$ and $k=k_N$, and let
$\kappa\in\mathcal K_{N,r}^{\mathrm{set}}$.  By \eqref{prof:eq:unnormalized-subset-lift} and the physical-level
Hoeffding decomposition \eqref{prof:eq:finitepopHoeffding-physical},
\[
 \operatorname{Ran}(\mathsf{Harm}_{r\to k})
 =\mathsf{Up}_{N,r\to k}\mathcal K_{N,r}^{\mathrm{set}}
 =\Hdeg_{N,r}.
\]
It therefore remains only to prove that the factor $c_{n,k,r}^{-1/2}$
normalizes the lift isometrically.

For functions $h,g$ on $\Omega_{N,j}^{\mathrm{set}}$, write
\[
 \langle h,g\rangle_{\mathrm{cnt},j}
 :=\sum_{A\in\Omega_{N,j}^{\mathrm{set}}}h(A)\overline{g(A)},
 \qquad
 \|h\|_{\mathrm{cnt},j}^2
 :=\langle h,h\rangle_{\mathrm{cnt},j}.
\]
In counting measure, set
\[
 H_j:=\mathsf{Add}_{j-1}\cdots\mathsf{Add}_r\kappa,
 \qquad r\le j\le k,
 \qquad H_r=\kappa.
\]
The commutator \eqref{prof:eq:add-del-commutator} and
$\mathsf{Del}_r\kappa=0$ imply, by induction,
\begin{equation}
 \mathsf{Del}_jH_j
 =(j-r)(n-j+1-r)H_{j-1}.
 \label{prof:eq:harmonic-lift-recursion}
\end{equation}
Since $H_j=\mathsf{Add}_{j-1}H_{j-1}$ and
$\mathsf{Del}_j=\mathsf{Add}_{j-1}^*$ in counting measure,
\begin{align*}
 \|H_j\|_{\mathrm{cnt},j}^2
 &=\langle H_j,H_j\rangle_{\mathrm{cnt},j}
 =\left\langle
    \mathsf{Add}_{j-1}H_{j-1},H_j
   \right\rangle_{\mathrm{cnt},j}\\
 &=\left\langle
    H_{j-1},\mathsf{Del}_jH_j
   \right\rangle_{\mathrm{cnt},j-1}
 =(j-r)(n-j+1-r)
   \|H_{j-1}\|_{\mathrm{cnt},j-1}^2,
\end{align*}
where the last line uses \eqref{prof:eq:harmonic-lift-recursion}.  Iterating
for $j=r+1,\ldots,k$ gives
\begin{align}
 \|H_k\|_{\mathrm{cnt},k}^2
 =\prod_{j=r+1}^{k}(j-r)(n-j+1-r)
   \|\kappa\|_{\mathrm{cnt},r}^2
 =(k-r)!(n-2r)_{\underline{k-r}}
   \|\kappa\|_{\mathrm{cnt},r}^2.
 \label{prof:eq:countingnormformula}
\end{align}

To relate $H_k$ to $\mathsf L_{k,r}$, expand the iterated add operator in
\eqref{prof:eq:unnormalized-subset-lift}:
\[
 H_k(A)
 =\sum_{\substack{
      B_r\subset B_{r+1}\subset\cdots\subset B_k=A\\
      |B_q|=q}}
   \kappa(B_r).
\]
For a fixed $r$-subset $B\subseteq A$, a chain from $B$ to $A$ is
determined uniquely by the order in which the $k-r$ elements of
$A\setminus B$ are inserted.  Hence exactly $(k-r)!$ chains begin at $B$,
and
\[
 H_k(A)=(k-r)!(\mathsf L_{k,r}\kappa)(A).
\]
Combining this identity with \eqref{prof:eq:countingnormformula} and converting to
the uniform probability measures yields
\begin{align*}
 \|\mathsf L_{k,r}\kappa\|_{\mathsf X_{N,k}^{\mathrm{set}}}^2
 &=\frac{1}{\binom nk((k-r)!)^2}
   \|H_k\|_{\mathrm{cnt},k}^2
 =\frac{(n-2r)_{\underline{k-r}}}
        {\binom nk(k-r)!}
   \|\kappa\|_{\mathrm{cnt},r}^2\\
 &=\frac{\binom{n-2r}{k-r}}{\binom nk}
   \|\kappa\|_{\mathrm{cnt},r}^2
 =\frac{\binom nr\binom{n-2r}{k-r}}{\binom nk}
   \|\kappa\|_{\mathsf X_{N,r}^{\mathrm{set}}}^2
 =c_{n,k,r}\|\kappa\|_{\mathsf X_{N,r}^{\mathrm{set}}}^2.
\end{align*}
Therefore $\mathsf{Harm}_{r\to k_N}=c_{n,k,r}^{-1/2}\mathsf L_{k,r}$ is an
isometry onto $\Hdeg_{N,r}$.  After the unitary identification
$\iota_{N,r}$, the map $U_{N,r}$ is unitary onto the same layer.  Hence
$U_{N,r}^*U_{N,r}=I$ on $\mathcal K_{N,r}^{\ord}$, while
$U_{N,r}U_{N,r}^*$ is the orthogonal projection onto $\Hdeg_{N,r}$ and
therefore equals $P_{N,r}$.  This proves
\eqref{prof:eq:harmonic-lift-projection-identities}.
\end{proof}

The subset formula for $\mathsf L_{k,r}$ also makes the harmonic lift
site-permutation-equivariant.  Indeed, the definition of the down map shows that
it commutes with every site permutation, so site permutations preserve the
harmonic kernels; then, for every site permutation $\sigma\in\Sym(V_N)$,
\begin{equation}
 \mathsf T_{\sigma}^{(k_N)}\mathsf{Harm}_{r\to k_N}
 =\mathsf{Harm}_{r\to k_N}\mathsf T_{\sigma}^{(r)}.
 \label{prof:eq:harmonic-lift-permutation-equivariance}
\end{equation}
Since the complete graph exclusion generator is an average of site
transpositions, summing
\eqref{prof:eq:harmonic-lift-permutation-equivariance} over transpositions
transports its action through the harmonic lift.  This is the mechanism used
in \cref{prof:sec:chaos-weight} to identify the degree-$r$ complete graph
eigenvalue.

\localheading{Compatibility of the slice and ordered projections.}
The projections $P_{N,r}$ and $\mathsf{Top}_{N,r}^{\ord}$ act on different spaces and
serve different roles.  The operator $\mathsf{Top}_{N,r}^{\ord}$ extracts the harmonic
kernel inside the ordered level-$r$ space
$\mathsf X_{N,r}^{\ord,\mathrm{sym}}$, whereas $P_{N,r}$ extracts the
lifted degree-$r$ component inside the fixed-population slice
$L^2(\Omega_{N,k_N},\pi_N)$.  The projection identities
\eqref{prof:eq:harmonic-lift-projection-identities} make the dictionary
exact: $U_{N,r}$ transports a harmonic kernel to its slice representative,
and $U_{N,r}^*$ recovers that kernel from the degree-$r$ slice component.

\subsection{Unitary chaos coordinates and their moment realization}

\begin{definition}[Isometric chaos coordinate]
\label{prof:def:chaos-coordinate}
For $f\in L^2(\Omega_{N,k_N},\pi_N)$ and $0\le r\le\ell_N$, define
\begin{equation}
 \widehat\Psi_{N,r}(f)
 :=U_{N,r}^*f
 =U_{N,r}^*P_{N,r}f
 \in\mathcal K_{N,r}^{\ord}.
 \label{prof:eq:chaos-coordinate-operator-definition}
\end{equation}
\end{definition}

Applying $U_{N,r}$ to \eqref{prof:eq:chaos-coordinate-operator-definition}
and using \eqref{prof:eq:harmonic-lift-projection-identities} gives
\begin{equation*}
 U_{N,r}\widehat\Psi_{N,r}(f)=P_{N,r}f
\end{equation*}
and, for every $\kappa\in\mathcal K_{N,r}^{\ord}$,
\[
 \langle\widehat\Psi_{N,r}(f),\kappa\rangle_2
 =\langle P_{N,r}f,U_{N,r}\kappa\rangle_{\pi_N}.
\]
Orthogonality of the slice layers and unitarity
of $U_{N,r}$ yield the exact Parseval identity
\[
 \|f\|_{L^2(\pi_N)}^2
 =\sum_{r=0}^{\ell_N}\|P_{N,r}f\|_{L^2(\pi_N)}^2
 =\sum_{r=0}^{\ell_N}\|\widehat\Psi_{N,r}(f)\|_2^2.
\]
\localheading{Isometric, factorial, and moment normalizations.}
The isometric coordinate $\widehat\Psi_{N,r}$ is the natural Hilbert space
coordinate.  To make the degree-$r$ Parseval contribution carry the standard
$1/r!$ exponential-generating weight, define its factorially rescaled form by
\[
 \Psi_{N,r}(f):=\sqrt{r!}\,\widehat\Psi_{N,r}(f).
\]
Since $U_{N,r}$ is unitary onto $\Hdeg_{N,r}$, this rescaling gives the
exact degreewise chaos isometry
\begin{equation}
 \|P_{N,r}f\|_{L^2(\pi_N)}^2
 =\|\widehat\Psi_{N,r}(f)\|_2^2
 =\frac1{r!}\|\Psi_{N,r}(f)\|_2^2.
 \label{prof:eq:exact-chaos-isometry}
\end{equation}
Thus the factor $1/r!$ comes solely from the factorial rescaling of the
isometric coordinate.

Now let $f$ be a probability density with respect to $\pi_N$, meaning
$f\ge0$ and $\mathbb E_{\pi_N}[f]=1$.  Define its normalized ordered
$r$-point moment on distinct coordinates by
\begin{equation}
 \mathcal M_{N,r}[f](x_1,\ldots,x_r)
 :=\left(\frac{n_N}{\beta_N}\right)^{r/2}
 \mathbb E_{f\pi_N}\left[\prod_{i=1}^r(\eta_{x_i}-\rho_N)\right].
 \label{prof:eq:general-normalized-moment}
\end{equation}
The scalar converting this moment normalization into the factorially rescaled
chaos coordinate is
\[
 \alpha_{N,r}
 :=\sqrt{r!}\,c_{n_N,k_N,r}^{-1/2}\binom{n_N}{r}
   \left(\frac{\beta_N}{n_N}\right)^{r/2}.
\]
Thus the corresponding multiplier for the isometric coordinate is
$\alpha_{N,r}/\sqrt{r!}$.  Equivalently,
\begin{equation}
 \alpha_{N,r}^2
 =\beta_N^r\frac{(n_N)_{\underline r}}{n_N^r}
   \frac{(n_N)_{\underline{2r}}}
        {(k_N)_{\underline r}(n_N-k_N)_{\underline r}}.
 \label{prof:eq:chaos-normalization}
\end{equation}
Indeed,
$r!\binom{n_N}{r}=(n_N)_{\underline r}$ and
$\binom{n_N}{k_N}/\binom{n_N-2r}{k_N-r}
=(n_N)_{\underline{2r}}/
[(k_N)_{\underline r}(n_N-k_N)_{\underline r}]$;
substitution into the preceding definition of $\alpha_{N,r}$ gives
\eqref{prof:eq:chaos-normalization}.
The normalization ingredients used in this section therefore have separate roles:
$c_{n_N,k_N,r}^{-1/2}$ normalizes the harmonic lift, $\sqrt{r!}$ passes to
the factorial rescaling, $(n_N/\beta_N)^{r/2}$ normalizes centered moments,
and $\alpha_{N,r}$ records their resulting conversion factor for the
factorially rescaled chaos coordinate.

\begin{lemma}[Algebraic probability-density-to-chaos interface]
\label{prof:lem:density-chaos-interface}
For every probability density $f$ with respect to $\pi_N$ and every
$0\le r\le\ell_N$, we have
\begin{equation}
 \widehat\Psi_{N,r}(f)
 =\frac{\alpha_{N,r}}{\sqrt{r!}}\mathsf{Top}_{N,r}^{\ord}\mathcal M_{N,r}[f],
 \qquad
 \Psi_{N,r}(f)
 =\alpha_{N,r}\mathsf{Top}_{N,r}^{\ord}\mathcal M_{N,r}[f].
 \label{prof:eq:general-density-chaos-interface}
\end{equation}
\end{lemma}

\begin{proof}
Fix $\kappa\in\mathcal K_{N,r}^{\ord}$ and identify it with its harmonic
set-kernel realization.  For each $r$-set $B$,
\[
 \prod_{x\in B}(\eta_x-\rho_N)
 =\sum_{C\subseteq B}(-\rho_N)^{r-|C|}
   \prod_{x\in C}\eta_x.
\]
Summing this identity against $\overline{\kappa(B)}$, and regrouping first by
$s=|C|$ then by the $s$-set $C$ gives
\begin{align*}
 \sum_{|B|=r}\overline{\kappa(B)}
 \mathbb E_{f\pi_N}\left[\prod_{x\in B}(\eta_x-\rho_N)\right]
 =
 \sum_{s=0}^r(-\rho_N)^{r-s}
 \sum_{|C|=s}
 \mathbb E_{f\pi_N}\left[\prod_{x\in C}\eta_x\right]
 \sum_{\substack{B\supseteq C\\|B|=r}}\overline{\kappa(B)}.
\end{align*}
If $s<r$, repeated deletion yields
\[
 \sum_{\substack{B\supseteq C\\|B|=r}}\overline{\kappa(B)}
 =\frac1{(r-s)!}
  \overline{(\mathsf{Del}_{s+1}\cdots\mathsf{Del}_r\kappa)(C)}=0.
\]
Indeed, each $r$-set $B\supseteq C$ is counted once for each ordering of the
$r-s$ elements of $B\setminus C$, hence exactly $(r-s)!$ times, and the final
quantity vanishes because $\mathsf{Del}_r\kappa=0$.  Thus every term with
$s<r$ disappears, and only $s=r$ remains.  Therefore
\[
 \sum_{|B|=r}\overline{\kappa(B)}
 \mathbb E_{f\pi_N}\left[\prod_{x\in B}(\eta_x-\rho_N)\right]
 =\sum_{|B|=r}\overline{\kappa(B)}
   \mathbb E_{f\pi_N}\left[\prod_{x\in B}\eta_x\right].
\]
Since the occupation variables take values in $\{0,1\}$,
$\prod_{x\in B}\eta_x=\mathbf 1_{\{B\subseteq\eta\}}$.  Hence
\begin{align}
 \sum_{|B|=r}\overline{\kappa(B)}
 \mathbb E_{f\pi_N}\left[\prod_{x\in B}(\eta_x-\rho_N)\right]
 &=\mathbb E_{f\pi_N}
   \bigg[\sum_{\substack{B\subseteq\eta\\|B|=r}}\overline{\kappa(B)}\bigg].
 \label{prof:eq:chaosidentical}
\end{align}
On the one hand, \eqref{prof:eq:general-normalized-moment}, symmetry of both kernels, and
$(n_N)_{\underline r}=r!\binom{n_N}{r}$ show that the left-hand side of \eqref{prof:eq:chaosidentical} equals
\begin{align*}
 \binom{n_N}{r}
   \left(\frac{\beta_N}{n_N}\right)^{r/2}
   \left\langle\mathcal M_{N,r}[f],\kappa\right\rangle_2.
\end{align*}
On the other hand, by the definition of $U_{N,r}$, the right-hand side of \eqref{prof:eq:chaosidentical} equals
\[
 c_{n_N,k_N,r}^{1/2}
 \left\langle P_{N,r}f,U_{N,r}\kappa\right\rangle_{\pi_N}.
\]
Equating the two sides of \eqref{prof:eq:chaosidentical} yields the identity
\begin{equation*}
 \left\langle P_{N,r}f,U_{N,r}\kappa\right\rangle_{\pi_N}
 =\frac{\alpha_{N,r}}{\sqrt{r!}}\left\langle\mathcal M_{N,r}[f],\kappa\right\rangle_2.
\end{equation*}
By \eqref{prof:eq:chaos-coordinate-operator-definition}, the left side is
$\langle\widehat\Psi_{N,r}(f),\kappa\rangle_2$.  Since the identity holds
for every harmonic $\kappa \in \mathcal K_{N,r}^{\ord}$, it gives
\[
 \widehat\Psi_{N,r}(f)
 =\frac{\alpha_{N,r}}{\sqrt{r!}}\mathsf{Top}_{N,r}^{\ord}\mathcal M_{N,r}[f].
\]
Multiplication by $\sqrt{r!}$ proves the second formula.
\end{proof}

\subsection{Independent and hard-core dynamic realizations}

For the remainder of this subsection, $r$ is fixed before $N\to\infty$.
At the fixed-population slice level, the exclusion generator preserves the
Hoeffding layers.  Indeed, each Hoeffding layer is invariant under
permutations of the site coordinates, while $L_N^{\SEP}$ is a sum of
transpositions of neighboring sites.  Hence
\[
 L_N^{\SEP}P_{N,r}=P_{N,r}L_N^{\SEP}
 \quad\text{on }L^2(\Omega_{N,k_N},\pi_N).
\]
The fixed-degree comparison is carried out instead in ordered $r$-particle
coordinates, where the independent and hard-core dynamics can be compared
directly.  These dynamics act on different ambient spaces:
\[
 \mathsf X_{N,r}^{\ind}
 :=L^2(V_N^r,n_N^{-r}),
 \qquad
 \mathsf X_{N,r}^{\hc}
 :=\mathsf X_{N,r}^{\ord}.
\]
Write $\mathsf X_{N,r}^{\ind,\mathrm{sym}}$ for the subspace of
permutation-symmetric functions in $\mathsf X_{N,r}^{\ind}$.  Throughout
this subsection, the inner products are probability-normalized:
\begin{equation*}
 \langle u,v\rangle_{\ind}
 :=\frac1{n_N^r}\sum_{\boldsymbol x\in V_N^r}
 u(\boldsymbol x)\overline{v(\boldsymbol x)},
 \qquad
 \langle f,g\rangle_{\hc}
 :=\frac1{(n_N)_{\underline r}}
 \sum_{\boldsymbol x\in\Omega_{N,r}^{\ord}}
 f(\boldsymbol x)\overline{g(\boldsymbol x)},
\end{equation*}
with $\|\cdot\|_{\ind}$ and $\|\cdot\|_{\hc}$ denoting the associated
norms.
The permutation-symmetric subspace of $\mathsf X_{N,r}^{\hc}$ is
$\mathsf X_{N,r}^{\ord,\mathrm{sym}}$, the symmetric ordered space
introduced above.  For
$\boldsymbol x=(x_1,\ldots,x_r)$ and $y\in V_N$, let
$\boldsymbol x^{a\to y}$ be obtained by replacing $x_a$ by $y$.

\localheading{Generators, forms, and semigroups.}
The positive independent generator and the positive ordered hard-core
generator are
\begin{align}
 (\mathcal L_{N,r}^{\ind}f)(\boldsymbol x)
 &=\sum_{a=1}^r\sum_{y\sim x_a}
   \bigl(f(\boldsymbol x)-f(\boldsymbol x^{a\to y})\bigr),
 \qquad \boldsymbol x\in V_N^r, \notag\\
 (\mathcal L_{N,r}^{\hc}f)(\boldsymbol x)
 &=\sum_{a=1}^r
   \sum_{\substack{y\sim x_a\\y\notin\{x_b:b\ne a\}}}
   \bigl(f(\boldsymbol x)-f(\boldsymbol x^{a\to y})\bigr),
 \qquad \boldsymbol x\in\Omega_{N,r}^{\ord}.
 \label{prof:eq:fd-hard-core-generator}
\end{align}
The hard-core generator simply omits jumps into occupied coordinates; it
adds neither killing nor a penalty potential.  The associated Dirichlet
forms are
\[
 \mathcal E_{N,r}^{\ind}(f,g)
 :=\langle f,\mathcal L_{N,r}^{\ind}g\rangle_{\ind},
 \qquad
 \mathcal E_{N,r}^{\hc}(f,g)
 :=\langle f,\mathcal L_{N,r}^{\hc}g\rangle_{\hc},
\]
with the abbreviations
$\mathcal E_{N,r}^{\ind}(f)=\mathcal E_{N,r}^{\ind}(f,f)$ and
$\mathcal E_{N,r}^{\hc}(f)=\mathcal E_{N,r}^{\hc}(f,f)$.
The associated heat semigroups are
\begin{align}
 \mathcal P_{N,r}^{\ind}(t)=e^{-t\mathcal L_{N,r}^{\ind}},
 \qquad
 \mathcal P_{N,r}^{\hc}(t)=e^{-t\mathcal L_{N,r}^{\hc}}.
 \label{prof:eq:heat-semigroups}
\end{align}

\localheading{Restriction and zero extension.}
The restriction map
$\mathsf R_{N,r}:\mathsf X_{N,r}^{\ind}\to\mathsf X_{N,r}^{\hc}$ is
\begin{align}
 (\mathsf R_{N,r}u)(\boldsymbol x)=u(\boldsymbol x),
 \qquad \boldsymbol x\in\Omega_{N,r}^{\ord}.
\end{align}
Because both ambient spaces carry probability-normalized measures, its
adjoint $\mathsf R_{N,r}^*:\mathsf X_{N,r}^{\hc}\to\mathsf X_{N,r}^{\ind}$ is the normalization-adjusted zero extension
\begin{equation}
 (\mathsf R_{N,r}^*f)(\boldsymbol x)
 =\frac{n_N^r}{(n_N)_{\underline r}}
   \one_{\Omega_{N,r}^{\ord}}(\boldsymbol x)f(\boldsymbol x).
 \label{prof:eq:restriction-adjoint}
\end{equation}
Whenever zero extension is used below, it means the normalization-adjusted
adjoint $\mathsf R_{N,r}^*$ in \eqref{prof:eq:restriction-adjoint}.

\localheading{Symmetric top Hoeffding sectors.}
For $u\in\mathsf X_{N,r}^{\ind}$ and $1\le a\le r$, let $\mathbb E_a u$
denote expectation in the $a$th coordinate with respect to the uniform
probability measure on $V_N$, with all other coordinates held fixed:
\begin{align}
 (\mathbb E_a u)(x_1,\ldots,x_r)
 :=\frac1{n_N}\sum_{y\in V_N}
 u(x_1,\ldots,x_{a-1},y,x_{a+1},\ldots,x_r).
 \label{prof:eq:top-exp}
\end{align}
Thus $\mathbb E_a$ is the conditional expectation operator associated with
the $a$th factor of the product probability space
$\operatorname{Unif}(V_N)^{\otimes r}$.
The product-top projection is
\begin{align}
 \mathsf{Top}_{N,r}^{\ind}:=\prod_{a=1}^r(I-\mathbb E_a),
 \label{prof:eq:product-top-proj}
\end{align}
and the symmetric product-top Hoeffding sector is the explicit subspace
\begin{equation*}
 \mathcal T_{N,r}^{\ind}
 :=\mathsf{Top}_{N,r}^{\ind}\mathsf X_{N,r}^{\ind,\mathrm{sym}}
 =\left\{u\in\mathsf X_{N,r}^{\ind,\mathrm{sym}}:
   \mathbb E_a u=0\ \text{for every }1\le a\le r\right\}.
\end{equation*}
The symmetric hard-core top Hoeffding sector is
\[
 \mathcal T_{N,r}^{\hc}
 :=\mathcal K_{N,r}^{\ord}
 =\Ran\mathsf{Top}_{N,r}^{\ord},
\]
where the orthogonal projector $\mathsf{Top}_{N,r}^{\ord}$ was defined in
\eqref{prof:eq:fd-top-projection-formula}.
Thus $\mathcal K_{N,r}^{\mathrm{set}}$ and
$\mathcal K_{N,r}^{\ord}$ are reserved for harmonic kernels, whereas
$\mathcal T_{N,r}^{\ind}$ and $\mathcal T_{N,r}^{\hc}$ denote the
dynamic top sectors.  The equality
$\mathcal T_{N,r}^{\hc}=\mathcal K_{N,r}^{\ord}$ identifies two roles on
the same ordered space; it is not a notational convention.
The ordered down map intertwines the hard-core generators at adjacent
particle levels:
\begin{equation}
 \mathsf{Down}_{N,r\to r-1}\mathcal L_{N,r}^{\hc}
 =\mathcal L_{N,r-1}^{\hc}\mathsf{Down}_{N,r\to r-1}
 \quad\text{on }\mathsf X_{N,r}^{\ord,\mathrm{sym}}.
 \label{prof:eq:hard-core-down-intertwining}
\end{equation}
For an edge $e=\{u,v\}\in E_N$, let $\tau_e$ be the site transposition
exchanging $u$ and $v$.  On symmetric ordered-tuple functions, the hard-core
generator has the representation
\[
 \mathcal L_{N,r}^{\hc}
 =\sum_{e\in E_N}\bigl(I-\mathsf T_{\tau_e}^{(r)}\bigr).
\]
Indeed, if exactly one endpoint of $e$ is occupied, the corresponding term
is the allowed exclusion jump across $e$; if both endpoints are occupied,
$\mathsf T_{\tau_e}^{(r)}$ merely exchanges two coordinates and therefore
acts trivially by symmetry.  The ordered down map is equivariant under this
site action:
\[
 \mathsf{Down}_{N,r\to r-1}\mathsf T_{\tau_e}^{(r)}
 =\mathsf T_{\tau_e}^{(r-1)}\mathsf{Down}_{N,r\to r-1}.
\]
This follows directly from \eqref{prof:eq:ordered-up-down-formulas}, since
$z\mapsto\tau_e z$ bijects the complement of an ordered $(r-1)$-tuple onto
the complement of its image under $\tau_e$.  Summing over $e\in E_N$ gives
\eqref{prof:eq:hard-core-down-intertwining}.  Expanding the finite-dimensional
semigroups in power series then also gives
\begin{equation}
 \mathsf{Down}_{N,r\to r-1}\mathcal P_{N,r}^{\hc}(t)
 =\mathcal P_{N,r-1}^{\hc}(t)\mathsf{Down}_{N,r\to r-1},
 \qquad t\ge0.
 \label{prof:eq:hard-core-down-semigroup-intertwining}
\end{equation}
Taking adjoints, using self-adjointness of the hard-core semigroups, and
iterating gives the companion up-intertwining
\begin{equation}
 \mathcal P_{N,q}^{\hc}(t)\mathsf{Up}_{N,j\to q}
 =\mathsf{Up}_{N,j\to q}\mathcal P_{N,j}^{\hc}(t),
 \qquad 1\le j\le q,
 \quad t\ge0.
 \label{prof:eq:up-intertwine}
\end{equation}
Since
$\Ran\mathsf{Top}_{N,r}^{\ord}=\ker\mathsf{Down}_{N,r\to r-1}$, if
$f\in\Ran\mathsf{Top}_{N,r}^{\ord}$, then
\[
 \mathsf{Down}_{N,r\to r-1}\mathcal L_{N,r}^{\hc}f
 =\mathcal L_{N,r-1}^{\hc}
   \mathsf{Down}_{N,r\to r-1}f
 =0.
\]
Thus $\Ran\mathsf{Top}_{N,r}^{\ord}$ is invariant under $\mathcal L_{N,r}^{\hc}$.
Because $\mathcal L_{N,r}^{\hc}$ is self-adjoint, its orthogonal complement
inside $\mathsf X_{N,r}^{\ord,\mathrm{sym}}$ is invariant as well.  Therefore
\begin{align}
 \mathsf{Top}_{N,r}^{\ord}\mathcal L_{N,r}^{\hc}
 =\mathcal L_{N,r}^{\hc}\mathsf{Top}_{N,r}^{\ord},
 \qquad
 \mathsf{Top}_{N,r}^{\ord}\mathcal P_{N,r}^{\hc}(t)
 =\mathcal P_{N,r}^{\hc}(t)\mathsf{Top}_{N,r}^{\ord}
 \quad (t\ge0).
 \label{prof:eq:HLcommute}
\end{align}
The permutation equivariance
\eqref{prof:eq:harmonic-lift-permutation-equivariance}, specialized to the
nearest-neighbor transpositions and summed over $e\in E_N$, gives the
corresponding generator intertwining on the top sector:
\[
L_N^{\SEP}U_{N,r}
 =U_{N,r}\mathcal L_{N,r}^{\hc}
 \quad\text{on }\mathcal T_{N,r}^{\hc}.
\]
Here the set/ordered identification $\iota_{N,r}$ identifies the
$r$-subset nearest-neighbor exchange generator with
$\mathcal L_{N,r}^{\hc}$ on symmetric ordered functions.
The fixed-degree spectral comparison in
\cref{prof:sec:fixed-degree} is carried out on these two top sectors.

\localheading{Finite-population Wick coordinate.}
Fix the source $S_N$, and suppress its superscript throughout this subsection:
\[
 F_N=\one_{S_N}-\rho_N,
 \qquad
 m_{N,t}=e^{-tL_N^{\RW}}F_N.
\]
For any mean-zero one-particle field $m\in\mathbb R_0^{V_N}$ and
$0\le r\le\ell_N$, define the
\emph{finite-population Wick coordinate}
\begin{equation}
 \mathscr W_{N,r}(m)
 :=\alpha_{N,r}\mathsf{Top}_{N,r}^{\ord}
 \left(\frac{n_N}{\beta_N}\right)^{r/2}
 \mathsf R_{N,r}m^{\otimes r}.
 \label{prof:eq:canonical-wick-coordinate}
\end{equation}
Thus $\mathscr W_{N,r}(m)$ is the top-Hoeffding coordinate of the
independent product tensor $m^{\otimes r}$ after restriction to ordered
distinct tuples.
The next lemma applies it to the one-particle heat profile $m=m_{N,t}$.

\begin{lemma}[Exact duality formula for the dynamic chaos coordinate]
\label{prof:lem:fd-exact-duality}
For every $N$, every $r\in\{0,\ldots,\ell_N\}$, and every $t\ge0$, we have
\begin{align}
 \Psi_{N,r}(h_{N,t}^{S_N})
 &=\alpha_{N,r}\mathsf{Top}_{N,r}^{\ord}
   \left(\frac{n_N}{\beta_N}\right)^{r/2}
   \mathcal P_{N,r}^{\hc}(t)\mathsf R_{N,r}F_N^{\otimes r},
 \label{prof:eq:fd-exact-hc-coordinate}\\
 \mathscr W_{N,r}(m_{N,t})
 &=\alpha_{N,r}\mathsf{Top}_{N,r}^{\ord}
   \left(\frac{n_N}{\beta_N}\right)^{r/2}
   \mathsf R_{N,r}\mathcal P_{N,r}^{\ind}(t)F_N^{\otimes r}.
 \label{prof:eq:fd-exact-ind-coordinate}
\end{align}
\end{lemma}

\begin{proof}
Apply \cref{prof:lem:density-chaos-interface} to the
specific density
$
 f=h_{N,t}^{S_N}=\frac{d\mu_{N,t}^{S_N}}{d\pi_N}
 $.
For a distinct ordered tuple $X=(x_1,\ldots,x_r)$, set
\[
 \mathsf{Dual}_r(\eta,X):=\prod_{a=1}^r(\eta_{x_a}-\rho_N).
\]
We now invoke the \emph{self-duality} of SEP as follows.
For an edge
$e=\{u,v\}\in E_N$, let $\tau_e$ be the transposition of $u$ and $v$.
Swapping the occupations at $u$ and $v$ gives
\[
 \mathsf{Dual}_r(\eta^{uv},X)=\mathsf{Dual}_r(\eta,\tau_e X),
\]
where $\tau_e$ acts coordinatewise on the ordered tuple.  Therefore, with
$L_N^{\SEP}$ acting on the configuration variable and
$\mathcal L_{N,r}^{\hc}$ on the ordered-tuple variable,
\[
 (L_N^{\SEP})_\eta \mathsf{Dual}_r(\eta,X)
 =\sum_{e\in E_N}\{\mathsf{Dual}_r(\eta,X)-\mathsf{Dual}_r(\eta,\tau_eX)\}
 =(\mathcal L_{N,r}^{\hc})_X \mathsf{Dual}_r(\eta,X).
\]
The last equality is the site transposition representation of
$\mathcal L_{N,r}^{\hc}$ on symmetric ordered-tuple functions established
above.  Since both state spaces are finite, exponentiating this generator
identity gives the semigroup duality relation.  Let $\mathbf X_t^X$ denote
the ordered $r$-particle exclusion process started from $X$, and write
$\mathbb E_X$ for expectation for this process.  
Under the deterministic
initial configuration $\eta_0=\one_{S_N}$, we have the identity
\begin{align}
 \mathbb E_{\mu_{N,t}^{S_N}}
 \left[\prod_{a=1}^r(\eta_{x_a}-\rho_N)\right]
 =\mathbb E_X\left[\prod_{a=1}^rF_N(\mathbf X_t^X(a))\right]
 =\mathbb E_X\left[\bigl(\mathsf R_{N,r}F_N^{\otimes r}\bigr)(\mathbf X_t^X)\right]
 =\bigl(\mathcal P_{N,r}^{\hc}(t)\mathsf R_{N,r}F_N^{\otimes r}\bigr)(X).
 \label{prof:eq:sep-self-duality}
\end{align}
Hence, by
\eqref{prof:eq:general-normalized-moment},
\[
 \mathcal M_{N,r}[h_{N,t}^{S_N}]
 =\left(\frac{n_N}{\beta_N}\right)^{r/2}
   \mathcal P_{N,r}^{\hc}(t)\mathsf R_{N,r}F_N^{\otimes r}.
\]
Substituting this identity into the second formula of
\eqref{prof:eq:general-density-chaos-interface} gives
\eqref{prof:eq:fd-exact-hc-coordinate}.

For independent walkers, tensorization gives
\[
 \mathcal P_{N,r}^{\ind}(t)F_N^{\otimes r}
 =(e^{-tL_N^{\RW}}F_N)^{\otimes r}
 =m_{N,t}^{\otimes r}.
\]
Substitution into \eqref{prof:eq:canonical-wick-coordinate} proves
\eqref{prof:eq:fd-exact-ind-coordinate}.
\end{proof}

\subsection{Static canonical tilt comparison}

The following lemma identifies the fixed-degree chaos coordinate of a
small canonical tilt.  The natural leading term in the cumulant expansion is
the Wick coordinate generated by the one-site marginal displacement
$\Marg_N(v_N)$.  The canonical response estimate then replaces this
displacement by its linear approximation $\beta_Nv_N$.  The proof uses the
canonical cumulant estimates from \cref{prof:lem:canonical-response}.

\begin{lemma}[Static canonical Wick coordinate at fixed degree]
\label{prof:lem:static-wick-coordinate}
Let $v_N\in\mathbb R_0^{V_N}$ satisfy $\|v_N\|_\infty\to0$ and
$\|v_N\|_{2,\mathrm{cnt}}=O(1)$.  For every fixed $R\in\mathbb N_0$, the
following convergence holds simultaneously over the finitely many degrees
$0\le r\le R$:
\begin{equation}
 \max_{0\le r\le R}
 \|\Psi_{N,r}(g_{N,v_N})-
 \mathscr W_{N,r}(\beta_Nv_N)\|_2\longrightarrow0.
 \label{prof:eq:static-wick-coordinate}
\end{equation}
\end{lemma}

\begin{proof}
The case $r=0$ is immediate.  Fix $r\ge1$, and write $v=v_N$.  The proof
has two steps.  First we show that the tilted canonical chaos coordinate is asymptotic to
$\mathscr W_{N,r}(m_v)$, where
$m_v:=\Marg_N(v)$ is the exact one-site marginal displacement.  We then use
the canonical response estimate to replace $m_v$ by $\beta_Nv$.

\localheading{Exact cumulant representation of the chaos coordinate.}
For $0\le\xi\le1$ and a nonempty block $B\subseteq[r]$, define
\[
 \mathfrak c_{\xi,B}(x_B)
 :=\operatorname{Cum}_{\xi v}
 \bigl((\eta_{x_i}-\rho_N)_{i\in B}\bigr),
\]
where the tilted-canonical cumulants were defined in
\eqref{prof:eq:canonical-cumulant-definition}. 
For a distinct ordered tuple $X=(x_1,\ldots,x_r)$, the definition of the
normalized moment, followed by the moment--cumulant identity
\eqref{prof:eq:moment-cumulant-identity} under $\nu_{N,v}$ with
$Y_i=\eta_{x_i}-\rho_N$, gives
\begin{align}
 \mathcal M_{N,r}[g_{N,v}](X)
 =\left(\frac{n_N}{\beta_N}\right)^{r/2}
   \mathbb E_{\nu_{N,v}}
   \left[\prod_{i=1}^r(\eta_{x_i}-\rho_N)\right]
 =\left(\frac{n_N}{\beta_N}\right)^{r/2}
   \sum_{\mathfrak p\in\operatorname{Part}([r])}
   \prod_{B\in\mathfrak p}\mathfrak c_{1,B}(x_B).
 \label{prof:eq:static-moment-cumulant}
\end{align}
Substituting \eqref{prof:eq:static-moment-cumulant} into the second formula of
\eqref{prof:eq:general-density-chaos-interface} in \cref{prof:lem:density-chaos-interface} gives
\begin{equation}
 \Psi_{N,r}(g_{N,v})
 =\alpha_{N,r}\mathsf{Top}_{N,r}^{\ord}
   \left(\frac{n_N}{\beta_N}\right)^{r/2}
   \sum_{\mathfrak p\in\operatorname{Part}([r])}
   \prod_{B\in\mathfrak p}\mathfrak c_{1,B}.
 \label{prof:eq:static-chaos-cumulant-representation}
\end{equation}
Let
\(
 \mathfrak p_{\mathrm{sing}}=\{\{1\},\ldots,\{r\}\}
\)
be the all-singleton partition.  Since
$\mathfrak c_{1,\{i\}}(x_i)=m_v(x_i)$, its contribution to
\eqref{prof:eq:static-chaos-cumulant-representation} is exactly
\[
 \alpha_{N,r}\mathsf{Top}_{N,r}^{\ord}
 \left(\frac{n_N}{\beta_N}\right)^{r/2}
 \mathsf R_{N,r}m_v^{\otimes r}
 =\mathscr W_{N,r}(m_v).
\]
Therefore
\begin{equation}
 \Psi_{N,r}(g_{N,v})-\mathscr W_{N,r}(m_v)
 =\alpha_{N,r}\mathsf{Top}_{N,r}^{\ord}
   \left(\frac{n_N}{\beta_N}\right)^{r/2}
   \sum_{\mathfrak p\ne\mathfrak p_{\mathrm{sing}}}
   \prod_{B\in\mathfrak p}\mathfrak c_{1,B}.
 \label{prof:eq:static-wick-remainder-exact}
\end{equation}
Thus the first task is to prove that the projected nonsingleton remainder in
\eqref{prof:eq:static-wick-remainder-exact} is $o(1)$.

The cumulant expansion is naturally expressed in terms of the exact tilted
marginal $m_v=\Marg_N(v)$, so we keep $m_v$ unchanged until the
nonsingleton remainder has been removed.  Since $\Marg_N(0)=0$, the
fundamental theorem of calculus and \eqref{prof:eq:canonical-covariance-response}
give, for all sufficiently
large $N$,
\[
 m_v-\beta_Nv
 =\int_0^1\bigl(\mathrm d\Marg_N(\xi v)-\beta_NI\bigr)v\,\dd\xi.
\]
Consequently,
\[
 \|m_v-\beta_Nv\|_{2,\mathrm{cnt}}
 \le C\|v\|_\infty\|v\|_{2,\mathrm{cnt}},
 \qquad
 \|m_v-\beta_Nv\|_\infty\le C\|v\|_\infty^2.
\]
Since \(\beta_N\) is uniformly bounded in the density window, and
\(\|v\|_\infty\to0\), this yields
\begin{equation}
 \|m_v-\beta_Nv\|_{2,\mathrm{cnt}}
 \le C\|v\|_\infty\|v\|_{2,\mathrm{cnt}},
 \qquad
 \|m_v\|_\infty\le C\|v\|_\infty.
 \label{prof:eq:static-one-site-response}
\end{equation}
In particular, $\|m_v\|_{2,\mathrm{cnt}}=O(1)$.

\localheading{Top projection of the nonsingleton remainder.}
For every nonsingleton block $B$, write
\[
 \mathfrak c_{1,B}=\mathfrak c_{0,B}+\Delta_B.
\]
Because $\nu_{N,0}=\pi_N$, the factor $\mathfrak c_{0,B}$ is the joint
cumulant under the stationary canonical measure $\pi_N$, that is, under the
untilted law.  By contrast,
$\Delta_B:=\mathfrak c_{1,B}-\mathfrak c_{0,B}$ records the change caused by
the tilt $v$.  The two terms play different roles.  The untilted canonical
cumulant $\mathfrak c_{0,B}$ is not discarded by a smallness estimate.
Instead, after summing over partitions the contributions obtained by choosing
$\mathfrak c_{0,B}$ from every nonsingleton block, the result lies in lower
Hoeffding degrees and is annihilated exactly by
$\mathsf{Top}_{N,r}^{\ord}$.  The tilt-response correction $\Delta_B$ is
small at the block normalization relevant to
\eqref{prof:eq:static-wick-remainder-exact}.  After this exact cancellation,
every remaining term therefore contains at least one small factor
$\Delta_B$.

We first prove this smallness.  Let $|B|=b\ge2$.  Differentiating the joint
cumulant generating function along the path $\xi v$ and using
\eqref{prof:eq:Frechet-cumulant-identity} give
\begin{align}
 \frac{\dd}{\dd\xi}\mathfrak c_{\xi,B}(x_B)
 &=\operatorname{Cum}_{\xi v}
   \bigl((\eta_{x_i}-\rho_N)_{i\in B},X_{N,v}\bigr)
 \notag\\
 &=\sum_zv(z)\,
   \operatorname{Cum}_{\xi v}
   \bigl((\eta_{x_i}-\rho_N)_{i\in B},\eta_z-\rho_N\bigr).
 \label{prof:eq:cumulant-variation}
\end{align}
For distinct $x_B$, define
\[
 C_{\xi,B}(x_B;z)
 :=\operatorname{Cum}_{\xi v}
   \bigl((\eta_{x_i}-\rho_N)_{i\in B},\eta_z-\rho_N\bigr).
\]
In \eqref{prof:eq:cumulant-variation}, the extra site $z$ is treated
differently according to whether it is new or collides with one of the $b$
sites already in $B$.  When
$z\notin\{x_i:i\in B\}$ there are $b+1$ distinct sites, so the stronger
distinct-site cumulant bound applies; when $z\in\{x_i:i\in B\}$ there are
only $b$ exceptional values of $z$, which compensates for the weaker
cumulant bound.  In the first case,
\eqref{prof:eq:canonical-distinct-cumulant-general} gives
$|C_{\xi,B}(x_B;z)|=O(n_N^{-b})$; Cauchy--Schwarz in $z$ yields
\[
 \left|\sum_{z\notin\{x_i:i\in B\}}v(z)C_{\xi,B}(x_B;z)\right|
 \le C_bn_N^{1/2-b}\|v\|_{2,\mathrm{cnt}}.
\]
In the collision case there are only $b$ possible values of $z$, and
$|C_{\xi,B}(x_B;z)|=O(n_N^{1-b})$, whence
\[
 \left|\sum_{z\in\{x_i:i\in B\}}v(z)C_{\xi,B}(x_B;z)\right|
 \le C_bn_N^{1-b}\|v\|_\infty.
\]
These bounds are uniform in $\xi\in[0,1]$.  Integrating
\eqref{prof:eq:cumulant-variation} and taking the ordered-tuple norm gives
\begin{equation}
 n_N^{b/2}\|\Delta_B\|_{\mathsf X_{N,b}^{\ord}}
 \le C_b\bigl(n_N^{(1-b)/2}\|v\|_{2,\mathrm{cnt}}
       +n_N^{1-b/2}\|v\|_\infty\bigr)=o(1).
 \label{prof:eq:normalized-cumulant-variation}
\end{equation}
The factor $n_N^{b/2}$ is exactly the block normalization induced by the
global factor $n_N^{r/2}$ in
\eqref{prof:eq:static-wick-remainder-exact}.

We next identify the exact cancellation behind the top projection.  Expand
every nonsingleton factor in
\eqref{prof:eq:static-wick-remainder-exact} as
$\mathfrak c_{0,B}+\Delta_B$, and first choose the untilted canonical term
$\mathfrak c_{0,B}$ from every nonsingleton block.  Group the resulting
contributions according to the number $s$ of singleton blocks.
Exchangeability of $\pi_N$ makes each untilted canonical cumulant associated
with a nonsingleton block depend only on that block's size.  After summing over all partitions with
exactly $s$ singleton blocks, the only coordinate dependence comes from the
$s$ singleton factors $m_v(x_i)$, and the contribution has the form
\begin{align*}
 \mathfrak a_{N,r,s}
 \binom rs^{-1}
 \sum_{\substack{I\subseteq[r]\\|I|=s}}
 \prod_{i\in I}m_v(x_i)
 &=\mathfrak a_{N,r,s}
   \mathsf{Up}_{N,s\to r}(m_v^{\otimes s})(x_1,\ldots,x_r),
\end{align*}
where $\mathfrak a_{N,r,s}$ absorbs the finitely many nonsingleton block size
patterns and the normalization of the iterated up map.  Because at least
one block is nonsingleton, $s\le r-2$.  Hence the entire contribution obtained by choosing only the untilted
canonical terms lies in $\operatorname{Ran}\mathsf{Up}_{N,s\to r}$ and is
annihilated exactly by $\mathsf{Top}_{N,r}^{\ord}$.  We have therefore shown
that, after this contribution is removed by the top projection, every
remaining expansion term contains at least one factor $\Delta_B$.

It remains only to check that multiplication by the other block factors
cannot destroy the $o(1)$ gain supplied by that $\Delta_B$.  For a fixed
partition $\mathfrak p\in\operatorname{Part}([r])$, if each block $B$
carries a kernel $K_B$ on $\Omega_{N,|B|}^{\ord}$, then dropping only the
distinctness constraints between different blocks gives
\begin{equation}
 n_N^{r/2}
 \left\|\prod_{B\in\mathfrak p}K_B(x_B)\right\|_{\mathsf X_{N,r}^{\ord}}
 \le C_r\prod_{B\in\mathfrak p}
 \left(n_N^{|B|/2}
 \|K_B\|_{\mathsf X_{N,|B|}^{\ord}}\right).
 \label{prof:eq:fixed-partition-product-bound}
\end{equation}
Indeed, after squaring, the sum over fully distinct $r$-tuples is bounded
by the product of the within-block distinct-tuple sums; the remaining
normalization ratio is
$\prod_{B\in\mathfrak p}(n_N)_{\underline{|B|}}/(n_N)_{\underline r}$,
which is bounded for fixed $r$.

The three possible block factors now have transparent sizes.  A singleton
block contributes
$\|m_v\|_{2,\mathrm{cnt}}=O(1)$ at the scaled level.  An untilted canonical cumulant factor associated with a block of size
$b\ge2$ contributes $O(n_N^{1-b/2})=O(1)$ by the cumulant bound
$O(n_N^{1-b})$.  Finally, every surviving term contains at least one
$\Delta_B$, whose scaled norm is $o(1)$ by
\eqref{prof:eq:normalized-cumulant-variation}.  The product estimate
\eqref{prof:eq:fixed-partition-product-bound} therefore makes every
surviving partition term $o(1)$ in the normalized $r$-point scale.  Since
there are
only finitely many partitions, $\mathsf{Top}_{N,r}^{\ord}$ is contractive,
$\alpha_{N,r}=O_{r,\rho_0}(1)$, and $\beta_N$ is bounded away from $0$ on the
standing density window, \eqref{prof:eq:static-wick-remainder-exact} yields
\begin{equation*}
 \left\|\Psi_{N,r}(g_{N,v})-\mathscr W_{N,r}(m_v)\right\|_2
 =o_{N,r}(1).
\end{equation*}
This completes the cumulant comparison.

\localheading{Replacement of the tilted marginal.}
It remains to replace the exact tilted marginal $m_v$ by its linear response
$\beta_Nv$.  For fixed $r\in\mathbb N_0$ and $a,b\in\mathbb R_0^{V_N}$, the tensor
telescoping identity
\[
 a^{\otimes r}-b^{\otimes r}
 =\sum_{j=1}^r
 a^{\otimes(j-1)}\otimes(a-b)\otimes b^{\otimes(r-j)}
\]
gives, for the $j$th summand,
\[
 \left\|a^{\otimes(j-1)}\otimes(a-b)\otimes b^{\otimes(r-j)}
 \right\|_{\ind}
 =n_N^{-r/2}\|a\|_{2,\mathrm{cnt}}^{j-1}
   \|a-b\|_{2,\mathrm{cnt}}
   \|b\|_{2,\mathrm{cnt}}^{r-j}.
\]
Since
\(
 \|\mathsf R_{N,r}\|^2=n_N^r/(n_N)_{\underline r}=1+O_r(n_N^{-1})
\), $\mathsf{Top}_{N,r}^{\ord}$ is contractive, and $\alpha_{N,r}$ is uniformly
bounded at fixed $r$, we obtain, on every counting norm ball of radius $B$,
\begin{equation}
 \|\mathscr W_{N,r}(a)-\mathscr W_{N,r}(b)\|_2
 \le C_{r,\rho_0,B}\|a-b\|_{2,\mathrm{cnt}}.
 \label{prof:eq:fixed-wick-lipschitz}
\end{equation}
Applying \eqref{prof:eq:fixed-wick-lipschitz} with $a=m_v$ and $b=\beta_Nv$, and using \eqref{prof:eq:static-one-site-response}, gives
\begin{align*}
 \|\Psi_{N,r}(g_{N,v})-\mathscr W_{N,r}(\beta_Nv)\|_2
 \le
 \|\Psi_{N,r}(g_{N,v})-\mathscr W_{N,r}(m_v)\|_2
 +
 \|\mathscr W_{N,r}(m_v)-\mathscr W_{N,r}(\beta_Nv)\|_2
 =o_{N,r}(1),
\end{align*}
where we used \eqref{prof:eq:static-one-site-response}.  Since $R$ is
fixed, taking the maximum over $0\le r\le R$ proves
\eqref{prof:eq:static-wick-coordinate}.
\end{proof}

We next compute the asymptotic norm of \(\mathscr W_{N,r}(a)\).  At fixed degree, the distinct-tuple restriction and the top Hoeffding projection preserve the leading product-tensor norm, yielding the following deterministic estimate.

\begin{lemma}[Norm of the finite-population Wick coordinate]
\label{prof:lem:static-wick-norm}
For each fixed $r$, every $B<\infty$, and every sequence
$\varepsilon_N\downarrow0$, uniformly over $a\in\mathbb R_0^{V_N}$ with
$\|a\|_{2,\mathrm{cnt}}\le B$ and $\|a\|_\infty\le\varepsilon_N$,
\begin{equation}
 \frac1{r!}\|\mathscr W_{N,r}(a)\|_2^2
 =\frac{(\beta_N^{-1}\|a\|_{2,\mathrm{cnt}}^2)^r}{r!}
  +o_{N,r}(1).
 \label{prof:eq:static-wick-norm}
\end{equation}
\end{lemma}

\begin{proof}
The case $r=0$ is immediate.  Assume $r\ge1$.  Fix $B<\infty$ and a sequence
$\varepsilon_N\downarrow0$, and let $a\in\mathbb R_0^{V_N}$ satisfy
$\|a\|_{2,\mathrm{cnt}}\le B$ and $\|a\|_\infty\le\varepsilon_N$.  By
\eqref{prof:eq:canonical-wick-coordinate},
\begin{equation}
 \|\mathscr W_{N,r}(a)\|_2^2
 =\alpha_{N,r}^2
  \left(\frac{n_N}{\beta_N}\right)^r
  \left\|\mathsf{Top}_{N,r}^{\ord}\mathsf R_{N,r}a^{\otimes r}\right\|_2^2.
 \label{prof:eq:wick-norm-start}
\end{equation}
We show successively that the distinct-tuple restriction and the top
projection do not change the leading norm; the remaining multiplier
$\alpha_{N,r}$ tends to $1$.

The unrestricted product sum equals $\|a\|_{2,\mathrm{cnt}}^{2r}$.  A tuple
omitted by the distinct-coordinate restriction contains a colliding pair,
and a union bound gives
\[
 \sum_{\boldsymbol x:\,\text{a collision}}
 \prod_{i=1}^r|a(x_i)|^2
 \le C_r\|a\|_\infty^2
       \|a\|_{2,\mathrm{cnt}}^{2r-2}
 \le C_r\varepsilon_N^2B^{2r-2}=o(1),
\]
Hence
\begin{equation}
 \left(\frac{n_N}{\beta_N}\right)^r
 \|\mathsf R_{N,r}a^{\otimes r}\|_{\mathsf X_{N,r}^{\ord}}^2
 =\beta_N^{-r}\|a\|_{2,\mathrm{cnt}}^{2r}+o_{N,r}(1).
 \label{prof:eq:raw-tensor-norm}
\end{equation}

We next show that the top projection changes this leading norm by only
$o(1)$.  The mean-zero condition gives the exact down map identity
\[
 \mathsf{Down}_{N,r\to r-1}\mathsf R_{N,r}a^{\otimes r}
 (x_1,\ldots,x_{r-1})
 =-\frac{a(x_1)\cdots a(x_{r-1})}{n_N-r+1}
   \sum_{i=1}^{r-1}a(x_i).
\]
Using $|\sum_{i=1}^{r-1}a(x_i)|^2\le
(r-1)^2\|a\|_\infty^2$ and then dropping the distinctness restriction,
we obtain
\begin{align*}
 n_N^r
 \left\|\mathsf{Down}_{N,r\to r-1}\mathsf R_{N,r}a^{\otimes r}\right\|_2^2
 &\le
 C_r\frac{n_N^r}{(n_N-r+1)^2(n_N)_{\underline{r-1}}}
 \|a\|_\infty^2
 \left(\sum_x|a(x)|^2\right)^{r-1}
 \notag\\
 &\le \frac{C_r}{n_N}\|a\|_\infty^2
       \|a\|_{2,\mathrm{cnt}}^{2r-2}
 \le \frac{C_r}{n_N}\varepsilon_N^2B^{2r-2}=o(1).
\end{align*}
Thus
\[
 n_N^{r/2}
 \left\|\mathsf{Down}_{N,r\to r-1}\mathsf R_{N,r}a^{\otimes r}\right\|_2
 =o(1).
\]
By \eqref{prof:eq:fd-down-up-eigenvalue} with $q=r$, for fixed $r$ and all
sufficiently large $N$ the spectrum of
$\mathsf{Down}_{N,r\to r-1}\mathsf{Up}_{N,r-1\to r}$ is bounded below by
$c_r>0$.  Hence its inverse has $L^2$ operator norm at most $C_r$.
Combining this with the exact projection formula
\eqref{prof:eq:fd-top-projection-formula} therefore gives
\[
 n_N^{r/2}
 \left\|(I-\mathsf{Top}_{N,r}^{\ord})\mathsf R_{N,r}a^{\otimes r}\right\|_2=o(1).
\]
Since $\mathsf{Top}_{N,r}^{\ord}$ is an orthogonal projection, combining this with
\eqref{prof:eq:raw-tensor-norm} yields
\begin{equation}
 \left(\frac{n_N}{\beta_N}\right)^r
 \left\|\mathsf{Top}_{N,r}^{\ord}\mathsf R_{N,r}a^{\otimes r}\right\|_2^2
 =\beta_N^{-r}\|a\|_{2,\mathrm{cnt}}^{2r}+o_{N,r}(1).
 \label{prof:eq:top-tensor-norm}
\end{equation}
Finally, $\alpha_{N,r}=1+O_{r,\rho_0}(n_N^{-1})$.  Substituting
\eqref{prof:eq:top-tensor-norm} into \eqref{prof:eq:wick-norm-start} gives
\[
 \|\mathscr W_{N,r}(a)\|_2^2
 =\bigl(\beta_N^{-1}\|a\|_{2,\mathrm{cnt}}^2\bigr)^r+o_{N,r}(1).
\]
All error bounds above depend only on $r$, $\rho_0$, $B$, and
$\varepsilon_N$, and vanish as $N\to\infty$.  Thus the convergence is
uniform over the stated class.  Dividing by $r!$ proves
\eqref{prof:eq:static-wick-norm}.
\end{proof}

\subsection{Profile-time specialization and canonical tail}

The preceding static comparison becomes directly applicable at the cutoff
window after inserting the calibrated profile field.  By
\cref{prof:lem:one-particle-bounds,prof:lem:calibration}, uniformly over
$S_N\subset V_N$ with $|S_N|=k_N$ and $s$ in a fixed compact set,
\[
 \|v_{N,t_N(s)}^{S_N}\|_\infty\longrightarrow0,
 \qquad
 \|v_{N,t_N(s)}^{S_N}\|_{2,\mathrm{cnt}}=O(1).
\]
Consequently, the assertion of
\cref{prof:lem:static-wick-coordinate} is uniform on this calibrated profile
family: if uniformity failed, choosing a violating source and time for each
$N$ along a subsequence would produce a sequence satisfying the lemma's
hypotheses but not its conclusion.

We record the resulting profile time
specialization first, followed by the source-uniform high-chaos tail of the
canonical tilt.

\begin{proposition}[Static comparison with the finite-population Wick coordinate]
\label{prof:prop:static-fixed-degree}
For every fixed $R\in\mathbb N_0$ and nonempty compact $J\subset\mathbb R$,
\begin{equation}
 \sup_{\substack{S_N\subset V_N,\ |S_N|=k_N\\ s\in J}}
 \max_{0\le r\le R}
 \|\Psi_{N,r}(g_{N,t_N(s)}^{S_N})
   -\mathscr W_{N,r}(m_{N,t_N(s)}^{S_N})\|_2
 \longrightarrow0.
 \label{prof:eq:static-wick-comparison}
\end{equation}
\end{proposition}

\begin{proof}
Apply the uniform calibrated-profile consequence of
\cref{prof:lem:static-wick-coordinate} just established.  Uniformly over $S_N\subset V_N$ with $|S_N|=k_N$, $s\in J$, and
$0\le r\le R$,
\begin{equation}
 \Psi_{N,r}(g_{N,t_N(s)}^{S_N})
 =\mathscr W_{N,r}(\beta_Nv_{N,t_N(s)}^{S_N})+o_{N,R,J}(1).
 \label{prof:eq:static-wick-intermediate}
\end{equation}
By \cref{prof:lem:one-particle-bounds,prof:lem:calibration},
\begin{equation}
 \sup_{\substack{S_N\subset V_N,\ |S_N|=k_N\\ s\in J}}
 \|\beta_Nv_{N,t_N(s)}^{S_N}-m_{N,t_N(s)}^{S_N}\|_{2,\mathrm{cnt}}
 \longrightarrow0,
 \label{prof:eq:calibrated-field-L2}
\end{equation}
and both one-particle fields have uniformly bounded counting measure norm.

Apply the fixed-degree Lipschitz estimate
\eqref{prof:eq:fixed-wick-lipschitz} to
\(a=\beta_Nv_{N,t_N(s)}^{S_N}\) and \(b=m_{N,t_N(s)}^{S_N}\), and combine it
with \eqref{prof:eq:static-wick-intermediate} and
\eqref{prof:eq:calibrated-field-L2}.  This proves
\eqref{prof:eq:static-wick-comparison} uniformly over the stated sources.
\end{proof}

\begin{lemma}[Tilt high-chaos tail]
\label{prof:lem:tilt-tail}
For every nonempty compact $J \subset \mathbb R$,
\begin{equation*}
 \lim_{R\to\infty}\limsup_{N\to\infty}
 \sup_{\substack{S_N\subset V_N,\ |S_N|=k_N\\ s\in J}}
 \sum_{r>R}\|P_{N,r}g_{N,t_N(s)}^{S_N}\|_2^2=0.
\end{equation*}
Moreover, for each fixed $r$, uniformly over the same sources and
$s\in J$,
\begin{equation}
 \|P_{N,r}g_{N,t_N(s)}^{S_N}\|_2^2
 =\frac{(\mathfrak q_N^{S_N}(s))^r}{r!}+o_{N,r,J}(1).
 \label{prof:eq:tilt-degree-norm}
\end{equation}
\end{lemma}

\begin{proof}
For fixed $r\in\mathbb N_0$, \cref{prof:lem:static-wick-coordinate} gives, uniformly
over $S_N\subset V_N$ with $|S_N|=k_N$ and $s\in J$,
\[
 \Psi_{N,r}(g_{N,t_N(s)}^{S_N})
 =\mathscr W_{N,r}(\beta_Nv_{N,t_N(s)}^{S_N})+o_{N,r,J}(1).
\]
The norm asymptotic \eqref{prof:eq:static-wick-norm} from
\cref{prof:lem:static-wick-norm} gives
\[
 \|\mathscr W_{N,r}(\beta_Nv_{N,t_N(s)}^{S_N})\|_2^2
 =(\mathfrak q_N^{S_N}(s))^r+o_{N,r,J}(1),
\]
where \eqref{prof:eq:calibrated-q}, together with the definition
\eqref{prof:eq:profile-coordinate}, identifies the one-particle squared norm
with $\mathfrak q_N^{S_N}(s)$.  The exact chaos isometry
\eqref{prof:eq:exact-chaos-isometry} then proves
\eqref{prof:eq:tilt-degree-norm}.  On the other hand,
the uniform log MGF estimate \eqref{prof:eq:uniform-mgf-gaussian} from
\cref{prof:lem:gaussianization}, applied at $z=1,2$, yields uniformly over
$S_N$ and $s\in J$
\begin{equation}
 \|g_{N,t_N(s)}^{S_N}\|_2^2
 =\frac{\mathbb E_{\pi_N}[e^{2X_{N,v_{N,t_N(s)}^{S_N}}}]}
        {(\mathbb E_{\pi_N}[e^{X_{N,v_{N,t_N(s)}^{S_N}}}])^2}
 =e^{\mathfrak q_N^{S_N}(s)+o(1)}.
 \label{prof:eq:tilt-total-second-moment}
\end{equation}
Choose $N_J$ so that $t_N(s)\ge0$ for every $N\ge N_J$ and $s\in J$.
By \cref{prof:lem:one-particle-bounds}, there is $C_J<\infty$ such that
\[
 \sup_{N\ge N_J}
 \sup_{\substack{S_N\subset V_N,\ |S_N|=k_N\\ s\in J}}
 \mathfrak q_N^{S_N}(s)\le C_J.
\]
For fixed $R\in\mathbb N_0$, subtract the sum of \eqref{prof:eq:tilt-degree-norm} over
$r\le R$ from \eqref{prof:eq:tilt-total-second-moment}.
Orthogonality gives
\[
 \limsup_{N\to\infty}
 \sup_{\substack{S_N\subset V_N,\ |S_N|=k_N\\ s\in J}}
 \sum_{r>R}\|P_{N,r}g_{N,t_N(s)}^{S_N}\|_2^2
 \le \sup_{0\le q\le C_J}
 \left(e^q-\sum_{r=0}^R\frac{q^r}{r!}\right).
\]
The right side tends to zero as $R\to\infty$.
\end{proof}

\section{Fixed-degree comparison of independent and exclusion dynamics}
\label{prof:sec:fixed-degree}

\Cref{prof:sec:chaos-interface} constructs the common coordinate space
$\mathcal K_{N,r}^{\ord}$ in which the exact exclusion chaos coordinate and
the independent product coordinate are compared.
In this section we execute the quantitative comparison at fixed degree
$r$, and then let $N\to\infty$.  Throughout the section, $r\in\mathbb N$ is fixed.
By the density window \eqref{prof:eq:density-window}, one has
$\ell_N\ge \rho_0n_N$, so for all sufficiently large $N$ the fixed degree
satisfies $r\le\ell_N$; all asymptotic statements below are understood on
this eventual range.  

The comparison has three stages.
First, the geometric defect estimates give the asymptotically sharp spectral bottom
$r\gamma_N-O(\gamma_N/n_N)$ by a Schur argument using a single
$\gamma_N$-shifted resolvent.  Second, the contribution generated by a finite
one-particle spectral truncation is compared directly by Duhamel's formula.
Third, a one-sided estimate on low hard-core eigenvectors controls the
remaining source contribution.  The final matching theorem combines this
dynamic result with the static canonical tilt comparison already proved in
\cref{prof:prop:static-fixed-degree}.

Recall the finite-population Wick coordinate
$\mathscr W_{N,r}(m)$ from \eqref{prof:eq:canonical-wick-coordinate}.  The
dynamic comparison between $\Psi_{N,r}(h_{N,t}^{S_N})$ and
$\mathscr W_{N,r}(m_{N,t}^{S_N})$ is furnished by
\cref{prof:lem:fd-exact-duality}, namely:
\begin{equation*}
 \Psi_{N,r}(h_{N,t}^{S_N})-\mathscr W_{N,r}(m_{N,t}^{S_N})
 =\alpha_{N,r}\beta_N^{-r/2}n_N^{r/2}
 \Bigl[
 \mathsf{Top}_{N,r}^{\ord}\mathcal P_{N,r}^{\hc}(t)\mathsf R_{N,r}F_N^{\otimes r}
 -\mathsf{Top}_{N,r}^{\ord}\mathsf R_{N,r}\mathcal P_{N,r}^{\ind}(t)F_N^{\otimes r}
 \Bigr].
\end{equation*}
The object of interest is the source-normalized hard-core/product
semigroup comparison inside the brackets.

We continue to write $\|\cdot\|_{\ind}$ and $\|\cdot\|_{\hc}$ for the
ambient norms fixed above, and use $\|\cdot\|_2$ when the ambient space is
unambiguous, in particular on the common harmonic coordinate space
$\mathcal K_{N,r}^{\ord}$.

\subsection{Geometric comparison estimates}

The comparison estimates below are proved on a fixed low-energy Fourier
window, where restriction to distinct tuples and the hard-core top projection
admit uniform pointwise and discrete gradient control.  We therefore first
localize each one-particle factor to energies of order $\gamma_N$; at fixed
degree, tensor products of these modes generate the finite-dimensional
comparison space used below.

Fix $K\ge0$ and consider the one-particle spectral interval
$[0,K\gamma_N]$.  Choose a smooth function $\vartheta_K:[0,\infty)\to[0,1]$ which equals $1$ on
$[0,K]$ and $0$ on $[K+1,\infty)$.  Let
\begin{equation}
 \Xi_{N,K}^{(1)}=\vartheta_K(\gamma_N^{-1}L_N^{\RW}),
 \qquad
 \Xi_{N,K}^{(1)}e_p=
 \begin{cases}
   e_p,& \lambda_N(p)\le K\gamma_N,\\
   0,& \lambda_N(p)\ge (K+1)\gamma_N.
 \end{cases}
 \label{prof:eq:fd-one-particle-cutoff}
\end{equation}
Thus $\Xi_{N,K}^{(1)}$ is the smooth low-energy spectral multiplier associated
with $\vartheta_K$: it acts as the identity below $K\gamma_N$, vanishes above
$(K+1)\gamma_N$, and interpolates on the intervening spectral band.
Recall the Fourier convention
\eqref{prof:eq:probability-fourier-basis}.  For $K_+\ge0$, let
$\mathscr V_{N,r}^{\ind}(K_+)\subset\mathcal T_{N,r}^{\ind}$ be the span of the
permutation symmetrizations of exactly $r$-fold character tensors whose total energy does not exceed $K_+ \gamma_N$:
\[
 e_{p_1}\otimes\cdots\otimes e_{p_r},
 \qquad p_a\ne0\ \text{for every }1\le a\le r,
 \qquad \sum_{a=1}^r\lambda_N(p_a)\le K_+\gamma_N.
\]
Above each factor is nonconstant, and the Fourier frequencies are allowed to repeat.
The dimension of $\mathscr V_{N,r}^{\ind}(K_+)$ is uniformly bounded for fixed $D,r,K_+$.  Define
\begin{align}
 \mathsf J_{N,r,K_+}:
 \mathscr V_{N,r}^{\ind}(K_+)\longrightarrow\mathcal T_{N,r}^{\hc},
 \qquad
 \mathsf J_{N,r,K_+}u
 :=\mathsf{Top}_{N,r}^{\ord}\mathsf R_{N,r}u.
 \label{prof:eq:JNrK}
\end{align}
Thus $\mathsf J_{N,r,K_+}$ first restricts a low-energy product mode to
ordered distinct tuples and then removes its lower hard-core Hoeffding
components.

For $q\ge1$, two tuples $\boldsymbol x,\boldsymbol y\in V_N^q$ are
\emph{product neighbors}, written $\boldsymbol x\sim\boldsymbol y$, if they
differ in exactly one coordinate and that coordinate moves across one
nearest-neighbor edge of $V_N$.  Whenever
$\sum_{\boldsymbol x\sim\boldsymbol y}$ is used for the product graph, it
runs over ordered neighboring pairs; the corresponding Dirichlet form
therefore carries the usual factor $1/2$.
For $q\ge1$, define the collision diagonal
\begin{equation*}
 \mathfrak D_{N,q}^{\mathrm{coll}}
 :=V_N^q\setminus\Omega_{N,q}^{\ord}
 = \{\boldsymbol x\in V_N^q: x_a=x_b \text{ for some } 1\le a<b\le q\}.
\end{equation*}
A blocked hard-core move can occur only when two coordinates are nearest
neighbors.  Accordingly, for $q\ge2$, define the corresponding contact set
\begin{equation}
 \mathfrak C_{N,q}^{\mathrm{nn}}
 :=\bigcup_{1\le a<b\le q}
   \{\boldsymbol x\in V_N^q:x_a\sim x_b\}.
 \label{prof:eq:nearest-neighbor-contact-set}
\end{equation}
Thus $\mathfrak D_{N,q}^{\mathrm{coll}}$ records repeated coordinates,
whereas $\mathfrak C_{N,q}^{\mathrm{nn}}$ records tuples on which at least
one selected pair of coordinates is adjacent.  We also write
\begin{equation}
 \theta_{N,q}:=\frac{(n_N)_{\underline q}}{n_N^q},
 \qquad 0\le q\le n_N,
 \label{prof:eq:fd-distinct-tuple-density}
\end{equation}
for the product probability that a uniformly sampled ordered $q$-tuple has
distinct coordinates.  In particular,
$\theta_{N,q}=1-O_q(n_N^{-1})$ at fixed $q$.
We write
$
 \mathcal D_D:=\{\pm\mathbf e_1,\ldots,\pm\mathbf e_D\}
$
for the set of nearest-neighbor displacement vectors.

\begin{lemma}[Top-compatible restriction of smooth product modes]
\label{prof:lem:fd-restriction}
For fixed $r\in\mathbb N$ and $K_+\ge0$, uniformly for
$u,v\in \mathscr V_{N,r}^{\ind}(K_+)$,
\begin{align}
 \left|\langle\mathsf J_{N,r,K_+}u,\mathsf J_{N,r,K_+}v\rangle_{\hc}
       -\langle u,v\rangle_{\ind}\right|
 &\le \frac{C_{D,r,K_+}}{n_N}\|u\|_{\ind}\|v\|_{\ind},
 \label{prof:eq:fd-gram-error}\\
 \left|\mathcal E_{N,r}^{\hc}(\mathsf J_{N,r,K_+}u,\mathsf J_{N,r,K_+}v)
       -\mathcal E_{N,r}^{\ind}(u,v)\right|
 &\le C_{D,r,K_+}\frac{\gamma_N}{n_N}
       \|u\|_{\ind}\|v\|_{\ind}.
 \label{prof:eq:fd-form-error}
\end{align}
\end{lemma}

\begin{proof}
The energy restriction defining $\mathscr V_{N,r}^{\ind}(K_+)$ and the dispersion
bound \eqref{prof:eq:dispersion-comparison} imply
$|\widetilde p_a|\le C_{K_+}$ for every character factor.  Thus the window
contains only $O_{D,r,K_+}(1)$ product Fourier modes.  Since $|e_p|=1$ and,
for every nearest-neighbor displacement $\boldsymbol\delta\in\mathcal D_D$,
\[
 |e_p(x+\boldsymbol\delta)-e_p(x)|\le C\frac{|\widetilde p|}{N},
\]
finite-dimensional Fourier norm equivalence gives
\begin{equation}
 \|u\|_\infty\le C_{D,r,K_+}\|u\|_{\ind},
 \qquad
 \max_{\boldsymbol x\sim\boldsymbol y}|u(\boldsymbol x)-u(\boldsymbol y)|
 \le C_{D,r,K_+}N^{-1}\|u\|_{\ind},
 \qquad
 \dim\mathscr V_{N,r}^{\ind}(K_+)\le C_{D,r,K_+}.
 \label{prof:eq:fd-finite-window-fourier-bounds}
\end{equation}
These are the finite-window $L^2$-to-$L^\infty$ bound, its discrete gradient
counterpart, and the uniform dimension bound used below.

We first compare restriction with the product geometry.  Since both spaces
carry uniform probability measures,
\begin{equation}
 \langle\mathsf R_{N,r}u,\mathsf R_{N,r}v\rangle_{\hc}
 =\theta_{N,r}^{-1}
  \langle\one_{\Omega_{N,r}^{\ord}}u,v\rangle_{\ind}.
 \label{prof:eq:fd-restriction-normalization}
\end{equation}
The collision diagonal has product probability $O_r(n_N^{-1})$.  Rewriting
\eqref{prof:eq:fd-restriction-normalization} as
\[
 \langle\mathsf R_{N,r}u,\mathsf R_{N,r}v\rangle_{\hc}
 -\langle u,v\rangle_{\ind}
 =\bigl(\theta_{N,r}^{-1}-1\bigr)\langle u,v\rangle_{\ind}
  -\theta_{N,r}^{-1}
   \langle\one_{\mathfrak D_{N,r}^{\mathrm{coll}}}u,v\rangle_{\ind},
\]
the finite-window $L^\infty$ bound in
\eqref{prof:eq:fd-finite-window-fourier-bounds} and
$\theta_{N,r}=1-O_r(n_N^{-1})$ give the quantitative restriction estimate
\begin{align}
 \left|
 \langle\mathsf R_{N,r}u,\mathsf R_{N,r}v\rangle_{\hc}
 -\langle u,v\rangle_{\ind}
 \right|
 \le \frac{C_{D,r,K_+}}{n_N}\|u\|_{\ind}\|v\|_{\ind}.
 \label{prof:eq:restriction-est}
\end{align}
Thus, before the top-Hoeffding projection is applied, restriction itself
already changes the inner product by only $O_{D,r,K_+}(n_N^{-1})$.

The same comparison holds for the
Dirichlet form.  Let $\mathcal E_{N,r}^{\mathrm{coll}}(u,v)$ denote the
product form contribution of edges that are not allowed hard-core edges
between two distinct tuples.  Then
\[
 \mathcal E_{N,r}^{\hc}(\mathsf R_{N,r}u,\mathsf R_{N,r}v)
 =\theta_{N,r}^{-1}
  \bigl(\mathcal E_{N,r}^{\ind}(u,v)
        -\mathcal E_{N,r}^{\mathrm{coll}}(u,v)\bigr).
\]
The product spectrum on $\mathscr V_{N,r}^{\ind}(K_+)$ is at most
$K_+\gamma_N$, so Cauchy--Schwarz for the product Dirichlet form gives
\[
 |\mathcal E_{N,r}^{\ind}(u,v)|
 \le K_+\gamma_N\|u\|_{\ind}\|v\|_{\ind}.
\]
Moreover, the product edge mass incident to a collision is
$O_{D,r}(n_N^{-1})$, and the discrete gradient estimate in
\eqref{prof:eq:fd-finite-window-fourier-bounds} gives
\[
 |\mathcal E_{N,r}^{\mathrm{coll}}(u,v)|
 \le C_{D,r,K_+}\frac{\gamma_N}{n_N}\|u\|_{\ind}\|v\|_{\ind}.
\]
Combining these bounds with
$\theta_{N,r}=1+O_r(n_N^{-1})$ yields
\begin{align}
 \left|
 \mathcal E_{N,r}^{\hc}(\mathsf R_{N,r}u,\mathsf R_{N,r}v)
 -\mathcal E_{N,r}^{\ind}(u,v)
 \right|
 \le C_{D,r,K_+}\frac{\gamma_N}{n_N}
       \|u\|_{\ind}\|v\|_{\ind}.
\label{prof:eq:energy-restriction}
\end{align}
Hence restriction alone incurs an $O_{D,r,K_+}(n_N^{-1})$ inner product
error and an $O_{D,r,K_+}(\gamma_Nn_N^{-1})$ form error.  

The remaining task
is only to remove the lower hard-core Hoeffding components created by
restriction without worsening these scales.
Define the down defect
\[
 \mathfrak d_u:=\mathsf{Down}_{N,r\to r-1}\mathsf R_{N,r}u.
\]
Because $u$ has zero average in each product coordinate,
\begin{equation}
 \mathfrak d_u(x_1,\ldots,x_{r-1})
 =\frac1{n_N-r+1}\sum_{z\notin \{x_1,\ldots, x_{r-1}\}}
   u(x_1,\ldots,x_{r-1},z)
   =-\frac1{n_N-r+1}\sum_{j=1}^{r-1}
   u(x_1,\ldots,x_{r-1},x_j).
 \label{prof:eq:fd-down-defect}
\end{equation}
Hence the two estimates in \eqref{prof:eq:fd-finite-window-fourier-bounds}
give
\[
 \|\mathfrak d_u\|_{\hc}
 \le C_{D,r,K_+}n_N^{-1}\|u\|_{\ind},
 \qquad
 \mathcal E_{N,r-1}^{\hc}(\mathfrak d_u)
 \le C_{D,r,K_+}\gamma_Nn_N^{-2}\|u\|_{\ind}^2.
\]
The projection formula \eqref{prof:eq:fd-top-projection-formula} yields
\begin{equation}
 (\mathsf J_{N,r,K_+}-\mathsf R_{N,r})u
 =-\mathsf{Up}_{N,r-1\to r}
 \bigl(\mathsf{Down}_{N,r\to r-1}
       \mathsf{Up}_{N,r-1\to r}\bigr)^{-1}\mathfrak d_u.
 \label{prof:eq:fd-lower-layer-correction}
\end{equation}
By \eqref{prof:eq:fd-down-up-eigenvalue}, the down--up operator here has
spectrum in $[c_r,1]$ for all sufficiently large $N$; by
\eqref{prof:eq:hard-core-down-intertwining} and its adjoint, it and its
inverse commute with $\mathcal L_{N,r-1}^{\hc}$.  Since this inverse is
uniformly bounded, adjointness of $\mathsf{Up}$ and $\mathsf{Down}$ transfers,
respectively, the $L^2$ and Dirichlet form bounds from
\eqref{prof:eq:fd-down-defect} and \eqref{prof:eq:fd-finite-window-fourier-bounds}
to the correction:
\begin{equation}
 \| (\mathsf J_{N,r,K_+}-\mathsf R_{N,r})u\|_{\hc}
 \le C_{D,r,K_+}n_N^{-1}\|u\|_{\ind}.
 \label{prof:eq:fd-top-correction-L2}
\end{equation}
\begin{equation}
 \mathcal E_{N,r}^{\hc}
 \bigl((\mathsf J_{N,r,K_+}-\mathsf R_{N,r})u\bigr)
 \le C_{D,r,K_+}\gamma_Nn_N^{-2}\|u\|_{\ind}^2.
 \label{prof:eq:fd-top-correction-energy}
\end{equation}
To finish, write
$\mathsf J_{N,r,K_+}=\mathsf R_{N,r}+
(\mathsf J_{N,r,K_+}-\mathsf R_{N,r})$.  For the inner-product comparison,
\eqref{prof:eq:restriction-est} and \eqref{prof:eq:fd-top-correction-L2},
followed by Cauchy--Schwarz, make each mixed term
$O_{D,r,K_+}(n_N^{-1})\|u\|_{\ind}\|v\|_{\ind}$ and the
correction--correction term $O_{D,r,K_+}(n_N^{-2})$.  Together with the
restriction error this proves \eqref{prof:eq:fd-gram-error}.  The same
expansion for the Dirichlet form uses \eqref{prof:eq:energy-restriction},
the $O_{D,r,K_+}(\gamma_N)$ restriction energy on the fixed Fourier window,
and \eqref{prof:eq:fd-top-correction-energy}; Cauchy--Schwarz then makes
each mixed term $O_{D,r,K_+}(\gamma_Nn_N^{-1})$ and the pure correction
$O_{D,r,K_+}(\gamma_Nn_N^{-2})$, after factoring out
$\|u\|_{\ind}\|v\|_{\ind}$.  This proves
\eqref{prof:eq:fd-form-error}.
\end{proof}

Our next target is a uniformly bounded extension of functions on the hard-core space to the full product space, stated as \cref{prof:lem:fd-local-fill} below.
To prepare for its proof, we need a bounded displacement routing lemma on the hard-core configuration space.
For this local routing argument, let $\mathrm{pr}_N:\mathbb Z^D\to V_N$ denote the quotient map.
For \(x,y\in V_N = (\mathbb{Z}/N\mathbb{Z})^D\), choose arbitrary integer lifts
\(\widetilde x,\widetilde y\in\mathbb Z^D\) satisfying
\(\mathrm{pr}_N(\widetilde x)=x\) and \(\mathrm{pr}_N(\widetilde y)=y\), and write
\begin{equation}
 d_N^\infty(x,y)
 :=\min_{z\in N\mathbb Z^D}
   \|\widetilde x-\widetilde y+z\|_\infty,
 \qquad
 B_N^\infty(x,R):=\{y\in V_N:d_N^\infty(x,y)\le R\}.
 \label{prof:eq:torus-linfty-metric}
\end{equation}
These definitions do not depend on the chosen lifts.

\begin{lemma}[Bounded-displacement routing]
\label{prof:lem:fd-local-routing}
Fix \(D\ge2\), \(r\in\mathbb N\), and \(R\ge0\).  There are
\(C=C(D,r,R)<\infty\) and \(N_0=N_0(D,r,R)\) with the following property.
If \(N\ge N_0\) and
\(\boldsymbol z,\boldsymbol w\in\Omega_{N,r}^{\ord}\) satisfy
\[
 \max_{i\in[r]}d_N^\infty(z_i,w_i)\le R,
\]
then they can be joined by a path
\(\boldsymbol z=\boldsymbol z^0,\ldots,\boldsymbol z^L=\boldsymbol w\)
in the ordered exclusion configuration graph such that
\begin{equation}
 L\le C,
 \qquad
 \max_{\substack{0\le a\le L\\i\in[r]}}
 d_N^\infty(z_i^a,z_i)\le C.
 \label{prof:eq:fd-local-routing}
\end{equation}
The path may be chosen deterministically from
\((\boldsymbol z,\boldsymbol w)\).
\end{lemma}

\begin{proof}
Put $b_{\mathrm{buf}}:=2r+4$.  The routing construction involves two stages.  First we
cluster together labels whose bounded source--target neighborhoods overlap;
each resulting cluster will lie in a uniformly bounded torus box whose graph
lifts without wraparound to an ordinary box in $\mathbb Z^D$.  Inside each
lifted box we then route the source and target occupied sets to the same
reference strip, correct the label order using an adjacent vacant strip, and
reverse the target routing.  The bookkeeping below makes these choices
uniform and deterministic.

To be precise now, we call $Q\subset V_N$ a \emph{local torus box} if
$Q=\mathrm{pr}_N(\widetilde Q)$ for an axis-parallel integer box
\[
 \widetilde Q=\prod_{a=1}^D([A_a,B_a]\cap\mathbb Z)
 \subset\mathbb Z^D,
 \quad \text{ where } B_a-A_a<\frac{N}{2} \quad \text{ for every } 1\le a\le D.
\]
For such a box, the restriction $\mathrm{pr}_N|_{\widetilde Q}$ is injective:
two points with the same image differ by an element of $N\mathbb Z^D$, while
each coordinate difference inside $\widetilde Q$ has absolute value $<N/2$.
Moreover, two points of $\widetilde Q$ project to nearest neighbors in $V_N$
only when they are already nearest neighbors in $\mathbb Z^D$.  Thus
$\mathrm{pr}_N|_{\widetilde Q}$ is a graph isomorphism onto $Q$; in
particular, a local torus box has exactly the graph geometry of its integer
representative, with no wraparound ambiguity.  Finally, any two integer box
representatives of the same $Q$ satisfying these side bounds differ by a
common translation in $N\mathbb Z^D$: after one pair of corresponding points
is aligned, injectivity and connectedness force the same translation across
the whole box.

\localheading{Initial boxes and the merger rule.}
Let \(\boldsymbol z,\boldsymbol w\in\Omega_{N,r}^{\ord}\) be given.
For each $i\in[r]$, choose the distance-realizing displacement
$\boldsymbol\delta_i=(\delta_{i,a})_{a=1}^D\in\mathbb Z^D$ with
\[
 w_i=z_i+\boldsymbol\delta_i\pmod{N\mathbb Z^D},
 \qquad
 \|\boldsymbol\delta_i\|_\infty=d_N^\infty(z_i,w_i)\le R.
\]
We consider all $N>2R$, in which case
$\boldsymbol\delta_i$ is unique.  Let
\[
 \widetilde Q_i^0
 :=\prod_{a=1}^D
 \bigl([\min\{0,\delta_{i,a}\}-b_{\mathrm{buf}},
          \max\{0,\delta_{i,a}\}+b_{\mathrm{buf}}]\cap\mathbb Z\bigr),
 \qquad
 Q_i^0:=z_i+\mathrm{pr}_N(\widetilde Q_i^0).
\]
Then $Q_i^0$ contains both $z_i$ and $w_i$; every coordinate side has at
least $2b_{\mathrm{buf}}+1=4r+9$ sites and at most
$
 w_{\mathrm{box}}:=R+2b_{\mathrm{buf}}+1
$
sites.

Next we define the merger rule for the local torus boxes $(Q_i^0)_i$.
Suppose
$Q=\mathrm{pr}_N(\widetilde Q)$ and
$Q'=\mathrm{pr}_N(\widetilde Q')$ intersect, and the coordinate sides of
$\widetilde Q$ and $\widetilde Q'$ contain at most $L$ and $L'$ sites,
respectively, with $L+L'<N/2$.  The intersection produces a translation vector
$\mathbf v\in N\mathbb Z^D$ for which $\widetilde Q$ and $\widetilde Q'+\mathbf v$
intersect.  This $\mathbf v$ is unique: if both $\mathbf v$ and $\mathbf v'$ work, then every
coordinate of $\mathbf v-\mathbf v'\in N\mathbb Z^D$ has absolute value less than
$L+L'<N/2$, hence $\mathbf v=\mathbf v'$.  Define $Q\vee Q'$ as the projection of the
smallest axis-parallel integer box containing
$\widetilde Q\cup(\widetilde Q'+\mathbf v)$.  Its coordinate sides contain at most
$L+L'$ sites.  Replacing either integer representative merely translates the
aligned pair by a common element of $N\mathbb Z^D$, so $Q\vee Q'$ is a
well-defined local torus box.

Observe that the merger rule preserves one simple invariant: a current box formed
from $m$ initial boxes contains the source and target sites belonging to
those $m$ labels, and each of its coordinate sides contains at most $m w_{\mathrm{box}}$
sites. 

Set $C_{\rm box}:=r w_{\mathrm{box}}$ and enlarge $N_0=N_0(D,r,R)$ so that
$N>\max\{2R,2C_{\rm box}\}$ whenever $N\ge N_0$.  Starting from
$Q_1^0,\ldots,Q_r^0$, repeatedly merge the lexicographically first
intersecting pair, with current boxes ordered by their constituent label
sets.  The invariant holds initially.  If boxes with $m$ and $m'$ initial
constituents are merged, then
$(m+m')w_{\mathrm{box}}\le r w_{\mathrm{box}}=C_{\rm box}<N/2$, so the merger rule applies and the
new box satisfies the same invariant with $m+m'$ constituents.  After at
most $r-1$ mergers, the surviving local torus boxes are pairwise disjoint.
Each therefore contains all source and target sites of its label cluster,
and its diameter and cardinality are bounded only in terms of $D,r,R$.

\localheading{Routing inside one final box.}
Fix one final torus box $Q$, and let $I_Q\subseteq[r]$ be its constituent
labels, $r_Q:=|I_Q|$, and $i_Q:=\min I_Q$.  Among the integer representatives
of $Q$, choose the unique one $\widetilde Q$ containing the canonical spatial
representative of $z_{i_Q}$.  By the local box observation above,
$\mathrm{pr}_N|_{\widetilde Q}:\widetilde Q\to Q$ is a graph isomorphism,
so we may route in $\widetilde Q$ and project every move back to $Q$.  Denote
by $\widetilde z_i,\widetilde w_i\in\widetilde Q$ the unique preimages of the
assigned source and target sites.

The initial buffer gives every side of $\widetilde Q$ at least $4r+9$ sites,
and mergers never decrease side lengths.  Write
\[
 \widetilde Q=\prod_{a=1}^D([A_a,B_a]\cap\mathbb Z),
\]
and, since $D\ge2$, set the point
\[
 \mathbf q:=(A_1+2r+4,\ldots,A_D+2r+4)\in\widetilde Q.
\]
Because $r_Q\le r$, both
\[
 P=\{\mathbf q,\mathbf q+\mathbf e_1,\ldots,\mathbf q+(r_Q-1)\mathbf e_1\},
 \qquad P+\mathbf e_2,
\]
lie in $\widetilde Q$.  We use $P$ as the \emph{reference strip} and
$P+\mathbf e_2$ as the adjacent \emph{buffer strip}.

The routing now has three phases: move the source occupied set to $P$,
reorder the labels along $P$, and then reverse a route from the target occupied set to $P$.  For
the first and third phases, forget the labels temporarily.  The unlabeled
$r_Q$-particle exclusion graph on the finite connected graph $\widetilde Q$
is connected because $r_Q<|\widetilde Q|$.  Indeed, a vacancy can be moved
along a spanning tree, and induction over its leaves places the particles at
any prescribed occupied set.  Choose, with a fixed lexicographic tie-breaking
rule, a shortest unlabeled path $\Gamma_{\boldsymbol z}^{\mathrm{route}}$ from the occupied set
of $(\widetilde z_i)_{i\in I_Q}$ to $P$, and a shortest unlabeled path
$\Gamma_{\boldsymbol w}^{\mathrm{route}}$ from the occupied set of
$(\widetilde w_i)_{i\in I_Q}$ to $P$.  Each occupied--vacant move transports
the unique label at its occupied endpoint, so these unlabeled paths determine
labeled paths.  At their endpoints the source and target labels appear in
possibly different orderings $\ell_{\boldsymbol z}$ and
$\ell_{\boldsymbol w}$ along $P$.

At either endpoint the occupied set is exactly $P$, so the buffer strip
$P+\mathbf e_2$ is vacant.  It can therefore be used to interchange adjacent
labels on $P$ by the four legal moves on the corresponding $2\times2$ face:
\[
 \begin{matrix} A&B\\ \circ&\circ\end{matrix}
 \longrightarrow
 \begin{matrix} \circ&B\\ A&\circ\end{matrix}
 \longrightarrow
 \begin{matrix} B&\circ\\ A&\circ\end{matrix}
 \longrightarrow
 \begin{matrix} B&\circ\\ \circ&A\end{matrix}
 \longrightarrow
 \begin{matrix} B&A\\ \circ&\circ\end{matrix}.
\]
A fixed adjacent transposition algorithm, for instance bubble sort, changes
$\ell_{\boldsymbol z}$ into $\ell_{\boldsymbol w}$ using at most
$r_Q(r_Q-1)/2$ such swaps.  Reversing the labeled path determined by
$\Gamma_{\boldsymbol w}^{\mathrm{route}}$ then reaches the target labeled configuration in
$Q$ exactly.

\localheading{Uniformity and determinism.}
Because the final torus boxes are pairwise disjoint, routing one label
cluster leaves every other cluster untouched.  Apply the preceding
construction successively to the final boxes, in lexicographic order of
their constituent label sets.  For fixed $D,r,R$, each final
box has uniformly bounded cardinality, so the selected shortest unlabeled
paths have uniformly bounded length; the total number of adjacent
transpositions is at most $r(r-1)/2$.  Every intermediate coordinate remains
inside the final box containing its source, whose diameter is uniformly
bounded as well.  These two observations give both bounds in
\eqref{prof:eq:fd-local-routing}.  Finally, every choice above was made by a
fixed deterministic rule, so the resulting path is a deterministic function
of $(\boldsymbol z,\boldsymbol w)$.
\end{proof}

With the routing mechanism of \cref{prof:lem:fd-local-routing} available,
we proceed to the extension result.

\begin{lemma}[Local extension]
\label{prof:lem:fd-local-fill}
Fix $D\ge2$ and $r$.  For all sufficiently large \(N\), there is a
particle-label-equivariant linear extension
\(
 \mathsf{Ext}_{N,r}:\mathsf X_{N,r}^{\hc}\to\mathsf X_{N,r}^{\ind}
\)
such that
\(
 \mathsf R_{N,r}\mathsf{Ext}_{N,r}=I,
\)
\begin{align}
 \|\mathsf{Ext}_{N,r}f\|_{\ind}^2
 &\le C_{D,r}\|f\|_{\hc}^2, \quad \text{and} \label{prof:eq:ext-local-fill}\\
 \mathcal E_{N,r}^{\ind}(\mathsf{Ext}_{N,r}f)
 &\le C_{D,r}\mathcal E_{N,r}^{\hc}(f).
 \label{prof:eq:fd-local-fill}
\end{align}
\end{lemma}

\begin{proof}
We first define a deterministic local collision-resolution map.  Given an
arbitrary product tuple, its role is to leave one label at each already
occupied site and move only the remaining duplicate labels to nearby vacant
sites, producing a distinct tuple at uniformly bounded displacement.  The
precise tie-breaking rule below is chosen only to make this resolution
deterministic.

Let
$
 \mathcal A_r=\{a\mathbf e_1+b\mathbf e_2: a,b\in \mathbb{Z},~0\le a,b\le2r\}
 $,
and order $\mathcal A_r$ lexicographically by $(a,b)$.  For
$\boldsymbol x=(x_1,\ldots,x_r) \in V_N^r$, write
$
 \mathsf{Occ}(\boldsymbol x):=\{x_1,\ldots,x_r\}$.
At each site $z\in \mathsf{Occ}(\boldsymbol x)$, designate the smallest index
$i$ with $x_i=z$ as the keeper of $z$.  Set $y_i=x_i$ for every keeper.
Process the remaining, nonkeeper labels in increasing index order.  For a
nonkeeper $i$, let $\mathbf a_i$ be the first element of $\mathcal A_r$ such that
\begin{equation}
 x_i+\mathbf a_i\notin
 \mathsf{Occ}(\boldsymbol x)
 \cup
 \{y_j:j<i\text{ and }j\text{ is a nonkeeper}\},
 \label{prof:eq:fd-resolution-reservation}
\end{equation}
and set $y_i=x_i+\mathbf a_i$.  Define
$\mathsf{Res}_0\boldsymbol x=(y_1,\ldots,y_r)$.
Thus $\mathsf{Res}_0$ is intended to pull apart collisions locally while fixing every
tuple that is already hard-core.

This rule is well defined for all sufficiently large $N$.  Indeed, if
$r_{\mathrm{occ}}:=|\mathsf{Occ}(\boldsymbol x)|$, then there are $r-r_{\mathrm{occ}}$ nonkeepers.  When a given
nonkeeper is processed, the forbidden set in
\eqref{prof:eq:fd-resolution-reservation} has cardinality at most
$
 r_{\mathrm{occ}}+(r-r_{\mathrm{occ}}-1)=r-1
 $,
whereas $|\mathcal A_r|=(2r+1)^2>r-1$.  For $N>4r$, the translate
$x_i+\mathcal A_r$ contains $|\mathcal A_r|$ distinct torus sites, so an
admissible $\mathbf a_i$ exists.  The keeper locations are the distinct sites in
$\mathsf{Occ}(\boldsymbol x)$, every nonkeeper target avoids all of those sites, and
each newly chosen nonkeeper target avoids all earlier nonkeeper targets.  Hence
$\mathsf{Res}_0\boldsymbol x\in\Omega_{N,r}^{\ord}$.  If $\boldsymbol x$ is already
distinct, every label is a keeper and $\mathsf{Res}_0\boldsymbol x=\boldsymbol x$.
Moreover,
\[
 \max_{i\in[r]}d_N^\infty(x_i,y_i)\le R_r:=2r.
\]

For a permutation $\varpi\in\Sym([r])$ set
$\mathsf{Res}_\varpi:=\mathsf P_\varpi^{-1}\mathsf{Res}_0\mathsf P_\varpi$. We define
\begin{align}
 (\mathsf{Ext}_{N,r}f)(\boldsymbol x)
 =\frac1{r!}\sum_{\varpi\in\Sym([r])}f(\mathsf{Res}_\varpi\boldsymbol x).
 \label{prof:eq:ExtNr}
\end{align}
Since every $\mathsf{Res}_\varpi$ fixes every distinct tuple,
$\mathsf R_{N,r}\mathsf{Ext}_{N,r}=I$.  Averaging these conjugates of $\mathsf{Res}_0$
makes $\mathsf{Ext}_{N,r}$ equivariant under particle-label permutations.

Every fiber of $\mathsf{Res}_\varpi$ has size at most
\begin{equation}
 B_{D,r}:=(4r+1)^{Dr}=|B_N^\infty(0,R_r)|^r,
 \label{prof:eq:fd-resolution-fiber}
\end{equation}
where the equality uses $R_r=2r$ and the standing eventual range $N>4r$.
Indeed, $\mathsf{Res}_\varpi\boldsymbol x=\boldsymbol z$ implies
$d_N^\infty(x_i,z_i)\le R_r$ for every $i$.  Jensen's inequality and the
probability normalizations thus give \eqref{prof:eq:ext-local-fill}:
\[
\|\mathsf{Ext}_{N,r}f\|_{\ind}^2
 \le \frac1{r!}\sum_{\varpi\in\Sym([r])}
      \frac1{n_N^r}\sum_{\boldsymbol x}
      |f(\mathsf{Res}_\varpi\boldsymbol x)|^2
 \le B_{D,r}\frac{(n_N)_{\underline r}}{n_N^r}
      \|f\|_{\hc}^2
 \le B_{D,r}\|f\|_{\hc}^2.
\]

Conjugation by $\mathsf P_\varpi$ merely relabels the coordinates, so the
preceding displacement estimate for $\mathsf{Res}_0$ gives
\[
 \max_{i\in[r]}d_N^\infty\bigl(x_i,(\mathsf{Res}_\varpi\boldsymbol x)_i\bigr)
 \le R_r,
 \qquad
 \max_{i\in[r]}d_N^\infty\bigl(y_i,(\mathsf{Res}_\varpi\boldsymbol y)_i\bigr)
 \le R_r.
\]
If $\boldsymbol x\sim\boldsymbol y$ in the product graph, then
$\max_{i\in [r]} d_N^\infty(x_i,y_i)\le1$.  The triangle inequality therefore gives
\[
 \max_{i\in[r]}d_N^\infty
 \bigl((\mathsf{Res}_\varpi\boldsymbol x)_i,(\mathsf{Res}_\varpi\boldsymbol y)_i\bigr)
 \le 2R_r+1.
\]
By \cref{prof:lem:fd-local-routing}, the distinct tuples
$\mathsf{Res}_\varpi\boldsymbol x$ and $\mathsf{Res}_\varpi\boldsymbol y$ are connected by
a deterministically chosen hard-core path
$\Gamma_\varpi(\boldsymbol x,\boldsymbol y)$ of uniformly bounded length.
For an edge $e=\{\boldsymbol z,\boldsymbol z'\}$ of the ordered hard-core
configuration graph, write
$|\nabla_e f|:=|f(\boldsymbol z')-f(\boldsymbol z)|$; this is independent of
which endpoint is listed first.  The path inequality then gives
\begin{equation}
 |f(\mathsf{Res}_\varpi\boldsymbol x)-f(\mathsf{Res}_\varpi\boldsymbol y)|^2
 \le C_{D,r}\sum_{e\in\Gamma_\varpi(\boldsymbol x,\boldsymbol y)}
       |\nabla_e f|^2.
 \label{prof:eq:fd-resolution-path}
\end{equation}
This path family has uniformly bounded congestion.  To see this, fix a
hard-core configuration edge $e$.  If
$e\in\Gamma_\varpi(\boldsymbol x,\boldsymbol y)$, then both tuples $\mathsf{Res}_\varpi\boldsymbol x$ and $\mathsf{Res}_\varpi\boldsymbol y$ lie within a
configuration graph distance bounded by $C_{D,r}$ from an endpoint of $e$.
Since the hard-core configuration graph has degree at most $2Dr$, there are
at most $C_{D,r}$ possible ordered pairs of such tuples.  By
\eqref{prof:eq:fd-resolution-fiber}, each such tuple has at most
$B_{D,r}$ preimages under $\mathsf{Res}_\varpi$; for each $\boldsymbol x$ there are at
most $2Dr$ product neighbors $\boldsymbol y$, and there are only $r!$
particle-label permutations $\varpi$.  Thus every hard-core edge is used by at most
$C_{D,r}B_{D,r}^2r!$ triples
$(\varpi,\boldsymbol x,\boldsymbol y)$.  Applying Jensen's inequality first
to the average over $\varpi$, summing
\eqref{prof:eq:fd-resolution-path} over product edges, and then using this
congestion bound gives
\[
 \frac1{2n_N^r}
 \sum_{\boldsymbol x\sim\boldsymbol y}
 |\mathsf{Ext}_{N,r}f(\boldsymbol x)
  -\mathsf{Ext}_{N,r}f(\boldsymbol y)|^2
 \le
 C_{D,r}\frac{(n_N)_{\underline r}}{n_N^r}
 \mathcal E_{N,r}^{\hc}(f).
\]
Recognizing that the left-hand side is
$\mathcal E_{N,r}^{\ind}(\mathsf{Ext}_{N,r}f)$ and using
$\theta_{N,r}\le1$ from \eqref{prof:eq:fd-distinct-tuple-density}, we obtain
\eqref{prof:eq:fd-local-fill}.
\end{proof}

We now use \cref{prof:lem:fd-local-fill} to prove a Schur form and resolvent comparison.
Recall from \eqref{prof:eq:restriction-adjoint} that $\mathsf R_{N,r}^*$ is the
normalization-adjusted zero extension, with factor
$n_N^r/(n_N)_{\underline r}=\theta_{N,r}^{-1}$.
 In the product space, the range
$\Ran\mathsf R_{N,r}^*$ is the subspace supported on distinct tuples, while
$(\Ran\mathsf R_{N,r}^*)^\perp$ is the repeated-tuple subspace.  
Based on the orthogonal decomposition
\[
 \mathsf X_{N,r}^{\ind}
 =\Ran\mathsf R_{N,r}^*\oplus(\Ran\mathsf R_{N,r}^*)^\perp,
\]
we consider the Schur complement of
$\mathcal L_{N,r}^{\ind}+\gamma_N$ onto the distinct-tuple component $\Ran\mathsf R_{N,r}^*$.
We transport the Schur form to $\mathsf X_{N,r}^{\hc}$ and denote
the resulting positive definite operator by $\mathsf S_{N,r,\gamma_N}$.  
The variational characterization is
\begin{equation}
 \langle f,\mathsf S_{N,r,\gamma_N}f\rangle_{\hc}
 =\inf_{\substack{u\in\mathsf X_{N,r}^{\ind}\\
                   \mathsf R_{N,r}u=f}}
   \left\{\mathcal E_{N,r}^{\ind}(u)+\gamma_N\|u\|_{\ind}^2\right\}.
 \label{prof:eq:fd-schur-variational-form}
\end{equation}
Thus the Schur form is the least shifted independent-particle energy needed to
extend prescribed hard-core data across the collision set.  
For $\star\in\{\ind,\hc\}$ and $h\in\mathsf X_{N,r}^{\star}$, write
\begin{equation}
 \|h\|_{-1,\star,\gamma_N}^2
 :=\langle h,(\mathcal L_{N,r}^{\star}+\gamma_N)^{-1}h\rangle_{\star}.
 \label{prof:eq:fd-shifted-resolvent-norm}
\end{equation}
Since $\gamma_N>0$, $\|\cdot\|_{-1,\star,\gamma_N}$ defines a norm
on each of the product and hard-core spaces.

\begin{lemma}[Schur and resolvent comparison]
\label{prof:lem:fd-schur-resolvent}
Fix $D\ge2$ and $r$.  For all sufficiently large $N$,
\begin{equation}
 \mathcal L_{N,r}^{\hc}+\gamma_N
 \ge c_{D,r}\mathsf S_{N,r,\gamma_N}.
 \label{prof:eq:fd-schur-form-comparison}
\end{equation}
Consequently, for every $b\in\mathsf X_{N,r}^{\hc}$,
\begin{equation}
 \|b\|_{-1,\hc,\gamma_N}^2
 \le C_{D,r}\|\mathsf R_{N,r}^*b\|_{-1,\ind,\gamma_N}^2,
 \label{prof:eq:fd-resolvent-transfer}
\end{equation}
where $\mathsf R_{N,r}^*$ is the normalization-adjusted zero extension
defined by \eqref{prof:eq:restriction-adjoint}.
\end{lemma}

\begin{proof}
Observe that $\mathsf{Ext}_{N,r}f$ defined in \eqref{prof:eq:ExtNr} is an admissible
competitor in the variational problem \eqref{prof:eq:fd-schur-variational-form}.
Hence, by the $L^2$ and energy bounds of \cref{prof:lem:fd-local-fill},
\[
 \begin{aligned}
 \langle f,\mathsf S_{N,r,\gamma_N}f\rangle_{\hc}
 &\le \mathcal E_{N,r}^{\ind}(\mathsf{Ext}_{N,r}f)
       +\gamma_N\|\mathsf{Ext}_{N,r}f\|_{\ind}^2\\
 &\le C_{D,r}\bigl(\mathcal E_{N,r}^{\hc}(f)
       +\gamma_N\|f\|_{\hc}^2\bigr).
 \end{aligned}
\]
Equivalently,
$\mathcal L_{N,r}^{\hc}+\gamma_N\ge c_{D,r}\mathsf S_{N,r,\gamma_N}$,
which is \eqref{prof:eq:fd-schur-form-comparison}.

The inverse of the transported Schur form is
determined by convex duality.  For every $b\in\mathsf X_{N,r}^{\hc}$,
\begin{align*}
 \langle b,\mathsf S_{N,r,\gamma_N}^{-1}b\rangle_{\hc}
 &=\sup_f\left\{2\Re\langle b,f\rangle_{\hc}
     -\inf_{\mathsf R_{N,r}u=f}
      \langle u,(\mathcal L_{N,r}^{\ind}+\gamma_N)u\rangle_{\ind}\right\}\\
 &=\sup_u\left\{2\Re\langle\mathsf R_{N,r}^*b,u\rangle_{\ind}
     -\langle u,(\mathcal L_{N,r}^{\ind}+\gamma_N)u\rangle_{\ind}\right\}\\
 &=\left\langle\mathsf R_{N,r}^*b,
   (\mathcal L_{N,r}^{\ind}+\gamma_N)^{-1}\mathsf R_{N,r}^*b\right\rangle_{\ind}.
\end{align*}
Hence, as an operator on the hard-core space,
\begin{equation}
 \mathsf S_{N,r,\gamma_N}^{-1}
 =\mathsf R_{N,r}(\mathcal L_{N,r}^{\ind}+\gamma_N)^{-1}
   \mathsf R_{N,r}^*.
 \label{prof:eq:fd-Schur-inverse-identity}
\end{equation}
Inverting \eqref{prof:eq:fd-schur-form-comparison} and using
\eqref{prof:eq:fd-Schur-inverse-identity} proves
\eqref{prof:eq:fd-resolvent-transfer}.
\end{proof}

The extension estimate \cref{prof:lem:fd-local-fill} and its Schur and resolvent consequences \cref{prof:lem:fd-schur-resolvent} complete the
first of two analytic inputs to \cref{prof:lem:fd-high-block}, which is the main result of this subsection. 
The second analytic input is a product grid trace estimate on collision diagonals, as follows.

\begin{lemma}[Collision-diagonal trace estimate]
\label{prof:lem:fd-trace}
We have
\begin{equation}
 \Theta_N:=\frac1{n_N}\sum_{p\in\T_N^D}
 \frac1{\lambda_N(p)+\gamma_N} 
 =
 \left\{\begin{array}{ll} O(\log N), & D=2,\\ O_D(1), & D\ge 3.\end{array}\right.
 \label{prof:eq:Theta_N}
\end{equation}
Then for $r\ge2$, every $u\in\mathsf X_{N,r}^{\ind}$, and every
$1\le i<j\le r$,
\begin{equation}
 \|\one_{\{x_i=x_j\}}u\|_{\ind}^2
 \le C_D\Theta_N
 \bigl(\mathcal E_{N,r}^{\ind}(u)+\gamma_N\|u\|_{\ind}^2\bigr).
 \label{prof:eq:fd-diagonal-trace}
\end{equation}
In particular, if the shifted energy $\mathcal E_{N,r}^{\ind}(u)+\gamma_N\|u\|_{\ind}^2$ on the right-hand side of \eqref{prof:eq:fd-diagonal-trace} is $O(\gamma_N)$, then the
left-hand side of \eqref{prof:eq:fd-diagonal-trace} is $o(1)$ for every $D\ge2$.
\end{lemma}

\begin{proof}
Fix $1\le i<j\le r$.  We reduce the trace on the pair diagonal
$\{x_i=x_j\}$ to a one-particle Fourier estimate in the relative displacement.
Fix $x_i$ and all coordinates $(x_a)_{a\notin\{i,j\}}$, and let only the
$j$th coordinate vary.  Writing
$
 z:=x_j-x_i\in V_N$,
define the resulting one-variable function by
\[
 g(z):=u(x_1,\ldots,x_i,\ldots,x_i+z,\ldots,x_r),
\]
where $x_i+z$ occupies the $j$th coordinate.  Thus the diagonal
$x_i=x_j$ corresponds exactly to $z=0$.

Using the characters $e_p$ from \eqref{prof:eq:probability-fourier-basis},
define the probability-normalized Fourier transform of $g$ by
\[
 \widehat g(p)
 :=\langle g,e_p\rangle_{\mathrm{av}}
 =\frac1{n_N}\sum_{z\in V_N}
   g(z)e^{-2\pi i p\cdot z/N},
 \qquad p\in\T_N^D.
\]
Since $(e_p)_p$ is orthonormal for the uniform probability measure on $V_N$,
Fourier inversion and Parseval give
\[
 g(0)=\sum_{p\in\T_N^D}\widehat g(p),
 \qquad
 \frac1{n_N}\sum_{z\in V_N}|g(z)|^2
 =\sum_{p\in\T_N^D}|\widehat g(p)|^2.
\]
Weighted Cauchy--Schwarz therefore yields
\begin{equation}
\begin{aligned}
 \frac1{n_N}|g(0)|^2
 &\le
 \left(\frac1{n_N}\sum_{p\in\T_N^D}
       \frac1{\lambda_N(p)+\gamma_N}\right)
 \left(\sum_{p\in\T_N^D}
       (\lambda_N(p)+\gamma_N)|\widehat g(p)|^2\right)\\
 &=\Theta_N
 \left(\sum_p\lambda_N(p)|\widehat g(p)|^2
       +\gamma_N\sum_p|\widehat g(p)|^2\right).
\end{aligned}
\label{prof:eq:weighted-CS}
\end{equation}

We now average this inequality over $x_i$ and all coordinates
$(x_a)_{a\notin\{i,j\}}$.  The averaged left-hand side is exactly
$\|\one_{\{x_i=x_j\}}u\|_{\ind}^2$, while the averaged second term in
parentheses is $\gamma_N\|u\|_{\ind}^2$.  For the first term, Parseval for
the one-particle Laplacian gives
\[
 \sum_p\lambda_N(p)|\widehat g(p)|^2
 =\frac1{2n_N}\sum_{z\in V_N}\sum_{\boldsymbol\delta\in\mathcal D_D}
   |g(z+\boldsymbol\delta)-g(z)|^2.
\]
Because $x_i$ is fixed, replacing $z$ by $z+\boldsymbol\delta$ changes only the
$j$th coordinate, from $x_j$ to $x_j+\boldsymbol\delta$.  Hence, after averaging over
the fixed coordinates,
\begin{align}
 \frac1{n_N^{r-1}}
 \sum_{\substack{x_i\in V_N\\(x_a)_{a\notin\{i,j\}}\in V_N^{r-2}}}
 \sum_{p\in\T_N^D}\lambda_N(p)|\widehat g(p)|^2 
 = \frac1{2n_N^r}\sum_{\boldsymbol x\in V_N^r}
 \sum_{\boldsymbol\delta\in\mathcal D_D}
 \left|u(\boldsymbol x^{j\to x_j+\boldsymbol\delta})-u(\boldsymbol x)\right|^2
 \le \mathcal E_{N,r}^{\ind}(u).
 \label{prof:eq:fd-relative-energy-bound}
\end{align}
Substitution of \eqref{prof:eq:fd-relative-energy-bound} into \eqref{prof:eq:weighted-CS} proves
\eqref{prof:eq:fd-diagonal-trace}.

It remains to estimate $\Theta_N$.  We first make the zero mode explicit.
Since $\gamma_N=4\sin^2(\pi/N)\asymp N^{-2}$,
\[
 \lambda_N(0)+\gamma_N=\gamma_N\ge cN^{-2}.
\]
For $p\ne0$, the lower bound in
\eqref{prof:eq:dispersion-comparison} gives
\[
 \lambda_N(p)+\gamma_N
 \ge \gamma_N\left(1+\frac4{\pi^2}|\widetilde p|^2\right)
 \ge c_DN^{-2}(1+|\widetilde p|^2).
\]
Thus the same lower bound holds for every $p\in\T_N^D$.  Since
$n_N=N^D$ and $p\mapsto\widetilde p$ identifies $\T_N^D$ with
$\mathcal Q_N$,
\[
 \Theta_N
 \le C_DN^{2-D}
 \sum_{\ell\in\mathcal Q_N}\frac1{1+|\ell|^2}.
\]
The representative set $\mathcal Q_N$ lies in a Euclidean ball of radius
$C_DN$.  For every integer $m\ge1$, the number of lattice points in the shell
$\{\ell\in\mathbb Z^D:m\le|\ell|<m+1\}$ is at most $C_Dm^{D-1}$.
Consequently,
\[
 \sum_{\ell\in\mathcal Q_N}\frac1{1+|\ell|^2}
 \le C_D\left(1+
 \sum_{m=1}^{C_DN}\frac{m^{D-1}}{1+m^2}\right).
\]
When $D=2$, the summand is $O(m^{-1})$, so the last quantity is
$O(\log N)$.  When $D\ge3$, it is bounded by
$C_Dm^{D-3}$, and hence
\[
 \sum_{m=1}^{C_DN}\frac{m^{D-1}}{1+m^2}
 \le C_D\sum_{m=1}^{C_DN}m^{D-3}
 =O_D(N^{D-2}).
\]
Substituting these estimates into the displayed lattice sum bound
$\Theta_N\le C_DN^{2-D}\sum_{\ell\in\mathcal Q_N}(1+|\ell|^2)^{-1}$ proves
\eqref{prof:eq:Theta_N}.  Since $\gamma_N\asymp N^{-2}$, it also gives
\[
 \gamma_N\Theta_N
 =\begin{cases}
   O(N^{-2}\log N),&D=2,\\
   O_D(N^{-2}),&D\ge3,
  \end{cases}
 \longrightarrow0.
\]
Therefore, whenever
$\mathcal E_{N,r}^{\ind}(u)+\gamma_N\|u\|_{\ind}^2=O(\gamma_N)$,
\eqref{prof:eq:fd-diagonal-trace} has an $o(1)$ right-hand side, proving the
final assertion.
\end{proof}

\localheading{Normalized comparison geometry.}
Recall the definition  $\mathsf J_{N,r,K_+}=\mathsf{Top}_{N,r}^{\ord}\mathsf R_{N,r}$ from \eqref{prof:eq:JNrK}.
The inner-product estimate \eqref{prof:eq:fd-gram-error} shows that, on the fixed low-energy window, $\mathsf J_{N,r,K_+}$ is asymptotically isometric.
For fixed $K_+\ge0$ and all sufficiently large $N$, we turn it into an exact isometry by defining the
\emph{Gram operator} and Gram-normalized comparison map by
\begin{equation}
 \mathsf G_{N,r,K_+}
 :=\mathsf J_{N,r,K_+}^*\mathsf J_{N,r,K_+},
 \qquad
 \widehat{\mathsf J}_{N,r,K_+}
 :=\mathsf J_{N,r,K_+}\mathsf G_{N,r,K_+}^{-1/2}.
 \label{prof:eq:fd-Gram-normalization}
\end{equation}
Thus
$\langle u,\mathsf G_{N,r,K_+}v\rangle_{\ind}
 =\langle\mathsf J_{N,r,K_+}u,\mathsf J_{N,r,K_+}v\rangle_{\hc}$.
Define the comparison range projections by
\begin{equation}
 \Pi_{N,r,K_+}^{\mathrm{cmp}}
 :=\widehat{\mathsf J}_{N,r,K_+}\widehat{\mathsf J}_{N,r,K_+}^*
 =\mathsf J_{N,r,K_+}\mathsf G_{N,r,K_+}^{-1}\mathsf J_{N,r,K_+}^*,
 \qquad
 \Pi_{N,r,K_+}^{\mathrm{cmp},\perp}
 :=I-\Pi_{N,r,K_+}^{\mathrm{cmp}}.
 \label{prof:eq:fd-comparison-projections}
\end{equation}
Because $\mathsf J_{N,r,K_+}$ takes values in
$\mathcal T_{N,r}^{\hc}$, throughout the fixed-degree spectral comparison
these projections are regarded as operators on the hard-core top sector
$\mathcal T_{N,r}^{\hc}$; in particular, the identity $I$ in
\eqref{prof:eq:fd-comparison-projections} is the identity on that sector.
\cref{prof:lem:fd-restriction} provides the following operator norm bounds.
First, by \eqref{prof:eq:fd-gram-error},
\begin{equation}
 \|\mathsf G_{N,r,K_+}-I\|_{\mathrm{op}}
 \le \frac{C_{D,r,K_+}}{n_N},
 \qquad
 \|\mathsf J_{N,r,K_+}\|_{\mathrm{op}}
 +\|\mathsf G_{N,r,K_+}^{-1}\|_{\mathrm{op}}
 +\|\mathsf G_{N,r,K_+}^{1/2}\|_{\mathrm{op}}
 +\|\mathsf G_{N,r,K_+}^{-1/2}\|_{\mathrm{op}}
 \le C_{D,r,K_+},
 \label{prof:eq:fd-Gram-uniform-bounds}
\end{equation}
for all sufficiently large $N$.  In particular,
\begin{equation}
 \widehat{\mathsf J}_{N,r,K_+}^*\widehat{\mathsf J}_{N,r,K_+}=I
 \quad\text{on }\mathscr V_{N,r}^{\ind}(K_+),
 \qquad
 \Pi_{N,r,K_+}^{\mathrm{cmp}}
 \text{ is the orthogonal projection onto }\Ran\mathsf J_{N,r,K_+}.
 \label{prof:eq:fd-Gram-isometry}
\end{equation}
Finally, combining \eqref{prof:eq:fd-form-error} with
\eqref{prof:eq:fd-Gram-uniform-bounds} and the product energy bound
$\mathcal L_{N,r}^{\ind}\le K_+\gamma_N$ on $\mathscr V_{N,r}^{\ind}(K_+)$ gives
\begin{equation}
 \left\|
 \Pi_{N,r,K_+}^{\mathrm{cmp}}\mathcal L_{N,r}^{\hc}
 \Pi_{N,r,K_+}^{\mathrm{cmp}}
 -\widehat{\mathsf J}_{N,r,K_+}\mathcal L_{N,r}^{\ind}
  \widehat{\mathsf J}_{N,r,K_+}^*
 \right\|_{\mathrm{op}}
 \le C_{D,r,K_+}\frac{\gamma_N}{n_N}.
 \label{prof:eq:fd-normalized-compressed-form}
\end{equation}
Here the independent generator in the second term is restricted to
$\mathscr V_{N,r}^{\ind}(K_+)$.

We have finally reached the main result of this geometric comparison subsection.

\begin{lemma}[Spectral complement above a fixed product window]
\label{prof:lem:fd-high-block}
Fix $r\in\mathbb N$ and an energy target $H_0\ge0$.  There is
$K_+=K_+(D,r,H_0)<\infty$ such that, as an operator inequality on
$\mathcal T_{N,r}^{\hc}$,
\begin{equation}
 \Pi_{N,r,K_+}^{\mathrm{cmp},\perp}\mathcal L_{N,r}^{\hc}\Pi_{N,r,K_+}^{\mathrm{cmp},\perp}
 \ge (H_0+2)\gamma_N\Pi_{N,r,K_+}^{\mathrm{cmp},\perp}
 \label{prof:eq:fd-high-block-gap}
\end{equation}
for all sufficiently large $N$.
\end{lemma}

\begin{proof}
Let $f\in\Ran\Pi_{N,r,K_+}^{\mathrm{cmp},\perp}$ with $\|f\|_{\hc}=1$.
By the ambient space convention following
\eqref{prof:eq:fd-comparison-projections}, one automatically has
$f\in\mathcal T_{N,r}^{\hc}$.  
This top sector restriction is essential to the proof: it is used first to show that the minimizing product extension has
negligible lower product-Hoeffding component, and again to identify
orthogonality to $\mathsf J_{N,r,K_+}\varphi$ with orthogonality to
$\mathsf R_{N,r}\varphi$ for low product modes $\varphi$.

If \(\mathcal E_{N,r}^{\hc}(f)>(H_0+2)\gamma_N\), then
\eqref{prof:eq:fd-high-block-gap} already holds for the vector $f$.  
Hence assume
\begin{align}
\mathcal E_{N,r}^{\hc}(f)\le (H_0+2)\gamma_N.
\label{prof:eq:E-contradict}
\end{align}
The admissible class in \eqref{prof:eq:fd-schur-variational-form} is
nonempty: by \cref{prof:lem:fd-local-fill}, the function
$\mathsf{Ext}_{N,r}f$ defined in \eqref{prof:eq:ExtNr} satisfies
$\mathsf R_{N,r}\mathsf{Ext}_{N,r}f=f$.  Moreover, for every
$v\in\mathsf X_{N,r}^{\ind}$,
\[
 \mathcal E_{N,r}^{\ind}(v)+\gamma_N\|v\|_{\ind}^2
 \ge \gamma_N\|v\|_{\ind}^2,
\]
and $\gamma_N>0$.  Thus the shifted quadratic form is positive definite
and coercive.  Since the admissible class is a closed affine subspace of the
finite-dimensional Hilbert space $\mathsf X_{N,r}^{\ind}$, the constrained
variational problem has a unique minimizer.  Denote it by
$u\in\mathsf X_{N,r}^{\ind}$; thus $\mathsf R_{N,r}u=f$.
Using $\mathsf{Ext}_{N,r}f$ as an admissible competitor, together with
\eqref{prof:eq:ext-local-fill} and \eqref{prof:eq:fd-local-fill}, gives
\begin{equation}
 \mathcal E_{N,r}^{\ind}(u)+\gamma_N\|u\|_{\ind}^2
 \le C_{D,r,H_0}\gamma_N.
 \label{prof:eq:fd-minimizer-energy}
\end{equation}
Since $f\in\mathcal T_{N,r}^{\hc}$ is permutation-symmetric, for
every particle-label permutation $\varpi\in\Sym([r])$, the function
$u\circ\mathsf P_\varpi$ also satisfies
$\mathsf R_{N,r}(u\circ\mathsf P_\varpi)=f$ and has the same shifted energy
as $u$.  Uniqueness therefore implies
$u\circ\mathsf P_\varpi=u$.  By the union bound over pair diagonals and \cref{prof:lem:fd-trace},
\begin{equation}
 \|\one_{\mathfrak D_{N,r}^{\mathrm{coll}}}u\|_{\ind}^2
 \le C_r\Theta_N
 \bigl(\mathcal E_{N,r}^{\ind}(u)+\gamma_N\|u\|_{\ind}^2\bigr)
 =o(1),
 \label{prof:eq:fd-minimizer-collision-leakage}
\end{equation}
where the last equality uses \eqref{prof:eq:fd-minimizer-energy} and the
final assertion of \cref{prof:lem:fd-trace}.  With $\theta_{N,r}$ defined by
\eqref{prof:eq:fd-distinct-tuple-density}, since $u=f$ on the distinct grid,
\begin{equation}
 \|u\|_{\ind}^2
 =\theta_{N,r}\|f\|_{\hc}^2
  +\|\one_{\mathfrak D_{N,r}^{\mathrm{coll}}}u\|_{\ind}^2
 =1+o(1).
 \label{prof:eq:fd-minimizer-norm}
\end{equation}
because $\theta_{N,r}=1-O_r(n_N^{-1})$.

We next show that the minimizing extension $u$ is asymptotically contained
in the product-top sector.  Let $G_a\subset V_N^{r-1}$ be the set on which the coordinates other than
$a$ are distinct, and write
$G_a^c:=V_N^{r-1}\setminus G_a$.  On $G_a$, the retained coordinates are distinct.  Split the defining
average into values $z$ that avoid the retained coordinates and the
$r-1$ collision values $z=x_b$, $b\ne a$.  On the first set, $u=f$, and
the hard-core top condition $f\in\mathcal T_{N,r}^{\hc}$ gives
\[
 \sum_{z\notin\{x_b:b\ne a\}}
 f(x_1,\ldots,x_{a-1},z,x_{a+1},\ldots,x_r)=0.
\]
Therefore using the expectation notation of \eqref{prof:eq:top-exp},
\[
 (\mathbb E_au)(x_1,\ldots,\widehat{x_a},\ldots,x_r)
 =\frac1{n_N}\sum_{b\ne a}
   u(x_1,\ldots,x_{a-1},x_b,x_{a+1},\ldots,x_r).
\]
Jensen's inequality on $G_a$ gives
\[
 \|\one_{G_a}\mathbb E_au\|_2^2
 \le \frac{C_r}{n_N}
 \sum_{b\ne a}\|\one_{\{x_a=x_b\}}u\|_2^2.
\]
On $G_a^c$, Jensen's inequality is applied directly to the defining average:
\[
 \|\one_{G_a^c}\mathbb E_au\|_2^2
 \le \|\one_{G_a^c}u\|_2^2
 \le \sum_{\substack{b<c\\b,c\ne a}}
      \|\one_{\{x_b=x_c\}}u\|_2^2.
\]
Consequently,
\begin{equation*}
 \|\mathbb E_au\|_2^2
 \le \frac{C_r}{n_N}
      \sum_{b\ne a}\|\one_{\{x_a=x_b\}}u\|_2^2
    +\sum_{\substack{b<c\\b,c\ne a}}
      \|\one_{\{x_b=x_c\}}u\|_2^2.
\end{equation*}
Every pair diagonal is contained in
$\mathfrak D_{N,r}^{\mathrm{coll}}$, so every term on the right is $o(1)$ by
\eqref{prof:eq:fd-minimizer-collision-leakage}.  Thus
$\|\mathbb E_au\|_2=o(1)$ for every $a$.
Now recall the product-top projection identity \eqref{prof:eq:product-top-proj}. 
Using the fact that the projections $I-\mathbb E_a$ commute for different $a$, we perform a telescoping argument on the products to find
\begin{equation}
 \|(I-\mathsf{Top}_{N,r}^{\ind})u\|_2
 \le
 \sum_{a=1}^r  \left\|\mathbb{E}_a \prod_{b<a} (I-\mathbb{E}_b)u\right\|_2
 \le\sum_{a=1}^r\|\mathbb E_au\|_2=o(1).
 \label{prof:eq:fd-product-top-leakage}
\end{equation}
This proves the claim that $u$ is asymptotically contained in the product-top sector.

It remains to show that the product-top component of $u$ has negligible
projection onto the low-energy product window $\mathscr V_{N,r}^{\ind}(K_+)$.  Fix
$\varphi\in \mathscr V_{N,r}^{\ind}(K_+)$.  By
\eqref{prof:eq:fd-Gram-isometry}, $\Pi_{N,r,K_+}^{\mathrm{cmp}}$ is the
orthogonal projection onto $\Ran\mathsf J_{N,r,K_+}$.  Since
$f\in\Ran\Pi_{N,r,K_+}^{\mathrm{cmp},\perp}$, it follows that
$
 \langle f,\mathsf J_{N,r,K_+}\varphi\rangle_{\hc}=0
 $.
Because
$f\in\mathcal T_{N,r}^{\hc}=\Ran\mathsf{Top}_{N,r}^{\ord}$ and
$\mathsf{Top}_{N,r}^{\ord}$ is an orthogonal projection,
\[
 \langle f,\mathsf R_{N,r}\varphi\rangle_{\hc}
 =\langle f,\mathsf{Top}_{N,r}^{\ord}\mathsf R_{N,r}\varphi\rangle_{\hc}
 =\langle f,\mathsf J_{N,r,K_+}\varphi\rangle_{\hc}
 =0.
\]
Finally, $\mathsf R_{N,r}u=f$.  Using the probability normalizations on the
product and hard-core spaces and the definition of $\theta_{N,r}$ in
\eqref{prof:eq:fd-distinct-tuple-density}, the contribution from the
distinct-tuple region is
\[
 \left\langle
 \one_{\Omega_{N,r}^{\ord}}u,
 \one_{\Omega_{N,r}^{\ord}}\varphi
 \right\rangle_{\ind}
 =\theta_{N,r}
 \langle \mathsf R_{N,r}u,\mathsf R_{N,r}\varphi\rangle_{\hc}
 =\theta_{N,r}
 \langle f,\mathsf R_{N,r}\varphi\rangle_{\hc}
 =0.
\]
Since
$V_N^r=\Omega_{N,r}^{\ord}\sqcup\mathfrak D_{N,r}^{\mathrm{coll}}$,
we therefore obtain
\[
 \langle u,\varphi\rangle_{\ind}
 =\left\langle
 \one_{\mathfrak D_{N,r}^{\mathrm{coll}}}u,
 \one_{\mathfrak D_{N,r}^{\mathrm{coll}}}\varphi
 \right\rangle_{\ind}.
\]
The collision factor of $u$ is already controlled by
\eqref{prof:eq:fd-minimizer-collision-leakage}.
Recall that $\mathscr V_{N,r}^{\ind}(K_+)$ is the low-energy product subspace
spanned by the permutation-symmetrized $r$-fold character tensors whose
total product energy is at most $K_+\gamma_N$.  The finite-window Fourier
bounds in \eqref{prof:eq:fd-finite-window-fourier-bounds} give
\[
 \|\varphi\|_\infty
 \le C_{D,r,K_+}\|\varphi\|_{\ind},
 \qquad \varphi\in\mathscr V_{N,r}^{\ind}(K_+).
\]
Moreover, by \eqref{prof:eq:fd-distinct-tuple-density}, the collision
diagonal has product probability
\[
 n_N^{-r}|\mathfrak D_{N,r}^{\mathrm{coll}}|
 =1-\theta_{N,r}=O_r(n_N^{-1}).
\]
Consequently, for every $\varphi\in\mathscr V_{N,r}^{\ind}(K_+)$ with
$\|\varphi\|_{\ind}=1$,
\[
 \|\one_{\mathfrak D_{N,r}^{\mathrm{coll}}}\varphi\|_{\ind}^2
 \le
 n_N^{-r}|\mathfrak D_{N,r}^{\mathrm{coll}}|\,\|\varphi\|_\infty^2
 \le C_{D,r,K_+}n_N^{-1}.
\]
Together with the collision-leakage estimate
\eqref{prof:eq:fd-minimizer-collision-leakage},
Cauchy--Schwarz therefore gives
\[
 \sup_{\substack{\varphi\in\mathscr V_{N,r}^{\ind}(K_+)\\
                   \|\varphi\|_{\ind}=1}}
 |\langle u,\varphi\rangle_{\ind}|=o(1).
\]

Let $\Pi_{N,r,K_+}^{\ind}$ be the orthogonal projection onto
$\mathscr V_{N,r}^{\ind}(K_+)$.  Since
$\mathscr V_{N,r}^{\ind}(K_+)\subset\mathcal T_{N,r}^{\ind}$, every
$\varphi\in\mathscr V_{N,r}^{\ind}(K_+)$ satisfies
$\mathsf{Top}_{N,r}^{\ind}\varphi=\varphi$.  Hence
\[
 \|\Pi_{N,r,K_+}^{\ind}\mathsf{Top}_{N,r}^{\ind}u\|_2
 =\sup_{\substack{\varphi\in\mathscr V_{N,r}^{\ind}(K_+)\\
                   \|\varphi\|_{\ind}=1}}
   |\langle u,\varphi\rangle_{\ind}|.
\]
The preceding supremum bound therefore yields
\begin{equation}
 \|\Pi_{N,r,K_+}^{\ind}\mathsf{Top}_{N,r}^{\ind}u\|_2=o(1).
 \label{prof:eq:fd-low-window-projection-leakage}
\end{equation}
The three vectors
\[
 (I-\mathsf{Top}_{N,r}^{\ind})u,\qquad
 \Pi_{N,r,K_+}^{\ind}\mathsf{Top}_{N,r}^{\ind}u,\qquad
 (I-\Pi_{N,r,K_+}^{\ind})\mathsf{Top}_{N,r}^{\ind}u
\]
are mutually orthogonal, and their sum is $u$.  Hence
\begin{align*}
 \|u\|_2^2
 =\|(I-\mathsf{Top}_{N,r}^{\ind})u\|_2^2
 +\|\Pi_{N,r,K_+}^{\ind}\mathsf{Top}_{N,r}^{\ind}u\|_2^2
 +\|(I-\Pi_{N,r,K_+}^{\ind})\mathsf{Top}_{N,r}^{\ind}u\|_2^2.
\end{align*}
The first two terms are $o(1)$ by
\eqref{prof:eq:fd-product-top-leakage} and
\eqref{prof:eq:fd-low-window-projection-leakage}, respectively.
Meanwhile, \eqref{prof:eq:fd-minimizer-norm} gives
$\|u\|_2^2=1+o(1)$, and thus
$
 \|(I-\Pi_{N,r,K_+}^{\ind})\mathsf{Top}_{N,r}^{\ind}u\|_2^2=1-o(1)
 $.
The product generator is at least $K_+\gamma_N$ on this last subspace.
Therefore
\[
 \mathcal E_{N,r}^{\ind}(u)+\gamma_N\|u\|_{\ind}^2
 \ge K_+\gamma_N(1-o(1))+\gamma_N(1+o(1))
 =(K_++1-o(1))\gamma_N.
\]
Combining this with
\eqref{prof:eq:fd-schur-form-comparison} gives
\[
 \mathcal E_{N,r}^{\hc}(f)+\gamma_N
 \ge c_{D,r}(K_++1-o(1))\gamma_N.
\]
Choose $K_+$ so that $c_{D,r}(K_++1)>H_0+4$.  Then, for all sufficiently
large $N$,
$
 \mathcal E_{N,r}^{\hc}(f)+\gamma_N
 >(H_0+3)\gamma_N$,
and hence
$
 \mathcal E_{N,r}^{\hc}(f)>(H_0+2)\gamma_N
 $.
This contradicts the assumed upper bound \eqref{prof:eq:E-contradict}.
Therefore every unit vector in
$\Ran\Pi_{N,r,K_+}^{\mathrm{cmp},\perp}$ has hard-core Rayleigh quotient
at least $(H_0+2)\gamma_N$, which is exactly the operator inequality
\eqref{prof:eq:fd-high-block-gap}.
\end{proof}

\subsection{Generator-intertwining defects and the asymptotically sharp spectral bottom}

For $u\in\mathscr V_{N,r}^{\ind}(K_+)$, define the exact generator intertwining
defect
\begin{equation}
 b_u:=\left(\mathsf J_{N,r,K_+}\mathcal L_{N,r}^{\ind}
      -\mathcal L_{N,r}^{\hc}\mathsf J_{N,r,K_+}\right)u.
 \label{prof:eq:fd-generator-defect-definition}
\end{equation}
Recall that the shifted resolvent norm $\| \cdot \|_{-1,\hc,\gamma_N}$ was defined in \eqref{prof:eq:fd-shifted-resolvent-norm}, and the orthogonal projection onto the comparison range $\Pi^{\mathrm{cmp}}_{N,r,K_+}$ was defined in \eqref{prof:eq:fd-comparison-projections}. 

\begin{lemma}[Generator-defect resolvent estimate with first-order cancellation]
\label{prof:lem:fd-generator-defect-resolvent}
For fixed $D\ge2$, $r\in\mathbb N$, and $K_+\ge0$,
\begin{equation}
 \|b_u\|_{-1,\hc,\gamma_N}^2
 \le C_{D,r,K_+}\frac{\gamma_N}{n_N}\|u\|_{\ind}^2,
 \qquad u\in \mathscr V_{N,r}^{\ind}(K_+).
 \label{prof:eq:fd-generator-defect-resolvent}
\end{equation}
Moreover the analogous estimate holds for the projection of $b_u$ orthogonal to the comparison range:
\begin{align}
 \|b_u-\Pi_{N,r,K_+}^{\mathrm{cmp}}b_u\|_{-1,\hc,\gamma_N}^2
 \le C_{D,r,K_+}\frac{\gamma_N}{n_N}\|u\|_{\ind}^2.
  \label{prof:eq:fd-generator-defect-resolvent-analog}
\end{align}
\end{lemma}

\begin{proof}
For $r=1$, the restriction map is the identity, and every
$u\in\mathscr V_{N,1}^{\ind}(K_+)$ already lies in the hard-core top sector, so
$\mathsf J_{N,1,K_+}u=u$.  The one-particle hard-core and independent
generators also coincide; hence $b_u=0$.  We may therefore assume $r\ge2$.

We separate the generator defect \eqref{prof:eq:fd-generator-defect-definition}
into
\begin{align*}
 b_u=b_{u,1}+b_{u,2},\qquad
 b_{u,1}
 &:=\mathsf R_{N,r}\mathcal L_{N,r}^{\ind}u
    -\mathcal L_{N,r}^{\hc}\mathsf R_{N,r}u,\\
 b_{u,2}
 &:=(\mathsf J_{N,r,K_+}-\mathsf R_{N,r})
       \mathcal L_{N,r}^{\ind}u
    -\mathcal L_{N,r}^{\hc}
       (\mathsf J_{N,r,K_+}-\mathsf R_{N,r})u.
\end{align*}
The correction term $b_{u,2}$ is controlled by the down defect and
top projection estimates and is smaller than the target resolvent scale.  The
substantive term is $b_{u,1}$: it is exactly the blocked-edge source, and its
first-order cancellation is exposed below in relative-coordinate Fourier
variables.

\localheading{Top-projection correction.}
Recall from \eqref{prof:eq:fd-down-defect} that
$
 \mathfrak d_u
 :=\mathsf{Down}_{N,r\to r-1}\mathsf R_{N,r}u
$
is given on distinct $(r-1)$-tuples by a sum of diagonal evaluations of
$u$.
Here we extend that formula to the full $(r-1)$-particle product grid by
\begin{align}
 \widetilde{\mathfrak d}_u(x_1,\ldots,x_{r-1})
 :=-\frac1{n_N-r+1}\sum_{j=1}^{r-1}
   u(x_1,\ldots,x_{r-1},x_j),
 \qquad
 (x_1,\ldots,x_{r-1})\in V_N^{r-1}.
 \label{prof:eq:tilde-d-u}
\end{align}
The finite-window Fourier bounds used in the proof of
\cref{prof:lem:fd-restriction}, specifically
\eqref{prof:eq:fd-finite-window-fourier-bounds}, also apply to these
diagonal evaluations.  Indeed, substituting $x_j$ for the last coordinate
only merges two bounded Fourier momenta, so
$\widetilde{\mathfrak d}_u$ belongs to a fixed $(r-1)$-particle Fourier
window depending only on $D,r,K_+$.  The explicit factor
$(n_N-r+1)^{-1}$ therefore gives
\[
 \|\widetilde{\mathfrak d}_u\|_{\ind}
 \le C_{D,r,K_+}n_N^{-1}\|u\|_{\ind}.
\]
Every Fourier mode in this window has product energy at most
$C_{D,r,K_+}\gamma_N$, and hence
\[
 \|\mathcal L_{N,r-1}^{\ind}\widetilde{\mathfrak d}_u\|_{\ind}
 \le C_{D,r,K_+}\frac{\gamma_N}{n_N}\|u\|_{\ind}.
\]
Since
$\|\mathsf R_{N,r-1}w\|_{\hc}
 \le \theta_{N,r-1}^{-1/2}\|w\|_{\ind}$
and $\theta_{N,r-1}^{-1}=1+O_r(n_N^{-1})$,
\begin{equation}
 \|\mathsf R_{N,r-1}\mathcal L_{N,r-1}^{\ind}
       \widetilde{\mathfrak d}_u\|_{\hc}
 \le C_{D,r,K_+}\frac{\gamma_N}{n_N}\|u\|_{\ind}.
 \label{prof:eq:fd-restricted-extended-down-defect-generator}
\end{equation}

We next compare the restricted product and hard-core generators by setting
\begin{align}
 \Delta_u
 :=\mathsf R_{N,r-1}\mathcal L_{N,r-1}^{\ind}
      \widetilde{\mathfrak d}_u
   -\mathcal L_{N,r-1}^{\hc}\mathfrak d_u.
 \label{prof:eq:Delta-u}
\end{align}
At a distinct $(r-1)$-tuple, the two generators have identical increments
for every move that remains distinct; thus $\Delta_u$ is exactly the sum of
the blocked product jumps into occupied neighboring coordinates.  For $r=2$
there are none.  For $r\ge3$, $\Delta_u$ is supported on the nearest-neighbor
contact set
$\mathfrak C_{N,r-1}^{\mathrm{nn}}$
\eqref{prof:eq:nearest-neighbor-contact-set}, whose hard-core probability is
$O_r(n_N^{-1})$; this follows from a union bound applied to
\eqref{prof:eq:nearest-neighbor-contact-set} and
$\theta_{N,r-1}^{-1}=1+O_r(n_N^{-1})$.

A blocked move changes at most two arguments in each diagonal evaluation of
$u$.  The discrete gradient bound in
\eqref{prof:eq:fd-finite-window-fourier-bounds} and the factor
$(n_N-r+1)^{-1}$ in
$\widetilde{\mathfrak d}_u$ \eqref{prof:eq:tilde-d-u} therefore give
\[
 |\Delta_u(\boldsymbol x)|
 \le C_{D,r,K_+}N^{-1}n_N^{-1}\|u\|_{\ind}
       \one_{\mathfrak C_{N,r-1}^{\mathrm{nn}}}(\boldsymbol x),
 \qquad
 r\ge3,\quad
 \boldsymbol x\in\Omega_{N,r-1}^{\ord}.
\]
Consequently,
\[
 \|\Delta_u\|_{\hc}
 \le C_{D,r,K_+}N^{-1}n_N^{-3/2}\|u\|_{\ind}
 \le C_{D,r,K_+}\frac{\gamma_N}{n_N}\|u\|_{\ind},
\]
where the last inequality uses $D\ge2$, $n_N=N^D$, and
$\gamma_N\asymp N^{-2}$.  The definition \eqref{prof:eq:Delta-u}, the
triangle inequality, and
\eqref{prof:eq:fd-restricted-extended-down-defect-generator} yield
\begin{equation}
 \|\mathcal L_{N,r-1}^{\hc}\mathfrak d_u\|_{\hc}
 \le C_{D,r,K_+}\frac{\gamma_N}{n_N}\|u\|_{\ind}.
 \label{prof:eq:fd-down-defect-generator}
\end{equation}

We transfer \eqref{prof:eq:fd-down-defect-generator} to the top projection
correction using the formula \eqref{prof:eq:fd-lower-layer-correction} from the proof of
\cref{prof:lem:fd-restriction}:
\[
 (\mathsf J_{N,r,K_+}-\mathsf R_{N,r})u
 =-\mathsf{Up}_{N,r-1\to r}
   \bigl(\mathsf{Down}_{N,r\to r-1}
         \mathsf{Up}_{N,r-1\to r}\bigr)^{-1}\mathfrak d_u.
\]
By the hard-core down--generator intertwining
\eqref{prof:eq:hard-core-down-intertwining} and its adjoint,
$\mathcal L_{N,r}^{\hc}$ commutes through the up map, while the down--up
operator and its inverse commute with
$\mathcal L_{N,r-1}^{\hc}$.  Moreover,
\eqref{prof:eq:fd-down-up-eigenvalue} gives a fixed-degree lower bound on the
spectrum of the down--up operator, so its inverse is uniformly bounded.
Thus
\[
 \mathcal L_{N,r}^{\hc}(\mathsf J_{N,r,K_+}-\mathsf R_{N,r})u
 =-\mathsf{Up}_{N,r-1\to r}
   \bigl(\mathsf{Down}_{N,r\to r-1}
         \mathsf{Up}_{N,r-1\to r}\bigr)^{-1}
   \mathcal L_{N,r-1}^{\hc}\mathfrak d_u,
\]
and \eqref{prof:eq:fd-down-defect-generator} gives
\begin{equation}
 \|\mathcal L_{N,r}^{\hc}
   (\mathsf J_{N,r,K_+}-\mathsf R_{N,r})u\|_{\hc}
 \le C_{D,r,K_+}\frac{\gamma_N}{n_N}\|u\|_{\ind}.
 \label{prof:eq:fd-projection-source-error}
\end{equation}

\localheading{Reduction to the blocked-edge source.}
We can now complete the estimate for the correction term $b_{u,2}$.  Its two
constituents are controlled separately.
On the one hand, \eqref{prof:eq:fd-projection-source-error} controls
$\mathcal L_{N,r}^{\hc}(\mathsf J_{N,r,K_+}-\mathsf R_{N,r})u$.
On the other hand, the product spectral window is invariant under
$\mathcal L_{N,r}^{\ind}$ and
$\|\mathcal L_{N,r}^{\ind}u\|_{\ind}
 \le K_+\gamma_N\|u\|_{\ind}$, so the
$L^2$ correction bound \eqref{prof:eq:fd-top-correction-L2}, applied to
$\mathcal L_{N,r}^{\ind}u$, controls
$(\mathsf J_{N,r,K_+}-\mathsf R_{N,r})
 \mathcal L_{N,r}^{\ind}u$
at the same scale.  Therefore
\[
 \|b_{u,2}\|_{\hc}
 \le C_{D,r,K_+}\frac{\gamma_N}{n_N}\|u\|_{\ind}.
\]
In conjunction with
$(\mathcal L_{N,r}^{\hc}+\gamma_N)^{-1}\le\gamma_N^{-1}I$, this implies
\[
 \|b_{u,2}\|_{-1,\hc,\gamma_N}^2
 \le C_{D,r,K_+}\frac{\gamma_N}{n_N^2}\|u\|_{\ind}^2,
\]
which is asymptotically negligible compared to the target scale in
\eqref{prof:eq:fd-generator-defect-resolvent}.

The triangle inequality for $\|\cdot\|_{-1,\hc,\gamma_N}$ therefore reduces
\eqref{prof:eq:fd-generator-defect-resolvent} to the estimate for
$b_{u,1}$.  At a distinct tuple, the product and hard-core generators agree
on every allowed move; hence $b_{u,1}$ consists exactly of product jumps into
occupied neighboring coordinates, which the hard-core generator suppresses.
Applying \eqref{prof:eq:fd-resolvent-transfer} with $b=b_{u,1}$ gives
\begin{equation}
 \|b_{u,1}\|_{-1,\hc,\gamma_N}^2
 \le C_{D,r}
      \|\mathsf R_{N,r}^*b_{u,1}\|_{-1,\ind,\gamma_N}^2,
 \qquad
 \mathsf R_{N,r}^*b_{u,1}
 =\theta_{N,r}^{-1}
   \one_{\Omega_{N,r}^{\ord}}b_{u,1}.
 \label{prof:eq:fd-generator-defect-transfer}
\end{equation}

\localheading{Reduction to one character tensor and one label pair.}
The map $u\mapsto\mathsf R_{N,r}^*b_{u,1}$ is linear on the
fixed-dimensional space $\mathscr V_{N,r}^{\ind}(K_+)$.  Choose an orthonormal
symmetrized character tensor basis.  Each basis vector is a normalized sum
of at most $r!$ unsymmetrized character tensors, and there are only
$\binom r2$ label pairs.  Since $D,r,K_+$ are fixed, linearity and
Cauchy--Schwarz reduce the product resolvent estimate in
\eqref{prof:eq:fd-generator-defect-transfer} to one unsymmetrized character
tensor and one pair $\{i,j\}$; the uniform component estimate below then
reconstructs $\mathsf R_{N,r}^*b_{u,1}$.

Use the unit-modulus characters $e_p$ from
\eqref{prof:eq:probability-fourier-basis}.  Let the characters in the
selected coordinates $i,j$ be $e_{p^{(1)}}$ and $e_{p^{(2)}}$,
respectively, and for each remaining coordinate
$a\notin\{i,j\}$ write $e_{q_a}$ for its character.  Set
\[
 \lambda_{\mathrm{pair}}
 :=\lambda_N(p^{(1)})+\lambda_N(p^{(2)})
 \le K_+\gamma_N.
\]
For $\boldsymbol\delta\in\mathcal D_D$, define the oriented contact set
\[
 \mathfrak C_{ij,\boldsymbol\delta}
 :=\{\boldsymbol x\in V_N^r:x_j=x_i+\boldsymbol\delta\}.
\]
For this character tensor, define the full-product comparison source
\begin{align}
 b_{ij}^{\ind}(\boldsymbol x)
 :=\left(\prod_{a\notin\{i,j\}}e_{q_a}(x_a)\right)
 \sum_{\boldsymbol\delta\in\mathcal D_D}
 \one_{\mathfrak C_{ij,\boldsymbol\delta}}(\boldsymbol x)
 e_{p^{(1)}}(x_i)e_{p^{(2)}}(x_i+\boldsymbol\delta)
 \bigl[2-e_{p^{(1)}}(\boldsymbol\delta)-e_{p^{(2)}}(-\boldsymbol\delta)\bigr].
 \label{prof:eq:fd-unrestricted-pair-source}
\end{align}
Its restriction
$
 b_{ij}^{\hc}:=\mathsf R_{N,r}b_{ij}^{\ind}
$
is the $\{i,j\}$-component of $b_{u,1}$ for the chosen character tensor,
and $\mathsf R_{N,r}^*b_{ij}^{\hc}$ is the corresponding exact component of
$\mathsf R_{N,r}^*b_{u,1}$.

\localheading{First-order cancellation and the relative-coordinate resolvent estimate.}
On $\mathfrak C_{ij,\boldsymbol\delta}$, the selected-pair factor of
$b_{ij}^{\ind}$ is
\[
 e_{p^{(1)}}(x_i)e_{p^{(2)}}(x_i+\boldsymbol\delta)
 \bigl[2-e_{p^{(1)}}(\boldsymbol\delta)-e_{p^{(2)}}(-\boldsymbol\delta)\bigr].
\]
After factoring out the common center-of-mass character
$e_{p^{(1)}+p^{(2)}}(x_i)$, the relative coefficient is
\[
 c_{\boldsymbol\delta}
 :=e_{p^{(2)}}(\boldsymbol\delta)
   \bigl(2-e_{p^{(1)}}(\boldsymbol\delta)-e_{p^{(2)}}(-\boldsymbol\delta)\bigr)
 =2e_{p^{(2)}}(\boldsymbol\delta)
   -e_{p^{(1)}+p^{(2)}}(\boldsymbol\delta)-1.
\]
The selected-pair contribution to $b_{ij}^{\ind}$ can therefore be
written explicitly in center-of-mass and relative coordinates; namely,
with $z=x_j-x_i$, we have
\[
 b_{ij}^{\ind}(\boldsymbol x)
 =\left(\prod_{a\notin\{i,j\}}e_{q_a}(x_a)\right)
   e_{p^{(1)}+p^{(2)}}(x_i)
   h_{ij}^{\mathrm{rel}}(z),
\qquad
\text{where }~ 
 h_{ij}^{\mathrm{rel}}(z)
 :=\sum_{\boldsymbol\delta\in\mathcal D_D}c_{\boldsymbol\delta}\,\one_{\{z=\boldsymbol\delta\}}.
\]
Its Fourier coefficient is
\[
 \widehat h_{ij}^{\mathrm{rel}}(k)
 :=\frac{1}{n_N}\sum_{z\in V_N}
      h_{ij}^{\mathrm{rel}}(z)e_{-k}(z)
 =\frac{1}{n_N}\sum_{\boldsymbol\delta\in\mathcal D_D}
      c_{\boldsymbol\delta} e_{-k}(\boldsymbol\delta).
\]
Fourier inversion in the relative coordinate gives
\[
 h_{ij}^{\mathrm{rel}}(z)
 =\sum_{k\in\T_N^D}\widehat h_{ij}^{\mathrm{rel}}(k)e_k(z),
\]
and therefore
\begin{equation}
 b_{ij}^{\ind}(\boldsymbol x)
 =\sum_{k\in\T_N^D}\widehat h_{ij}^{\mathrm{rel}}(k)
   \left(\prod_{a\notin\{i,j\}}e_{q_a}(x_a)\right)
   e_{p^{(1)}+p^{(2)}-k}(x_i)e_k(x_j).
   \label{prof:eq:Fourier-bind}
\end{equation}
Thus, after the complementary momenta and total momentum of the selected
pair are fixed, the full-product Fourier coefficient on the relative-momentum
fiber indexed by $k$ is exactly $\widehat h_{ij}^{\mathrm{rel}}(k)$.

We next estimate this relative Fourier coefficient.
A basic upper bound is
 \begin{align*}
 c_{\boldsymbol\delta}
  =(e_{p^{(2)}}(\boldsymbol\delta)-1)
    +e_{p^{(2)}}(\boldsymbol\delta)(1-e_{p^{(1)}}(\boldsymbol\delta)) 
    \quad \Longrightarrow\quad
 |c_{\boldsymbol\delta}|^2
  \le
    2|e_{p^{(2)}}(\boldsymbol\delta)-1|^2
    +2|1-e_{p^{(1)}}(\boldsymbol\delta)|^2.
 \end{align*}
As $\mathcal D_D=\{\pm\mathbf e_a:1\le a\le D\}$ and
$
 \sum_{\boldsymbol\delta\in\mathcal D_D}|1-e_p(\boldsymbol\delta)|^2
 =2\lambda_N(p)
 $,
it follows that
$
 \sum_{\boldsymbol\delta\in\mathcal D_D}|c_{\boldsymbol\delta}|^2
 \le4\lambda_{\mathrm{pair}}
 $.
Moreover, for each coordinate direction,
\[
 c_{\mathbf e_a}+c_{-\mathbf e_a}
 =2\bigl(e_{p^{(2)}}(\mathbf e_a)+e_{p^{(2)}}(-\mathbf e_a)-2\bigr)
  -\bigl(e_{p^{(1)}+p^{(2)}}(\mathbf e_a)
         +e_{p^{(1)}+p^{(2)}}(-\mathbf e_a)-2\bigr).
\]
The torus dispersion satisfies the elementary subadditivity bound
\begin{equation}
 \lambda_N(p+q)
 \le 2\lambda_N(p)+2\lambda_N(q),
 \qquad p,q\in\T_N^D,
 \label{prof:eq:fd-dispersion-subadditivity}
\end{equation}
because, coordinatewise,
$|1-e_{p+q}(\mathbf e_a)|\le |1-e_p(\mathbf e_a)|+|1-e_q(\mathbf e_a)|$.
Since
$|e_p(\mathbf e_a)+e_p(-\mathbf e_a)-2|=2-2\cos(2\pi p_a/N)$,
\eqref{prof:eq:fd-dispersion-subadditivity} gives
\[
 \sum_{a=1}^D|c_{\mathbf e_a}+c_{-\mathbf e_a}|
 \le 2\lambda_N(p^{(2)})
      +\lambda_N(p^{(1)}+p^{(2)})
 \le 4\lambda_{\mathrm{pair}}.
\]
Thus
$
 \left|\sum_{\boldsymbol\delta} c_{\boldsymbol\delta}\right|
 \le C\lambda_{\mathrm{pair}}
 $.
Write
\[
 \sum_{\boldsymbol\delta\in\mathcal D_D}c_{\boldsymbol\delta} e_{-k}(\boldsymbol\delta)
 =
 \sum_{\boldsymbol\delta\in\mathcal D_D}
 c_{\boldsymbol\delta}\bigl(e_{-k}(\boldsymbol\delta)-1\bigr)
 +\sum_{\boldsymbol\delta\in\mathcal D_D}c_{\boldsymbol\delta}.
\]
Because $\mathcal D_D$ has $2D$ elements,
\[
 \left|
 \sum_{\boldsymbol\delta\in\mathcal D_D}
 c_{\boldsymbol\delta}\bigl(e_{-k}(\boldsymbol\delta)-1\bigr)
 \right|
 \le
 \left(\sum_{\boldsymbol\delta}|c_{\boldsymbol\delta}|^2\right)^{1/2}
 \left(\sum_{\boldsymbol\delta}|e_{-k}(\boldsymbol\delta)-1|^2\right)^{1/2}
 \le C\sqrt{\lambda_{\mathrm{pair}}\lambda_N(k)}.
\]
Since
$|\sum_{\boldsymbol\delta} c_{\boldsymbol\delta}|\le C\lambda_{\mathrm{pair}}$,
division by $n_N$ and squaring give
\begin{equation}
 |\widehat h_{ij}^{\mathrm{rel}}(k)|^2
 \le \frac{C_{K_+}}{n_N^2}
 \bigl(
   \lambda_{\mathrm{pair}}\lambda_N(k)
   +\lambda_{\mathrm{pair}}^2
 \bigr).
 \label{prof:eq:fd-dipole-symbol}
\end{equation}

Fix the momenta $(q_a)_{a\notin\{i,j\}}$ of the remaining coordinates and
the total momentum
$
 p_{\mathrm{tot}}:=p^{(1)}+p^{(2)}
 $
of the selected pair.  Put
$
 \lambda_{\mathrm{rest}}
 :=\sum_{a\notin\{i,j\}}\lambda_N(q_a)$.
After these quantities are fixed, the Fourier variable associated with the
pair-local source in the relative coordinate is the relative momentum $k$ of
the selected pair: its two momenta are $k$ and $p_{\mathrm{tot}}-k$.  The
corresponding product eigenvalue is therefore
\[
 \varepsilon_{p_{\mathrm{tot}},\mathrm{rest}}(k)
 :=\lambda_N(k)
   +\lambda_N(p_{\mathrm{tot}}-k)
   +\lambda_{\mathrm{rest}}.
\]
All three summands are nonnegative, so
\[
 \varepsilon_{p_{\mathrm{tot}},\mathrm{rest}}(k)+\gamma_N
 \ge\lambda_N(k)+\gamma_N.
\]
By the product Fourier expansion \eqref{prof:eq:Fourier-bind}, the dipole bound
\eqref{prof:eq:fd-dipole-symbol}, Parseval, and diagonalization of
the product resolvent, we obtain
\begin{align*}
 \|b_{ij}^{\ind}\|_{-1,\ind,\gamma_N}^2
 &=\sum_k
   \frac{|\widehat h_{ij}^{\mathrm{rel}}(k)|^2}
        {\varepsilon_{p_{\mathrm{tot}},\mathrm{rest}}(k)+\gamma_N}
 \le
 \frac{C_{K_+}}{n_N^2}
 \sum_k
 \frac{\lambda_{\mathrm{pair}}\lambda_N(k)
       +\lambda_{\mathrm{pair}}^2}
      {\lambda_N(k)+\gamma_N}\\
 &\le
 \frac{C_{K_+}}{n_N^2}
 \left(
   \lambda_{\mathrm{pair}}n_N
   +\lambda_{\mathrm{pair}}^2
    \sum_k\frac1{\lambda_N(k)+\gamma_N}
 \right)
 \le
 C_{K_+}\frac{\lambda_{\mathrm{pair}}}{n_N}
 \bigl(1+\lambda_{\mathrm{pair}}\Theta_N\bigr),
\end{align*}
where $\Theta_N$ was defined in \eqref{prof:eq:Theta_N} along with the
estimate.  Since
$\lambda_{\mathrm{pair}}\le K_+\gamma_N$, the last display is bounded by
$C_{D,K_+}\gamma_N/n_N$ for every $D\ge2$.  Consequently,
\[
 \|b_{ij}^{\ind}\|_{-1,\ind,\gamma_N}^2
 \le C_{D,K_+}\frac{\gamma_N}{n_N}.
\]

\localheading{Restoring the exact zero extension.}
It remains to pass from $b_{ij}^{\ind}$ to the exact zero extension
$\mathsf R_{N,r}^*b_{ij}^{\hc}$.  By
\eqref{prof:eq:restriction-adjoint},
\begin{equation}
 \mathsf R_{N,r}^*b_{ij}^{\hc}
 =\theta_{N,r}^{-1}
  \one_{\Omega_{N,r}^{\ord}}b_{ij}^{\ind}
 =b_{ij}^{\ind}+\delta b_{ij},
 \quad\text{where~ }
 \delta b_{ij}
 :=(\theta_{N,r}^{-1}-1)b_{ij}^{\ind}
  -\theta_{N,r}^{-1}
   \one_{\mathfrak D_{N,r}^{\mathrm{coll}}}b_{ij}^{\ind}.
 \label{prof:eq:fd-generator-defect-zero-extension-decomposition}
\end{equation}
The decomposition
\eqref{prof:eq:fd-generator-defect-zero-extension-decomposition}
exhausts every collision stratum, including collisions between two of the
remaining $r-2$ coordinates, collisions between a remaining coordinate and
either coordinate of the selected pair, and all multiple intersections of
such diagonals.

As \eqref{prof:eq:fd-unrestricted-pair-source} shows,
$b_{ij}^{\ind}$ is supported on
$\bigcup_{\boldsymbol\delta\in\mathcal D_D}\mathfrak C_{ij,\boldsymbol\delta}$.
This union has product probability $O_D(n_N^{-1})$, and the squared
character difference is $O(\lambda_{\mathrm{pair}})$.  Therefore
\[
 \|b_{ij}^{\ind}\|_{\ind}^2
 \le C_D\frac{\lambda_{\mathrm{pair}}}{n_N}
 \le C_{D,K_+}\frac{\gamma_N}{n_N}.
\]
Since $|\theta_{N,r}^{-1}-1|\le C_rn_N^{-1}$,
\[
 \|(\theta_{N,r}^{-1}-1)b_{ij}^{\ind}\|_{\ind}^2
 \le C_{D,r,K_+}\gamma_Nn_N^{-3}.
\]
Conditional on the selected-pair contact, any additional coordinate equality
has probability $O_r(n_N^{-1})$.  A union bound over the finitely many extra
pair diagonals gives
\[
 \|\one_{\mathfrak D_{N,r}^{\mathrm{coll}}}
   b_{ij}^{\ind}\|_{\ind}^2
 \le C_{D,r,K_+}\gamma_Nn_N^{-2}.
\]
Together with
\eqref{prof:eq:fd-generator-defect-zero-extension-decomposition},
we deduce that
\[
 \|\delta b_{ij}\|_{\ind}^2
 \le C_{D,r,K_+}\gamma_Nn_N^{-2}.
\]
The crude bound
$(\mathcal L_{N,r}^{\ind}+\gamma_N)^{-1}\le\gamma_N^{-1}I$
then gives
\[
 \|\delta b_{ij}\|_{-1,\ind,\gamma_N}^2
 \le C_{D,r,K_+}n_N^{-2}
 \le C_{D,r,K_+}\frac{\gamma_N}{n_N},
\]
where the last inequality uses $D\ge2$.  The triangle inequality for
$\|\cdot\|_{-1,\ind,\gamma_N}$ and
\eqref{prof:eq:fd-generator-defect-zero-extension-decomposition} now give
\[
 \|\mathsf R_{N,r}^*b_{ij}^{\hc}\|_{-1,\ind,\gamma_N}^2
 \le C_{D,r,K_+}\frac{\gamma_N}{n_N}.
\]
Summing these bounds over the fixed basis and the fixed set of pairs, and
applying Cauchy--Schwarz to the coefficient vector, bounds
$\|\mathsf R_{N,r}^*b_{u,1}\|_{-1,\ind,\gamma_N}^2$.
The transfer estimate
\eqref{prof:eq:fd-generator-defect-transfer} proves
\eqref{prof:eq:fd-generator-defect-resolvent}.

\localheading{Projection off the comparison range.}
Finally we prove the analog
\eqref{prof:eq:fd-generator-defect-resolvent-analog}.
We first control $\Pi_{N,r,K_+}^{\mathrm{cmp}}b_u$.
For a unit vector
$\varphi\in\mathscr V_{N,r}^{\ind}(K_+)$, the definition
\eqref{prof:eq:fd-generator-defect-definition} and self-adjointness of the
two generators give
\begin{align*}
 \langle b_u,\mathsf J_{N,r,K_+}\varphi\rangle_{\hc}
 &=
 \langle
   \mathsf J_{N,r,K_+}\mathcal L_{N,r}^{\ind}u,
   \mathsf J_{N,r,K_+}\varphi
 \rangle_{\hc}
 -
 \mathcal E_{N,r}^{\hc}
 \bigl(
   \mathsf J_{N,r,K_+}u,
   \mathsf J_{N,r,K_+}\varphi
 \bigr)\\
 &=
 \bigl[
 \langle
   \mathsf J_{N,r,K_+}\mathcal L_{N,r}^{\ind}u,
   \mathsf J_{N,r,K_+}\varphi
 \rangle_{\hc}
 -
 \langle\mathcal L_{N,r}^{\ind}u,\varphi\rangle_{\ind}
 \bigr]\\
 &\quad-
 \bigl[
 \mathcal E_{N,r}^{\hc}
 \bigl(
   \mathsf J_{N,r,K_+}u,
   \mathsf J_{N,r,K_+}\varphi
 \bigr)
 -
 \mathcal E_{N,r}^{\ind}(u,\varphi)
 \bigr].
\end{align*}
The difference of inner products in the first bracket is controlled by
\eqref{prof:eq:fd-gram-error}, using
$\|\mathcal L_{N,r}^{\ind}u\|
 \le K_+\gamma_N\|u\|$; the difference of Dirichlet forms in the second
bracket is controlled by \eqref{prof:eq:fd-form-error}.  Hence
\[
 \|\mathsf J_{N,r,K_+}^*b_u\|_{\ind}
 =
 \sup_{\substack{
        \varphi\in\mathscr V_{N,r}^{\ind}(K_+)\\
        \|\varphi\|_{\ind}=1}}
 |\langle b_u,\mathsf J_{N,r,K_+}\varphi\rangle_{\hc}|
 \le
 C_{D,r,K_+}\frac{\gamma_N}{n_N}\|u\|_{\ind}.
\]
Using the definition \eqref{prof:eq:fd-comparison-projections} of $ \Pi_{N,r,K_+}^{\mathrm{cmp}}$
and the uniform bounds \eqref{prof:eq:fd-Gram-uniform-bounds}, we obtain
\begin{equation}
 \|\Pi_{N,r,K_+}^{\mathrm{cmp}}b_u\|_{\hc}
 \le C_{D,r,K_+}\frac{\gamma_N}{n_N}\|u\|_{\ind}.
 \label{prof:eq:fd-projected-defect-norm}
\end{equation}
Finally,
$(\mathcal L_{N,r}^{\hc}+\gamma_N)^{-1}\le\gamma_N^{-1}I$,
so \eqref{prof:eq:fd-projected-defect-norm} gives
\[
 \|\Pi_{N,r,K_+}^{\mathrm{cmp}}b_u\|_{-1,\hc,\gamma_N}^2
 \le
 C\gamma_Nn_N^{-2}\|u\|_{\ind}^2.
\]
Hence
\[
 \|b_u-\Pi_{N,r,K_+}^{\mathrm{cmp}}b_u\|_{-1,\hc,\gamma_N}
 \le
 \|b_u\|_{-1,\hc,\gamma_N}
 +
 \gamma_N^{-1/2}
 \|\Pi_{N,r,K_+}^{\mathrm{cmp}}b_u\|_{\hc},
\]
and the right-hand side is
$O(\sqrt{\gamma_N/n_N})\|u\|_{\ind}$ by \eqref{prof:eq:fd-generator-defect-resolvent} and \eqref{prof:eq:fd-projected-defect-norm}.
This proves
\eqref{prof:eq:fd-generator-defect-resolvent-analog}.
\end{proof}

The high-block estimate \eqref{prof:eq:fd-high-block-gap}, the normalized
comparison estimate \eqref{prof:eq:fd-normalized-compressed-form}, and the two
generator defect resolvent bounds from
\cref{prof:lem:fd-generator-defect-resolvent}, namely
\eqref{prof:eq:fd-generator-defect-resolvent} and \eqref{prof:eq:fd-generator-defect-resolvent-analog},
now feed three direct consequences.  We first apply the Schur estimate at the
single energy relevant to the spectral bottom.  We then split the tensor source
using a finite one-particle spectral cutoff: the truncated contribution is
compared by Duhamel's formula, while the omitted contribution is handled by a
one-sided estimate on low hard-core eigenvectors.
For the Schur step, we will need to bound the inverse of a compressed positive
operator using the full inverse already controlled by the resolvent estimates.
The following elementary inequality supplies exactly that comparison.

\begin{lemma}[Inverse of a compression]
\label{prof:lem:fd-compression-inverse}
Let $T>0$ be positive definite on a finite-dimensional Hilbert space, and
let $Q$ be an orthogonal projection.  On $\Ran Q$ we have the operator inequality
\begin{equation}
 (QTQ|_{\Ran Q})^{-1}\le QT^{-1}Q.
 \label{prof:eq:fd-compression-inverse}
\end{equation}
\end{lemma}

\begin{proof}
For $y\in\Ran Q$, the variational formula for the inverse gives
\[
 \langle y,(QTQ|_{\Ran Q})^{-1}y\rangle
 =\sup_{x\in\Ran Q}
 \{2\operatorname{Re}\langle y,x\rangle-\langle x,Tx\rangle\}.
\]
Enlarging the supremum to the full Hilbert space gives
$\langle y,T^{-1}y\rangle$, which proves
\eqref{prof:eq:fd-compression-inverse}.
\end{proof}

For \cref{prof:prop:fd-top-gap,prof:lem:fd-low-source,prof:lem:fd-one-sided-eigenvector,prof:lem:fd-high-source}
below, abbreviate
the two top sector generators by
\[
 L_{\ind}:=\mathcal L_{N,r}^{\ind}\big|_{\mathcal T_{N,r}^{\ind}},
 \qquad
 L_{\hc}:=\mathcal L_{N,r}^{\hc}\big|_{\mathcal T_{N,r}^{\hc}}.
\]
Because every coordinate in $\mathcal T_{N,r}^{\ind}$ lies in the mean-zero
one-particle sector, tensorization of the one-particle spectral gap gives the
basic product-top bound
\begin{equation}
 L_{\ind}\ge r\gamma_N I
 \quad\text{on }\mathcal T_{N,r}^{\ind}.
 \label{prof:eq:fd-product-top-gap}
\end{equation}
Indeed, $L_{\ind}$ is the sum of the $r$ commuting one-particle Laplacians,
and each contributes at least $\gamma_N$ on its mean-zero factor.

We record the shifted-compression consequence used twice below.  Let $H\ge0$, and let $Q$ be
an orthogonal projection such that
\[
 QL_{\hc}Q\ge(H+2)\gamma_N Q
 \quad\text{on }\Ran Q,
\]
and let $z\le H\gamma_N$.  Spectral calculus for the compressed operator gives
\[
 Q(L_{\hc}-z)Q
 \ge \frac{2}{H+3}Q(L_{\hc}+\gamma_N)Q
 \quad\text{on }\Ran Q.
\]
Inverting this inequality and then applying
\cref{prof:lem:fd-compression-inverse} to $T=L_{\hc}+\gamma_N$ yields
\begin{equation}
 \bigl(Q(L_{\hc}-z)Q\big|_{\Ran Q}\bigr)^{-1}
 \le \frac{H+3}{2}
 Q(L_{\hc}+\gamma_N)^{-1}Q
 \quad\text{on }\Ran Q.
 \label{prof:eq:fd-shifted-compression-resolvent}
\end{equation}

\begin{proposition}[Asymptotically sharp spectral bottom of the ordered chaos sector]
\label{prof:prop:fd-top-gap}
For every fixed $r\in\mathbb N$, there is $C_{D,r}<\infty$ such that
\begin{equation}
 \inf\Spec\left(
 \mathcal L_{N,r}^{\hc}\big|_{\mathcal T_{N,r}^{\hc}}
 \right)
 \ge r\gamma_N-C_{D,r}\frac{\gamma_N}{n_N}
 \label{prof:eq:fd-top-gap}
\end{equation}
for all sufficiently large $N$.  Consequently, for every nonempty compact
$J\subset\mathbb R$,
\begin{equation}
 \sup_{s\in J}
 n_N^{r/2}
 \left\|
 e^{-t_N(s)\mathcal L_{N,r}^{\hc}}
 \big|_{\mathcal T_{N,r}^{\hc}}
 \right\|_{\mathrm{op}}
 \le C_{D,r,J}.
 \label{prof:eq:fd-top-damping}
\end{equation}
\end{proposition}

\begin{proof}
Apply \cref{prof:lem:fd-high-block} with target $r$, and let
$K_*=K_+(D,r,r)$.  We adopt the shorthands
$
 \Pi_*:=\Pi_{N,r,K_*}^{\mathrm{cmp}}
 $ and
 $
 \Pi_*^{\perp}:=\Pi_{N,r,K_*}^{\mathrm{cmp},\perp}
 $.
By \eqref{prof:eq:fd-normalized-compressed-form},
\begin{equation}
 \left\|\Pi_*L_{\hc}\Pi_*
 -\widehat{\mathsf J}_{N,r,K_*}L_{\ind}
  \widehat{\mathsf J}_{N,r,K_*}^*\right\|_{\mathrm{op}}
 \le C\frac{\gamma_N}{n_N}.
 \label{prof:eq:fd-direct-compressed-error}
\end{equation}
By \eqref{prof:eq:fd-product-top-gap} and \eqref{prof:eq:fd-direct-compressed-error}, the triangle inequality gives
\begin{equation}
 \Pi_*L_{\hc}\Pi_*
 \ge\left(r\gamma_N-C_1\frac{\gamma_N}{n_N}\right)
 \Pi_*.
 \label{prof:eq:fd-direct-P-gap}
\end{equation}
The high-block \cref{prof:lem:fd-high-block} gives
\begin{equation}
 \Pi_*^{\perp}L_{\hc}\Pi_*^{\perp}
 \ge(r+2)\gamma_N\Pi_*^{\perp}.
 \label{prof:eq:fd-direct-Q-gap}
\end{equation}

We now present the key Schur complement argument. Set
$
 z=r\gamma_N-C_*\frac{\gamma_N}{n_N}
 $,
where $C_*>C_1$ will be fixed below.  
Put
$
 \mathsf H_z:=L_{\hc}-z
 $,
and decompose the top sector orthogonally as
$
 \mathcal T_{N,r}^{\hc}=\Ran\Pi_*\oplus\Ran\Pi_*^{\perp}.
 $
Relative to this splitting,
\[
 \mathsf H_z=
 \begin{pmatrix}
  \Pi_*\mathsf H_z\Pi_* & \Pi_*L_{\hc}\Pi_*^{\perp}\\
  \Pi_*^{\perp}L_{\hc}\Pi_* & \Pi_*^{\perp}\mathsf H_z\Pi_*^{\perp}
 \end{pmatrix}.
\]
Note that the scalar shift does not appear in the off-diagonal blocks because
$\Pi_*\Pi_*^{\perp}=0$.  
We claim that the complementary block
$\Pi_*^{\perp}\mathsf H_z\Pi_*^{\perp}$ is positive.
Then the Schur complement criterion
reduces the desired inequality $\mathsf H_z\ge0$ to
\begin{equation}
 \Pi_*\mathsf H_z\Pi_*
 -\Pi_*L_{\hc}\Pi_*^{\perp}
  \bigl(\Pi_*^{\perp}\mathsf H_z\Pi_*^{\perp}
        \big|_{\Ran\Pi_*^{\perp}}\bigr)^{-1}
  \Pi_*^{\perp}L_{\hc}\Pi_*
 \ge0
 \quad\text{on }\Ran\Pi_*.
 \label{prof:eq:fd-direct-schur-target}
\end{equation}
We now verify positivity of the complementary block and then bound the
Schur correction in \eqref{prof:eq:fd-direct-schur-target}.

Since $z\le r\gamma_N$, the complementary gap
\eqref{prof:eq:fd-direct-Q-gap} and the shifted-compression estimate
\eqref{prof:eq:fd-shifted-compression-resolvent}, with
$H=r$ and $Q=\Pi_*^{\perp}$, show at once that
$\Pi_*^{\perp}\mathsf H_z\Pi_*^{\perp}$ is strictly positive and that
\begin{equation}
 \bigl(\Pi_*^{\perp}\mathsf H_z\Pi_*^{\perp}\big|_{\Ran\Pi_*^{\perp}}\bigr)^{-1}
 \le \frac{r+3}{2}
 \Pi_*^{\perp}(L_{\hc}+\gamma_N)^{-1}\Pi_*^{\perp}.
 \label{prof:eq:fd-direct-shift-inverse}
\end{equation}

Take $\psi\in\Ran\Pi_*$.  Since
$\Ran\Pi_*=\Ran\widehat{\mathsf J}_{N,r,K_*}$, choose
$u\in\mathscr V_{N,r}^{\ind}(K_*)$ such that
$\psi=\widehat{\mathsf J}_{N,r,K_*}u$.  By
\eqref{prof:eq:fd-Gram-normalization}, if
$v:=\mathsf G_{N,r,K_*}^{-1/2}u$, then
$
 \psi=\mathsf J_{N,r,K_*}v
 $.
The product spectral window $\mathscr V_{N,r}^{\ind}(K_*)$ is invariant under
$L_{\ind}$, so
$\mathsf J_{N,r,K_*}L_{\ind}v\in\Ran\Pi_*$.  Therefore the generator
defect $b_v$ from \eqref{prof:eq:fd-generator-defect-definition} satisfies
\[
 \Pi_*^{\perp}L_{\hc}\psi
 =-\Pi_*^{\perp}b_v.
\]
Applying \eqref{prof:eq:fd-direct-shift-inverse} to this vector and then
using \eqref{prof:eq:fd-generator-defect-resolvent-analog}, gives
\begin{align*}
 &\left\langle \Pi_*^{\perp}L_{\hc}\psi,
   \bigl(\Pi_*^{\perp}\mathsf H_z\Pi_*^{\perp}\big|_{\Ran\Pi_*^{\perp}}\bigr)^{-1}
   \Pi_*^{\perp}L_{\hc}\psi\right\rangle_{\hc}
 =\left\langle \Pi_*^{\perp}b_v,
   \bigl(\Pi_*^{\perp}\mathsf H_z\Pi_*^{\perp}\big|_{\Ran\Pi_*^{\perp}}\bigr)^{-1}
   \Pi_*^{\perp}b_v\right\rangle_{\hc}\\
 &\qquad \le \frac{r+3}{2}
 \left\langle \Pi_*^{\perp}b_v,
   \Pi_*^{\perp}(L_{\hc}+\gamma_N)^{-1}\Pi_*^{\perp}b_v
 \right\rangle_{\hc}
 \le C_2\frac{\gamma_N}{n_N}\|v\|_{\ind}^2.
\end{align*}
By \eqref{prof:eq:fd-Gram-uniform-bounds} and
\eqref{prof:eq:fd-Gram-isometry},
$\|v\|_{\ind}\le C\|\psi\|_{\hc}$.  Hence from the last display we obtain the operator inequality
\begin{equation}
 \Pi_*L_{\hc}\Pi_*^{\perp}
 \bigl(\Pi_*^{\perp}\mathsf H_z\Pi_*^{\perp}
       \big|_{\Ran\Pi_*^{\perp}}\bigr)^{-1}
 \Pi_*^{\perp}L_{\hc}\Pi_*
 \le C_3\frac{\gamma_N}{n_N}\Pi_*.
 \label{prof:eq:fd-direct-schur-error}
\end{equation}
On the other hand, \eqref{prof:eq:fd-direct-P-gap} gives
\[
 \Pi_*\mathsf H_z\Pi_*
 =\Pi_*(L_{\hc}-z)\Pi_*
 \ge(C_*-C_1)\frac{\gamma_N}{n_N}\Pi_*.
\]
Combining this with \eqref{prof:eq:fd-direct-schur-error}, we can bound the left-hand side
of \eqref{prof:eq:fd-direct-schur-target} below by
\[
 (C_*-C_1-C_3)\frac{\gamma_N}{n_N}\Pi_*.
\]
Choose $C_*>C_1+C_3$.  Then the Schur complement is nonnegative, while
the positivity supplied by \eqref{prof:eq:fd-shifted-compression-resolvent} shows that the complementary block
$\Pi_*^{\perp}\mathsf H_z\Pi_*^{\perp}$ is strictly positive.  The
Schur complement criterion therefore yields $\mathsf H_z\ge0$, proving
\eqref{prof:eq:fd-top-gap}.

Finally,
$\gamma_Nt_N(s)=(\log n_N+s)/2$, so \eqref{prof:eq:fd-top-gap} implies 
\begin{align*}
 n_N^{r/2}\|e^{-t_N(s)L_{\hc}}\|_{\mathrm{op}}
 \le n_N^{r/2}
 \exp\left\{-t_N(s)
   \left(r\gamma_N-C\frac{\gamma_N}{n_N}\right)\right\}
 =e^{-rs/2}
 \exp\left\{C\frac{\log n_N+s}{2n_N}\right\},
\end{align*}
which is bounded uniformly for $s\in J$.
\end{proof}

\subsection{One-particle spectral truncation and source-specific completion}

Recall that $F_N=F_N^{S_N}=\one_{S_N}-\rho_N$, and recall the smooth low-energy spectral multiplier $\Xi_{N,K}^{(1)}$ defined in \eqref{prof:eq:fd-one-particle-cutoff}.
For fixed $K$ and deterministic source $S_N$, put
$F_{N,K}=\Xi_{N,K}^{(1)}F_N$ and
\[
 \mathsf{Rem}_{N,r,K}:=F_N^{\otimes r}-F_{N,K}^{\otimes r}.
\]
The telescoping identity applied to the tensor product gives
\begin{equation}
 \mathsf{Rem}_{N,r,K}
 =\sum_{j=1}^r
 F_{N,K}^{\otimes(j-1)}\otimes(F_N-F_{N,K})
 \otimes F_N^{\otimes(r-j)}.
 \label{prof:eq:fd-source-tensor-telescoping}
\end{equation}
By the defining property of \eqref{prof:eq:fd-one-particle-cutoff},
$F_N-F_{N,K}$ has one-particle spectral support in
$(K\gamma_N,\infty)$.  In the $j$th summand of
\eqref{prof:eq:fd-source-tensor-telescoping}, this is precisely the $j$th
tensor factor.  Each of the other $r-1$ factors is a copy of either
$F_{N,K}$ or $F_N$.  Since $F_N$ has zero spatial mean and
$\Xi_{N,K}^{(1)}$ is a spectral multiplier, $F_{N,K}$ also has zero spatial
mean.  Thus each of these other factors has one-particle spectral support in
$[\gamma_N,\infty)$.  Therefore the $j$th summand in \eqref{prof:eq:fd-source-tensor-telescoping} 
 has total independent-walker energy strictly larger
than $((j-1)+K+(r-j))\gamma_N = (K+r-1)\gamma_N$, which implies that
\begin{equation}
 \one_{[0,(K+r-1)\gamma_N]}(\mathcal L_{N,r}^{\ind})
 \mathsf{Rem}_{N,r,K}=0.
 \label{prof:eq:fd-source-tensor-spectral-separation}
\end{equation}
Define the corresponding sources in $\mathcal T_{N,r}^{\hc}$ by
\begin{align}
 \mathscr S_{N,r,K}^{\mathrm{low}}
 :=\mathsf{Top}_{N,r}^{\ord}\mathsf R_{N,r}F_{N,K}^{\otimes r},
 \qquad
 \mathscr S_{N,r,K}^{\mathrm{high}}
 :=\mathsf{Top}_{N,r}^{\ord}\mathsf R_{N,r}\mathsf{Rem}_{N,r,K}.
 \label{prof:eq:low-high-sources}
\end{align}
Here ``low'' and ``high'' refer only to this source decomposition; they do not
assert membership of $\mathscr S_{N,r,K}^{\mathrm{low}}$ or
$\mathscr S_{N,r,K}^{\mathrm{high}}$ in low- or high-energy spectral subspaces
of $L_{\hc}$.

Recall the heat semigroups $\mathcal{P}^{\star}_{N,r}(t)=e^{-t\mathcal L_{N,r}^{\star}}$ for $\star\in \{\hc, \ind\}$. This notation was introduced in  \eqref{prof:eq:heat-semigroups}.

\begin{lemma}[Source-specific fixed-window Duhamel comparison]
\label{prof:lem:fd-low-source}
For fixed $r\in\mathbb N$, $K\ge0$, and nonempty compact $J\subset\mathbb R$,
\begin{align}
 \sup_{\substack{S_N\subset V_N,\ |S_N|=k_N\\ s\in J}}
 n_N^{r/2}
 \left\|
 \mathcal P_{N,r}^{\hc}(t_N(s))\mathscr S_{N,r,K}^{\mathrm{low}}
 -\mathsf{Top}_{N,r}^{\ord}\mathsf R_{N,r}
  \mathcal P_{N,r}^{\ind}(t_N(s))F_{N,K}^{\otimes r}
 \right\|_2
 \le
 C_{D,r,K,J}\frac{\log n_N}{\sqrt{n_N}}.
 \label{prof:eq:fd-low-source}
\end{align}
\end{lemma}

\begin{proof}
Put $u=F_{N,K}^{\otimes r}$.  The two terms in \eqref{prof:eq:fd-low-source} are the endpoints of
one interpolation.  For $0\le a\le t$, set
\[
 \Upsilon(a):=e^{-(t-a)L_{\hc}}
 \mathsf{Top}_{N,r}^{\ord}\mathsf R_{N,r}e^{-aL_{\ind}}u.
\]
By \eqref{prof:eq:low-high-sources},
\[
 \Upsilon(0)=e^{-tL_{\hc}}\mathscr S_{N,r,K}^{\mathrm{low}},
 \qquad
 \Upsilon(t)=\mathsf{Top}_{N,r}^{\ord}\mathsf R_{N,r}e^{-tL_{\ind}}u.
\]
Differentiating gives
\[
 \Upsilon'(a)
 =e^{-(t-a)L_{\hc}}
 \bigl(L_{\hc}\mathsf{Top}_{N,r}^{\ord}\mathsf R_{N,r}
       -\mathsf{Top}_{N,r}^{\ord}\mathsf R_{N,r}L_{\ind}\bigr)
 e^{-aL_{\ind}}u.
\]
Now \eqref{prof:eq:fd-one-particle-cutoff} implies that $u$ has total
independent-walker spectral support below $r(K+1)\gamma_N$, and this support
is preserved by $e^{-aL_{\ind}}$.  Hence
$e^{-aL_{\ind}}u\in\mathscr V_{N,r}^{\ind}(r(K+1))$, so by
\eqref{prof:eq:JNrK} and using the generator intertwining defect notation \eqref{prof:eq:fd-generator-defect-definition}, we can write
\[
 \Upsilon'(a)=-e^{-(t-a)L_{\hc}}b_{e^{-aL_{\ind}}u}.
\]
Integrating from $0$ to $t$ gives the Duhamel identity
\begin{equation}
 e^{-tL_{\hc}}\mathscr S_{N,r,K}^{\mathrm{low}}
 -\mathsf{Top}_{N,r}^{\ord}\mathsf R_{N,r}e^{-tL_{\ind}}u
 =\int_0^t e^{-(t-a)L_{\hc}}b_{e^{-aL_{\ind}}u}\,\dd a.
 \label{prof:eq:fd-direct-duhamel}
\end{equation}

We now estimate the Duhamel integrand in
\eqref{prof:eq:fd-direct-duhamel}.
 By the definition \eqref{prof:eq:fd-shifted-resolvent-norm},
\[
 \|(L_{\hc}+\gamma_N)^{-1/2}b_{e^{-aL_{\ind}}u}\|_{\hc}
 =\|b_{e^{-aL_{\ind}}u}\|_{-1,\hc,\gamma_N}.
\]
Since $e^{-aL_{\ind}}u\in\mathscr V_{N,r}^{\ind}(r(K+1))$,
\eqref{prof:eq:fd-product-top-gap} and \eqref{prof:eq:fd-generator-defect-resolvent} gives
\begin{equation}
\|b_{e^{-aL_{\ind}}u}\|_{-1,\hc,\gamma_N}
 \le C\sqrt{\frac{\gamma_N}{n_N}}
 e^{-r\gamma_Na}\|u\|_{\ind}.
 \label{prof:eq:fd-direct-defect-norm}
\end{equation}
Recall that $L_{\hc} \ge \lambda_0:= r\gamma_N-C\frac{\gamma_N}{n_N}$ from
\cref{prof:prop:fd-top-gap}.  By the spectral theorem, for $\tau>0$,
\[
 \|e^{-\tau L_{\hc}}(L_{\hc}+\gamma_N)^{1/2}\|_{\mathrm{op}}
 =\sup_{\lambda\in\Spec(L_{\hc})}
   e^{-\tau\lambda}\sqrt{\lambda+\gamma_N}
 \le \sup_{\lambda\ge\lambda_0}
   e^{-\tau\lambda}\sqrt{\lambda+\gamma_N}.
\]
Writing $x=\lambda-\lambda_0\ge0$ and using
$\sqrt{x+\lambda_0+\gamma_N}\le\sqrt{x}+\sqrt{\lambda_0+\gamma_N}$,
we obtain
\[
 \sup_{\lambda\ge\lambda_0}e^{-\tau\lambda}\sqrt{\lambda+\gamma_N}
 \le e^{-\lambda_0\tau}
 \left\{
   \sup_{x\ge0}e^{-\tau x}\sqrt{x}
   +\sqrt{\lambda_0+\gamma_N}
 \right\}.
\]
The first supremum is $O(\tau^{-1/2})$, while
$\lambda_0+\gamma_N\le C_r\gamma_N$ for all sufficiently large $N$.
Hence
\begin{equation}
 \|e^{-\tau L_{\hc}}(L_{\hc}+\gamma_N)^{1/2}\|_{\mathrm{op}}
 \le C_r e^{-\lambda_0\tau}
 \left(\tau^{-1/2}+\sqrt{\gamma_N}\right).
 \label{prof:eq:fd-direct-smoothing}
\end{equation}

Put $\tau=t-a$.  Since
$L_{\hc}+\gamma_N$ is strictly positive,
\[
 e^{-\tau L_{\hc}}b_{e^{-aL_{\ind}}u}
 =\left(e^{-\tau L_{\hc}}(L_{\hc}+\gamma_N)^{1/2}\right)
  \left( (L_{\hc}+\gamma_N)^{-1/2}b_{e^{-aL_{\ind}}u}\right).
\]
Therefore
\begin{align*}
 \|e^{-\tau L_{\hc}}b_{e^{-aL_{\ind}}u}\|_{\hc}
 &\le
 \|e^{-\tau L_{\hc}}(L_{\hc}+\gamma_N)^{1/2}\|_{\mathrm{op}}
 \|(L_{\hc}+\gamma_N)^{-1/2}b_{e^{-aL_{\ind}}u}\|_{\hc}\\
 &\le
 C\sqrt{\frac{\gamma_N}{n_N}}
 e^{-\lambda_0\tau}e^{-r\gamma_Na}
 \left(\tau^{-1/2}+\sqrt{\gamma_N}\right)\|u\|_{\ind},
\end{align*}
where we used \eqref{prof:eq:fd-direct-smoothing} and \eqref{prof:eq:fd-direct-defect-norm} in the last inequality.
Because $a=t-\tau$,
\[
 e^{-\lambda_0\tau}e^{-r\gamma_Na}
 =e^{-r\gamma_Nt}e^{(r\gamma_N-\lambda_0)\tau}
 \le e^{-r\gamma_Nt}e^{C\gamma_N\tau/n_N}.
\]
Applying the triangle inequality to the integral in
\eqref{prof:eq:fd-direct-duhamel} now gives
\begin{align*}
 \|e^{-tL_{\hc}}\mathscr S_{N,r,K}^{\mathrm{low}}
 -\mathsf{Top}_{N,r}^{\ord}\mathsf R_{N,r}e^{-tL_{\ind}}u\|_{\hc}
 &\le C\sqrt{\frac{\gamma_N}{n_N}}e^{-r\gamma_Nt}\|u\|_{\ind}
 \int_0^t e^{C\gamma_N\tau/n_N}
 \left(\tau^{-1/2}+\sqrt{\gamma_N}\right)\,\dd\tau.
\end{align*}
At $t=t_N(s)$ with $s\in J$, one has
$\gamma_Nt_N(s)=(\log n_N+s)/2$, so
$e^{C\gamma_N\tau/n_N}$ is uniformly bounded for
$0\le\tau\le t_N(s)$ and $s\in J$.  Moreover,
\[
 \int_0^{t_N(s)}
 \left(\tau^{-1/2}+\sqrt{\gamma_N}\right)\,\dd\tau
 =2\sqrt{t_N(s)}+\sqrt{\gamma_N}\,t_N(s)
 \le C_J\frac{\log n_N}{\sqrt{\gamma_N}}.
\]
Consequently,
\[
 \|e^{-t_N(s)L_{\hc}}\mathscr S_{N,r,K}^{\mathrm{low}}
 -\mathsf{Top}_{N,r}^{\ord}\mathsf R_{N,r}e^{-t_N(s)L_{\ind}}u\|_{\hc}
 \le C_{D,r,K,J}
 \frac{\log n_N}{\sqrt{n_N}}
 e^{-r\gamma_Nt_N(s)}\|u\|_{\ind}.
\]
Finally we multiply both sides of the last display by $n_N^{r/2}$.
On the one hand, $n_N^{r/2}e^{-r\gamma_Nt_N(s)}=e^{-rs/2}$ is uniformly bounded for $s\in J$.  
On the other hand,
$\|u\|_{\ind}=\|F_{N,K}\|_{2,\mathrm{av}}^r$ is uniformly bounded because
$\Xi_{N,K}^{(1)}$ is an $L^2$ contraction and
$\|F_N\|_{2,\mathrm{av}}^2=\rho_N(1-\rho_N)\le1/4$.
These uniform bounds can be absorbed into $C_{D,r,K,J}$, thereby yielding the claimed inequality \eqref{prof:eq:fd-low-source}.
\end{proof}

The high-source estimate for $e^{-tL_{\hc}}$ requires only one-sided control of
low-energy eigenvectors of $L_{\hc}$, which is the content of the next two lemmas,
\cref{prof:lem:fd-one-sided-eigenvector,prof:lem:fd-high-source}.
For $H>r$, we let
$K_H:=K_+(D,r,H)$ be the cutoff furnished by
\cref{prof:lem:fd-high-block}, and use the shorthands
\[
 \Pi_H^{\mathrm{cmp}}:=\Pi_{N,r,K_H}^{\mathrm{cmp}},
 \qquad
 \Pi_H^{\mathrm{cmp},\perp}:=\Pi_{N,r,K_H}^{\mathrm{cmp},\perp}.
\]

\begin{lemma}[One-sided low-eigenvector leakage]
\label{prof:lem:fd-one-sided-eigenvector}
Fix $H>r$, with $K_H$, $\Pi_H^{\mathrm{cmp}}$, and $\Pi_H^{\mathrm{cmp},\perp}$ as just defined.
For all sufficiently large $N$, if $L_{\hc}\psi=\lambda\psi$,
$0\le\lambda\le H\gamma_N$, and
\[
 \psi=\psi_{\mathrm{cmp}}+\psi_{\perp},
 \quad\text{where }~
 \psi_{\mathrm{cmp}}=\Pi_H^{\mathrm{cmp}}\psi
 ~\text{ and }~
 \psi_{\perp}=\Pi_H^{\mathrm{cmp},\perp}\psi,
\]
then
\begin{equation}
 \|\psi_{\perp}\|_{\hc}
 \le C_{D,r,H}n_N^{-1/2}\|\psi_{\mathrm{cmp}}\|_{\hc}.
 \label{prof:eq:fd-one-sided-eigenvector}
\end{equation}
Consequently, $\Pi_H^{\mathrm{cmp}}$ is injective on
$\one_{[0,H\gamma_N]}(L_{\hc})\mathcal T_{N,r}^{\hc}$ and
\begin{equation}
 \operatorname{rank}\one_{[0,H\gamma_N]}(L_{\hc})
 \le\dim\mathscr V_{N,r}^{\ind}(K_H)=O_{D,r,H}(1).
 \label{prof:eq:fd-one-sided-rank}
\end{equation}
\end{lemma}

\begin{proof}
Left-multiplying the eigenvalue equation $L_{\hc}\psi=\lambda\psi$ by
$\Pi_H^{\mathrm{cmp},\perp}$ and using
$\psi=\psi_{\mathrm{cmp}}+\psi_{\perp}$ gives the projected eigenvalue equation
\begin{equation}
 \Pi_H^{\mathrm{cmp},\perp}(L_{\hc}-\lambda)\Pi_H^{\mathrm{cmp},\perp}\psi_{\perp}
 =-\Pi_H^{\mathrm{cmp},\perp}L_{\hc}\psi_{\mathrm{cmp}}.
 \label{prof:eq:fd-one-sided-projected-eigenvalue}
\end{equation}
Set
\[
 \mathcal L_{\perp,\lambda}
 :=\Pi_H^{\mathrm{cmp},\perp}(L_{\hc}-\lambda)\Pi_H^{\mathrm{cmp},\perp}
 \quad\text{on }\Ran\Pi_H^{\mathrm{cmp},\perp}.
\]
Applying the high-block estimate \eqref{prof:eq:fd-high-block-gap} with
$H_0=H$ gives
$\Pi_H^{\mathrm{cmp},\perp}L_{\hc}\Pi_H^{\mathrm{cmp},\perp}\ge(H+2)\gamma_N I$ on
$\Ran\Pi_H^{\mathrm{cmp},\perp}$.  Since $\lambda\le H\gamma_N$,
\begin{align}
 \mathcal L_{\perp,\lambda}\ge2\gamma_N I.
 \label{prof:eq:L_perp_lambda-bound}
\end{align}
The same high-block estimate and
\eqref{prof:eq:fd-shifted-compression-resolvent}, now with
$Q=\Pi_H^{\mathrm{cmp},\perp}$ and $z=\lambda\le H\gamma_N$, also give
\begin{equation}
 \mathcal L_{\perp,\lambda}^{-1}
 \le\frac{H+3}{2}
 \Pi_H^{\mathrm{cmp},\perp}(L_{\hc}+\gamma_N)^{-1}\Pi_H^{\mathrm{cmp},\perp}
 \quad\text{on }\Ran\Pi_H^{\mathrm{cmp},\perp}.
 \label{prof:eq:L_perp_lambda-inverse}
\end{equation}

We next express the coupling term on the right-hand side of
\eqref{prof:eq:fd-one-sided-projected-eigenvalue} through the generator defect.
Using the Gram normalization \eqref{prof:eq:fd-Gram-normalization}, write
\[
 \psi_{\mathrm{cmp}}
 =\widehat{\mathsf J}_{N,r,K_H}u
 =\mathsf J_{N,r,K_H}v,
 \qquad
 v=\mathsf G_{N,r,K_H}^{-1/2}u.
\]
The product spectral window $\mathscr V_{N,r}^{\ind}(K_H)$ is invariant under
$L_{\ind}$, so
$\mathsf J_{N,r,K_H}L_{\ind}v\in\Ran\Pi_H^{\mathrm{cmp}}$ by
\eqref{prof:eq:fd-Gram-isometry}.  Projecting the generator defect identity
\eqref{prof:eq:fd-generator-defect-definition} onto $\Ran\Pi_H^{\mathrm{cmp},\perp}$
therefore gives
\begin{align}
 \Pi_H^{\mathrm{cmp},\perp}L_{\hc}\psi_{\mathrm{cmp}}
 =-\Pi_H^{\mathrm{cmp},\perp}b_v.
 \label{prof:eq:proj-defect}
\end{align}

We now estimate
$\psi_{\perp}=\Pi_H^{\mathrm{cmp},\perp}\psi$, the component of $\psi$ orthogonal to the
comparison range $\Ran\Pi_H^{\mathrm{cmp}}$.  The lower bound \eqref{prof:eq:L_perp_lambda-bound} implies, by spectral
calculus on $\Ran\Pi_H^{\mathrm{cmp},\perp}$,
\begin{equation}
 \mathcal L_{\perp,\lambda}^{-2}
 \le(2\gamma_N)^{-1}\mathcal L_{\perp,\lambda}^{-1}.
 \label{prof:eq:inverse-square-bound}
\end{equation}
By \eqref{prof:eq:fd-one-sided-projected-eigenvalue},
\begin{align}
 \psi_{\perp}
 =-\mathcal L_{\perp,\lambda}^{-1}
   \Pi_H^{\mathrm{cmp},\perp}L_{\hc}\psi_{\mathrm{cmp}}.
   \label{prof:eq:proj-gen-defect}
\end{align}
Using \cref{prof:eq:proj-gen-defect,prof:eq:inverse-square-bound,prof:eq:L_perp_lambda-inverse} and then \eqref{prof:eq:proj-defect} gives
\begin{align*}
 \|\psi_{\perp}\|_{\hc}^2
 &=\langle \Pi_H^{\mathrm{cmp},\perp}L_{\hc}\psi_{\mathrm{cmp}},
   \mathcal L_{\perp,\lambda}^{-2}
   \Pi_H^{\mathrm{cmp},\perp}L_{\hc}\psi_{\mathrm{cmp}}\rangle_{\hc}\\
 &\le\frac1{2\gamma_N}
   \langle \Pi_H^{\mathrm{cmp},\perp}L_{\hc}\psi_{\mathrm{cmp}},
   \mathcal L_{\perp,\lambda}^{-1}
   \Pi_H^{\mathrm{cmp},\perp}L_{\hc}\psi_{\mathrm{cmp}}\rangle_{\hc}\\
 &\le\frac{C_H}{\gamma_N}
   \langle \Pi_H^{\mathrm{cmp},\perp}b_v,
   (L_{\hc}+\gamma_N)^{-1}\Pi_H^{\mathrm{cmp},\perp}b_v\rangle_{\hc}
 \le\frac{C_{D,r,H}}{n_N}\|v\|_{\ind}^2,
\end{align*}
where the last inequality is
\eqref{prof:eq:fd-generator-defect-resolvent-analog}.
Finally, \eqref{prof:eq:fd-Gram-isometry} gives
$\|u\|_{\ind}=\|\psi_{\mathrm{cmp}}\|_{\hc}$, while
\eqref{prof:eq:fd-Gram-uniform-bounds} and
$v=\mathsf G_{N,r,K_H}^{-1/2}u$ give
$\|v\|_{\ind}\le C_{D,r,H}\|u\|_{\ind}$.  Hence
\eqref{prof:eq:fd-one-sided-eigenvector} follows.

It remains to prove the rank assertion on the entire low spectral subspace.
Let
\[
 \Pi_H^{\mathrm{spec}}:=\one_{[0,H\gamma_N]}(L_{\hc})
 \quad\text{on }\mathcal T_{N,r}^{\hc}.
\]
Suppose that $f\in\Ran\Pi_H^{\mathrm{spec}}$ and $\Pi_H^{\mathrm{cmp}}f=0$.  Then
$f\in\Ran\Pi_H^{\mathrm{cmp},\perp}$, so the high-block estimate
\eqref{prof:eq:fd-high-block-gap}, with $H_0=H$, gives
\[
 \langle f,L_{\hc}f\rangle_{\hc}
 \ge (H+2)\gamma_N\|f\|_{\hc}^2.
\]
On the other hand, since $f\in\Ran\Pi_H^{\mathrm{spec}}$, spectral calculus gives
\[
 \langle f,L_{\hc}f\rangle_{\hc}
 \le H\gamma_N\|f\|_{\hc}^2.
\]
Hence $f=0$.  Thus $\Pi_H^{\mathrm{cmp}}$ is injective on $\Ran\Pi_H^{\mathrm{spec}}$.
Consequently,
\[
 \operatorname{rank}\one_{[0,H\gamma_N]}(L_{\hc})
 \le \dim\Ran\Pi_H^{\mathrm{cmp}}
 =\dim\mathscr V_{N,r}^{\ind}(K_H),
\]
which proves \eqref{prof:eq:fd-one-sided-rank}.
\end{proof}

\begin{lemma}[Fixed-cutoff high-source leakage]
\label{prof:lem:fd-high-source}
For every fixed $r\in\mathbb N$ and nonempty compact $J\subset\mathbb R$, there is a finite
$K_{\mathrm{src}}=K_{\mathrm{src}}(D,r)$ such that, uniformly over deterministic
sources $S_N\subset V_N$ with $|S_N|=k_N$,
\begin{align}
 \sup_{s\in J}n_N^{r/2}
 \left\|\mathcal P_{N,r}^{\hc}(t_N(s))
 \mathscr S_{N,r,K_{\mathrm{src}}}^{\mathrm{high}}\right\|_2&\longrightarrow0,
 \label{prof:eq:fd-high-source-hc}\\
 \sup_{s\in J}n_N^{r/2}
 \left\|\mathsf{Top}_{N,r}^{\ord}\mathsf R_{N,r}
 \mathcal P_{N,r}^{\ind}(t_N(s))
 \mathsf{Rem}_{N,r,K_{\mathrm{src}}}\right\|_2&\longrightarrow0.
 \label{prof:eq:fd-high-source-ind}
\end{align}
\end{lemma}

\begin{proof}
Set $H:=r+1$ and let
$K_H=K_+(D,r,H)$ be the cutoff fixed above and used in
\cref{prof:lem:fd-one-sided-eigenvector}.
Fix $K_{\mathrm{src}}>1$ so large that
\begin{equation}
 K_{\mathrm{src}}+r-1>K_H.
 \label{prof:eq:fd-fixed-cutoff-separation}
\end{equation}
By the Fourier character convention \eqref{prof:eq:probability-fourier-basis}
and the dispersion comparison \eqref{prof:eq:dispersion-comparison}, the
one-particle spectral space below $(K_{\mathrm{src}}+1)\gamma_N$ has
uniformly bounded dimension and is spanned by unit-modulus characters.
Together with $|F_N|\le1$, the tensor decomposition
\eqref{prof:eq:fd-source-tensor-telescoping}, and the high-source definition
\eqref{prof:eq:low-high-sources}, this gives uniformly
\begin{equation}
 \|F_{N,K_{\mathrm{src}}}\|_\infty+
 \|\mathsf{Rem}_{N,r,K_{\mathrm{src}}}\|_\infty+
 \|\mathscr S_{N,r,K_{\mathrm{src}}}^{\mathrm{high}}\|_{\hc}
 \le C_{D,r,K_{\mathrm{src}}}.
 \label{prof:eq:fd-fixed-source-bounds}
\end{equation}
The spectral separation
\eqref{prof:eq:fd-source-tensor-spectral-separation}, with
$K=K_{\mathrm{src}}$, shows that
$\mathsf{Rem}_{N,r,K_{\mathrm{src}}}$ has no product spectral component at
energy at most $(K_{\mathrm{src}}+r-1)\gamma_N$.

We first show that the hard-core high source has only a small component in
the comparison range.  More precisely, our target is
\begin{equation}
 \left\|\Pi_H^{\mathrm{cmp}}
 \mathscr S_{N,r,K_{\mathrm{src}}}^{\mathrm{high}}\right\|_{\hc}
 =\sup_{\substack{\varphi\in\mathscr V_{N,r}^{\ind}(K_H)\\
                  \|\varphi\|_{\ind}=1}}
 \left|\left\langle\widehat{\mathsf J}_{N,r,K_H}\varphi,
 \mathscr S_{N,r,K_{\mathrm{src}}}^{\mathrm{high}}\right\rangle_{\hc}\right|
 \le Cn_N^{-1}.
 \label{prof:eq:fd-high-source-low-pairing}
\end{equation}
Here the equality follows from the isometry
\eqref{prof:eq:fd-Gram-isometry}, whose range is
$\Ran\Pi_H^{\mathrm{cmp}}$.  To estimate the displayed supremum, fix a unit
$\varphi\in\mathscr V_{N,r}^{\ind}(K_H)$ and put
$w=\mathsf G_{N,r,K_H}^{-1/2}\varphi$, which is uniformly bounded in norm by
\eqref{prof:eq:fd-Gram-uniform-bounds}.  By the definition \eqref{prof:eq:JNrK},
$\mathsf J_{N,r,K_H}w\in\Ran\mathsf{Top}_{N,r}^{\ord}=\mathcal T_{N,r}^{\hc}$, so
\begin{align*}
 &\left\langle\widehat{\mathsf J}_{N,r,K_H}\varphi,
 \mathscr S_{N,r,K_{\mathrm{src}}}^{\mathrm{high}}\right\rangle_{\hc}
 =\left\langle\mathsf J_{N,r,K_H}w,
 \mathsf R_{N,r}\mathsf{Rem}_{N,r,K_{\mathrm{src}}}\right\rangle_{\hc}\\
 &\qquad=\left\langle (\mathsf J_{N,r,K_H}-\mathsf R_{N,r})w,
 \mathsf R_{N,r}\mathsf{Rem}_{N,r,K_{\mathrm{src}}}\right\rangle_{\hc}
 +
 \left\langle\mathsf R_{N,r}w,
 \mathsf R_{N,r}\mathsf{Rem}_{N,r,K_{\mathrm{src}}}\right\rangle_{\hc}.
\end{align*}
The first, top correction, term is $O(n_N^{-1})$ because by \eqref{prof:eq:fd-top-correction-L2},
\[
 \|(\mathsf J_{N,r,K_H}-\mathsf R_{N,r})w\|_{\hc}\le Cn_N^{-1},
\]
and
\eqref{prof:eq:fd-fixed-source-bounds} states a uniform bound for the $\infty$-norm of $\mathsf{Rem}_{N,r,K_{\mathrm{src}}}$.  For the second, restriction, term, the normalization identity
\eqref{prof:eq:fd-restriction-normalization} gives
\begin{align*}
 \left\langle\mathsf R_{N,r}w,
 \mathsf R_{N,r}\mathsf{Rem}_{N,r,K_{\mathrm{src}}}\right\rangle_{\hc}
=\theta_{N,r}^{-1}
 \left(\langle w,\mathsf{Rem}_{N,r,K_{\mathrm{src}}}\rangle_{\ind}
 -\left\langle\one_{\mathfrak D_{N,r}^{\mathrm{coll}}}w,
   \mathsf{Rem}_{N,r,K_{\mathrm{src}}}\right\rangle_{\ind}\right).
\end{align*}
The first inner product vanishes by product spectral orthogonality,
\eqref{prof:eq:fd-source-tensor-spectral-separation}, and
\eqref{prof:eq:fd-fixed-cutoff-separation}.  By \eqref{prof:eq:fd-distinct-tuple-density}, the collision set has product
probability $1-\theta_{N,r}=O_r(n_N^{-1})$, while the finite-window
$L^2$-to-$L^\infty$ bound \eqref{prof:eq:fd-finite-window-fourier-bounds}
and \eqref{prof:eq:fd-fixed-source-bounds} give uniform pointwise bounds.
Combining the two estimates proves
\eqref{prof:eq:fd-high-source-low-pairing}.

Let
$
 d_{N,H}:=\dim\Ran\Pi_H^{\mathrm{spec}}$.
Choose an orthonormal eigenbasis
$(\psi_j)_{j=1}^{d_{N,H}}$ of $\Ran\Pi_H^{\mathrm{spec}}$.
By \eqref{prof:eq:fd-one-sided-rank}, $d_{N,H}\le C_{D,r,H}$.  Decompose
$
 \psi_j=\psi_{j,\mathrm{cmp}}+\psi_{j,\perp}
$
relative to $\Pi_H^{\mathrm{cmp}}\oplus\Pi_H^{\mathrm{cmp},\perp}$.  By \eqref{prof:eq:fd-Gram-isometry},
$\widehat{\mathsf J}_{N,r,K_H}$ is an isometry onto
$\Ran\Pi_H^{\mathrm{cmp}}$.  Hence there is a unique
$\varphi_j\in\mathscr V_{N,r}^{\ind}(K_H)$ such that
\[
 \psi_{j,\mathrm{cmp}}
 =\widehat{\mathsf J}_{N,r,K_H}\varphi_j,
 \qquad
 \|\varphi_j\|_{\ind}
 =\|\psi_{j,\mathrm{cmp}}\|_{\hc}
 \le \|\psi_j\|_{\hc}=1.
\]
Applying \eqref{prof:eq:fd-high-source-low-pairing} to $\varphi_j$ and
using homogeneity therefore gives
\[
 \left|\left\langle
 \psi_{j,\mathrm{cmp}},
 \mathscr S_{N,r,K_{\mathrm{src}}}^{\mathrm{high}}
 \right\rangle_{\hc}\right|
 \le Cn_N^{-1}\|\varphi_j\|_{\ind}
 \le Cn_N^{-1}.
\]
Hence by the triangle inequality and Cauchy--Schwarz,
\begin{align*}
 |\langle\psi_j,
   \mathscr S_{N,r,K_{\mathrm{src}}}^{\mathrm{high}}\rangle_{\hc}|
 &\le
 |\langle\psi_{j,\mathrm{cmp}},
   \mathscr S_{N,r,K_{\mathrm{src}}}^{\mathrm{high}}\rangle_{\hc}|
 +\|\psi_{j,\perp}\|_{\hc}
  \|\mathscr S_{N,r,K_{\mathrm{src}}}^{\mathrm{high}}\|_{\hc}\\
 &\le Cn_N^{-1}+Cn_N^{-1/2}
 \le Cn_N^{-1/2},
\end{align*}
where the second term uses
\eqref{prof:eq:fd-one-sided-eigenvector} and
\eqref{prof:eq:fd-fixed-source-bounds}.  Summing the squared coefficients
over the uniformly bounded number $d_{N,H}$ of eigenvectors gives
\begin{equation}
 \|\Pi_H^{\mathrm{spec}}
   \mathscr S_{N,r,K_{\mathrm{src}}}^{\mathrm{high}}\|_{\hc}
 \le Cn_N^{-1/2}.
 \label{prof:eq:fd-fixed-high-lowprojection}
\end{equation}
For this low spectral component, the sharp damping
\eqref{prof:eq:fd-top-damping} and
\eqref{prof:eq:fd-fixed-high-lowprojection} give directly
\[
 n_N^{r/2}
 \|e^{-t_N(s)L_{\hc}}\Pi_H^{\mathrm{spec}}
   \mathscr S_{N,r,K_{\mathrm{src}}}^{\mathrm{high}}\|_{\hc}
 \le C_{D,r,J}n_N^{-1/2}.
\]
For the high spectral $I-\Pi_H^{\mathrm{spec}}$ component, the spectrum of $L_{\hc}$ is
larger than $H\gamma_N$.  Therefore
\begin{align*}
 n_N^{r/2}\|e^{-t_N(s)L_{\hc}}(I-\Pi_H^{\mathrm{spec}})
 \mathscr S_{N,r,K_{\mathrm{src}}}^{\mathrm{high}}\|_{\hc}
 &\le C n_N^{r/2}e^{-H\gamma_Nt_N(s)}\\
 &=C e^{-Hs/2}n_N^{(r-H)/2}
 \le C_Jn_N^{(r-H)/2}\longrightarrow0,
\end{align*}
because $H=r+1$.  This proves \eqref{prof:eq:fd-high-source-hc}.

Finally, for the remainder term, the spectral support identity
\eqref{prof:eq:fd-source-tensor-spectral-separation} gives
\begin{align*}
 n_N^{r/2}\left\|
 \mathsf{Top}_{N,r}^{\ord}\mathsf R_{N,r}e^{-t_N(s)L_{\ind}}
 \mathsf{Rem}_{N,r,K_{\mathrm{src}}}\right\|_{\hc}
 &\le C_r n_N^{r/2}
 e^{-(K_{\mathrm{src}}+r-1)\gamma_Nt_N(s)}
 \|\mathsf{Rem}_{N,r,K_{\mathrm{src}}}\|_{\ind} \\
 &\le C_{r,J}n_N^{-(K_{\mathrm{src}}-1)/2}\longrightarrow0,
\end{align*}
which proves \eqref{prof:eq:fd-high-source-ind}.
\end{proof}

\subsection{Fixed-degree dynamic comparison and matching}

\begin{proposition}[Fixed-degree dynamic chaos comparison]

For every fixed $R\in\mathbb N_0$ and nonempty compact $J\subset\mathbb R$,
\begin{equation}
 \sup_{\substack{S_N\subset V_N,\ |S_N|=k_N\\ s\in J}}
 \max_{0\le r\le R}
 \|\Psi_{N,r}(h_{N,t_N(s)}^{S_N})-
   \mathscr W_{N,r}(m_{N,t_N(s)}^{S_N})\|_2
 \longrightarrow0.
 \label{prof:eq:dynamic-wick-comparison}
\end{equation}
\end{proposition}

\begin{proof}
For $r=0$, both coordinates in \eqref{prof:eq:dynamic-wick-comparison} equal
$1$.  For $r=1$, one has
$\mathcal L_{N,1}^{\hc}=\mathcal L_{N,1}^{\ind}=L_N^{\RW}$ and
$\mathsf R_{N,1}=I$, so the two identities in
\cref{prof:lem:fd-exact-duality} coincide.  Thus the degrees $r=0,1$ are exact.

Fix $r\ge2$.  By the density window and the exact formula
\eqref{prof:eq:chaos-normalization}, the factors
$\alpha_{N,r}$ and $\beta_N^{\pm1}$ are uniformly bounded at fixed $r$.
Hence \cref{prof:lem:fd-exact-duality} reduces the claim to the corresponding
source-normalized semigroup comparison.  Choose once and for all the finite
cutoff $K_{\mathrm{src}}=K_{\mathrm{src}}(D,r)$ supplied by
\cref{prof:lem:fd-high-source}, and write
\[
 F_N^{\otimes r}
 =F_{N,K_{\mathrm{src}}}^{\otimes r}
  +\mathsf{Rem}_{N,r,K_{\mathrm{src}}}.
\]
By \cref{prof:eq:low-high-sources,prof:eq:HLcommute}, the triangle inequality
gives, with $t=t_N(s)$,
\begin{align*}
&n_N^{r/2}\Bigl\|
 \mathsf{Top}_{N,r}^{\ord}\mathcal P_{N,r}^{\hc}(t)
   \mathsf R_{N,r}F_N^{\otimes r}
 -\mathsf{Top}_{N,r}^{\ord}\mathsf R_{N,r}
   \mathcal P_{N,r}^{\ind}(t)F_N^{\otimes r}
 \Bigr\|_2\\
&\quad\le n_N^{r/2}\Bigl\|
 \mathcal P_{N,r}^{\hc}(t)
   \mathscr S_{N,r,K_{\mathrm{src}}}^{\mathrm{low}}
 -\mathsf{Top}_{N,r}^{\ord}\mathsf R_{N,r}
   \mathcal P_{N,r}^{\ind}(t)F_{N,K_{\mathrm{src}}}^{\otimes r}
 \Bigr\|_2\\
&\qquad+n_N^{r/2}\Bigl\|
 \mathcal P_{N,r}^{\hc}(t)
   \mathscr S_{N,r,K_{\mathrm{src}}}^{\mathrm{high}}
 \Bigr\|_2
 +n_N^{r/2}\Bigl\|
 \mathsf{Top}_{N,r}^{\ord}\mathsf R_{N,r}
   \mathcal P_{N,r}^{\ind}(t)
   \mathsf{Rem}_{N,r,K_{\mathrm{src}}}
 \Bigr\|_2.
\end{align*}
The first term is controlled by \cref{prof:lem:fd-low-source}; the second
and third terms converge to zero uniformly by
\cref{prof:lem:fd-high-source}.  Thus the fixed-$r$ comparison tends to zero
uniformly over deterministic sources $S_N\subset V_N$ with $|S_N|=k_N$ and
$s\in J$.  Since only finitely many degrees $r\le R$ occur, taking the
maximum proves \eqref{prof:eq:dynamic-wick-comparison}.
\end{proof}

\begin{theorem}[Fixed-degree chaos matching]
\label{prof:thm:fixed-degree-matching}
For every fixed $R\in\mathbb N_0$ and nonempty compact $J\subset\mathbb R$,
\begin{equation*}
 \sup_{\substack{S_N\subset V_N,\ |S_N|=k_N\\ s\in J}}
 \sum_{r=0}^R
 \|P_{N,r}h_{N,t_N(s)}^{S_N}-P_{N,r}g_{N,t_N(s)}^{S_N}\|_2^2
 \longrightarrow0.
\end{equation*}
\end{theorem}

\begin{proof}
Fix $R\in\mathbb N_0$ and a nonempty compact $J\subset\mathbb R$.
For $t=t_N(s)$ and $m=m_{N,t}^{S_N}$, the triangle inequality gives
\[
 \|\Psi_{N,r}(h_{N,t}^{S_N})-\Psi_{N,r}(g_{N,t}^{S_N})\|_2
 \le
 \|\Psi_{N,r}(h_{N,t}^{S_N})-\mathscr W_{N,r}(m)\|_2
 +
 \|\Psi_{N,r}(g_{N,t}^{S_N})-\mathscr W_{N,r}(m)\|_2.
\]
Taking the maximum over $0\le r\le R$ and the supremum over the admissible
sources and $s\in J$, the first term tends to zero by
\eqref{prof:eq:dynamic-wick-comparison}, while the second tends to zero by
\eqref{prof:eq:static-wick-comparison}.  Hence
\begin{equation}
 \sup_{\substack{S_N\subset V_N,\ |S_N|=k_N\\ s\in J}}
 \max_{0\le r\le R}
 \|\Psi_{N,r}(h_{N,t_N(s)}^{S_N})
   -\Psi_{N,r}(g_{N,t_N(s)}^{S_N})\|_2
 \longrightarrow0.
 \label{prof:eq:fixed-degree-uniform-conv}
\end{equation}

By the exact chaos isometry \eqref{prof:eq:exact-chaos-isometry},
\[
 \sum_{r=0}^R\|P_{N,r}(h-g)\|_{L^2(\pi_N)}^2
 =\sum_{r=0}^R\frac1{r!}
   \|\Psi_{N,r}(h)-\Psi_{N,r}(g)\|_2^2
 \le e\max_{0\le r\le R}
   \|\Psi_{N,r}(h)-\Psi_{N,r}(g)\|_2^2.
\]
Apply this with $h=h_{N,t_N(s)}^{S_N}$ and $g=g_{N,t_N(s)}^{S_N}$,
take the same supremum, and use
\eqref{prof:eq:fixed-degree-uniform-conv}.
\end{proof}

The fixed-degree matching theorem \cref{prof:thm:fixed-degree-matching}
closes the bounded-degree comparison.  The canonical tilt tail was already
established in \cref{prof:lem:tilt-tail}.  The remaining input is a
source-uniform high-degree bound for the exclusion density, supplied by the
warm-start module below.

\section{Profile-coordinate warm start: target and conventions}
\label{prof:sec:warm-start}

Throughout the warm-start module, $V_N=\T_N^D$, the density assumption
\eqref{prof:eq:density-window} remains in force, and $t_N(s)$ is the
global time from \eqref{prof:eq:profile-coordinate}.  Since $N$ is large but fixed, we use the shorter notation $n=n_N$, $k=k_N$, $\rho=\rho_N$, $\mathcal L_1=L_N^{\RW}$, and $\gamma=\gamma_N$.
Here a \emph{warm start} means a source-uniform $L^2(\pi_N)$ bound for the exclusion density at an earlier fixed profile coordinate; see \cref{prof:thm:warm-start} below.

\localheading{Abbreviations for the warm-start module.}
Functions on
\(\Omega_{N,k}\) are measured in \(L^2(\pi_N)\), ordered pair functions
are measured in
\(\mathsf X_{N,2}^{\ord}=L^2(\Omega_{N,2}^{\ord},\upsilon_{N,2}^{\ord})\),
and one-particle fields in the profile coordinate estimates are measured
with \(\|\cdot\|_{2,\mathrm{av}}\), unless counting measure is explicitly
indicated.  Recall from \eqref{prof:eq:probability-fourier-basis} that
\(e_p(x)=e^{2\pi i p\cdot x/N}=\sqrt n\,\chi_p(x)\).  Thus
\((e_p)_p\) is orthonormal for the uniform probability measure on \(V_N\),
whereas \((\chi_p)_p\) is the counting measure normalization used in
\cref{prof:cor:first-shell-profile}.
We abbreviate the one-particle quantities from
\eqref{prof:eq:one-particle-energy-oscillation} by
\begin{equation}
\mathcal{E}_1(f)
:=\mathcal E_N^{\RW}(f)
=
 \langle f,\mathcal L_1f\rangle_{\mathrm{av}}
 =\frac1{2n}\sum_{x\sim y}|f(x)-f(y)|^2,
 \label{prof:eq:local-normalization}
\end{equation}
and
\begin{equation}
 \omega_1(f):=\omega_N(f)=\max_{x\sim y}|f(x)-f(y)|.
 \label{prof:eq:omega1}
\end{equation}
The edge sum in \eqref{prof:eq:local-normalization} is over ordered neighboring pairs.

For a deterministic \(k\)-set \(S\subset V_N\), define $F_S=\one_S-\rho$; then $n^{-1}\sum_xF_S(x)=0$ and $\|F_S\|_{2,\mathrm{av}}^2=\rho(1-\rho)$.  We abbreviate \(h_t^S=h_{N,t}^S\).

\begin{theorem}[Profile-coordinate warm start]
\label{prof:thm:warm-start}
For every fixed profile coordinate $s_0<-1$, there is
$C_{D,\rho_0,s_0}<\infty$ such that for all sufficiently large $N$,
\begin{equation}
\sup_{\substack{S\subset V_N\\ |S|=k}}
\norm{h_{t_N(s_0)}^S}_{L^2(\pi_N)}^2
\le C_{D,\rho_0,s_0}.
 \label{prof:eq:warm-start}
\end{equation}
\end{theorem}

The proof is deferred to the end of \cref{prof:sec:strong-rayleigh}.
Sections~7--10 establish the four inputs used there: the pair energy estimate
(\cref{prof:sec:pair-energy}), exponential weight propagation
(\cref{prof:sec:chaos-weight}), the terminal row bound
(\cref{prof:sec:row}), and the Strong--Rayleigh moment and terminal chaos
closure (\cref{prof:sec:strong-rayleigh}).

The only non-elementary external theorem specific to the warm-start argument is
\cite[Proposition~5.1]{BBL09}, invoked in \cref{prof:sec:strong-rayleigh}.

\section{A cutoff window pair energy estimate}
\label{prof:sec:pair-energy}

This section proves the pair energy estimate
\cref{prof:thm:pair-energy}.  We specialize to degree two the ordered tuple
space and operators defined in
\eqref{prof:eq:ordered-up-down-formulas},
\eqref{prof:eq:fd-hard-core-generator}, and
\eqref{prof:eq:restriction-adjoint}, using the following abbreviations:
\[
 \Omega_2:=\Omega_{N,2}^{\ord},
 \qquad
 \upsilon_2^{\ord}:=\upsilon_{N,2}^{\ord},
 \qquad
 \mathsf R_2:=\mathsf R_{N,2},
\]
\[
 \mathcal L_2^{\hc}:=\mathcal L_{N,2}^{\hc},
 \qquad
 \mathcal P_2(t):=\mathcal P_{N,2}^{\hc}(t) = e^{-t\mathcal L^{\hc}_2},
 \qquad
 \mathcal E_2(f):=\mathcal E_{N,2}^{\hc}(f)
 =\langle f,\mathcal L_2^{\hc}f\rangle_{L^2(\upsilon_2^{\ord})}.
\]
Also abbreviate \(\mathsf{Down}_{2\to1}:=\mathsf{Down}_{N,2\to1}\).  Recall from \eqref{prof:eq:ordered-up-down-formulas} that $(\mathsf{Down}_{2\to1}f)(x)=(n-1)^{-1}\sum_{y\ne x}f(x,y)$.
The symmetric ordered top-pair space is the already defined sector
\begin{equation*}
 \mathcal T_{N,2}^{\hc}
 =\left\{f\in L^2(\Omega_2,\upsilon_2^{\ord}):
 f(x,y)=f(y,x),\ \mathsf{Down}_{2\to1}f=0\right\}.
\end{equation*}
By the definition of $\mathcal T_{N,2}^{\hc}$ in
\cref{prof:sec:chaos-interface}, $\mathsf{Top}_{N,2}^{\ord}$ is its orthogonal
projection.  \eqref{prof:eq:HLcommute} with $r=2$ gives the local
specialization
\begin{equation}
 \mathsf{Top}_{N,2}^{\ord}\mathcal L_2^{\hc}=\mathcal L_2^{\hc}\mathsf{Top}_{N,2}^{\ord},
 \qquad
 \mathsf{Top}_{N,2}^{\ord}\mathcal P_2(t)=\mathcal P_2(t)\mathsf{Top}_{N,2}^{\ord}.
 \label{prof:eq:pair-top-commute}
\end{equation}

For a fixed deterministic source $S$, retain the one-particle profile notation
from \eqref{prof:eq:profile} and abbreviate
\begin{align*}
 m_t:=m_{N,t}^{S}=e^{-t\mathcal L_1}F_S,
 \qquad
 \varphi_t:=\sqrt{n/\beta_N}\,m_t
 =\mathcal M_{N,1}[h_t^S].
\end{align*}
Because $F_S$ is mean zero, $m_t$ converges to zero in $L^2$ as $t\to\infty$.
Moreover,
\[
 \frac{\dd}{\dd t}\|m_t\|_{2,\mathrm{av}}^2
 =-2\mathcal E_1(m_t).
\]
Therefore we have the following time-integrated one-particle energy bound
\begin{align}
 \int_0^\infty\mathcal E_1(m_\tau)\,\dd\tau
 =\frac12\|F_S\|_{2,\mathrm{av}}^2
 =\frac12\rho(1-\rho)
 \le\frac18.
 \label{prof:eq:integrated-one-particle-energy}
\end{align}

For a one-particle field $f$, define the contact source
\begin{equation*}
 \mathcal B[f](x,y)
 =\one_{\{x\sim y\}}(f(x)-f(y))^2.
\end{equation*}
Define the unprojected pair discrepancy and its top projection by
\begin{equation}
 \overline G_2(t)
 :=\mathcal P_2(t)\mathsf R_2(\varphi_0\otimes\varphi_0)
 -\mathsf R_2(\varphi_t\otimes\varphi_t),
 \qquad
 G_2(t):=\mathsf{Top}_{N,2}^{\ord}\overline G_2(t).
 \label{prof:eq:G2}
\end{equation}
Thus $\overline G_2$ is the discrepancy between hard-core pair dynamics and the
restriction to $\Omega_2$ of the independent two-walk evolution, both
started from the same initial data $\varphi_0\otimes\varphi_0$.

\begin{lemma}[Exact pair generator defect identity and Duhamel formula]
\label{prof:lem:pair-generator-defect-duhamel}
On ordered distinct pairs,
\begin{equation}
 (\mathcal L_2^{\hc}\mathsf R_2-\mathsf R_2\mathcal L_2^{\ind})(f\otimes f)
 =\mathcal B[f],
 \label{prof:eq:pair-generator-defect}
\end{equation}
where $\mathcal L_2^{\ind}=\mathcal L_1\otimes I+I\otimes \mathcal L_1$.  Consequently,
the function $G_2$ in
\eqref{prof:eq:G2} satisfies the differential equation
\begin{align}
G_2'(t) = -\mathcal L_2^{\hc} G_2(t) - \mathsf{Top}_{N,2}^{\ord} \mathcal B[\varphi_t],
 \label{prof:eq:G2eqn}
\end{align}
from which it follows that
\begin{equation}
 G_2(t)
 =-\mathsf{Top}_{N,2}^{\ord}\int_0^t\, \mathcal P_2(t-\tau)\mathcal B[\varphi_\tau] \,\dd\tau.
 \label{prof:eq:pair-duhamel}
\end{equation}
\end{lemma}

\begin{proof}
Every independent move that preserves distinct occupancies also occurs in the
hard-core generator, and hence contributes nothing to the left-hand side of
\eqref{prof:eq:pair-generator-defect}.
This describes all such moves in the case $x\not\sim y$.
On the other hand, if $x\sim y$, the two suppressed independent moves contribute
\[
 \{f(x)f(y)-f(y)^2\}+\{f(x)f(y)-f(x)^2\}
 =-(f(x)-f(y))^2
\]
to $\mathcal L_2^{\ind}(f\otimes f)(x,y)$, while they are omitted by the hard-core
generator.
This proves \eqref{prof:eq:pair-generator-defect}.  Equivalently, the blocked
independent moves enter $\mathcal L_2^{\ind}(f\otimes f)$ with a negative sign,
so subtracting the independent generator from the hard-core generator produces
the positive contact source $\mathcal B[f]$.

Since $\partial_t\varphi_t=-\mathcal L_1\varphi_t$, tensorization gives
\[
 \partial_t\bigl[\mathsf R_2(\varphi_t\otimes\varphi_t)\bigr]
 =-\mathsf R_2\mathcal L_2^{\ind}(\varphi_t\otimes\varphi_t).
\]
Differentiating \eqref{prof:eq:G2}, using the commutation relation
\eqref{prof:eq:pair-top-commute}, and then adding and subtracting
$\mathcal L_2^{\hc}\mathsf{Top}_{N,2}^{\ord}\mathsf R_2(\varphi_t\otimes\varphi_t)$, gives
\begin{align*}
 G_2'(t)
 &=-\mathcal L_2^{\hc}\mathsf{Top}_{N,2}^{\ord}
   e^{-t\mathcal L_2^{\hc}}\mathsf R_2(\varphi_0\otimes\varphi_0)
   +\mathsf{Top}_{N,2}^{\ord}\mathsf R_2\mathcal L_2^{\ind}
     (\varphi_t\otimes\varphi_t)\\
 &=-\mathcal L_2^{\hc}G_2(t)
   -\mathsf{Top}_{N,2}^{\ord}
    \bigl(\mathcal L_2^{\hc}\mathsf R_2
          -\mathsf R_2\mathcal L_2^{\ind}\bigr)
     (\varphi_t\otimes\varphi_t)\\
 &=-\mathcal L_2^{\hc}G_2(t)
   -\mathsf{Top}_{N,2}^{\ord}\mathcal B[\varphi_t],
\end{align*}
where the last line uses \eqref{prof:eq:pair-generator-defect}.  This is
\eqref{prof:eq:G2eqn}.  Since $G_2(0)=0$, variation of constants gives
\eqref{prof:eq:pair-duhamel}.
\end{proof}

For every fixed profile coordinate $s_1\in\mathbb R$, the centralized one-particle estimate
\eqref{prof:eq:profile-one-particle-explicit} and
$\varphi_t=\sqrt{n/\beta_N}\,m_t$ give, for all sufficiently large $N$,
\begin{equation}
 \|\varphi_{t_N(s_1)}\|_{2,\mathrm{av}}^2+\|\varphi_{t_N(s_1)}\|_\infty^2
 \le C_{D,\rho_0}e^{-s_1}.
 \label{prof:eq:terminal-one-particle}
\end{equation}

\begin{lemma}[Contact-source bounds]
\label{prof:lem:source-estimates}
There is $C_{D,\rho_0}<\infty$ such that
\begin{align}
 \|\mathcal B[f]\|_2^2
 &\le \frac{2}{n-1}\,\omega_1(f)^2\mathcal E_1(f),
 \label{prof:eq:contact-square-exact}\\
 \|\mathcal B[\varphi_t]\|_2^2
 &\le C_{D,\rho_0}n\,\omega_1(m_t)^2\mathcal E_1(m_t),
 \qquad t\ge0.
 \label{prof:eq:source-master}
\end{align}
Consequently,
\begin{align}
 \|\mathcal B[\varphi_t]\|_2^2
 &\le C_{D,\rho_0}n\mathcal E_1(m_t),
 && t\ge0,
 \label{prof:eq:rough-source}\\
 \|\mathcal B[\varphi_t]\|_2^2
 &\le C_{D,\rho_0}n\gamma^2e^{-4\gamma t},
 && t\ge \gamma^{-1}.
 \label{prof:eq:late-source}
\end{align}
The first consequence is global in time, whereas the second uses the long-time smoothing estimate \eqref{prof:eq:long-one-particle}.
\end{lemma}

\begin{proof}
By the probability normalization on $\Omega_2$,
\[
 \|\mathcal B[f]\|_2^2
 =\frac1{n(n-1)}\sum_{x\sim y}|f(x)-f(y)|^4.
\]
Since
$
 |f(x)-f(y)|^4
 \le \omega_1(f)^2|f(x)-f(y)|^2$,
and the edge sum is over ordered neighboring pairs,
\eqref{prof:eq:local-normalization} gives
\[
 \|\mathcal B[f]\|_2^2
 \le \frac{2}{n-1}\,\omega_1(f)^2\mathcal E_1(f),
\]
which is \eqref{prof:eq:contact-square-exact}.

Since $\varphi_t=\sqrt{n/\beta_N}\,m_t$,
we have the scaling identities
$
 \omega_1(\varphi_t)^2
 =\frac{n}{\beta_N}\omega_1(m_t)^2
$
and
$
 \mathcal E_1(\varphi_t)
 =\frac{n}{\beta_N}\mathcal E_1(m_t)
 $.
Therefore \eqref{prof:eq:contact-square-exact} gives
\[
 \|\mathcal B[\varphi_t]\|_2^2
 \le
 \frac{2n^2}{(n-1)\beta_N^2}
 \omega_1(m_t)^2\mathcal E_1(m_t)
 \le C_{D,\rho_0}n\,\omega_1(m_t)^2\mathcal E_1(m_t),
\]
where the last inequality uses $\beta_N\asymp_{\rho_0}1$.
This is \eqref{prof:eq:source-master}.  Because
$-\rho\le F_S\le1-\rho$ and the Markov semigroup preserves this interval,
$\operatorname{osc}(m_t)\le1$, and hence $\omega_1(m_t)\le1$.
Substitution into \eqref{prof:eq:source-master} proves \eqref{prof:eq:rough-source}.

The long-time estimate \eqref{prof:eq:long-one-particle}, proved in
\cref{prof:lem:one-particle-bounds}, is exactly
$\omega_1(m_t)^2+\mathcal E_1(m_t)\le C_D\gamma e^{-2\gamma t}$ in the present abbreviations.  Combining it with
\eqref{prof:eq:source-master} gives \eqref{prof:eq:late-source}.
\end{proof}

\subsection{High pair spectrum}

Let
\begin{equation*}
 \mathcal L_{2,\mathrm{top}}^{\hc}
 :=\mathcal L_2^{\hc}\big|_{\mathcal T_{N,2}^{\hc}}.
\end{equation*}
Fix $H\ge1$ and define, by functional calculus on
$\mathcal T_{N,2}^{\hc}$, the high and low spectral projections
\begin{equation}
 \Pi_{\mathrm{hi},H}:=\one_{[H\gamma,\infty)}
   (\mathcal L_{2,\mathrm{top}}^{\hc}),
 \qquad
 \Pi_{\mathrm{lo},H}:=\one_{[0,H\gamma)}
   (\mathcal L_{2,\mathrm{top}}^{\hc}).
 \label{prof:eq:pair-spectral-projections}
\end{equation}
By the defining equation \eqref{prof:eq:G2}, $G_2(t)\in\mathcal T_{N,2}^{\hc}$, and we put $G_{2,\mathrm{hi}}(t)=\Pi_{\mathrm{hi},H}G_2(t)$ and $G_{2,\mathrm{lo}}(t)=\Pi_{\mathrm{lo},H}G_2(t)$.

\begin{lemma}[High-spectrum semigroup smoothing bound]
\label{prof:lem:high-kernel}
For fixed $D\ge2$ and $H\ge1$, there is $C_{D,H}<\infty$ such that
\begin{equation}
 \| (\mathcal L_2^{\hc})^{1/2}\Pi_{\mathrm{hi},H}e^{-u\mathcal L_2^{\hc}}\|_{2\to2}
 \le \mathsf K_H(u):=C_{D,H}(1+u)^{-1/2}e^{-H\gamma u/4}
 \label{prof:eq:high-kernel}
\end{equation}
for all $u\ge0$.  Moreover,
\begin{equation}
 \int_0^\infty\, \mathsf K_H(u)\,\dd u\le C_{D,H}\gamma^{-1/2}.
 \label{prof:eq:high-kernel-L1}
\end{equation}
\end{lemma}

\begin{proof}
The ordered two-particle exclusion graph has degree at most $4D$, so the corresponding graph Laplacian $\mathcal L_2^{\hc}$ has maximum absolute row sum at most $8D$.
This implies the spectral radius bound $\Spec(\mathcal L_2^{\hc})\subset[0,8D]$.
If $\Pi_{\mathrm{hi},H}=0$, then \eqref{prof:eq:high-kernel} is immediate.
Otherwise $H\gamma\le 8D$, and spectral calculus gives
\[
 \left\|
 (\mathcal L_2^{\hc})^{1/2}\Pi_{\mathrm{hi},H}e^{-u\mathcal L_2^{\hc}}
 \right\|_{2\to2}
 \le\sup_{\lambda\in[H\gamma,8D]}\sqrt\lambda\,e^{-u\lambda}.
\]
Note the elementary bound $\sup_{\lambda\ge0}\sqrt\lambda e^{-u\lambda/2}=C u^{-1/2}$. 
As a result, for $u\ge1$ we have
\[
 \sqrt\lambda\,e^{-u\lambda}
 =\bigl(\sqrt\lambda\,e^{-u\lambda/2}\bigr)e^{-u\lambda/2}
 \le C u^{-1/2}e^{-H\gamma u/2}.
\]
For $0\le u\le1$, we apply the crude spectral upper bound to get
$
 \sqrt\lambda e^{-u\lambda}\le\sqrt{8D}$.
On the same interval,
$(1+u)^{-1/2}e^{-H\gamma u/4}\ge
2^{-1/2}e^{-2DH}$, so, after enlarging $C_{D,H}$,
this is bounded by the right-hand side of
\eqref{prof:eq:high-kernel}.
To make the $\gamma^{-1/2}$ scale explicit, split the integral at
$u=\gamma^{-1}$:
\[
 \int_0^\infty (1+u)^{-1/2}e^{-H\gamma u/4}\,\dd u
 \le \int_0^{\gamma^{-1}}(1+u)^{-1/2}\,\dd u
 +\int_{\gamma^{-1}}^\infty u^{-1/2}e^{-H\gamma u/4}\,\dd u
 \le C_H\gamma^{-1/2},
\]
where the second integral is estimated after the change of variables
$v=\gamma u$.  This proves \eqref{prof:eq:high-kernel-L1}.
\end{proof}

\begin{proposition}[High-spectrum cutoff estimate]
\label{prof:prop:high-spectrum}
Fix $A>0$.  One can choose $H=H(D,A)$ so that, uniformly for
\begin{equation}
 t\ge A\gamma^{-1}\log n,
 \label{prof:eq:lower-window}
\end{equation}
all deterministic $S\subset V_N$ with $|S|=k$, and all sufficiently large $N$,
\begin{equation}
 \mathcal E_2(G_{2,\mathrm{hi}}(t))
 \le C_{D,\rho_0,A}\,n\gamma e^{-4\gamma t}.
 \label{prof:eq:high-spectrum}
\end{equation}
\end{proposition}

\begin{proof}
We first make the spectral calculus reduction explicit.  Applying
$\Pi_{\mathrm{hi},H}$ to the Duhamel formula in
\cref{prof:lem:pair-generator-defect-duhamel}, namely \eqref{prof:eq:pair-duhamel}, using the commutation relation
\eqref{prof:eq:pair-top-commute}, and using that $\Pi_{\mathrm{hi},H}$ is a spectral
projection of $\mathcal L_{2,\mathrm{top}}^{\hc}$ gives
\begin{equation*}
 G_{2,\mathrm{hi}}(t)
 =-\int_0^t
 \Pi_{\mathrm{hi},H}\mathcal P_2(t-\tau)
 \mathsf{Top}_{N,2}^{\ord}\mathcal B[\varphi_\tau] \,\dd\tau.
\end{equation*}
The projection $\Pi_{\mathrm{hi},H}$ commutes with the semigroup and with
$(\mathcal L_2^{\hc})^{1/2}$ on the top-pair sector.  Therefore, using the definition of $\mathcal E_2$,
Minkowski's integral inequality, the kernel estimate from
\cref{prof:lem:high-kernel}, namely \eqref{prof:eq:high-kernel}, and the contractivity of the orthogonal projection $\mathsf{Top}_{N,2}^{\ord}$, we obtain
\begin{align*}
 \mathcal E_2(G_{2,\mathrm{hi}}(t))^{1/2}
 &=\left\|
 \int_0^t\,
 (\mathcal L_2^{\hc})^{1/2}\Pi_{\mathrm{hi},H}
 e^{-(t-\tau)\mathcal L_2^{\hc}}
 \mathsf{Top}_{N,2}^{\ord}\mathcal B[\varphi_\tau] \,\dd\tau
 \right\|_2 \notag\\
 &\le\int_0^t\,
 \left\|(\mathcal L_2^{\hc})^{1/2}\Pi_{\mathrm{hi},H}
 e^{-(t-\tau)\mathcal L_2^{\hc}}\right\|_{2\to2}
 \left\|\mathsf{Top}_{N,2}^{\ord}\mathcal B[\varphi_\tau]\right\|_2\,\dd\tau
 \le\int_0^t\,\mathsf K_H(t-\tau)\|\mathcal B[\varphi_\tau]\|_2\,\dd\tau.
\end{align*}
Weighted Cauchy--Schwarz, followed by
\eqref{prof:eq:high-kernel-L1}, now yields
\begin{align}
 \mathcal E_2(G_{2,\mathrm{hi}}(t))
 &\le \left(\int_0^t\, \mathsf K_H(t-\tau)\,\dd\tau\right)
 \left(\int_0^t\, \mathsf K_H(t-\tau)
 \|\mathcal B[\varphi_\tau]\|_2^2\,\dd\tau\right) \notag\\
 &\le C_{D,H}\gamma^{-1/2}
 \int_0^t\,\mathsf K_H(t-\tau)\|\mathcal B[\varphi_\tau]\|_2^2\,\dd\tau.
 \label{prof:eq:high-weighted}
\end{align}
Set
\[
 I_{\mathrm{early}}(t)
 :=\int_0^{t/2}\,\mathsf K_H(t-\tau)\|\mathcal B[\varphi_\tau]\|_2^2\,\dd\tau,
 \qquad
 I_{\mathrm{late}}(t)
 :=\int_{t/2}^{t}\,\mathsf K_H(t-\tau)\|\mathcal B[\varphi_\tau]\|_2^2\,\dd\tau.
\]
We estimate these two terms separately.

\emph{Early half.}
For $0\le\tau\le t/2$, one has $t-\tau\ge t/2$.  Hence, using the rough source estimate from
\cref{prof:lem:source-estimates}, namely \eqref{prof:eq:rough-source},
together with \eqref{prof:eq:high-kernel} and the time-integrated energy
bound \eqref{prof:eq:integrated-one-particle-energy}, we obtain
\begin{align*}
 I_{\mathrm{early}}(t)
 \le C_{D,\rho_0,H}n e^{-H\gamma t/8}
 \int_0^{t/2}\mathcal E_1(m_\tau)\,\dd\tau
 \le C_{D,\rho_0,H}n e^{-H\gamma t/8}.
\end{align*}
Its contribution to the right-hand side of
\eqref{prof:eq:high-weighted} is therefore at most
$
 C_{D,\rho_0,H}n\gamma^{-1/2}e^{-H\gamma t/8}
$.
Relative to the target $n\gamma e^{-4\gamma t}$ in \eqref{prof:eq:high-spectrum}, this has ratio
$
 C_{D,\rho_0,H}\gamma^{-3/2}
 e^{-(H/8-4)\gamma t}$.
Since $\gamma\asymp_D n^{-2/D}$ and
$t\ge A\gamma^{-1}\log n$, the said ratio is bounded by
$
 C_{D,\rho_0,H}n^{3/D-A(H/8-4)}
$.
Choose $H=H(D,A)>32$ so that
$A(H/8-4)>3/D$.  Then the early-half contribution is bounded by
$C_{D,\rho_0,A}n\gamma e^{-4\gamma t}$ for all sufficiently large $N$.

\emph{Late half.}
Condition \eqref{prof:eq:lower-window} implies, for all sufficiently large $N$,
that $\tau\ge t/2\ge\gamma^{-1}$ throughout $[t/2,t]$.  We may therefore use
the long-time source estimate \eqref{prof:eq:late-source} from
\cref{prof:lem:source-estimates}.  With $u=t-\tau$,
\begin{align*}
 I_{\mathrm{late}}(t)
 &\le C_{D,\rho_0,H}n\gamma^2e^{-4\gamma t}
 \int_0^{t/2}(1+u)^{-1/2}
 e^{-(H/4-4)\gamma u}\,\dd u \notag\\
 &\le C_{D,\rho_0,H}n\gamma^2e^{-4\gamma t}
 \int_0^\infty(1+u)^{-1/2}
 e^{-(H/4-4)\gamma u}\,\dd u
 \le C_{D,\rho_0,H}n\gamma^{3/2}e^{-4\gamma t},
\end{align*}
where the last inequality uses $H>16$ and the change of variables
$v=(H/4-4)\gamma u$.  Substitution into
\eqref{prof:eq:high-weighted} shows that the late-half contribution is at most
$
 C_{D,\rho_0,H}n\gamma e^{-4\gamma t}.
 $
Adding the completed early- and late-half estimates proves
\eqref{prof:eq:high-spectrum}.
\end{proof}

\subsection{Low pair spectrum}

The low-spectrum argument requires finite-volume smoothing on the ordered
two-particle exclusion graph.  The needed Nash inequality and heat kernel
consequences are stated in \cref{prof:lem:pair-nash,prof:lem:low-geometry}.  In
those two results, $\|\cdot\|_{p,\mathrm{cnt}}$ denotes the counting measure
$\ell^p$ norm on the relevant pair space, and
$\langle\cdot,\cdot\rangle_{\mathrm{cnt}}$ denotes the corresponding
counting inner product.  We write
\[
 \mathcal E_{2,\mathrm{cnt}}^{\hc}(u)
 :=\langle u,\mathcal L_2^{\hc}u\rangle_{\mathrm{cnt}},
 \qquad
 \mathcal E_{2,\mathrm{cnt}}^{\ind}(g)
 :=\langle g,\mathcal L_2^{\ind}g\rangle_{\mathrm{cnt}}
\]
for the corresponding counting measure Dirichlet forms.  Relative to the
probability-normalized forms fixed in \cref{prof:sec:chaos-interface},
\begin{equation*}
 \mathcal E_{2,\mathrm{cnt}}^{\hc}(u)=n(n-1)\mathcal E_2(u),
 \qquad
 \mathcal E_{2,\mathrm{cnt}}^{\ind}(g)=n^2\mathcal E_{N,2}^{\ind}(g).
\end{equation*}
Thus the hard-core and independent counting normalizations differ by the
cardinalities of $\Omega_2$ and $V_N^2$, respectively.

The general extension $\mathsf{Ext}_{N,r}$ from
\cref{prof:lem:fd-local-fill} is designed for fixed-degree $L^2$ and
Dirichlet form comparison.  The Nash argument below additionally needs a
counting measure $L^1$ bound.  At degree two it is therefore useful to employ
the following simpler diagonal extension, specialized to the pair space and
chosen precisely to preserve both $L^1$ and $L^2$ at bounded cost.

\begin{lemma}[Diagonal extension and pair Nash inequality]
\label{prof:lem:pair-nash}
There is a linear diagonal pair extension $\mathsf{Ext}_2^{\mathrm{diag}}$ from functions on $\Omega_2$ to
functions on $V_N^2$, agreeing off the diagonal, such that for $q=1,2$,
\begin{equation}
 \|\mathsf{Ext}_2^{\mathrm{diag}}u\|_{q,\mathrm{cnt}}\le C_D\|u\|_{q,\mathrm{cnt}},
 \qquad
 \mathcal E_{2,\mathrm{cnt}}^{\ind}(\mathsf{Ext}_2^{\mathrm{diag}}u)
 \le C_D\mathcal E_{2,\mathrm{cnt}}^{\hc}(u).
 \label{prof:eq:diagonal-extension}
\end{equation}
Consequently,
\begin{equation}
 \|u\|_{2,\mathrm{cnt}}^2
 \le C_D [\mathcal E_{2,\mathrm{cnt}}^{\hc}(u)]^{D/(D+1)}
       \|u\|_{1,\mathrm{cnt}}^{2/(D+1)}
 +\frac{C_D}{|\Omega_2|}\|u\|_{1,\mathrm{cnt}}^2.
 \label{prof:eq:pair-nash}
\end{equation}
\end{lemma}

\begin{proof}
For $x\in V_N$, let
\[
 \mathsf{Nbr}_x^{\mathrm{pair}}
 =\{(x+\boldsymbol\delta,x),(x,x+\boldsymbol\delta):\boldsymbol\delta\in\mathcal D_D\}.
\]
Set $\mathsf{Ext}_2^{\mathrm{diag}}u=u$ off the diagonal and
\[
 (\mathsf{Ext}_2^{\mathrm{diag}}u)(x,x)
 =\frac1{4D}\sum_{\boldsymbol z\in\mathsf{Nbr}_x^{\mathrm{pair}}}u(\boldsymbol z).
\]
We first verify the norm bounds.  Every ordered contact pair belongs to
exactly two of the sets $\mathsf{Nbr}_x^{\mathrm{pair}}$.  Hence
\begin{align*}
 \sum_x|\mathsf{Ext}_2^{\mathrm{diag}}u(x,x)|
 &\le \frac1{4D}\sum_x\sum_{\boldsymbol z\in\mathsf{Nbr}_x^{\mathrm{pair}}}|u(\boldsymbol z)|
 \le C_D\|u\|_{1,\mathrm{cnt}},\\
 \sum_x|\mathsf{Ext}_2^{\mathrm{diag}}u(x,x)|^2
 &\le \frac1{4D}\sum_x\sum_{\boldsymbol z\in\mathsf{Nbr}_x^{\mathrm{pair}}}|u(\boldsymbol z)|^2
 \le C_D\|u\|_{2,\mathrm{cnt}}^2,
\end{align*}
where the second line uses Jensen's inequality.  Adding the unchanged
off-diagonal contribution proves the two norm bounds in
\eqref{prof:eq:diagonal-extension}.

We next compare the Dirichlet energies.
The gradient terms which require attention are $\mathsf{Ext}_2^{\mathrm{diag}}u(x,x)-u(\boldsymbol w)$ for $\boldsymbol w\in \mathsf{Nbr}_x^{\mathrm{pair}}$.
For an independent pair edge joining $(x,x)$ to $\boldsymbol w$, Jensen's inequality
gives
\begin{equation}
 |\mathsf{Ext}_2^{\mathrm{diag}}u(x,x)-u(\boldsymbol w)|^2
 \le \frac1{4D}
 \sum_{\boldsymbol z\in\mathsf{Nbr}_x^{\mathrm{pair}}}|u(\boldsymbol z)-u(\boldsymbol w)|^2.
 \label{prof:eq:diagonal-jensen}
\end{equation}
For $\boldsymbol z,\boldsymbol w\in\mathsf{Nbr}_x^{\mathrm{pair}}$, each particle moves by at most two lattice steps in the torus $\ell^\infty$ metric $d_N^\infty$ defined in \eqref{prof:eq:torus-linfty-metric}, so
\[
 \max_{i=1,2}d_N^\infty(z_i,w_i)\le2.
\]
For all sufficiently large $N$, \cref{prof:lem:fd-local-routing}, applied
with $r=2$ and $R=2$, therefore provides a deterministic hard-core path
$\Gamma^{\mathrm{pair}}_{x;\boldsymbol z,\boldsymbol w}=(\zeta_0,\ldots,\zeta_L)$ from
$\boldsymbol z$ to $\boldsymbol w$ such that $L\le C_D$ and every vertex of
the path remains within $C_D$ of $x$.  Writing
$a_j=\{\zeta_{j-1},\zeta_j\}$ for its consecutive hard-core edges,
Cauchy--Schwarz along the path gives
\[
 |u(\boldsymbol z)-u(\boldsymbol w)|^2
 \le C_D\sum_{a\in\Gamma^{\mathrm{pair}}_{x;\boldsymbol z,\boldsymbol w}}
      |\nabla_a u|^2.
\]
The same locality gives the required bounded congestion.  Indeed, if a
hard-core edge $a$ occurs in a path indexed by $x$, then $x$ lies within a
fixed $D$-dependent distance of an endpoint of $a$, so there are only
$O_D(1)$ possible choices of $x$.  For each such $x$, the set
$\mathsf{Nbr}_x^{\mathrm{pair}}$ has cardinality $4D$, and hence only
$O_D(1)$ pairs $(\boldsymbol z,\boldsymbol w)$ can index a path through $a$.
Thus every hard-core edge occurs in at most $O_D(1)$ paths after summing
\eqref{prof:eq:diagonal-jensen} over $x$ and $\boldsymbol w$, and the
bounded-congestion sum proves the energy bound in
\eqref{prof:eq:diagonal-extension}.  The routing lemma was stated for all
sufficiently large $N$; enlarging $C_D$ handles the finitely many smaller
tori, so the displayed estimate holds uniformly in $N$ as stated.

It remains to prove the Nash inequality on the product torus.  Index the
counting measure orthonormal characters on $V_N^2$ by
$(p^{(1)},p^{(2)})\in(\T_N^D)^2$; each character has pointwise modulus $n^{-1}$
and is an eigenfunction of $\mathcal L_2^{\ind}$ with eigenvalue
$
 \lambda_N^{(2)}(p^{(1)},p^{(2)}):=\lambda_N(p^{(1)})+\lambda_N(p^{(2)})$.
For $N^{-1}\le R\le1$, split the Fourier expansion of $g$ into the zero
mode $(p^{(1)},p^{(2)})=(0,0)$, the nonzero low modes with
$0<\lambda_N^{(2)}\le R^2$, and the high modes with
$\lambda_N^{(2)}>R^2$.  Parseval gives
\[
\begin{aligned}
 \|g\|_{2,\mathrm{cnt}}^2
 &=|\widehat g(0,0)|^2
 +\sum_{0<\lambda_N^{(2)}\le R^2}|\widehat g|^2
 +\sum_{\lambda_N^{(2)}>R^2}|\widehat g|^2.
\end{aligned}
\]
Every Fourier coefficient satisfies
$|\widehat g|\le n^{-1}\|g\|_{1,\mathrm{cnt}}$.  Moreover,
\eqref{prof:eq:dispersion-comparison} and $\gamma\asymp_DN^{-2}$ imply that, for a
low pair, each component momentum $p^{(1)}$ and $p^{(2)}$ lies in a
$D$-dimensional ball of radius $O_D(NR)$.  Hence the $2D$-dimensional product
count contains at most $C_Dn^2R^{2D}$ nonzero low pairs.  On the high spectrum,
$\lambda_N^{(2)}>R^2$.  Hence the three Parseval contributions satisfy, respectively,
\[
\begin{aligned}
 |\widehat g(0,0)|^2
 &\le \frac1{n^2}\|g\|_{1,\mathrm{cnt}}^2,
 &&\text{(zero mode)},\\
 \sum_{0<\lambda_N^{(2)}\le R^2}|\widehat g|^2
 &\le C_DR^{2D}\|g\|_{1,\mathrm{cnt}}^2,
 &&\text{(nonzero low spectrum)},\\
 \sum_{\lambda_N^{(2)}>R^2}|\widehat g|^2
 &\le R^{-2}\sum_{\lambda_N^{(2)}>R^2}\lambda_N^{(2)}|\widehat g|^2
 \le R^{-2}\mathcal E_{2,\mathrm{cnt}}^{\ind}(g),
 &&\text{(high spectrum)}.
\end{aligned}
\]
Adding these three estimates gives, for $N^{-1}\le R\le1$,
\begin{equation}
 \|g\|_{2,\mathrm{cnt}}^2
 \le \frac1{n^2}\|g\|_{1,\mathrm{cnt}}^2
      +C_DR^{2D}\|g\|_{1,\mathrm{cnt}}^2
      +C_DR^{-2}\mathcal E_{2,\mathrm{cnt}}^{\ind}(g).
 \label{prof:eq:product-nash-three-term}
\end{equation}
The claim is trivial for $g=0$, so assume $g\ne0$.
To balance the two $R$-dependent terms on the right-hand side, set
\[
 R_*=\left(
 \frac{\mathcal E_{2,\mathrm{cnt}}^{\ind}(g)}{\|g\|_{1,\mathrm{cnt}}^2}
 \right)^{1/(2D+2)}.
\]
If $R_*$ lies in $[N^{-1},1]$, choosing $R=R_*$ yields
\begin{equation}
 \|g\|_{2,\mathrm{cnt}}^2
 \le C_D[\mathcal E_{2,\mathrm{cnt}}^{\ind}(g)]^{D/(D+1)}
       \|g\|_{1,\mathrm{cnt}}^{2/(D+1)}
 +\frac{C_D}{n^2}\|g\|_{1,\mathrm{cnt}}^2.
 \label{prof:eq:product-nash}
\end{equation}
If $R_*<N^{-1}$, choose
$R=N^{-1}$; then
$R^{-2}\mathcal E_{2,\mathrm{cnt}}^{\ind}(g)\le n^{-2}\|g\|_{1,\mathrm{cnt}}^2$.
If $R_*>1$, choose $R=1$.  Then \eqref{prof:eq:product-nash-three-term} gives
\[
 \|g\|_{2,\mathrm{cnt}}^2
 \le C_D\bigl(\|g\|_{1,\mathrm{cnt}}^2
                  +\mathcal E_{2,\mathrm{cnt}}^{\ind}(g)\bigr).
\]
Since $R_*>1$ is exactly
$\mathcal E_{2,\mathrm{cnt}}^{\ind}(g)>\|g\|_{1,\mathrm{cnt}}^2$,
the right-hand side is at most
$C_D\mathcal E_{2,\mathrm{cnt}}^{\ind}(g)$.  To absorb this remaining
energy term into the Nash interpolation term, use the bounded product
spectrum: $\lambda_N(p)\le4D$ implies
$\lambda_N^{(2)}\le8D$, and therefore
\[
 \mathcal E_{2,\mathrm{cnt}}^{\ind}(g)
 \le8D\|g\|_{2,\mathrm{cnt}}^2
 \le8D\|g\|_{1,\mathrm{cnt}}^2.
\]
Consequently,
\[
 \mathcal E_{2,\mathrm{cnt}}^{\ind}(g)
 \le (8D)^{1/(D+1)}
 [\mathcal E_{2,\mathrm{cnt}}^{\ind}(g)]^{D/(D+1)}
 \|g\|_{1,\mathrm{cnt}}^{2/(D+1)}.
\]
Combining the last two displays proves \eqref{prof:eq:product-nash} also in
the endpoint regime $R_*>1$.  Thus \eqref{prof:eq:product-nash} holds in every regime.
Applying it to $\mathsf{Ext}_2^{\mathrm{diag}}u$, using
\eqref{prof:eq:diagonal-extension}, and recalling that
$|\Omega_2|=n(n-1)\asymp n^2$ proves \eqref{prof:eq:pair-nash}.
\end{proof}

\begin{remark}
The exponents in \eqref{prof:eq:pair-nash} are the exponents in the Nash inequality
at dimension $2D$.  Indeed, the classical $d$-dimensional form
$\|f\|_2^2\lesssim \mathcal E(f)^{d/(d+2)}\|f\|_1^{4/(d+2)}$
gives the powers $D/(D+1)$ and $2/(D+1)$ when $d=2D$.  The additional
$|\Omega_2|^{-1}\|u\|_{1,\mathrm{cnt}}^2$ term is the finite-volume correction from
the constant Fourier mode.
\end{remark}

Let $p_t^{\hc}(\boldsymbol z,\boldsymbol w)$ be the transition kernel of
$\mathcal P_2(t)=e^{-t\mathcal L_2^{\hc}}$ with respect to counting measure.

\begin{lemma}[Pair heat kernel, low-spectrum dimension, and $L^\infty$ control]
\label{prof:lem:low-geometry}
Uniformly in $N,t$, and $\boldsymbol z\in\Omega_2$,
\begin{equation}
 p_t^{\hc}(\boldsymbol z,\boldsymbol z)
 \le C_D\{(1+t)^{-D}+|\Omega_2|^{-1}\}.
 \label{prof:eq:pair-heat-kernel}
\end{equation}
For fixed $H\ge1$, let the low hard-core pair spectral space be
\[
 \mathscr V_{N,2}^{\hc}(H)
 :=\operatorname{Ran}\Pi_{\mathrm{lo},H}\subset\mathcal T_{N,2}^{\hc}.
\]
Then
\begin{equation}
 \dim\mathscr V_{N,2}^{\hc}(H)\le C_{D,H},
 \quad \text{ and }
 \|u\|_\infty\le C_{D,H}\|u\|_2
 \quad \text{for all } u\in\mathscr V_{N,2}^{\hc}(H).
 \label{prof:eq:low-rank-linfty}
\end{equation}
\end{lemma}

\begin{proof}
For $f_t(\boldsymbol w)=p_t^{\hc}(\boldsymbol z,\boldsymbol w)$, symmetry gives
\[
 w(t):=\|f_t\|_{2,\mathrm{cnt}}^2=p_{2t}^{\hc}(\boldsymbol z,\boldsymbol z),
 \qquad
 w'(t)=-2\mathcal E_{2,\mathrm{cnt}}^{\hc}(f_t) \le 0,
 \qquad
 \|f_t\|_{1,\mathrm{cnt}}=1.
\]
Fix, for the remainder of this argument, a constant $C_D$ for which the
numbered Nash estimate \eqref{prof:eq:pair-nash} holds.  Since
$\|f_t\|_{1,\mathrm{cnt}}=1$, that estimate gives
\[
 w(t)\le C_D\mathcal E_{2,\mathrm{cnt}}^{\hc}(f_t)^{D/(D+1)}+\frac{C_D}{|\Omega_2|}.
\]
Suppose that $w(t)\ge2C_D/|\Omega_2|$.  Then
\[
 \frac12w(t)
 \le w(t)-\frac{C_D}{|\Omega_2|}
 \le C_D\mathcal E_{2,\mathrm{cnt}}^{\hc}(f_t)^{D/(D+1)}.
\]
Raising both sides to the power $(D+1)/D$ therefore yields, for some
$c_D^{\mathrm{Nash}}>0$ depending only on $D$,
$
 \mathcal E_{2,\mathrm{cnt}}^{\hc}(f_t)\ge c_D^{\mathrm{Nash}}w(t)^{1+1/D}
 $.
Plugging this inequality into $w'(t)=-2\mathcal E_{2,\mathrm{cnt}}^{\hc}(f_t)$ yields
\begin{align}
 w'(t)\le-2c_D^{\mathrm{Nash}}w(t)^{1+1/D}.
 \label{prof:eq:pair-kernel-ode}
\end{align}
Set
$
 t_{\mathrm{Nash}}=\inf\left\{t\ge0:
 w(t)\le\frac{2C_D}{|\Omega_2|}\right\}$,
with the convention $t_{\mathrm{Nash}}=\infty$ if this set is empty.
For $0\le t<t_{\mathrm{Nash}}$,
\eqref{prof:eq:pair-kernel-ode} holds, and hence
\[
 \frac{\dd}{\dd t}w(t)^{-1/D}
 =-\frac1D w(t)^{-1-1/D}w'(t)
 \ge\frac{2c_D^{\mathrm{Nash}}}{D}.
\]
Integrating from $0$ to $t$ and using $w(0)=1$ gives
\[
 w(t)^{-1/D}\ge 1+\frac{2c_D^{\mathrm{Nash}}}{D}t
 \quad \Longrightarrow \quad
 w(t)\le C_D(1+t)^{-D}.
\]
If $t_{\mathrm{Nash}}<\infty$, continuity of $w$ gives
$w(t_{\mathrm{Nash}})\le2C_D/|\Omega_2|$, and monotonicity of $w$ then gives
\[
 w(t)\le w(t_{\mathrm{Nash}})\le\frac{2C_D}{|\Omega_2|},
 \qquad t\ge t_{\mathrm{Nash}}.
\]
(The same conclusion holds from time $0$ if $t_{\mathrm{Nash}}=0$.)  Combining the two
regimes yields
\[
 w(t)\le C_D\bigl\{(1+t)^{-D}+|\Omega_2|^{-1}\bigr\},
 \qquad t\ge0.
\]
Finally, since $w(s)=p_{2s}^{\hc}(\boldsymbol z,\boldsymbol z)$,
\[
 p_t^{\hc}(\boldsymbol z,\boldsymbol z)
 =w(t/2)
 \le C_D\bigl\{(1+t/2)^{-D}+|\Omega_2|^{-1}\bigr\}
 \le C_D\bigl\{(1+t)^{-D}+|\Omega_2|^{-1}\bigr\},
\]
after enlarging $C_D$ by a dimension-dependent factor, whence the on-diagonal heat kernel upper bound
\eqref{prof:eq:pair-heat-kernel}.

We now derive separately the trace, $L^2$-to-$L^\infty$, and rank
consequences for $\operatorname{Ran}\Pi_{\mathrm{lo},H}$ in  \eqref{prof:eq:low-rank-linfty}.
Set
$t_*=\gamma^{-1}$.  Since $\gamma\asymp_DN^{-2}$ and $|\Omega_2|\asymp n^2$,
\eqref{prof:eq:pair-heat-kernel} gives
\begin{equation}
 \operatorname{Tr}(e^{-t_*\mathcal L_2^{\hc}})
 =\sum_{\boldsymbol z\in\Omega_2}p_{t_*}^{\hc}(\boldsymbol z,\boldsymbol z)
 \le C_D.
 \label{prof:eq:low-heat-trace}
\end{equation}
Because the kernel relative to the uniform probability measure is
$|\Omega_2|p_t^{\hc}$, the semigroup identity gives
\begin{equation}
 \|e^{-t_*\mathcal L_2^{\hc}}\|_{L^2(\upsilon_2^{\ord})\to L^\infty}^2
 =\sup_{\boldsymbol z\in\Omega_2}|\Omega_2|p_{2t_*}^{\hc}(\boldsymbol z,\boldsymbol z)
 \le C_D.
 \label{prof:eq:low-semigroup-linfty}
\end{equation}
If $u\in\mathscr V_{N,2}^{\hc}(H)$, define
$v=e^{t_*\mathcal L_2^{\hc}}u$ by spectral calculus on that subspace.
Then $u=e^{-t_*\mathcal L_2^{\hc}}v$ and
$\|v\|_2\le e^H\|u\|_2$, so
\eqref{prof:eq:low-semigroup-linfty} yields
$\|u\|_\infty\le C_{D,H}\|u\|_2$.
Finally, for every eigenvalue $\lambda<H\gamma$,
$1\le e^He^{-t_*\lambda}$.  Therefore
\[
 \dim\mathscr V_{N,2}^{\hc}(H)
 =\operatorname{Tr}\Pi_{\mathrm{lo},H}
 \le e^H\operatorname{Tr}(e^{-t_*\mathcal L_2^{\hc}})
 \le C_{D,H}
\]
by \eqref{prof:eq:low-heat-trace}.  This proves
\eqref{prof:eq:low-rank-linfty}.
\end{proof}

\begin{lemma}[Compression of the ordered nearest-neighbor contact source]
\label{prof:lem:low-shell}
For every $H\ge1$ and every one-particle field $f$,
\begin{equation}
 \|\Pi_{\mathrm{lo},H}\mathsf{Top}_{N,2}^{\ord}\mathcal B[f]\|_2
 \le \frac{C_{D,H}}n\mathcal E_1(f).
 \label{prof:eq:low-shell}
\end{equation}
\end{lemma}

\begin{proof}
Let $u\in\mathscr V_{N,2}^{\hc}(H)$ satisfy $\|u\|_2=1$.
Since $\mathcal B[f]$ is supported on
$\mathfrak C_{N,2}^{\mathrm{nn}}$ from
\eqref{prof:eq:nearest-neighbor-contact-set},
\eqref{prof:eq:low-rank-linfty} and the ordered edge normalization give
\[
 |\langle u,\mathcal B[f]\rangle|
 \le\frac{\|u\|_\infty}{n(n-1)}
      \sum_{x\sim y}|f(x)-f(y)|^2
 \le\frac{C_{D,H}}n\mathcal E_1(f).
\]
By orthogonal projection duality,
\begin{align*}
 \|\Pi_{\mathrm{lo},H}\mathsf{Top}_{N,2}^{\ord}\mathcal B[f]\|_2
 =\sup_{\substack{u\in\mathscr V_{N,2}^{\hc}(H)\\\|u\|_2=1}}
   |\langle u,\mathcal B[f]\rangle|
\le\frac{C_{D,H}}n\mathcal E_1(f),
\end{align*}
which is \eqref{prof:eq:low-shell}.
\end{proof}

\begin{proposition}[Low-spectrum cutoff estimate]
\label{prof:prop:low-spectrum}
For fixed $H\ge1$ and $M\in\mathbb R$, uniformly for
\begin{equation}
 0\le t\le\frac{\log n+M}{2\gamma},
 \label{prof:eq:upper-window}
\end{equation}
all deterministic $S\subset V_N$ with $|S|=k$, and all sufficiently large $N$,
\begin{equation}
 \mathcal E_2(G_{2,\mathrm{lo}}(t))
 \le C_{D,\rho_0,H,M}\gamma n e^{-4\gamma t}.
 \label{prof:eq:low-spectrum}
\end{equation}
\end{proposition}

\begin{proof}
Recall the lower bound for the fixed-degree-$r$ spectral bottom \eqref{prof:eq:fd-top-gap}.
Set $r=2$, and choose $C_D^{\mathrm{edge}}>0$ such that, on
$\mathcal T_{N,2}^{\hc}=\Ran\mathsf{Top}_{N,2}^{\ord}$,
\begin{align}
 \mathcal L_{2,\mathrm{top}}^{\hc}
 \ge 2\gamma-C_D^{\mathrm{edge}}\frac{\gamma}{n}.
 \label{prof:eq:specbot2}
\end{align}
Set
\[
 \varepsilon=C_D^{\mathrm{edge}}\frac{\gamma}{n},
 \qquad
 \underline\lambda_2=2\gamma-\varepsilon.
\]
Then $\underline\lambda_2>0$ for all sufficiently large $N$.  Moreover,
on the time interval \eqref{prof:eq:upper-window},
\begin{align}
 e^{2\varepsilon t}
 \le
 \exp\!\left\{C_{D,M}\frac{\log n}{n}\right\}
 =O_{D,M}(1),
 \label{prof:eq:low-edge-correction}
\end{align}
and in fact this factor tends to one as $N\to\infty$.

We first show the spectral calculus reduction.  Applying $\Pi_{\mathrm{lo},H}$ to
\eqref{prof:eq:pair-duhamel} and using
\eqref{prof:eq:pair-top-commute} gives
\[
 G_{2,\mathrm{lo}}(t)
 =-\int_0^t
 \mathcal P_2(t-\tau)
 \Pi_{\mathrm{lo},H}\mathsf{Top}_{N,2}^{\ord}\mathcal B[\varphi_\tau] \,\dd\tau.
\]
The lower spectral bound \eqref{prof:eq:specbot2} implies
\[
 \|e^{-u\mathcal L_2^{\hc}}\Pi_{\mathrm{lo},H}\|_{2\to2}
 \le e^{-\underline\lambda_2u},
 \qquad u\ge0,
\]
while \cref{prof:lem:low-shell}, through the contact shell estimate
\eqref{prof:eq:low-shell}, gives
\[
 \|\Pi_{\mathrm{lo},H}\mathsf{Top}_{N,2}^{\ord}\mathcal B[\varphi_\tau]\|_2
 \le\frac{C_{D,H}}n\mathcal E_1(\varphi_\tau).
\]
Minkowski's integral inequality therefore yields
\begin{equation}
 \|G_{2,\mathrm{lo}}(t)\|_2
 \le\frac{C_{D,H}}n\int_0^t
 e^{-\underline\lambda_2(t-\tau)}\mathcal E_1(\varphi_\tau)\,\dd\tau.
 \label{prof:eq:low-duhamel}
\end{equation}

Expand
$
 \varphi_0=\sum_{p\ne0}b_pe_p
$
in the basis \eqref{prof:eq:probability-fourier-basis}.  Then
$\mathcal L_1e_p=\lambda_N(p)e_p$ and, by the exact normalization of
$\beta_N$,
\[
 \sum_{p\ne0}|b_p|^2
 =\|\varphi_0\|_{2,\mathrm{av}}^2
 =n-1.
\]
For $\lambda\ge\gamma$, define
\[
 J_\lambda(t)
 :=\lambda\int_0^t\,
 e^{-\underline\lambda_2(t-\tau)}e^{-2\lambda\tau}\,\dd\tau.
\]
Since $\lambda\ge\gamma$,
$
 2\lambda-\underline\lambda_2
 =2(\lambda-\gamma)+\varepsilon
 \ge\varepsilon>0$.
Hence direct integration gives
\begin{equation*}
 J_\lambda(t)
 =\lambda e^{-\underline\lambda_2t}
 \frac{1-e^{-(2\lambda-\underline\lambda_2)t}}
 {2\lambda-\underline\lambda_2}.
\end{equation*}
We now treat the two spectral regimes.  If $\gamma\le\lambda\le2\gamma$, then
$2\lambda-\underline\lambda_2\ge0$, and hence
$
 J_\lambda(t)
 \le\lambda t e^{-\underline\lambda_2t}
 \le2\gamma t e^{-\underline\lambda_2t}
 $.
If $\lambda>2\gamma$, then
$2\lambda-\underline\lambda_2\ge\lambda$, so
$
 J_\lambda(t)\le e^{-\underline\lambda_2t}
 $.
Thus, uniformly for every $\lambda\ge\gamma$,
\begin{equation*}
 J_\lambda(t)
 \le C(1+\gamma t)e^{-\underline\lambda_2t}.
\end{equation*}
Since
\[
 \mathcal E_1(\varphi_\tau)
 =\sum_{p\ne0}\lambda_N(p)|b_p|^2e^{-2\tau\lambda_N(p)},
\]
substitution into \eqref{prof:eq:low-duhamel} gives
\begin{align*}
 \|G_{2,\mathrm{lo}}(t)\|_2
 \le\frac{C_{D,H}}n
 \sum_{p\ne0}|b_p|^2J_{\lambda_N(p)}(t)
 \le C_{D,H}\frac{n-1}{n}
 (1+\gamma t)e^{-\underline\lambda_2t}
 \le C_{D,H}(1+\gamma t)e^{-\underline\lambda_2t}.
\end{align*}
Because the spectrum of $G_{2,\mathrm{lo}}(t)$ lies below $H\gamma$,
\begin{align}
 \mathcal E_2(G_{2,\mathrm{lo}}(t))
 \le H\gamma\|G_{2,\mathrm{lo}}(t)\|_2^2
 \le C_{D,H}\gamma(1+\gamma t)^2e^{-2\underline\lambda_2t}
 =C_{D,H}\gamma(1+\gamma t)^2
   e^{-4\gamma t}e^{2\varepsilon t}.
   \label{prof:eq:E2-G2-bound}
\end{align}
On the interval \eqref{prof:eq:upper-window}, the correction factor
$e^{2\varepsilon t}$ is uniformly bounded thanks to
\eqref{prof:eq:low-edge-correction}.  Moreover,
$
 (1+\gamma t)^2\le C_M(1+\log n)^2\le C_Mn
$
for all sufficiently large $n$.  
Using these two bounds on \eqref{prof:eq:E2-G2-bound} proves \eqref{prof:eq:low-spectrum}.
\end{proof}

\begin{theorem}[Cutoff-window pair energy]
\label{prof:thm:pair-energy}
Fix $D\ge2$, $\rho_0\in(0,1/2]$, $A\in(0,1/2)$, and $M\in\mathbb R$.  Uniformly
over all deterministic $S\subset V_N$ with $|S|=k$
(hence $\rho_0n\le|S|\le(1-\rho_0)n$ by the standing density window),
all sufficiently large $N$, and all
\begin{equation}
 A\gamma^{-1}\log n
 \le t\le\frac{\log n+M}{2\gamma},
 \label{prof:eq:pair-window}
\end{equation}
one has
\begin{equation}
 \mathcal E_2(G_2(t))
 \le C_{D,\rho_0,A,M}\gamma n e^{-4\gamma t}.
 \label{prof:eq:pair-energy-window}
\end{equation}
\end{theorem}

\begin{proof}
Choose $H=H(D,A)$ as in \cref{prof:prop:high-spectrum}.  The projections
in \eqref{prof:eq:pair-spectral-projections} commute with $\mathcal L_2^{\hc}$ and
are orthogonal for the Dirichlet form.  Hence
$
 \mathcal E_2(G_2(t))
 =\mathcal E_2(G_{2,\mathrm{lo}}(t))+\mathcal E_2(G_{2,\mathrm{hi}}(t))
 $.
The first term is bounded by \cref{prof:prop:low-spectrum}, and the second by
\cref{prof:prop:high-spectrum}.  Together they prove \eqref{prof:eq:pair-energy-window}.
\end{proof}

\section{Chaos generating functions and exponential weight propagation}\label{prof:sec:chaos-weight}

This section uses the complete graph exclusion process for one genuinely
all-degree purpose: to obtain a degree-linear spectral gap and propagate an
exponential chaos weight.  We continue to use the warm-start abbreviations
$n=n_N$, $k=k_N$, $\rho=\rho_N$, and $\gamma=\gamma_N$.

\localheading{Complete-graph spectrum.}
The site permutation action $\mathsf T_{\sigma}^{(q)}$ was fixed in
\cref{prof:sec:chaos-interface}.  The normalized complete graph exchange
generator at subset level $q$ is
\[
 \mathcal L_{N,q}^{\mathrm{cg}}
 :=\frac1{n_N}\sum_{\{x,y\}\subset V_N}
   (I-\mathsf T_{\tau_{xy}}^{(q)}).
\]
For $1\le q\le n_N$ and $A\in\Omega_{N,q}^{\mathrm{set}}$, direct
substitution into the normalized up--down formulas gives
\[
 (\mathsf{Up}_{N,q-1\to q}\mathsf{Down}_{N,q\to q-1}f)(A)
 =\frac{qf(A)+\displaystyle\sum_{x\in A}\sum_{y\notin A}
          f((A\setminus\{x\})\cup\{y\})}
        {q(n_N-q+1)}.
\]
Only transpositions with exactly one endpoint in $A$ change the set, so
rearranging yields
\begin{equation}
 \mathcal L_{N,q}^{\mathrm{cg}}
 =\frac{q(n_N-q+1)}{n_N}
  \left(I-\mathsf{Up}_{N,q-1\to q}
             \mathsf{Down}_{N,q\to q-1}\right).
 \label{prof:eq:complete-graph-up-down}
\end{equation}
Let $0\le j\le q\le\ell_N$.  If $j<q$, every vector in the degree-$j$
summand at level $q$ is $\mathsf{Up}_{N,q-1\to q}h$ for a degree-$j$
vector $h$ at level $q-1$.  Hence \cref{prof:lem:fd-down-up-spectrum}
shows that $\mathsf{Up}_{N,q-1\to q}\mathsf{Down}_{N,q\to q-1}$ acts there
by $a_{q,j,N}$; for $j=q$ it vanishes because the top degree is harmonic.
Substitution into \eqref{prof:eq:complete-graph-up-down} gives
\begin{equation}
 \mathcal L_{N,q}^{\mathrm{cg}}
 \big|_{\mathsf{Up}_{N,j\to q}\mathcal K_{N,j}^{\mathrm{set}}}
 =\frac{j(n_N-j+1)}{n_N}\,I,
 \qquad 0\le j\le q\le\ell_N.
 \label{prof:eq:complete-graph-degree-spectrum}
\end{equation}
Thus the complete graph eigenvalue depends on the Hoeffding degree $j$ but
not on the ambient subset level $q$.

The harmonic lift is site-permutation-equivariant by
\eqref{prof:eq:harmonic-lift-permutation-equivariance}.  Summing that
identity over all site transpositions transports
\eqref{prof:eq:complete-graph-degree-spectrum} to the physical slice:
\begin{equation}
 \mathcal L_{N,k_N}^{\mathrm{cg}}\big|_{\Hdeg_{N,r}}
 =\frac{r(n_N-r+1)}{n_N}\,I,
 \qquad 0\le r\le\ell_N.
 \label{prof:eq:complete-graph-exclusion-eigenvalue}
\end{equation}
This is the complete graph input used below to obtain a degree-linear lower
bound for nearest-neighbor exclusion.

\begin{proposition}[Complete-graph to torus form comparison]
\label{prof:prop:complete-graph-comparison}
For every function $f$ on the standing $k$-particle slice,
\begin{equation}
 \langle f,\mathcal L_{N,k}^{\mathrm{cg}}f\rangle
 \le C_DN^2\langle f,L_N^{\SEP}f\rangle.
 \label{prof:eq:complete-graph-form-comparison}
\end{equation}
\end{proposition}

\begin{proof}
\localheading{Path comparison.}
Let $f$ be a function on the $k$-particle slice.  For each unordered endpoint pair $\{x,y\}$, choose once and for all an
orientation $(x,y)$.  In each coordinate choose the signed shortest
displacement in the half-open representative interval
$\{-\lfloor(N-1)/2\rfloor,\ldots,\lfloor N/2\rfloor\}$; when $N$ is even,
this convention resolves the antipodal tie by choosing $+N/2$.  Moving the
coordinates in the order $1,\ldots,D$ defines a deterministic
coordinate-ordered shortest torus path
\[
 \Gamma_{x,y}^{\mathrm{site}}=(z_0,\ldots,z_L),
 \qquad z_0=x,\quad z_L=y,\quad L\le DN/2.
\]
Write $\epsilon_i=\{z_{i-1},z_i\}$ and form the adjacent transposition word
\[
 \mathsf W_{x,y}=(a_1,\ldots,a_{s_{x,y}})
 :=(\epsilon_1,\ldots,\epsilon_L,\epsilon_{L-1},\ldots,\epsilon_1),
 \qquad s_{x,y}=2L-1\le DN.
\]
Then $\tau_{a_{s_{x,y}}}\cdots\tau_{a_1}=\tau_{xy}$.  For a slice configuration
$A$, define
\[
 A_0=A,
 \qquad
 A_j=\tau_{a_j}A_{j-1},
 \quad 1\le j\le s_{x,y}.
\]
Thus $A_{s_{x,y}}=\tau_{xy}A$; after deleting any consecutive repetitions
from $(A_0,\ldots,A_{s_{x,y}})$, the resulting sequence is a nearest-neighbor
exclusion path.  Cauchy--Schwarz along the unreduced sequence gives
\[
 |f(A)-f(\tau_{xy}A)|^2
 =\left|\sum_{j=1}^{s_{x,y}}\bigl(f(A_{j-1})-f(A_j)\bigr)\right|^2
 \le s_{x,y}\sum_{j=1}^{s_{x,y}}|f(A_{j-1})-f(A_j)|^2.
\]
Because the uniform $k$-subset measure is invariant under every site
transposition, averaging over $A$ yields
\[
 \mathbb E_{\upsilon_{N,k}^{\mathrm{set}}}\bigl|f-\mathsf T_{\tau_{xy}}^{(k)}f\bigr|^2
 \le s_{x,y}\sum_{j=1}^{s_{x,y}}
 \mathbb E_{\upsilon_{N,k}^{\mathrm{set}}}\bigl|f-\mathsf T_{\tau_{a_j}}^{(k)}f\bigr|^2.
\]

It remains to control congestion.  For a nearest-neighbor edge
$e\in E_N$, let $N_e(\mathsf W_{x,y}):=\#\{j:a_j=e\}$.
The forward-and-return construction gives $N_e(\mathsf W_{x,y})\le2$.
Fix $e$ in coordinate direction $a$.  If the coordinate-ordered path
$\Gamma_{x,y}^{\mathrm{site}}$ uses $e$ during its $a$th coordinate segment, then the first
$a-1$ target coordinates and the last $D-a$ source coordinates are fixed by
$e$.  The remaining transverse coordinates are the $a-1$ unused source
coordinates and the $D-a$ unused target coordinates, hence $D-1$ coordinates
in total and $N^{D-1}$ choices.  For the longitudinal coordinate, the ordered pair $(x_a,y_a)$ has at most
$N^2$ possible values in total, so in particular at most $N^2$ choices can
have their selected shortest arc contain the fixed edge $e$; the half-open
shortest-displacement convention above removes the only antipodal ambiguity.
Thus
\[
 \#\bigl\{\{x,y\}:N_e(\mathsf W_{x,y})>0\bigr\}
 \le C_DN^{D-1}N^2=C_DnN,
\]
and using $s_{x,y}\le DN$ we find
\[
 \sup_{e\in E_N}
 \sum_{\{x,y\}\subset V_N}s_{x,y}\,N_e(\mathsf W_{x,y})
 \le C_DnN^2.
\]
Consequently,
\begin{align*}
 \langle f,\mathcal L_{N,k}^{\mathrm{cg}}f\rangle
 &=\frac1{2n}\sum_{\{x,y\}\subset V_N}
   \mathbb E_{\upsilon_{N,k}^{\mathrm{set}}}\bigl|f-\mathsf T_{\tau_{xy}}^{(k)}f\bigr|^2\\
 &\le\frac1{2n}\sum_{e\in E_N}
   \left(\sum_{\{x,y\}\subset V_N}s_{x,y}\,N_e(\mathsf W_{x,y})\right)
   \mathbb E_{\upsilon_{N,k}^{\mathrm{set}}}\bigl|f-\mathsf T_{\tau_e}^{(k)}f\bigr|^2
 \le C_DN^2\langle f,L_N^{\SEP}f\rangle.
\qedhere
\end{align*} 
\end{proof}

Decompose $h_t^S$ orthogonally into its chaos degrees, and use
the isometric chaos coordinates $\widehat\Psi_{N,r}$ from
\cref{prof:def:chaos-coordinate}.  The degree-$r$ gap proved below damps the
$r$th chaos component by a factor exponential in $r(t-\tau)$.  Weighting the
$r$th squared norm by $z^r$ packages all degrees so that this damping can be
absorbed by increasing the weight $z$.  This is the mechanism behind the
exponential weight propagation estimate.  For $z\ge0$ and $t\ge0$, define the
weighted chaos generating function
\begin{equation*}
 Z_S(z,t):=\sum_{r=0}^{\ell_N}z^r
 \norm{\widehat\Psi_{N,r}(h_t^S)}_2^2,
\end{equation*}
and observe that $
 Z_S(1,t)=\norm{h_t^S}_{L^2(\pi_N)}^2$.
Because $Z_S(\cdot,t)$ is a finite polynomial for each $N$, the word ``weight'' below refers to the exponential factor placed on the chaos degree, not to a radius of convergence.
\begin{lemma}[Linear degree gap and exponential weight propagation]
\label{prof:lem:weight-propagation}
There is $c_D>0$ such that the SEP generator on degree $r$ satisfies
\begin{equation}
L_N^{\SEP}\big|_{\Hdeg_{N,r}}\ge c_Dr\gamma I,
\qquad 1\le r\le \ell_N.
 \label{prof:eq:degree-gap}
\end{equation}
Consequently, for $z\ge0$ and $0\le\tau\le t$,
\begin{equation}
Z_S\bigl(ze^{2c_D\gamma(t-\tau)},t\bigr)\le Z_S(z,\tau).
 \label{prof:eq:weight-propagation}
\end{equation}
\end{lemma}

\begin{proof}
By the exact complete graph eigenvalue
\eqref{prof:eq:complete-graph-exclusion-eigenvalue}, the form comparison from
\cref{prof:prop:complete-graph-comparison}, namely
\eqref{prof:eq:complete-graph-form-comparison}, and
$\frac{r(n-r+1)}n\ge\frac r2$ for $r\le\ell_N\le n/2$, every
$F\in\Hdeg_{N,r}$ satisfies
\[
 \frac r2\norm{F}_2^2
 \le \langle F,\mathcal L_{N,k}^{\mathrm{cg}}F\rangle
 \le C_DN^2\langle F,L_N^{\SEP}F\rangle.
\]
Hence
$
 \langle F,L_N^{\SEP}F\rangle
 \ge c_DrN^{-2}\norm{F}_2^2
 \ge c_Dr\gamma\norm{F}_2^2,
$
after adjusting $c_D$ using $\gamma\asymp_DN^{-2}$. This yields
\eqref{prof:eq:degree-gap}.

\localheading{Propagation of the exponential chaos weight.}
The exclusion semigroup preserves each chaos degree.  For $r\ge1$, spectral
calculus and \eqref{prof:eq:degree-gap} give
\[
 \norm{\widehat\Psi_{N,r}(h_t^S)}_2^2
 \le e^{-2c_Dr\gamma(t-\tau)}
      \norm{\widehat\Psi_{N,r}(h_\tau^S)}_2^2,
\]
while the degree-zero component is unchanged.
Multiplying by the terminal weight
$\bigl(ze^{2c_D\gamma(t-\tau)}\bigr)^r$ cancels the degree-dependent decay:
\[
 \bigl(ze^{2c_D\gamma(t-\tau)}\bigr)^r
 \norm{\widehat\Psi_{N,r}(h_t^S)}_2^2
 \le z^r\norm{\widehat\Psi_{N,r}(h_\tau^S)}_2^2.
\]
Summing over $0\le r\le\ell_N$ proves
\eqref{prof:eq:weight-propagation}.
\qedhere
\end{proof}

\begin{remark}
Both \cref{prof:prop:fd-top-gap,prof:lem:weight-propagation} give lower bounds
on the spectral bottom of exclusion in the degree-$r$ chaos sector.
\Cref{prof:prop:fd-top-gap} is asymptotically sharp for each fixed $r$, whereas
\cref{prof:lem:weight-propagation} is uniform for all
$1\le r\le\ell_N$.  The complete graph comparison used in
\eqref{prof:eq:degree-gap} yields a dimension-dependent constant $c_D$ that
is not sharp; its sole role here is the uniform linear growth in $r$ needed
for \eqref{prof:eq:weight-propagation}.
\end{remark}

\section{From pair energy to a terminal row bound}\label{prof:sec:row}

The local notation $n=n_N$, $k=k_N$, $\rho=\rho_N$, and $\gamma=\gamma_N$
remains in force in this section.  The pair-energy argument gives global
$L^2$ control of the degree-two discrepancy $G_2$.  Here we strengthen this
to uniform control in one coordinate.  The key input is a short-time
smoothing estimate that converts the available pair $L^2$ bound into the row
bound defined below.

We also use the sharper fixed-degree spectral bottom at degree two.  Setting
$r=2$ in \eqref{prof:eq:fd-top-gap}, for all sufficiently large $N$,
\[
 \mathcal L_{N,2}^{\hc}\big|_{\mathcal T_{N,2}^{\hc}}
 \ge \left(2-C_{D,2}n^{-1}\right)\gamma I
 \ge \gamma I.
\]
Thus every $f\in\mathcal T_{N,2}^{\hc}$ satisfies the pair Poincar\'e
estimate
\begin{equation}
 \norm{f}_2^2\le\gamma^{-1}\mathcal E_2(f),
 \label{prof:eq:pair-Poincare}
\end{equation}
for all sufficiently large $N$.  This is the only pair Poincar\'e estimate
used below.

For $f\in\mathsf X_{N,2}^{\ord,\mathrm{sym}}$, set
\[
\norm{f}_{\mathrm{row}}^2
:=\max_x\frac1{n-1}\sum_{y\ne x}|f(x,y)|^2.
\]
Throughout this section, $\|f\|_2$ for a pair function means
$\|f\|_{L^2(\upsilon_2^{\ord})}$.
\begin{lemma}[Row smoothing]
\label{prof:lem:row-smoothing}
There is $a_D>0$ such that, for $t_{\mathrm{row}}=a_D\gamma^{-1}$ and every
$f\in\mathsf X_{N,2}^{\ord,\mathrm{sym}}$,
\begin{align}
\norm{\mathcal P_2(t)f}_{\mathrm{row}}&\le\norm{f}_{\mathrm{row}},
\qquad t\ge0,
 \label{prof:eq:row-contract}\\
\norm{\mathcal P_2(t_{\mathrm{row}})f}_{\mathrm{row}}&\le C_D\norm{f}_2.
 \label{prof:eq:row-smooth}
\end{align}
We also have
\begin{align}
\norm{\mathsf{Top}_{N,2}^{\ord}f}_{\mathrm{row}}\le C\norm{f}_{\mathrm{row}}.
\label{prof:eq:top-row-bound}
\end{align}
\end{lemma}

\begin{proof}
For a symmetric pair function $f$, set
\[
 R_f:=\mathsf{Down}_{2\to1}(|f|^2),
 \qquad
 R_f(x)=\frac1{n-1}\sum_{y\ne x}|f(x,y)|^2.
\]
Because $\mathcal P_2(t)$ is a Markov semigroup, Jensen's inequality gives the
pointwise bound
$
 |\mathcal P_2(t)f|^2\le \mathcal P_2(t)|f|^2
 $.
Applying $\mathsf{Down}_{2\to1}$ and using the semigroup intertwining
\eqref{prof:eq:hard-core-down-semigroup-intertwining} yields
\begin{equation}
 R_{\mathcal P_2(t)f}
 \le e^{-t\mathcal L_1}R_f.
 \label{prof:eq:row-down-semigroup-bound}
\end{equation}
If $p_t(x,z)$ denotes the one-particle heat kernel, then
\eqref{prof:eq:row-down-semigroup-bound} gives the row kernel bound
\begin{equation}
 \frac1{n-1}\sum_{y\ne x}|\mathcal P_2(t)f(x,y)|^2
 \le\sum_zp_t(x,z)R_f(z).
 \label{prof:eq:row-kernel-bound}
\end{equation}
Taking the maximum over $x$ proves \eqref{prof:eq:row-contract}.

Choose \(a_D\ge2\), and put \(t_{\mathrm{row}}=a_D\gamma^{-1}\).  In the
unit-modulus character convention,
\[
 p_{t_{\mathrm{row}}}(x,z)
 =\frac1n\sum_{p}e^{-t_{\mathrm{row}}\lambda_N(p)}
   e_p(x)\overline{e_p(z)}.
\]
Therefore
\[
 \sup_{x,z}p_{t_{\mathrm{row}}}(x,z)
 \le\frac1n\left(1+\sum_{p\ne0}e^{-t_{\mathrm{row}}\lambda_N(p)}\right)
 =\frac1n\left(1+\sum_{p\ne0}
   e^{-2(t_{\mathrm{row}}/2)\lambda_N(p)}\right).
\]
Since $t_{\mathrm{row}}/2\ge\gamma^{-1}$, the long-time heat trace bound
\eqref{prof:eq:long-heat-trace} from
\cref{prof:lem:one-particle-spectral-sums} gives
$
 \sup_{x,z}p_{t_{\mathrm{row}}}(x,z)\le\frac{C_D}{n}
 $.
Substitution into \eqref{prof:eq:row-kernel-bound} yields \eqref{prof:eq:row-smooth}:
\[
 \frac1{n-1}\sum_{y\ne x}|\mathcal P_2(t_{\mathrm{row}})f(x,y)|^2
 \le\frac{C_D}{n}\sum_zR_f(z)=C_D\|f\|_2^2.
\]

Next we prove \eqref{prof:eq:top-row-bound}.  Specializing the general up--down projection
\eqref{prof:eq:fd-top-projection-formula} to $r=2$ gives
\begin{equation}
\begin{aligned}
 \mathsf{Top}_{N,2}^{\ord}
 &=I-\mathsf{Up}_{N,1\to2}
 \bigl(\mathsf{Down}_{N,2\to1}\mathsf{Up}_{N,1\to2}\bigr)^{-1}
 \mathsf{Down}_{N,2\to1}.
\end{aligned}
 \label{prof:eq:pair-top-projection-formula}
\end{equation}
By \eqref{prof:eq:fd-down-up-eigenvalue},
$\mathsf{Down}_{N,2\to1}\mathsf{Up}_{N,1\to2}$ acts by $1$ on constants and
by $(n-2)/(2(n-1))$ on mean-zero one-particle fields.
Hence for every one-particle field $g$,
\[
 \bigl(\mathsf{Down}_{N,2\to1}\mathsf{Up}_{N,1\to2}\bigr)^{-1}g
 =\frac{2(n-1)}{n-2}g
 -\frac n{n-2}\left(\frac1n\sum_z g(z)\right).
\]
By our standing assumption $D\ge 2$ and $N\ge 3$, $n=N^D \ge 3^D >3$, and therefore
\[
 \left\|\bigl(\mathsf{Down}_{N,2\to1}\mathsf{Up}_{N,1\to2}\bigr)^{-1}g\right\|_\infty
 \le C\|g\|_\infty.
\]
The ordered formula for $\mathsf{Down}_{N,2\to1}$ and Cauchy--Schwarz give
\[
 \|\mathsf{Down}_{N,2\to1}f\|_\infty
 =\max_x\left|\frac1{n-1}\sum_{y\ne x}f(x,y)\right|
 \le\|f\|_{\mathrm{row}},
\]
while $(\mathsf{Up}_{N,1\to2}g)(x,y)=\frac12(g(x)+g(y))$ implies
$\|\mathsf{Up}_{N,1\to2}g\|_{\mathrm{row}}\le\|g\|_\infty$.  Applying these bounds in
\eqref{prof:eq:pair-top-projection-formula} yields \eqref{prof:eq:top-row-bound}:
\[
 \|\mathsf{Top}_{N,2}^{\ord}f\|_{\mathrm{row}}
 \le \|f\|_{\mathrm{row}}
 +\left\|\mathsf{Up}_{N,1\to2}
 \bigl(\mathsf{Down}_{N,2\to1}\mathsf{Up}_{N,1\to2}\bigr)^{-1}
 \mathsf{Down}_{N,2\to1}f\right\|_{\mathrm{row}}
 \le C\|f\|_{\mathrm{row}}. \qedhere
\]
\end{proof}

The preceding row-smoothing estimates now convert the pair-energy control into a
terminal row-norm estimate for the degree-two top pair discrepancy $G_2$ defined
in \eqref{prof:eq:G2}.

\begin{proposition}[Terminal pair row bound]
For every fixed profile coordinate $s_1<-1$,
\begin{equation}
\norm{G_2(t_N(s_1))}_{\mathrm{row}}
\le C_{D,\rho_0}\frac{e^{-s_1}}{\sqrt n}
 \label{prof:eq:terminal-row}
\end{equation}
uniformly over deterministic $S\subset V_N$ with $|S|=k$ and all sufficiently large $N$.
\end{proposition}

The factor $e^{-s_1}$ worsens as $s_1$ decreases because $t_N(s_1)$ is then an
earlier physical time.  In the later warm-start argument, $s_1$ is fixed
independently of $N$, so this factor remains a fixed constant.

\begin{proof}
Let $t_-=t_N(s_1)-t_{\mathrm{row}}$, where $t_{\mathrm{row}} = a_D \gamma^{-1}$ as in \cref{prof:lem:row-smoothing}.
Since
$2\gamma t_-=\log n+s_1-2a_D$, for fixed $s_1$ and all sufficiently large
$N$ one has $t_-\ge\frac14\gamma^{-1}\log n$.  Moreover, $2\gamma t_-=\log n+s_1-2a_D\le\log n+1$, so $t_-\le(\log n+1)/(2\gamma)$.  Hence $t_-$ lies in the window
\eqref{prof:eq:pair-window} with $A=\frac 14$ and $M=1$.  By
\cref{prof:thm:pair-energy} and
\eqref{prof:eq:pair-Poincare},
\begin{align}
\norm{G_2(t_-)}_2^2
\le C_D\gamma^{-1}\mathcal E_2(G_2(t_-))
\le C_{D,\rho_0}ne^{-4\gamma t_-}
\le C_{D,\rho_0}\frac{e^{-2s_1}}n.
 \label{prof:eq:L2-bound-G2t-}
\end{align}
Set $T:=t_N(s_1)$.  Using the commutation
\eqref{prof:eq:HLcommute}, the Duhamel formula
\eqref{prof:eq:pair-duhamel} may be written with the top projection inside the
integral.  Splitting at $t_-$ then gives
\begin{align*}
G_2(T)
&=-\int_0^{t_-}\mathcal P_2(T-\tau)
   \mathsf{Top}_{N,2}^{\ord}\mathcal B[\varphi_\tau] \,\dd\tau
  -\int_{t_-}^{T}\mathcal P_2(T-\tau)
   \mathsf{Top}_{N,2}^{\ord}\mathcal B[\varphi_\tau] \,\dd\tau\\
&=-\mathcal P_2(T-t_-)
  \int_0^{t_-}\mathcal P_2(t_--\tau)
   \mathsf{Top}_{N,2}^{\ord}\mathcal B[\varphi_\tau] \,\dd\tau
  -\int_{t_-}^{T}\mathcal P_2(T-\tau)
   \mathsf{Top}_{N,2}^{\ord}\mathcal B[\varphi_\tau] \,\dd\tau\\
&=\mathcal P_2(T-t_-)G_2(t_-)
  -\int_{t_-}^{T}\mathcal P_2(T-\tau)
   \mathsf{Top}_{N,2}^{\ord}\mathcal B[\varphi_\tau] \,\dd\tau.
\end{align*}
Above, the second equality uses
$T-\tau=(T-t_-)+(t_--\tau)$ for $0\le\tau\le t_-$ and the semigroup
property $\mathcal P_2(t+t')= \mathcal P_2(t) \mathcal P_2(t')$, while the third equality is exactly
\eqref{prof:eq:pair-duhamel} evaluated at $t_-$.  Since
$T-t_-=t_{\mathrm{row}}$, we obtain
\begin{align}
G_2(t_N(s_1))
&=\mathcal P_2(t_{\mathrm{row}})G_2(t_-)
-\int_{t_-}^{t_N(s_1)}\mathcal P_2(t_N(s_1)-\tau)
 \mathsf{Top}_{N,2}^{\ord}\mathcal B[\varphi_\tau]\,\dd\tau.
 \label{prof:eq:G2tNs1}
\end{align}
The smoothing estimate \eqref{prof:eq:row-smooth}, together with the
$L^2$ bound \eqref{prof:eq:L2-bound-G2t-}, shows that the first term on the right-hand side of \eqref{prof:eq:G2tNs1} has row norm at most
$C_{D,\rho_0}e^{-s_1}n^{-1/2}$.
As for the integral term in \eqref{prof:eq:G2tNs1}, the contact source has at most $2D$ nonzero entries in
each row, and therefore
\[
 \|\mathcal B[\varphi_\tau]\|_{\mathrm{row}}^2
 \le \frac{2D}{n-1}\,\omega_1(\varphi_\tau)^4,
\]
with $\omega_1$ defined at \eqref{prof:eq:omega1}.
Using the row boundedness of $\mathsf{Top}_{N,2}^{\ord}$ \eqref{prof:eq:top-row-bound}, we obtain
\[
\norm{\mathsf{Top}_{N,2}^{\ord}\mathcal B[\varphi_\tau]}_{\mathrm{row}}
\le C_Dn^{-1/2}\omega_1(\varphi_\tau)^2.
\]
Moreover, for all sufficiently large $N$ one has $t_-\ge\gamma^{-1}$.
Since $\varphi_\tau=\sqrt{n/\beta_N}\,m_\tau$, the long-time one-particle
estimate \eqref{prof:eq:long-one-particle} gives, for
$\tau\in[t_-,t_N(s_1)]$,
\begin{equation}
 \omega_1(\varphi_\tau)^2
 =\frac n{\beta_N}\omega_1(m_\tau)^2
 \le C_{\rho_0}n\gamma e^{-2\gamma\tau}
 \le C_{\rho_0}n\gamma e^{-2\gamma t_-}
 =C_{\rho_0}e^{2a_D}\gamma e^{-s_1}
 \le C_{D,\rho_0}\gamma e^{-s_1}.
 \label{prof:eq:terminal-window-edge-oscillation}
\end{equation}
Combining \eqref{prof:eq:terminal-window-edge-oscillation} with the row contraction
\eqref{prof:eq:row-contract} and the identity
$t_N(s_1)-t_-=a_D\gamma^{-1}$ gives
\begin{align*}
 &\left\|
 \int_{t_-}^{t_N(s_1)}\mathcal P_2(t_N(s_1)-\tau)
 \mathsf{Top}_{N,2}^{\ord}\mathcal B[\varphi_\tau]\,\dd\tau
 \right\|_{\mathrm{row}}
 \le
 C_D n^{-1/2}\int_{t_-}^{t_N(s_1)}\omega_1(\varphi_\tau)^2\,\dd\tau
 \le C_{D,\rho_0}\frac{e^{-s_1}}{\sqrt n}.
\end{align*}
This completes the estimates on the two right-hand side terms in \eqref{prof:eq:G2tNs1}, and thus proves
\eqref{prof:eq:terminal-row}.
\end{proof}

\section{Strong--Rayleigh closure}\label{prof:sec:strong-rayleigh}

We now convert the terminal row bound \eqref{prof:eq:terminal-row} for the
degree-two discrepancy $G_2$ into a uniform terminal exponential-chaos bound
using Strong--Rayleigh structure.  Throughout this section, the warm-start
abbreviations $n=n_N$, $k=k_N$, $\rho=\rho_N$, and $\gamma=\gamma_N$ remain in force.

Strong--Rayleigh estimates apply to the full normalized moments.  For distinct
$X=(x_1,\ldots,x_r)$, the self-dual representation of
\eqref{prof:eq:general-normalized-moment} is
\begin{equation}
\mathcal M_{N,r}[h_t^S](X)
=\left(\frac n{\beta_N}\right)^{r/2}
\mathbb E_X\left[\prod_{i=1}^rF_S(\mathbf X_t^X(i))\right].
 \label{prof:eq:warm-start-normalized-moment}
\end{equation}
At degree two,
\begin{equation}
\overline G_2(t;x,y)
=\mathcal M_{N,2}[h_t^S](x,y)-\varphi_t(x)\varphi_t(y),
\label{prof:eq:G2M2}
\end{equation}
by \eqref{prof:eq:sep-self-duality}, and
$G_2=\mathsf{Top}_{N,2}^{\ord}\overline G_2$.  Thus the row estimate controls
only the top pair component, whereas the Strong--Rayleigh argument requires the
full $\overline G_2$.  The first step is therefore to recover its constant and
degree-one Hoeffding components.

\begin{lemma}[Exact lower pair layers]
\label{prof:lem:pair-layers}
For every fixed profile coordinate $s_1<-1$ and all sufficiently large $N$,
at time $t=t_N(s_1)$ we have
\begin{equation}
\overline G_2(t_N(s_1); x,y)=c_{s_1}+u_{s_1}^{(1)}(x)+u_{s_1}^{(1)}(y)+G_2(t_N(s_1);x,y),
 \label{prof:eq:pair-decomp}
\end{equation}
where $u_{s_1}^{(1)}:V_N\to\mathbb R$ is mean zero,
\begin{align}
c_{s_1}&=-1+\frac{\norm{\varphi_{t_N(s_1)}}_{2,\mathrm{av}}^2}{n-1},
 \label{prof:eq:pair-constant-layer}\\
\norm{u_{s_1}^{(1)}}_\infty
&\le C_{\rho_0}\left(
\frac{e^{-s_1/2}}{\sqrt n}+\frac{e^{-s_1}}n\right).
 \label{prof:eq:pair-linear-layer-bound}
\end{align}
\end{lemma}

\begin{proof}
Set $t=t_N(s_1)$ throughout this proof.  We identify the constant and degree-one Hoeffding
components of $\overline G_2(t)$; the remaining degree-two component is
$G_2(t)=\mathsf{Top}_{N,2}^{\ord}\overline G_2(t)$.

By the $r=2$ self-duality identity \eqref{prof:eq:sep-self-duality},
$
 \mathcal M_{N,2}[h_t^S]
 =\mathcal P_2(t)\mathsf R_2(\varphi_0\otimes\varphi_0)
 $.
The two-particle generator is self-adjoint in
$L^2(\upsilon_2^{\ord})$, and $\mathcal P_2(t)\mathbf 1=\mathbf 1$.
Therefore the uniform pair-space average is preserved in time:
\[
\begin{aligned}
 \mathbb E_{(x,y)\sim\upsilon_2^{\ord}}
 \!\left[\mathcal M_{N,2}[h_t^S](x,y)\right]
 &=\left\langle \mathbf 1,
   \mathcal P_2(t)\mathsf R_2(\varphi_0\otimes\varphi_0)
   \right\rangle_{L^2(\upsilon_2^{\ord})}\\
 &=\left\langle \mathcal P_2(t)\mathbf 1,
   \mathsf R_2(\varphi_0\otimes\varphi_0)
   \right\rangle_{L^2(\upsilon_2^{\ord})}
 =\mathbb E_{(x,y)\sim\upsilon_2^{\ord}}
 \!\left[\mathcal M_{N,2}[h_0^S](x,y)\right].
\end{aligned}
\]
At time zero,
\[
 \mathbb E_{(x,y)\sim\upsilon_2^{\ord}}\!\left[\mathcal M_{N,2}[h_0^S](x,y)\right]
 =\frac1{n(n-1)}\sum_{x\ne y}\varphi_0(x)\varphi_0(y)
 =\frac{(\sum_x\varphi_0(x))^2-\sum_x\varphi_0(x)^2}{n(n-1)}.
\]
Since $\sum_x\varphi_0(x)=0$ and
\[
 \sum_x\varphi_0(x)^2
 =\frac n{\beta_N}\sum_xF_S(x)^2
 =\frac n{\beta_N}\,n\rho(1-\rho)
 =n(n-1),
\]
we obtain
\[
 \mathbb E_{(x,y)\sim\upsilon_2^{\ord}}\!\left[\mathcal M_{N,2}[h_t^S](x,y)\right]
 =-1.
\]
Meanwhile, $\sum_x\varphi_t(x)=0$, so
\[
 \mathbb E_{(x,y)\sim\upsilon_2^{\ord}}
 [\varphi_t(x)\varphi_t(y)]
 =-\frac{\|\varphi_t\|_{2,\mathrm{av}}^2}{n-1}.
\]
Using \eqref{prof:eq:G2M2}, the constant chaos component is therefore
\[
 c_{s_1}
 =\mathbb E_{(x,y)\sim\upsilon_2^{\ord}}
 \!\left[\overline G_2(t;x,y)\right]
 =-1+\frac{\|\varphi_t\|_{2,\mathrm{av}}^2}{n-1},
\]
which is \eqref{prof:eq:pair-constant-layer}.

For the degree-one component, use the down map from
\eqref{prof:eq:ordered-up-down-formulas} and set
\[
 d_t:=\mathsf{Down}_{2\to1}\overline G_2(t),
 \qquad
 d_t(x)=\frac1{n-1}\sum_{y\ne x}\overline G_2(t;x,y).
\]
The elementary one-particle identity used here and again in the moment envelope proof is
\begin{align}
 F_S^2=\rho(1-\rho)+(1-2\rho)F_S.
 \label{prof:eq:F_S^2}
\end{align}
Since $e^{-t\mathcal L_1}\mathbf 1=\mathbf 1$ and
$e^{-t\mathcal L_1}F_S=m_t$, the last display \eqref{prof:eq:F_S^2} propagates to
\begin{equation}
 e^{-t\mathcal L_1}F_S^2
 =\rho(1-\rho)+(1-2\rho)m_t.
 \label{prof:eq:F_S_squared_heat}
\end{equation}
Equivalently, because
$\varphi_0=\sqrt{n/\beta_N}\,F_S$ and
$\beta_N=\rho(1-\rho)n/(n-1)$,
\begin{equation}
 e^{-t\mathcal L_1}[\varphi_0^2]
 =(n-1)+\frac n{\beta_N}(1-2\rho)m_t.
 \label{prof:eq:phi0-square-heat}
\end{equation}

We now compute the two terms in
$d_t=\mathsf{Down}_{2\to1}\mathcal M_{N,2}[h_t^S]
-\mathsf{Down}_{2\to1}\mathsf R_2(\varphi_t\otimes\varphi_t)$.
Since $\sum_y\varphi_0(y)=0$,
\[
 \bigl(\mathsf{Down}_{2\to1}
       \mathsf R_2(\varphi_0\otimes\varphi_0)\bigr)(x)
 =\frac{\varphi_0(x)}{n-1}\sum_{y\ne x}\varphi_0(y)
 =-\frac{\varphi_0(x)^2}{n-1}.
\]
The down--semigroup intertwining
\eqref{prof:eq:hard-core-down-semigroup-intertwining} and \eqref{prof:eq:phi0-square-heat} therefore give
\[
 \bigl(\mathsf{Down}_{2\to1}\mathcal M_{N,2}[h_t^S]\bigr)(x)
 =-\frac1{n-1}e^{-t\mathcal L_1}[\varphi_0^2](x)
 =-1-\frac{1-2\rho}{\rho(1-\rho)}m_t(x).
\]
Likewise, $\sum_y\varphi_t(y)=0$ gives
\[
 \bigl(\mathsf{Down}_{2\to1}
       \mathsf R_2(\varphi_t\otimes\varphi_t)\bigr)(x)
 =-\frac{\varphi_t(x)^2}{n-1}
 =-\frac{m_t(x)^2}{\rho(1-\rho)}.
\]
Subtracting the two down map identities yields
\begin{equation}
 d_t(x)=-1-\frac{1-2\rho}{\rho(1-\rho)}m_t(x)
       +\frac1{\rho(1-\rho)}m_t(x)^2.
 \label{prof:eq:pair-down-layer}
\end{equation}

Define
\begin{align}
 u_{s_1}^{(1)}:=\frac{n-1}{n-2}(d_t-c_{s_1}).
 \label{prof:eq:pair-linear-field}
\end{align}
Averaging the definition of $d_t$ over $x$ gives
$n^{-1}\sum_x d_t(x)=c_{s_1}$, so $u_{s_1}^{(1)}$ is mean zero.  By
\eqref{prof:eq:fd-down-up-eigenvalue} with $(q,j)=(2,1)$,
\[
 2\mathsf{Down}_{2\to1}\mathsf{Up}_{N,1\to2}u_{s_1}^{(1)}
 =\frac{n-2}{n-1}u_{s_1}^{(1)}
 =d_t-c_{s_1}.
\]
Consequently
\[
 \mathsf{Down}_{2\to1}
 \left(\overline G_2(t)-c_{s_1}-2\mathsf{Up}_{N,1\to2}u_{s_1}^{(1)}\right)=0.
\]
The bracketed function is symmetric, hence lies in
$\ker\mathsf{Down}_{2\to1}=\Ran\mathsf{Top}_{N,2}^{\ord}$.  Moreover,
$c_{s_1}+2\mathsf{Up}_{N,1\to2}u_{s_1}^{(1)}$ lies in
$\Ran\mathsf{Up}_{N,1\to2}$.  The orthogonal decomposition encoded by
\eqref{prof:eq:fd-top-projection-formula} therefore identifies the bracketed
function with $G_2(t)$.  Pointwise, this gives \eqref{prof:eq:pair-decomp}.

Substituting \eqref{prof:eq:pair-down-layer} and
\eqref{prof:eq:pair-constant-layer} into the definition of $u_{s_1}^{(1)}$ \eqref{prof:eq:pair-linear-field} gives
\begin{equation}
 u_{s_1}^{(1)}(x)=\frac{n-1}{n-2}
 \left(
 -\frac{1-2\rho}{\rho(1-\rho)}m_t(x)
 +\frac{m_t(x)^2}{\rho(1-\rho)}
 -\frac{\|\varphi_t\|_{2,\mathrm{av}}^2}{n-1}
 \right).
 \label{prof:eq:pair-degree-one-field-exact}
\end{equation}
By \eqref{prof:eq:terminal-one-particle}, the relation
$m_t=\sqrt{\beta_N/n}\,\varphi_t$, and the uniform upper and lower bounds on
$\rho(1-\rho)$ and $\beta_N$ in the density window, we obtain the upper bounds
\[
 \|m_t\|_\infty\le C_{\rho_0}e^{-s_1/2}n^{-1/2},
 \qquad
 \|m_t^2\|_\infty\le C_{\rho_0}e^{-s_1}n^{-1},
\qquad
 \frac{\|\varphi_t\|_{2,\mathrm{av}}^2}{n-1}
 \le C_De^{-s_1}n^{-1}.
\]
Substituting these bounds into \eqref{prof:eq:pair-degree-one-field-exact}
yields \eqref{prof:eq:pair-linear-layer-bound}.
\end{proof}

\begin{theorem}[All-degree moment envelope]
\label{prof:thm:moment-envelope}
There is $C=C(D,\rho_0)$ such that, for every fixed profile coordinate
$s_1<-1$, every $1\le r\le \ell_N$, every deterministic
$S\subset V_N$ with $|S|=k$, and all sufficiently large $N$,
\begin{equation}
\norm{\mathcal M_{N,r}[h_{t_N(s_1)}^S]}_\infty^2
\le\bigl[C(r+e^{-2s_1})\bigr]^r.
 \label{prof:eq:moment-envelope}
\end{equation}
\end{theorem}

\begin{proof}
Fix a distinct ordered configuration $X=(x_1,\ldots,x_r)\in\Omega_{N,r}^{\ord}$,
and write $\mathbf X_t^X=(\mathbf X_t^X(1),\ldots,\mathbf X_t^X(r))$ for the
coordinates of the ordered hard-core process introduced in
\eqref{prof:eq:sep-self-duality}.  Write
$\underline X=\{x_1,\ldots,x_r\}$ and
$\underline{\mathbf X}_t^X:=\{\mathbf X_t^X(1),\ldots,\mathbf X_t^X(r)\}$ for the
corresponding unlabeled initial set and occupied set at time $t$.  Forgetting the
labels turns $\mathbf X_t^X$ into the ordinary $r$-particle exclusion process started
from $\underline X$; in particular, the law of $\underline{\mathbf X}_t^X$ depends only
on $\underline X$, not on the ordering chosen in $X$.  Set
$\mathsf N_S(t):=|\underline{\mathbf X}_t^X\cap S|$.  Our purpose in invoking
Strong--Rayleigh theory is to show that this scalar count has a
Poisson--binomial law.  That factorization will turn the centered moments of
$\mathsf N_S(t)$ into quantities controlled by its first two moments.  To obtain it, we
encode the exclusion law by its multivariate occupation polynomial and then
specialize that polynomial to the one-variable probability generating
function of $\mathsf N_S(t)$.

For a probability measure $\mu$ on subsets
of $V_N$, its multiaffine occupation generating polynomial is
\[
 \mathscr P_\mu(\boldsymbol z):=\sum_{A\subseteq V_N}\mu(A)\prod_{x\in A}z_x.
\]
The measure $\mu$ is called \emph{Strong--Rayleigh} when $\mathscr P_\mu$ is real stable,
that is, nonzero whenever every variable has positive imaginary part.  For
the point mass at $\underline X$ one has
$\mathscr P_\mu(\boldsymbol z)=\prod_{x\in\underline X}z_x$, which is real stable.
\localheading{Poisson--binomial factorization.}
The preservation theorem for finite symmetric exclusion with symmetric
nonnegative exchange rates \cite[Proposition~5.1]{BBL09} therefore applies
to the nearest-neighbor torus process started from this deterministic
configuration.  Let $\mu_t^{\underline X}$ denote the law of the unlabeled occupied set
$\underline{\mathbf X}_t^X$.  In the occupation polynomial $\mathscr P_{\mu_t^{\underline X}}$, set
$z_x=z$ for $x\in S$ and $z_x=1$ for $x\notin S$.  For every occupied set
$A\subseteq V_N$, this specialization sends its monomial to
$
 \prod_{x\in A}z_x=z^{|A\cap S|}
 $.
This gives exactly the probability generating function of $\mathsf N_S(t)$:
\[
 \mathcal G_{\underline X,t}(z)
 :=\mathscr P_{\mu_t^{\underline X}}\bigl((z)_{x\in S},(1)_{x\notin S}\bigr)
 =\sum_{A\subseteq V_N}\mu_t^{\underline X}(A)z^{|A\cap S|}
 =\mathbb E_X[z^{\mathsf N_S(t)}].
\]

To transfer stability to this one-variable count polynomial, note that the
two operations in the definition of $\mathcal G_{\underline X,t}$ preserve real
stability: specialize the
variables outside $S$ to the real value $1$, and then identify all variables
in $S$ with the common variable $z$.  For the first operation, approximate
$1$ by $1+i\varepsilon$ and let $\varepsilon\downarrow0$; Hurwitz's theorem gives a
stable limit, with the zero polynomial alternative excluded by
$\mathscr P_{\mu_t^{\underline X}}(\boldsymbol 1)=1$.  The second operation is immediate,
since $\operatorname{Im}z>0$ places every remaining variable in the upper
half-plane.  Thus $\mathcal G_{\underline X,t}$ is a real stable univariate polynomial,
hence is real-rooted.  Its coefficients are nonnegative, so it has no positive
zero.  Consequently every zero of $\mathcal G_{\underline X,t}$ is nonpositive.  Finally,
$\mathcal G_{\underline X,t}(1)=1$.  If its polynomial degree is $d\le r$, these facts
give
\[
 \mathcal G_{\underline X,t}(z)=c_{\underline X,t}\prod_{j=1}^d(z+a_j^{\mathrm{PB}}),
 \qquad a_j^{\mathrm{PB}}\ge0.
\]
Normalization at $z=1$ implies
$c_{\underline X,t}\prod_{j=1}^d(1+a_j^{\mathrm{PB}})=1$.  Thus, with $\theta_j^{\mathrm{PB}}=(1+a_j^{\mathrm{PB}})^{-1}\in[0,1]$, one has $(z+a_j^{\mathrm{PB}})/(1+a_j^{\mathrm{PB}})=1-\theta_j^{\mathrm{PB}}+\theta_j^{\mathrm{PB}}z$.
Appending $\theta_j^{\mathrm{PB}}=0$ for $d<j\le r$ gives
\[
\mathcal G_{\underline X,t}(z)
 =\prod_{j=1}^r(1-\theta_j^{\mathrm{PB}}+\theta_j^{\mathrm{PB}}z),
 \qquad \theta_j^{\mathrm{PB}}\in[0,1].
\]
To identify this law explicitly, let $B_1,\ldots,B_r$ be independent
Bernoulli random variables with $\mathbb P(B_j=1)=\theta_j^{\mathrm{PB}}$.  Their sum has
probability generating function
\[
 \mathbb E\!\left[z^{B_1+\cdots+B_r}\right]
 =\prod_{j=1}^r\mathbb E[z^{B_j}]
 =\prod_{j=1}^r(1-\theta_j^{\mathrm{PB}}+\theta_j^{\mathrm{PB}}z)
 =\mathcal G_{\underline X,t}(z).
\]
Since both random variables take values in $\{0,\ldots,r\}$, equality of
their probability generating functions is equality of the coefficients and
therefore equality in law.  Hence $\mathsf N_S(t)$ has the \emph{Poisson--binomial law}, that
is, the law of a sum of independent Bernoulli variables with possibly
different success probabilities.  The parameters $\theta_1^{\mathrm{PB}},\ldots,\theta_r^{\mathrm{PB}}$
represent only this scalar count law; they are auxiliary Bernoulli parameters
and are not attached to the labeled particles of $\mathbf X_t^X$.

\localheading{Centered product identity.}
Let $I_i=\one_{\{\mathbf X_t^X(i)\in S\}}$ and
$z_\rho=-(1-\rho)/\rho$.  For every realization,
$
 \prod_{i=1}^r(I_i-\rho)
 =(-\rho)^rz_\rho^{\mathsf N_S(t)}$.
Consequently, if $\varepsilon_j=\theta_j^{\mathrm{PB}}-\rho$, then
\begin{align}
 \mathbb E_X\left[\prod_{i=1}^r(I_i-\rho)\right]
 =(-\rho)^r\mathcal G_{\underline X,t}(z_\rho)
 =\prod_{j=1}^r\bigl[-\rho(1-\theta_j^{\mathrm{PB}}+\theta_j^{\mathrm{PB}}z_\rho)\bigr]
  =\prod_{j=1}^r(\theta_j^{\mathrm{PB}}-\rho)
 =\prod_{j=1}^r\varepsilon_j.
 \label{prof:eq:centered-product-factorization}
\end{align}
Substituting \eqref{prof:eq:centered-product-factorization} into the
definition \eqref{prof:eq:warm-start-normalized-moment} gives
\begin{equation}
 \mathcal M_{N,r}[h_t^S](X)
 =\left(\frac n{\beta_N}\right)^{r/2}
 \prod_{j=1}^r\varepsilon_j.
 \label{prof:eq:centered-parameter-product-identity}
\end{equation}

\localheading{The centered count: first and second moments.}
We now return to the ordered process $\mathbf X_t^X$ and, recalling
$F_S=\one_S-\rho$, set
\[
 Y_t:=\mathsf N_S(t)-r\rho=\sum_{i=1}^rF_S(\mathbf X_t^X(i)),
 \qquad m_i:=m_t(x_i).
\]
This symmetric centered count is the bridge between the Poisson--binomial and
ordered exclusion descriptions.  We compute its mean and variance in both
representations; all pair sums below are empty when $r=1$.

On the Poisson--binomial side, the independent Bernoulli variables introduced
above give $Y_t\stackrel{\mathrm d}=\sum_{j=1}^r(B_j-\rho)$.  Hence
\begin{align}
 \mathbb E_X[Y_t]
 =\sum_{j=1}^r\varepsilon_j, \qquad
 \operatorname{Var}_X(Y_t)
 =\sum_{j=1}^r\theta_j^{\mathrm{PB}}(1-\theta_j^{\mathrm{PB}})
 =r\rho(1-\rho)
   +(1-2\rho)\sum_{j=1}^r\varepsilon_j
   -\sum_{j=1}^r\varepsilon_j^2.
   \label{prof:eq:PBEV}
\end{align}

On the ordered-exclusion side, the iterated up--semigroup
intertwining \eqref{prof:eq:up-intertwine} reduces symmetric one- and
two-particle observables of the $r$-particle process to the corresponding
lower-particle semigroups.  The first moment is therefore
\[
 \mathbb E_X[Y_t]
 =r\,\mathsf{Up}_{N,1\to r}
   \bigl(e^{-t\mathcal L_1}F_S\bigr)(X)
 =\sum_{i=1}^rm_i,
\]
where we used \eqref{prof:eq:iterated-up-average}.  Comparing with the
Poisson--binomial mean in \eqref{prof:eq:PBEV} gives
\begin{equation}
 \sum_{j=1}^r\varepsilon_j=\sum_{i=1}^rm_i.
 \label{prof:eq:first-parameter-identity}
\end{equation}

For the second moment, write
$Y_i(t):=F_S(\mathbf X_t^X(i))$, so that $Y_t=\sum_iY_i(t)$.  Applying
\eqref{prof:eq:up-intertwine} at level one and then
\eqref{prof:eq:iterated-up-average} gives the diagonal contribution
explicitly as
\[
 \mathbb E_X\!\left[\sum_{i=1}^rY_i(t)^2\right]
 =r\,\mathsf{Up}_{N,1\to r}
   \bigl(e^{-t\mathcal L_1}F_S^2\bigr)(X).
\]
Using \eqref{prof:eq:F_S_squared_heat} in the last display yields
\[
 \mathbb E_X\!\left[\sum_{i=1}^rY_i(t)^2\right]
 =r\rho(1-\rho)+(1-2\rho)\sum_{i=1}^rm_i.
\]
For the off-diagonal part, if $r\ge2$, then
\eqref{prof:eq:iterated-up-average}, \eqref{prof:eq:up-intertwine}, and the
pair identity \eqref{prof:eq:G2M2} give
\begin{align*}
 \mathbb E_X\!\left[
 \sum_{1\le i<j\le r}Y_i(t)Y_j(t)\right]
 &=\sum_{1\le i<j\le r}
 \bigl(\mathcal P_2(t)\mathsf R_2(F_S\otimes F_S)\bigr)(x_i,x_j)\\
 &=\sum_{1\le i<j\le r}
 \left(
 \frac{\beta_N}{n}\overline G_2(t;x_i,x_j)+m_i m_j
 \right).
\end{align*}
For $r=1$ the pair sum is empty.  Combining the diagonal and off-diagonal
parts yields, for every $r\ge1$,
\begin{align*}
 \mathbb E_X[Y_t^2]
 =r\rho(1-\rho)+(1-2\rho)\sum_{i=1}^rm_i
 +2\sum_{1\le i<j\le r}
 \left(
 \frac{\beta_N}{n}\overline G_2(t;x_i,x_j)+m_i m_j
 \right).
\end{align*}
Subtracting
$\bigl(\mathbb E_X[Y_t]\bigr)^2=(\sum_i m_i)^2$ from the last display yields
\begin{equation*}
 \operatorname{Var}_X(Y_t)
 =r\rho(1-\rho)+(1-2\rho)\sum_{i=1}^rm_i
 -\sum_{i=1}^rm_i^2
 +\frac{2\beta_N}{n}
   \sum_{1\le i<j\le r}\overline G_2(t;x_i,x_j).
\end{equation*}
Comparing this ordered-process variance with the Poisson--binomial variance
in \eqref{prof:eq:PBEV}, and using
\eqref{prof:eq:first-parameter-identity}, the common constant and linear
terms cancel explicitly, leaving
\[
 \sum_{j=1}^r\varepsilon_j^2
 =\sum_{i=1}^rm_i^2
  -\frac{2\beta_N}{n}
    \sum_{1\le i<j\le r}\overline G_2(t;x_i,x_j).
\]
Since $\varphi_t(x_i)^2=(n/\beta_N)m_i^2$, multiplication by
$n/\beta_N$ gives
\begin{equation}
 \frac n{\beta_N}\sum_{j=1}^r\varepsilon_j^2
 =\sum_{i=1}^r\varphi_t(x_i)^2
 -2\sum_{1\le i<j\le r}\overline G_2(t;x_i,x_j).
 \label{prof:eq:centered-parameter-variance-identity}
\end{equation}

\localheading{From the variance identity to the moment envelope.}
Set $t=t_N(s_1)$.  Summing the decomposition in
\cref{prof:lem:pair-layers} over the pairs of coordinates of $X$ gives
\[
 \sum_{1\le i<j\le r}\overline G_2(t;x_i,x_j)
 =\binom r2 c_{s_1}
 +(r-1)\sum_{i=1}^r u_{s_1}^{(1)}(x_i)
 +\sum_{1\le i<j\le r}G_2(t;x_i,x_j).
\]
Substituting this identity into
\eqref{prof:eq:centered-parameter-variance-identity} and dividing by $r$ yields
\begin{align}
 \frac1r\frac n{\beta_N}\sum_{j=1}^r\varepsilon_j^2
 ={}&\frac1r\sum_{i=1}^r\varphi_t(x_i)^2
 -(r-1)c_{s_1}
 -\frac{2(r-1)}r\sum_{i=1}^r u_{s_1}^{(1)}(x_i) 
 -\frac2r\sum_{1\le i<j\le r}G_2(t;x_i,x_j).
 \label{prof:eq:variance-four-contributions}
\end{align}
We bound the four terms on the right-hand side of \eqref{prof:eq:variance-four-contributions} in order: the one-particle square term, the constant pair layer, the degree-one pair layer, and the degree-two pair layer.

For the one-particle term,
\[
 \frac1r\sum_{i=1}^r\varphi_t(x_i)^2
 \le\|\varphi_t\|_\infty^2
 \le C_De^{-s_1}
\]
by \eqref{prof:eq:terminal-one-particle}.  For the constant pair component,
\[
 (r-1)|c_{s_1}|=O(r+e^{-s_1}),
\]
because $c_{s_1}=-1+O(e^{-s_1}/n)$.  For the degree-one pair component,
\eqref{prof:eq:pair-linear-layer-bound}, $r\le n$, and the quadratic form of
Young's inequality give
\begin{equation}
 \frac{2(r-1)}r\left|\sum_{i=1}^r u_{s_1}^{(1)}(x_i)\right|
 \le2(r-1)\|u_{s_1}^{(1)}\|_\infty
 \le C_{\rho_0}\left(\frac{r}{\sqrt n}e^{-s_1/2}
      +\frac rn e^{-s_1}\right).
 \label{prof:eq:degree-one-pair-contribution-preyoung}
\end{equation}
To simplify the last upper bound, apply Young's inequality
$uv\le\frac12(u^2+v^2)$ with
$u=r/\sqrt n$ and $v=e^{-s_1/2}$.  Since $r\le n$,
\[
 \frac{r}{\sqrt n}e^{-s_1/2}
 \le \frac12\left(\frac{r^2}{n}+e^{-s_1}\right)
 \le \frac12\left(r+e^{-s_1}\right),
 \qquad
 \frac rn e^{-s_1}\le e^{-s_1}.
\]
Substituting these two bounds into \eqref{prof:eq:degree-one-pair-contribution-preyoung} gives
\begin{equation*}
 \frac{2(r-1)}r\left|\sum_{i=1}^r u_{s_1}^{(1)}(x_i)\right|
 \le C_{\rho_0}(r+e^{-s_1}).
\end{equation*}
For the degree-two pair component, fix $i$.  Because the points
$x_1,\ldots,x_r$ are distinct, the points $x_j$ with $j\ne i$ are
distinct elements of $V_N\setminus\{x_i\}$. 
Then using the symmetry of $G_2(t)$, Cauchy--Schwarz and
the definition of the row norm, we find
\begin{align*}
 2\sum_{1\le i<j\le r} |G_2(t;x_i,x_j)|
 &=\sum_{i=1}^r \sum_{j\ne i}^r |G_2(t;x_i,x_j)|
 \le \sum_{i=1}^r \sqrt{r-1}
   \left(\sum_{j\ne i}|G_2(t;x_i,x_j)|^2\right)^{1/2}\\
 &\le \sum_{i=1}^r \sqrt{r-1}
   \left(\sum_{y\ne x_i}|G_2(t;x_i,y)|^2\right)^{1/2}
 \le r \sqrt{r-1}\sqrt{n-1}\,\|G_2(t)\|_{\mathrm{row}}.
\end{align*}
The last term in \eqref{prof:eq:variance-four-contributions} carries an additional factor $1/r$, so using \eqref{prof:eq:terminal-row},
$\sqrt{r-1}\le\sqrt r$, and $\sqrt{(n-1)/n}\le1$,
\[
 \left|\frac2r\sum_{1\le i<j\le r}G_2(t;x_i,x_j)\right|
 \le \sqrt{r-1}\sqrt{n-1}\,\|G_2(t)\|_{\mathrm{row}}
 \le C_{D,\rho_0} \sqrt{r} e^{-s_1}.
\]
Applying $uv\le\frac12(u^2+v^2)$ once more, now with
$u=\sqrt r$ and $v=e^{-s_1}$, yields
$
 \sqrt r\,e^{-s_1}
 \le \frac12\left(r+e^{-2s_1}\right)$.
Hence the absolute value of the last term in
\eqref{prof:eq:variance-four-contributions} is at most
$C_{D,\rho_0}(r+e^{-2s_1})$.

Combining these four estimates gives an upper bound for \eqref{prof:eq:variance-four-contributions}:
\begin{equation}
 \frac1r\frac n{\beta_N}
 \sum_{j=1}^r\varepsilon_j^2
 \le C(r+e^{-2s_1}).
 \label{prof:eq:variance-bound}
\end{equation}
Finally, the arithmetic--geometric mean (AM-GM) inequality gives
\[
 \prod_{j=1}^r\varepsilon_j^2
 \le\left(\frac1r\sum_{j=1}^r\varepsilon_j^2\right)^r.
\]
Therefore \eqref{prof:eq:centered-parameter-product-identity} and \eqref{prof:eq:variance-bound} imply
\[
 |\mathcal M_{N,r}[h_t^S](X)|^2
 \le\left(\frac1r\frac n{\beta_N}
              \sum_{j=1}^r\varepsilon_j^2\right)^r
 \le\bigl[C(r+e^{-2s_1})\bigr]^r.
\]
Taking the supremum over $X$ proves \eqref{prof:eq:moment-envelope}.
\end{proof}

\localheading{All-degree chaos normalization.}
The fixed-degree asymptotic $\alpha_{N,r}=1+O_r(n^{-1})$ used earlier is
not uniform in $r$.  For the all-degree sum below we instead record the
geometric bound obtained directly from \eqref{prof:eq:chaos-normalization}.
For every $0\le r\le\ell_N$, using
$(a)_{\underline r}\ge(a/e)^r$ for $a\in\{k,n-k\}$,
$(n)_{\underline r}\le n^r$, $(n)_{\underline{2r}}\le n^{2r}$,
$\beta_N\le2\rho(1-\rho)$, and $k=\rho n$, one has
\begin{equation}
 \alpha_{N,r}^2\le (2e^2)^r=:C_0^r.
 \label{prof:eq:all-degree-alpha-bound}
\end{equation}
This estimate is uniform simultaneously in $N$, the density window, and all
admissible degrees $r$.

\begin{proposition}[Uniform terminal chaos weight]
\label{prof:prop:terminal-chaos-weight}
There exists $z_*=z_*(D,\rho_0)\in(0,1)$ such that for every
fixed profile coordinate $s_1<-1$, there exist
$N_0=N_0(D,\rho_0,s_1)\in\mathbb N$ and $K_{s_1}<\infty$ such that
\begin{equation}
\sup_{N\ge N_0}\sup_{\substack{S\subset V_N\\|S|=k}}
Z_S(z_*,t_N(s_1))\le K_{s_1}.
 \label{prof:eq:terminal-chaos-weight}
\end{equation}
\end{proposition}

\begin{proof}
At $t=t_N(s_1)$, apply \cref{prof:lem:density-chaos-interface} to
$f=h_t^S$.  The self-dual representation \eqref{prof:eq:warm-start-normalized-moment} gives
\[
 \widehat\Psi_{N,r}(h_t^S)
 =\frac{\alpha_{N,r}}{\sqrt{r!}}
   \mathsf{Top}_{N,r}^{\ord}\mathcal M_{N,r}[h_t^S].
\]
By the all-degree normalization estimate
\eqref{prof:eq:all-degree-alpha-bound} and the fact that
$\mathsf{Top}_{N,r}^{\ord}$ is an orthogonal projection in
the uniform distinct-tuple probability norm,
\[
 Z_S(z,t)
 \le\sum_{r=0}^{\ell_N}
 \frac{(C_0z)^r}{r!}
 \|\mathcal M_{N,r}[h_t^S]\|_\infty^2.
\]
Combining this inequality at $t=t_N(s_1)$ with
\cref{prof:thm:moment-envelope} yields
\[
 Z_S(z,t_N(s_1))
 \le\sum_{r=0}^{\ell_N}\frac{(C_0z)^r}{r!}
 [C(r+e^{-2s_1})]^r.
\]
Write $A_{s_1}=e^{-2s_1}$, so the right-hand side is
\begin{align}
 \sum_{r=0}^{\ell_N}\frac{(C_0z)^r}{r!}[C(r+A_{s_1})]^r.
 \label{prof:eq:final-sum}
\end{align}
Choose $z_*\in(0,1)$ so that $2eC_0Cz_*<1/2$.  For $r\ge A_{s_1}$,
$r+A_{s_1}\le2r$ and $r!\ge(r/e)^r$, so the summand in \eqref{prof:eq:final-sum} is at most $2^{-r}$.
For the finitely many integers $0\le r<A_{s_1}$, the summands are bounded by
\[
 \frac{(C_0z_*)^r}{r!}[C(r+A_{s_1})]^r,
\]
which depends on $s_1$ but not on $N$ or $S$.  
This shows that \eqref{prof:eq:final-sum} is summable, and the resulting sum $K_{s_1}$ holds uniformly over deterministic $S\subset V_N$ with $|S|=k$ and all sufficiently large $N$.
\end{proof}

We can now prove the profile coordinate warm-start bound.

\begin{proof}[Proof of \cref{prof:thm:warm-start}]
Let $\delta_*:=c_D^{-1}\log(z_*^{-1})$,
where $c_D$ is the degree-gap constant in
\cref{prof:lem:weight-propagation}.
Set $s_1:=s_0-\delta_*$.
Then $s_1<s_0<-1$, and
$
 t_N(s_0)-t_N(s_1)=\delta_*/(2\gamma)
 $.
Apply the exponential weight propagation inequality \eqref{prof:eq:weight-propagation}
with initial time $t_N(s_1)$, terminal time $t_N(s_0)$, and
weight $z_*$.  Since $z_*e^{c_D\delta_*}=1$, the terminal weight estimate from
\cref{prof:prop:terminal-chaos-weight}, namely
\eqref{prof:eq:terminal-chaos-weight}, yields directly
\[
 Z_S(1,t_N(s_0))
 \le Z_S(z_*,t_N(s_1))
 \le K_{s_1}.
\]
Because $Z_S(1,t)=\norm{h_t^S}_{L^2(\pi_N)}^2$, this proves
\eqref{prof:eq:warm-start}, uniformly over deterministic $S\subset V_N$ with
$|S|=k$ and all sufficiently large
$N$.
\end{proof}

\section{Completion of the main proofs}
\label{prof:sec:completion-proof}

We return to the global notation
$n_N,k_N,\rho_N,$ and $\gamma_N$, and to the source-explicit calibrated
objects $v_{N,t}^{S_N}$ and $g_{N,t}^{S_N}$ introduced in
\cref{prof:def:tilt}.

The profile coordinate warm-start theorem \cref{prof:thm:warm-start}, combined with
the linear chaos gap \eqref{prof:eq:degree-gap}, gives exponential damping of
the high-chaos tail throughout the cutoff window.

\begin{proposition}[High-degree damping after the warm start]
\label{prof:prop:degree-damping}
Let $J\subset\mathbb R$ be nonempty and compact.  Choose a fixed profile coordinate
$s_0<-1$ with $s_0<\inf J$, and put
$\delta_J:=\inf_{s\in J}(s-s_0)>0$.  Then there exists $N_0=N_0(D,\rho_0,s_0)$ such that, for every $R$
and every $N\ge N_0$,
\begin{equation*}
 \sup_{\substack{S_N\subset V_N,\ |S_N|=k_N\\ s\in J}}
 \sum_{r>R}\|P_{N,r}h_{N,t_N(s)}^{S_N}\|_2^2
 \le C_{D,s_0,\rho_0}e^{-c_DR\delta_J}.
\end{equation*}
\end{proposition}

\begin{proof}
By \eqref{prof:eq:degree-gap}, the semigroup restricted to degree $r$
satisfies
\[
 \|P_{N,r}h_{N,t_N(s)}^{S_N}\|_2^2
 \le e^{-2c_Dr\gamma_N(t_N(s)-t_N(s_0))}
      \|P_{N,r}h_{N,t_N(s_0)}^{S_N}\|_2^2.
\]
For $s\in J$ and $r>R$, the exponential factor is at most
$e^{-c_DR\delta_J}$.  Summing over $r>R$, using orthogonality of the chaos layers, and then
applying the warm-start estimate \eqref{prof:eq:warm-start}, for $N\ge N_0$, proves \cref{prof:prop:degree-damping}.
\end{proof}

We combine bounded-degree matching with the two high-degree tail estimates.

\begin{proof}[Proof of \cref{prof:thm:main}]
Fix a nonempty compact set $J\subset\mathbb R$.  By
\eqref{prof:eq:one-particle-bounds},
\[
 \sup_{\substack{S_N\subset V_N,\ |S_N|=k_N\\ s\in J}}
 \|m_{N,t_N(s)}^{S_N}\|_\infty
 \le C_J n_N^{-1/2}\longrightarrow0.
\]
Hence \cref{prof:lem:calibration} implies that, for all sufficiently large
$N$, the calibrated fields $v_{N,t_N(s)}^{S_N}$ and densities
$g_{N,t_N(s)}^{S_N}$ are defined uniformly over the same sources and
$s\in J$.  Let $s_0$ and $\delta_J$ be as in
\cref{prof:prop:degree-damping}.  For an integer $R\ge0$, define
\begin{align*}
 A_{N,R}(J)
 &:=\sup_{\substack{S_N\subset V_N,\ |S_N|=k_N\\ s\in J}}
   \sum_{r=0}^R
   \|P_{N,r}(h_{N,t_N(s)}^{S_N}-g_{N,t_N(s)}^{S_N})\|_2^2,\\
 B_{N,R}(J)
 &:=\sup_{\substack{S_N\subset V_N,\ |S_N|=k_N\\ s\in J}}
   \sum_{r>R}\|P_{N,r}h_{N,t_N(s)}^{S_N}\|_2^2,\\
 C_{N,R}(J)
 &:=\sup_{\substack{S_N\subset V_N,\ |S_N|=k_N\\ s\in J}}
   \sum_{r>R}\|P_{N,r}g_{N,t_N(s)}^{S_N}\|_2^2.
\end{align*}
Orthogonality of the chaos projections and
\(\|a-b\|_2^2\le2\|a\|_2^2+2\|b\|_2^2\) give the explicit bound
\begin{equation}
 \sup_{\substack{S_N\subset V_N,\ |S_N|=k_N\\ s\in J}}
 \|h_{N,t_N(s)}^{S_N}-g_{N,t_N(s)}^{S_N}\|_{L^2(\pi_N)}^2
 \le A_{N,R}(J)+2B_{N,R}(J)+2C_{N,R}(J).
 \label{prof:eq:main-three-term-bound}
\end{equation}
For every fixed $R\in\mathbb N_0$, the fixed-degree matching theorem
\cref{prof:thm:fixed-degree-matching} gives
\begin{equation}
 A_{N,R}(J)\longrightarrow0.
 \label{prof:eq:main-low-degree-limit}
\end{equation}
The high-degree damping estimate
\cref{prof:prop:degree-damping} gives, uniformly for all sufficiently large $N$,
\begin{equation*}
 B_{N,R}(J)\le C_{D,s_0,\rho_0}e^{-c_DR\delta_J},
\end{equation*}
and the tilt-tail estimate \cref{prof:lem:tilt-tail} gives
\begin{equation}
 \lim_{R\to\infty}\limsup_{N\to\infty}C_{N,R}(J)=0.
 \label{prof:eq:main-tilt-tail}
\end{equation}
Taking the upper limit as $N\to\infty$ in
\eqref{prof:eq:main-three-term-bound}, using
\eqref{prof:eq:main-low-degree-limit}--\eqref{prof:eq:main-tilt-tail}, and
then letting $R\to\infty$ proves
\eqref{prof:eq:L2-local-equilibrium}.

For every $S_N$ and $s$, the reverse triangle inequality for total
variation gives
\begin{align*}
 \left|
 \|\mu_{N,t_N(s)}^{S_N}-\pi_N\|_{\TV}
 -\frac12\mathbb E_{\pi_N}|g_{N,t_N(s)}^{S_N}-1|
 \right|
 \le
 \|\mu_{N,t_N(s)}^{S_N}-g_{N,t_N(s)}^{S_N}\pi_N\|_{\TV}\\
 =
 \frac12\|h_{N,t_N(s)}^{S_N}-g_{N,t_N(s)}^{S_N}\|_{L^1(\pi_N)}
 \le
 \frac12\|h_{N,t_N(s)}^{S_N}-g_{N,t_N(s)}^{S_N}\|_{L^2(\pi_N)}.
\end{align*}
Taking the same uniform supremum and applying
\eqref{prof:eq:L2-local-equilibrium} proves
\eqref{prof:eq:TV-tilt-reduction}.

Finally, fix a deterministic source sequence $(S_N)$ and a fixed
$s\in\mathbb R$, and suppose that $\mathfrak q_N^{S_N}(s)\to\mathfrak q$ for some scalar $\mathfrak q\ge0$.
By \eqref{prof:eq:calibrated-q},
\[
 \beta_N\|v_{N,t_N(s)}^{S_N}\|_{2,\mathrm{cnt}}^2
 =\mathfrak q_N^{S_N}(s)+o(1)\longrightarrow\mathfrak q.
\]
Since the corresponding Gaussianization parameter satisfies
\(\varsigma_N^2=\beta_N\|v_{N,t_N(s)}^{S_N}\|_{2,\mathrm{cnt}}^2\to\mathfrak q\) and \(\varsigma_N\ge0\), we have
\(\varsigma_N\to\sqrt{\mathfrak q}\).  The Gaussianization lemma
\cref{prof:lem:gaussianization}, applied to
\(v_N=v_{N,t_N(s)}^{S_N}\), therefore yields
\[
 \frac12\mathbb E_{\pi_N}|g_{N,t_N(s)}^{S_N}-1|
 \longrightarrow
 2\PhiGauss\!\left(\frac{\sqrt{\mathfrak q}}{2}\right)-1.
\]
Combining this limit with \eqref{prof:eq:TV-tilt-reduction} proves
\eqref{prof:eq:explicit-profile}.
\end{proof}

\subsection{Profile corollaries and source examples}
\label{prof:sec:profile-corollary-proofs}

We derive the two profile refinements stated after the main theorem.  Their
proofs use the established $L^2$ local equilibrium theorem and reduce to the
canonical tilt and the one-particle Fourier evolution.

\begin{proof}[Proof of \cref{prof:cor:uniform-profile}]
By \cref{prof:thm:main}, it remains to evaluate the total variation distance
of the calibrated tilt uniformly for $s\in J$.  The one-particle bounds and
calibration estimates in
\cref{prof:lem:one-particle-bounds,prof:lem:calibration} give
\[
 \sup_{s\in J}\|v_{N,t_N(s)}^{S_N}\|_\infty\longrightarrow0,
 \qquad
 \sup_{s\in J}\|v_{N,t_N(s)}^{S_N}\|_{2,\mathrm{cnt}}<\infty,
\]
while \eqref{prof:eq:calibrated-q} gives the corresponding
calibrated-variance comparison uniformly on $J$.
Apply \eqref{prof:eq:uniform-tilt-TV} with
$\mathcal I_N=J$ and $v_{N,s}=v_{N,t_N(s)}^{S_N}$.  It gives, uniformly for $s\in J$,
\[
 \frac12\mathbb E_{\pi_N}|g_{N,t_N(s)}^{S_N}-1|
 =2\PhiGauss\!\left(
 \frac{\sqrt{\beta_N\|v_{N,t_N(s)}^{S_N}\|_{2,\mathrm{cnt}}^2}}2
 \right)-1+o(1).
\]
Combine \eqref{prof:eq:calibrated-q} and
\eqref{prof:eq:uniform-q-convergence} with this identity and the uniform continuity of $a\mapsto2\PhiGauss(\sqrt a/2)-1$ on bounded
intervals.  The total variation reduction
\eqref{prof:eq:TV-tilt-reduction} then proves
\eqref{prof:eq:uniform-explicit-profile}.
\end{proof}

\begin{proof}[Proof of \cref{prof:cor:first-shell-profile}]
Parseval's identity and \eqref{prof:eq:profile} give
\begin{equation*}
 \mathfrak q_N(s)=\beta_N^{-1}\sum_{p\ne0}
 e^{-2t_N(s)\lambda_N(p)}|\widehat F_N(p)|^2.
\end{equation*}
The contribution of $p\in\mathcal S_N^{(1)}$ is exactly
$e^{-s}a_N(S_N)$, because
$e^{-2\gamma_Nt_N(s)}=n_N^{-1}e^{-s}$.  For all sufficiently large $N$,
every nonzero momentum outside $\mathcal S_N^{(1)}$ has eigenvalue at least
$2\gamma_N$.  Indeed, if two coordinates are nonzero, each contributes at
least $\gamma_N$.  If exactly one coordinate is nonzero with absolute torus
representative at least two, then
$
 2-2\cos(4\pi/N)=4\cos^2(\pi/N)\gamma_N\ge2\gamma_N
$
for all sufficiently large $N$.  Hence, uniformly for $s\in J$,
\[
 \sum_{p\notin\{0\}\cup\mathcal S_N^{(1)}}
 e^{-2t_N(s)\lambda_N(p)}|\widehat F_N(p)|^2
 \le C_Jn_N^{-2}\|F_N\|_{2,\mathrm{cnt}}^2
 \le C_Jn_N^{-1}.
\]
The density window bounds $\beta_N^{-1}$, proving
\eqref{prof:eq:first-shell-reduction}.  The profile conclusion follows from
\cref{prof:cor:uniform-profile}.
\end{proof}

\localheading{Slab source.}
Consider the slab sequence described after
\cref{prof:cor:first-shell-profile}; in particular, it satisfies the standing
density window when the profile theorem is invoked.  Thus
$w_N/N\to\rho\in(0,1)$ and
$
 S_N=\{x\in\T_N^D:0\le x_1<w_N\}
 $.
Then $\rho_N=w_N/N$, and the coefficients $\widehat F_N(\pm\mathbf e_j)$ vanish
for $j\ge2$, while
\[
 \frac1{\sqrt{n_N}}\widehat F_N(\mathbf e_1)
 =\frac1N\sum_{a=0}^{w_N-1}e^{-2\pi ia/N}
 \longrightarrow e^{-\pi i\rho}\frac{\sin(\pi\rho)}{\pi}.
\]
The coefficient at $-\mathbf e_1$ is its complex conjugate.  Since
$\beta_N\to\rho(1-\rho)$, substitution in
\eqref{prof:eq:first-shell-amplitude} proves
\eqref{prof:eq:slab-amplitude}.  In particular, the half-torus amplitude is
$8/\pi^2$.

\localheading{Sources orthogonal to the first nonzero eigenspace.}
If $\widehat F_N(p)=0$ for every $p\in\mathcal S_N^{(1)}$, then
\eqref{prof:eq:first-shell-reduction} gives
$\mathfrak q_N(s)=O_J(n_N^{-1})$.  \Cref{prof:cor:uniform-profile} therefore
shows that the total variation distance tends to zero uniformly for $s$ in compact subsets of the cutoff window.  For even $N$, the checkerboard source
$
 S_N=\{x\in\T_N^D:x_1+\cdots+x_D\ \text{is even}\}
$
has density $1/2$.  Its centered indicator is a nonzero scalar multiple of the character at
momentum $(N/2,\ldots,N/2)$, which is orthogonal to every $\chi_p$ with $p\in\mathcal S_N^{(1)}$.
Thus the checkerboard discrepancy is orthogonal to $\mathscr E_N^{(1)}$.

\appendix

\section{Canonical cumulants and response estimates}

The purpose of this appendix is to prove the finite-population response
package stated in \cref{prof:lem:canonical-response}.  The only genuinely
nonproduct feature that must be controlled is the dependence created by
conditioning independent Bernoulli variables on their total population.
We first isolate that dependence on a fixed set of marked sites, and then
transfer the resulting conditional cumulant estimate to the canonical tilt.

For a fixed marked set, the population constraint enters its marginal only
through the deleted-sum masses $q_A(h)$.  We control their finite differences
by a centered exponential tilt and a differentiated local limit estimate,
then convert those bounds into a connected-interaction estimate for marked
cumulants.  The final two subsections transfer that estimate to the canonical
tilt and to the response bounds used in
\cref{prof:sec:canonical-response-calibration}.

\subsection{Support-sensitive conditional Bernoulli cumulants}
\label{prof:app:conditional-bernoulli-cumulants}

Fix $\underline p\in(0,1/2)$ and an integer $M\ge1$.  Let $\mathcal I$ be a finite set of
cardinality $n$, let $(\mathsf B_x)_{x\in\mathcal I}$ be independent Bernoulli variables,
and assume
\begin{equation}
 \underline p\le p_x:=\Pp(\mathsf B_x=1)\le 1-\underline p,
 \qquad
 \sum_{x\in\mathcal I}p_x=k\in\mathbb Z.
 \label{prof:canred:eq:standing}
\end{equation}
For every $U\subseteq\mathcal I$, write
$
 \mathsf N_U:=\sum_{x\in U}\mathsf B_x$.
Thus $\mathsf N_{\mathcal I}$ is the total Bernoulli count.  

We denote conditional
cumulants under $\Pp(\,\cdot\mid \mathsf N_{\mathcal I}=k)$ by $\Cum_k$.
To prove \eqref{prof:eq:canonical-distinct-cumulant-general}, it suffices
in the present notation to show
\[
 \abs{\Cum_k(\mathsf B_{x_1},\ldots,\mathsf B_{x_j})}
 \le C_{M,\underline p}n^{1-d},
 \qquad
 1\le j\le M,
\]
whenever $x_1,\ldots,x_j$ involve exactly $d$ distinct sites.
To this end, we shall fix the distinct marked set
\[
 A:=\{x_1,\ldots,x_j\} \subset \mathcal I,\qquad d:=|A|\le M,
\]
and write $A^c:=\mathcal I\setminus A$ with the complement taken relative to $\mathcal I$.
Define
\[
 q_A(h):=\Pp(\mathsf N_{A^c}=k-h),\qquad h\in\{0,1,\ldots,d\}.
\]
This scalar is the key reduction: once the marked configuration is fixed,
its Hamming weight $h$ is the only information seen by the population
constraint, so all dependence among the marked variables is carried by the
variation of $q_A(h)$ with $h$.  Equivalently, $q_A(h)$ is the
Poisson--binomial point mass of the unmarked coordinates, and it has the
coefficient--product representation
\[
 q_A(h)
 =\sum_{\substack{\Omega\subseteq A^c\\ |\Omega|=k-h}}
   \prod_{x\in \Omega}p_x
   \prod_{y\in A^c\setminus \Omega}(1-p_y)
 =[z^{\,k-h}]\prod_{x\in A^c}(1-p_x+p_xz).
\]

Write $\mathsf B_A:=(\mathsf B_x)_{x\in A}$.  For a marked configuration
$\boldsymbol\xi=(\xi_x)_{x\in A}\in\{0,1\}^A$, write
$|\boldsymbol\xi|:=\sum_{x\in A}\xi_x$ for its Hamming weight.  
Independence of the $\mathsf B_x$ under $\Pp$ then gives the exact marked marginal:
\begin{equation}
\begin{aligned}
 \Pp(\mathsf B_A=\boldsymbol\xi\mid \mathsf N_{\mathcal I}=k)
 &=
 \frac{\Pp(\mathsf B_A=\boldsymbol\xi, ~\mathsf N_{\mathcal I}=k)}{\Pp(\mathsf N_{\mathcal I}=k)}
 =
 \frac{\Pp(\mathsf B_A=\boldsymbol\xi,~\mathsf N_{A^c} = k-|\boldsymbol\xi|)}{\Pp(\mathsf N_{\mathcal I}=k)}\\
 &=
 \frac{\Pp(\mathsf B_A=\boldsymbol\xi) \Pp(\mathsf N_{A^c}=k-|\boldsymbol\xi|)}{\Pp(\mathsf N_{\mathcal I}=k)}
 =\frac{1}{\Pp(\mathsf N_{\mathcal I}=k)}
 \left(\prod_{x\in A}p_x^{\xi_x}(1-p_x)^{1-\xi_x}\right)
 q_A(|\boldsymbol\xi|).
 \end{aligned}
 \label{prof:canred:eq:marked-marginal}
\end{equation}
In the final expression in \eqref{prof:canred:eq:marked-marginal}, the product
$
 \prod_{x\in A}p_x^{\xi_x}(1-p_x)^{1-\xi_x}
$
is the original product--Bernoulli probability of the marked configuration
$\boldsymbol\xi$, while the normalizing factor $\Pp(\mathsf N_{\mathcal I}=k)^{-1}$ is independent
of $\boldsymbol\xi$.  Hence the only factor that is not coordinatewise multiplicative
in the marked variables is $q_A(|\boldsymbol\xi|)$.  Because $q_A(|\boldsymbol\xi|)$ depends on the
marked configuration only through its total occupation $|\boldsymbol\xi|$, it records
exactly how conditioning on $\mathsf N_{\mathcal I}=k$ couples the marked coordinates.
The remainder of the subsection quantifies this coupling by
controlling finite differences of $h\mapsto \log q_A(h)$.

Throughout the proofs below, constants may depend on $M$ and $\underline p$, but not
on $n$, the marked set $A$, or the individual parameters $p_x$.
The saddle-point, local-mass, finite-difference, and Gibbs arguments are used
only for $n\ge n_0(M,\underline p)$, where the relevant deleted point masses are
positive.  The finitely many smaller values of $n$ are handled only at the
level of the final conditional cumulant estimate in
\cref{prof:canred:lem:hypergraph}, where enlarging the constant is legitimate.

For $t\in\mathbb R$ and $x\in A^c$, define the common exponential tilt
\begin{align}
p_x(t):=\frac{p_xe^t}{1-p_x+p_xe^t}.
\label{prof:eq:exp-tilt}
\end{align}
Let $\Pp_t$ be the product law under which the variables
$(\mathsf B_x)_{x\in A^c}$ are independent Bernoulli random variables with
parameters $(p_x(t))_{x\in A^c}$, and write $\E_t$ and $\operatorname{Var}_t$, respectively, for expectation and variance under
$\Pp_t$.  Set
\[
\mathfrak k(t):=\log \E[e^{t \mathsf N_{A^c}}]=\sum_{x\in A^c}\log(1-p_x+p_xe^t),
 \qquad
 \mathfrak m(t):=\mathfrak k'(t),
 \qquad
 \mathfrak v(t):=\mathfrak k''(t).
\]
Then $\mathfrak m(t)$ and $\mathfrak v(t)$ are respectively the mean and
variance of the unmarked count $\mathsf N_{A^c}$ under $\Pp_t$:
\[
\mathfrak m(t)=\E_t[\mathsf N_{A^c}]=\sum_{x\in A^c}p_x(t),
\qquad
\mathfrak v(t)=\operatorname{Var}_t(\mathsf N_{A^c})
=\sum_{x\in A^c}p_x(t)(1-p_x(t)).
\]

To estimate $q_A(h)$ uniformly for the finitely many values
$0\le h\le d$, we reweight the unmarked Bernoulli variables so that the
requested count $k-h$ is centered at the tilted mean.
For real $h\in[0,d]$, define the \emph{saddle action}
\[
 \mathfrak A_h(t):=\mathfrak k(t)-t(k-h) = \sum_{x\in A^c}\log(1-p_x+p_xe^t) - t(k-h).
\]
For integer $h$, the coefficient--product representation of $q_A(h)$ makes
the origin of this exponent explicit.  Put
\[
 G_A(z):=\prod_{x\in A^c}(1-p_x+p_xz),
 \qquad
 q_A(h)=[z^{\,k-h}]G_A(z).
\]
Cauchy's coefficient formula on the circle $|z|=e^t$, parametrized by
$z=e^{t+i\theta}$, gives
\begin{align*}
 q_A(h)
 &=\frac1{2\pi}\int_{-\pi}^{\pi}
   G_A(e^{t+i\theta})e^{-(k-h)(t+i\theta)}\,\dd\theta\\
 &=e^{\mathfrak A_h(t)}
   \frac1{2\pi}\int_{-\pi}^{\pi}
   e^{-i(k-h)\theta}
   \prod_{x\in A^c}
   \frac{1-p_x+p_xe^{t+i\theta}}
        {1-p_x+p_xe^t}\,\dd\theta.
\end{align*}
The normalized one-site factor satisfies
\[
 \frac{1-p_x+p_xe^{t+i\theta}}{1-p_x+p_xe^t}
 =1-p_x(t)+p_x(t)e^{i\theta}=\E_t[e^{i\theta \mathsf B_x}],
\]
so the product in the angular integral is the characteristic function of
$\mathsf N_{A^c}$ under $\Pp_t$.  Fourier inversion therefore gives, for
every integer $h\in\{0,\ldots,d\}$,
\begin{equation}
 q_A(h)=
 e^{\mathfrak A_h(t)} \frac{1}{2\pi} \int_{-\pi}^{\pi}\, \mathbb{E}_t[e^{i\theta(\mathsf N_{A^c}-(k-h))}]\,\dd\theta
 =e^{\mathfrak A_h(t)}
 \Pp_t(\mathsf N_{A^c}=k-h).
 \label{prof:canred:eq:coefficient-tilt-identity}
\end{equation}
We choose the \emph{centering parameter} $t_{\mathrm{cen}}(h)$ (equivalently, the
saddle-point parameter in the coefficient integral) by the stationary
condition
\[
 \mathfrak A_h'(t_{\mathrm{cen}}(h))=0
 \qquad\Longleftrightarrow\qquad
 \mathfrak m(t_{\mathrm{cen}}(h))=k-h.
\]
Thus, for integer $h$, choosing $t=t_{\mathrm{cen}}(h)$ makes the target count in
\eqref{prof:canred:eq:coefficient-tilt-identity} equal to its mean:
\[
 \E_{t_{\mathrm{cen}}(h)}[\mathsf N_{A^c}]=k-h.
\]
The same equation for real $h$ gives the smooth interpolation needed to
turn estimates on $q_A(0),\ldots,q_A(d)$ into derivative estimates in $h$.
To differentiate the deleted-sum ratios with respect to $h$, we first
control the centering parameter $t_{\mathrm{cen}}(h)$ and its derivatives:
it exists uniformly, stays of order $n^{-1}$, and each additional
$h$-derivative gains another factor $n^{-1}$.

\begin{lemma}[Existence and regularity of the centering parameter]
\label{prof:canred:lem:saddle}
There are $n_0=n_0(M,\underline p)$, $t_0=t_0(\underline p)>0$, and constants
$C_r=C_{M,\underline p,r}$ such that the following holds for $n\ge n_0$.
For every marked set $A$ with $|A|\le M$ and every real
$h\in[0,|A|]$, there is a unique centering parameter
$t_{\mathrm{cen}}(h)\in[-t_0,t_0]$ satisfying
\begin{equation}
 \mathfrak m(t_{\mathrm{cen}}(h))=k-h.
 \label{prof:canred:eq:saddle-eq}
\end{equation}
Moreover,
\begin{equation}
 \mathfrak v(t)\asymp_{\underline p}n
 \qquad (|t|\le t_0),
 \label{prof:canred:eq:deleted-var-comp}
\end{equation}
and, for every fixed $1\le r\le M+1$,
\begin{equation}
 \sup_{0\le h\le|A|}\abs{t_{\mathrm{cen}}^{(r)}(h)}
 \le C_r n^{-r}.
 \label{prof:canred:eq:centering-derivatives}
\end{equation}
In particular, $\sup_{0\le h\le |A|}|t_{\mathrm{cen}}(h)|\le C_{M,\underline p}n^{-1}$.
\end{lemma}

\begin{proof}
Choose $t_0>0$, depending only on $\underline p$, so small that the exponential tilt \eqref{prof:eq:exp-tilt} satisfies
\[
 \underline p_{\mathrm{tilt}}\le p_x(t)\le1-\underline p_{\mathrm{tilt}}
 \qquad (x\in A^c,\ |t|\le t_0)
\]
for some $\underline p_{\mathrm{tilt}}=\underline p_{\mathrm{tilt}}(\underline p)\in(0,1/2)$.  Hence
$
 \underline p_{\mathrm{tilt}}(1-\underline p_{\mathrm{tilt}})|A^c|\le \mathfrak v(t)\le |A^c|/4
 $.
For $n\ge2M$, $|A^c|=n-|A|\asymp_Mn$, which proves
\eqref{prof:canred:eq:deleted-var-comp}.

At $t=0$,
$
 \mathfrak m(0)=\sum_{x\in A^c}p_x
 =k-\sum_{x\in A}p_x$,
so
$
 |(k-h)-\mathfrak m(0)|
 =\abs{\sum_{x\in A}p_x-h}\le  M$.
Since $\mathfrak m'(t)=\mathfrak v(t)\ge c n$ on $[-t_0,t_0]$, the image
$\mathfrak m([-t_0,t_0])$ contains the interval
$[\mathfrak m(0)-c\,n\,t_0,\mathfrak m(0)+c\,n\,t_0]$.  For $n$ large enough this contains every
$k-h$, $0\le h\le|A|$.  
Strict monotonicity of $\mathfrak m$ proves uniqueness of $t_{\mathrm{cen}}(h)$.  

Since $\mathfrak m$ is
real-analytic, and $\mathfrak m'=\mathfrak v$ is bounded away from zero on
$[-t_0,t_0]$, the inverse function theorem implies that
$t_{\mathrm{cen}}=\mathfrak m^{-1}(k-\cdot)$ is real-analytic, hence $C^\infty$, on
$[0,|A|]$.
This justifies the repeated differentiations of $t_{\mathrm{cen}}(h)$ carried out below, after we estimate the magnitude of $t_{\mathrm{cen}}(h)$.

Indeed, for some $\xi_h$ between $0$ and $t_{\mathrm{cen}}(h)$, the mean value theorem gives
\[
 \abs{(k-h)-\mathfrak m(0)}
 =\abs{\mathfrak m(t_{\mathrm{cen}}(h))-\mathfrak m(0)}
 =\mathfrak v(\xi_h)|t_{\mathrm{cen}}(h)|.
\]
The left-hand side is at most $M$, while
$\mathfrak v(\xi_h)\ge c n$ by
\eqref{prof:canred:eq:deleted-var-comp}.  Hence
\[
 |t_{\mathrm{cen}}(h)|\le \frac{M}{cn}\le C_{M,\underline p}n^{-1}.
\]

We next differentiate the implicit equation
\eqref{prof:canred:eq:saddle-eq} repeatedly. 
 At first order,
\begin{align}
\mathfrak v(t_{\mathrm{cen}}(h))t_{\mathrm{cen}}'(h)=-1
 \quad \Longrightarrow \quad
 t_{\mathrm{cen}}'(h)= -\mathfrak v(t_{\mathrm{cen}}(h))^{-1},
\end{align}
so $|t_{\mathrm{cen}}'|\le Cn^{-1}$ again by
\eqref{prof:canred:eq:deleted-var-comp}.  
For every fixed $q\ge2$, $\mathfrak k^{(q)}(t)$ is the order-$q$ cumulant of
$\mathsf N_{A^c}$ under the product law $\Pp_t$, hence the sum of
$|A^c|=O(n)$ Bernoulli cumulants of order $q$.  Since a fixed-order
Bernoulli cumulant is a bounded polynomial of its parameter, uniformly for
$|t|\le t_0$,
\begin{equation}
 \sup_{|t|\le t_0}|\mathfrak k^{(q)}(t)|\le C_q n
 \qquad(q\ge2\text{ fixed}).
 \label{prof:canred:eq:deleted-logpart-derivatives}
\end{equation}

We now prove \eqref{prof:canred:eq:centering-derivatives} by strong induction on the
order of differentiation.  At each order, the highest derivative
$t_{\mathrm{cen}}^{(r)}$ appears linearly with coefficient
$\mathfrak k''(t_{\mathrm{cen}})=\mathfrak v(t_{\mathrm{cen}})$; all remaining terms involve only
lower derivatives of $t_{\mathrm{cen}}$.
Let $r\ge2$ and differentiate the equation
$\mathfrak k'(t_{\mathrm{cen}}(h))=k-h$ w.r.t.\@ $h$ a total of $r$ times.  The unique term containing
$t_{\mathrm{cen}}^{(r)}$ is
\[
 \mathfrak k''(t_{\mathrm{cen}})t_{\mathrm{cen}}^{(r)}
 =\mathfrak v(t_{\mathrm{cen}})t_{\mathrm{cen}}^{(r)}.
\]
Every other term contains one factor $\mathfrak k^{(q)}(t_{\mathrm{cen}})$ with
$q\ge3$ and a product
$t_{\mathrm{cen}}^{(a_1)}\cdots t_{\mathrm{cen}}^{(a_\ell)}$, where
$a_1+\cdots+a_\ell=r$ and every $a_i<r$.  Assume inductively that
$|t_{\mathrm{cen}}^{(a)}|\le C_an^{-a}$ for all $a<r$.  The product of the lower-order
$t_{\mathrm{cen}}$-derivatives is then $O(n^{-r})$, while
\eqref{prof:canred:eq:deleted-logpart-derivatives} contributes at most
$O(n)$.  Hence the sum of all terms not containing $t_{\mathrm{cen}}^{(r)}$ is
$O_r(n^{1-r})$.  Solving for $t_{\mathrm{cen}}^{(r)}$ and using
$\mathfrak v(t_{\mathrm{cen}})\asymp n$ from
\eqref{prof:canred:eq:deleted-var-comp} gives
\[
 |t_{\mathrm{cen}}^{(r)}|
 \le \frac{C_r n^{1-r}}{c n}
 \le C_r n^{-r}.
\]
This closes the induction and proves
\eqref{prof:canred:eq:centering-derivatives}.
\end{proof}

By \cref{prof:canred:lem:saddle}, the centering map
$h\mapsto t_{\mathrm{cen}}(h)$ satisfies \eqref{prof:canred:eq:saddle-eq}, with its
fixed-order derivatives controlled by \eqref{prof:canred:eq:centering-derivatives}.
For integer $h$, set $t=t_{\mathrm{cen}}(h)$ in
\eqref{prof:canred:eq:coefficient-tilt-identity}.  By the centering equation
\eqref{prof:canred:eq:saddle-eq}, $k-h=\mathfrak m(t_{\mathrm{cen}}(h))$, so the
probability factor becomes
$
 \Pp_{t_{\mathrm{cen}}(h)}\bigl(\mathsf N_{A^c}=\mathfrak m(t_{\mathrm{cen}}(h))\bigr)
 $.
Thus we need to estimate the probability that, under the tilted law
$\Pp_{t_{\mathrm{cen}}(h)}$, the Poisson--binomial count $\mathsf N_{A^c}$ takes its
mean.  Because the tilt parameter will later vary with $h$ through
$t=t_{\mathrm{cen}}(h)$, we need this local estimate uniformly in $t$, together with
control of a fixed number of $t$-derivatives.  

Under the tilted product law $\Pp_t$, set
\begin{equation}
 Z_t:=\frac{\mathsf N_{A^c}-\mathfrak m(t)}{\sqrt{\mathfrak v(t)}},
 \qquad
 \psi_t(y):=\E_t[e^{iyZ_t}],\quad y\in\mathbb R.
\end{equation}
Thus $\psi_t$ is the characteristic function of the centered, variance-one
count $Z_t$.  Define
\begin{equation}
 \mathfrak h(t):=\frac1{\sqrt{2\pi}}
 \int_{-\pi\sqrt{\mathfrak v(t)}}^{\pi\sqrt{\mathfrak v(t)}}
 \psi_t(y)\,\dd y.
 \label{prof:canred:eq:central-factor-def}
\end{equation}
If $\mathfrak m(t)$ is an integer, ordinary Fourier inversion followed by
$y=\sqrt{\mathfrak v(t)}\,\theta$ gives
\begin{equation}
 \Pp_t(\mathsf N_{A^c}=\mathfrak m(t))
 =\frac{\mathfrak h(t)}{\sqrt{2\pi \mathfrak v(t)}}.
 \label{prof:canred:eq:central-factor-prob}
\end{equation}
When $\mathfrak m(t)$ is noninteger, $\mathfrak h(t)$ no longer represents
a lattice point probability, but \eqref{prof:canred:eq:central-factor-def}
still supplies a smooth real-valued interpolation in $t$.  At integer $h$,
\eqref{prof:canred:eq:central-factor-prob} evaluated at $t=t_{\mathrm{cen}}(h)$ gives
\[
 \Pp_{t_{\mathrm{cen}}(h)}\bigl(\mathsf N_{A^c}=k-h\bigr)
 =\frac{\mathfrak h(t_{\mathrm{cen}}(h))}
 {\sqrt{2\pi\mathfrak v(t_{\mathrm{cen}}(h))}}.
\]
This is the form used in \cref{prof:canred:lem:finite-differences} to build a
smooth interpolation of $\log q_A(h)$.  In particular,
$\log\mathfrak h(t_{\mathrm{cen}}(h))$ must be differentiated with respect to $h$ there;
Fa\`a di Bruno then combines the $t$-derivative bounds in
\cref{prof:canred:lem:central} with
\eqref{prof:canred:eq:centering-derivatives}.

The normalization in \eqref{prof:canred:eq:central-factor-def} is chosen so
that the standard normal characteristic function $e^{-y^2/2}$ would give
$\mathfrak h(t)=1$ if integrated over the whole real line.  Thus the local
Gaussian limit takes the form $\mathfrak h(t)\approx1$.  The next lemma proves
this comparison together with the required derivative bounds.

\begin{lemma}[Differentiated central local estimate]
\label{prof:canred:lem:central}
Under the assumptions of \cref{prof:canred:lem:saddle}, after decreasing $t_0$ if
necessary,
\begin{equation}
 \sup_{|t|\le t_0}|\mathfrak h(t)-1|\le C_{M,\underline p}n^{-1/2},
 \label{prof:canred:eq:central-factor-close}
\end{equation}
and, for every fixed $1\le r\le M+1$,
\begin{equation}
 \sup_{|t|\le t_0}|\mathfrak h^{(r)}(t)|\le C_{M,\underline p,r}.
 \label{prof:canred:eq:central-factor-derivative}
\end{equation}
Consequently, for all $n$ sufficiently large, we have $\mathfrak h(t)\in[\frac12,\frac32]$ for $|t|\le t_0$, and
\begin{equation}
 \sup_{|t|\le t_0}
 \left|\frac{\dd^r}{\dd t^r}\log \mathfrak h(t)\right|
 \le C_{M,\underline p,r}
 \qquad(0\le r\le M+1).
 \label{prof:canred:eq:log-central-factor-derivative}
\end{equation}
\end{lemma}

\begin{proof}
\localheading{Estimates of standardized cumulants.}
For $j\ge2$, set
\[
 a_j(t):=\mathfrak k^{(j)}(t)\,\mathfrak v(t)^{-j/2}.
\]
Since $\psi_t$ is the characteristic function of $Z_t$, $a_j(t)$ is exactly
the order-$j$ cumulant of $Z_t$; in particular $a_2\equiv1$.  Equivalently,
because $\psi_t(0)=1$, there is an analytic branch of $\log\psi_t$ near the
origin with $\log\psi_t(0)=0$, and
\[
 \left.\frac{\dd^j}{\dd y^j}\log\psi_t(y)\right|_{y=0}
 =i^j a_j(t),\qquad j\ge2.
\]
Thus the variance normalization makes the quadratic coefficient equal to
$-\frac12$, while the $a_j$, $j\ge3$, are the higher-order coefficients measuring
the departure from the standard normal characteristic function
$e^{-y^2/2}$.

Let us estimate the sizes of $a_j(t)$ and their $t$-derivatives.  By
\eqref{prof:canred:eq:deleted-logpart-derivatives}, for every fixed $\ell\ge0$,
\[
 \sup_{|t|\le t_0}|\partial_t^\ell\mathfrak k^{(j)}(t)|\le C_{j,\ell}n,
 \qquad
 \sup_{|t|\le t_0}|\partial_t^\ell\mathfrak v(t)|\le C_\ell n,
\]
while \eqref{prof:canred:eq:deleted-var-comp} gives $\mathfrak v(t)\asymp n$.
The one-variable Fa\`a di Bruno formula applied to
$x\mapsto x^{-\alpha}$ therefore gives, for every fixed $\ell$ and $\alpha>0$,
\begin{equation}
 \sup_{|t|\le t_0}
 \left|\partial_t^\ell\bigl(\mathfrak v(t)^{-\alpha}\bigr)\right|
 \le C_{\ell,\alpha}n^{-\alpha}.
 \label{prof:canred:eq:inverse-variance-derivatives}
\end{equation}
Indeed, in every Fa\`a di Bruno term, the negative power of $\mathfrak v$
contributes exactly enough factors of $n^{-1}$ to compensate for the
$O(n)$ derivatives of $\mathfrak v$.  
Leibniz' rule applied to
$a_j=\mathfrak k^{(j)}\mathfrak v^{-j/2}$ then yields
\begin{equation}
 \sup_{|t|\le t_0}|\partial_t^r a_j(t)|
 \le C_{j,r}n^{1-j/2}
 \qquad(j\ge3,\ r\le M+1).
 \label{prof:canred:eq:standardized-cumulants}
\end{equation}
Thus the cubic standardized cumulant $a_3(t)$ is $O(n^{-1/2})$, and every higher
standardized cumulant has a smaller order in $n$.

\localheading{Central/tail Fourier split.}
Set
\begin{equation}
 y_{\max}(t):=\pi\sqrt{\mathfrak v(t)},
 \qquad
 b_n:=C_{\mathrm{tail}}\sqrt{\log n},
 \label{prof:canred:eq:fourier-scales}
\end{equation}
where $C_{\mathrm{tail}}=C_{\mathrm{tail}}(M,\underline p)$ will be chosen below.  We
call $|y|\le b_n$ the \emph{central interval}.
Since \eqref{prof:canred:eq:deleted-var-comp} gives $\mathfrak v(t)\asymp n$,
\eqref{prof:canred:eq:fourier-scales} implies $b_n<y_{\max}(t)$ uniformly for all sufficiently large $n$.
Using \eqref{prof:canred:eq:central-factor-def} and the standard
Gaussian integral identity gives the exact decomposition
\begin{equation}
\begin{aligned}
 \mathfrak h(t)-1
 =\frac1{\sqrt{2\pi}}
 \int_{|y|\le b_n}\bigl(\psi_t(y)-e^{-y^2/2}\bigr)\,\dd y
 +\frac1{\sqrt{2\pi}}
 \int_{b_n<|y|\le y_{\max}(t)}\psi_t(y)\,\dd y
 -\frac1{\sqrt{2\pi}}\int_{|y|>b_n}e^{-y^2/2}\,\dd y.
\end{aligned}
 \label{prof:canred:eq:central-tail-decomposition}
\end{equation}
Thus \eqref{prof:canred:eq:central-tail-decomposition} reduces the proof of the lemma to
three tasks: compare $\psi_t$ with the standard Gaussian
characteristic function on the central interval; control the complementary
Fourier integral; and control the ordinary Gaussian tail.  
For the derivatives of $\mathfrak h(t)$, the only additional issue is that
the endpoints $\pm y_{\max}(t)$ of the second integral in
\eqref{prof:canred:eq:central-tail-decomposition} also depend on $t$; this is
handled after the complementary-region estimate.

\localheading{Central Gaussian approximation.}
The product structure of the tilted law gives, with
$\theta=y/\sqrt{\mathfrak v(t)}$,
\begin{align}
 \psi_t(y)
 =e^{-iy\mathfrak m(t)/\sqrt{\mathfrak v(t)}}
  \prod_{x\in A^c}
  \left(1-p_x(t)+p_x(t)e^{iy/\sqrt{\mathfrak v(t)}}\right),
  \label{prof:eq:product-psi_t}
\end{align}
where the factors indexed by $x$ are the individual Bernoulli characteristic
functions under $\Pp_t$.  On the central interval,
$|\theta|\le b_n/\sqrt{\mathfrak v(t)}=O(\sqrt{\log n/n})$, and
$p_x(t)\in[\underline p_{\mathrm{tilt}},1-\underline p_{\mathrm{tilt}}]$.  Since
$|1-p_x(t)+p_x(t)e^{i\theta}-1|\le|e^{i\theta}-1|=o(1)$ uniformly on the
central interval, every site factor lies in the disk $\{z\in\mathbb C:|z-1|<\frac12\}$ for all
sufficiently large $n$.  The principal logarithm is analytic on this disk and
vanishes at $1$.  Using this branch for each site factor, the principal logarithm of the
standardized characteristic function is
\[
\begin{aligned}
 \log\psi_t(y)
 &:=-i\frac{y\mathfrak m(t)}{\sqrt{\mathfrak v(t)}}
   +\sum_{x\in A^c}
    \log\!\left(1-p_x(t)+p_x(t)e^{iy/\sqrt{\mathfrak v(t)}}\right)\\
 &=\mathfrak k\!\left(t+i\frac{y}{\sqrt{\mathfrak v(t)}}\right)-\mathfrak k(t)
   -i\frac{y}{\sqrt{\mathfrak v(t)}}\mathfrak m(t),
\end{aligned}
\]
where the second line uses the corresponding analytic extension of
$\mathfrak k$.  In particular, $\log\psi_t(0)=0$ on the central interval.

Let $J_{\mathrm{Tay}}:=M+4$ and put
$z_t(y):=iy/\sqrt{\mathfrak v(t)}$.  Taylor's formula with integral remainder
gives, uniformly on the central interval,
\begin{equation}
 \log\psi_t(y)
 =-\frac{y^2}{2}
  +\sum_{j=3}^{J_{\mathrm{Tay}}}\frac{i^j}{j!}a_j(t)y^j
  +\mathcal R_{J_{\mathrm{Tay}},t}(y).
 \label{prof:canred:eq:standardized-log-expansion}
\end{equation}
Here the $j=0$ and $j=1$ Taylor terms cancel the subtracted
$\mathfrak k(t)$ and $z_t(y)\mathfrak m(t)$ terms in $\log\psi_t(y)$, while the
$j=2$ term equals $-y^2/2$ because $\mathfrak k''(t)=\mathfrak v(t)$.
The remainder $\mathcal R_{J_{\mathrm{Tay}},t}$ in \eqref{prof:canred:eq:standardized-log-expansion} is the
integral remainder in Taylor's formula for
$u\mapsto\mathfrak k(t+u)$ at $u=0$, evaluated at $u=z_t(y)$.  More
generally, for every integer $J\ge2$, define
\begin{equation}
\begin{aligned}
 \mathcal R_{J,t}(y)
 &:=\frac{z_t(y)^{J+1}}{J!}
   \int_0^1(1-\lambda)^J
   \mathfrak k^{(J+1)}\!\bigl(t+\lambda z_t(y)\bigr)\,\dd\lambda.
\end{aligned}
 \label{prof:canred:eq:taylor-remainder-def}
\end{equation}

We claim that
\begin{equation}
\begin{aligned}
 |\partial_t^r\mathcal R_{J_{\mathrm{Tay}},t}(y)|
 &\le C_r n^{-(J_{\mathrm{Tay}}-1)/2}|y|^{J_{\mathrm{Tay}}+1},
 \qquad 0\le r\le M+1,
 \quad |y|\le b_n.
\end{aligned}
 \label{prof:canred:eq:taylor-remainder-bound}
\end{equation}
To see this we separately estimate the factor $z_t(y)^{J+1}$ and the integral in
\eqref{prof:canred:eq:taylor-remainder-def}.  First, since
$z_t(y)=iy\,\mathfrak v(t)^{-1/2}$,
\eqref{prof:canred:eq:inverse-variance-derivatives} gives, for every fixed
$\ell\ge1$,
\[
 |\partial_t^\ell z_t(y)|\le C_\ell|z_t(y)|.
\]
Repeated Leibniz differentiation of the product of $J+1$ copies of $z_t(y)$
then yields
\[
 \bigl|\partial_t^\ell z_t(y)^{J+1}\bigr|
 \le C_{J,\ell}|z_t(y)|^{J+1}.
\]
For the integral factor, we first note that the single-site function
$u\mapsto\log(1-p_x+p_xe^u)$ admits analytic continuation with uniformly
bounded fixed derivatives in a fixed complex neighborhood of
$[-t_0,t_0]$.  Since $|z_t(y)|=O(b_n/\sqrt n)=o(1)$ on the central interval,
the points $t+\lambda z_t(y)$ remain in this neighborhood for all $\lambda\in[0,1]$ and
all sufficiently large $n$.  Summing the single-site bounds then gives
\[
 \bigl|\mathfrak k^{(m)}(t+\lambda z_t(y))\bigr|\le C_m n
\]
uniformly in $\lambda\in[0,1]$ and $|y|\le b_n$, for every fixed $m$.  
Now the map
$t\mapsto t+\lambda z_t(y)$ has first derivative $O(1)$ and every higher fixed
derivative $O_\ell(|z_t(y)|)$.  Fa\`a di Bruno's formula therefore shows
that, for every fixed $\ell$,
\[
 \left|
 \partial_t^\ell\mathfrak k^{(J+1)}\!\bigl(t+\lambda z_t(y)\bigr)
 \right|
 \le C_{J,\ell}n.
\]
Applying Leibniz' rule to \eqref{prof:canred:eq:taylor-remainder-def}, we can bound each
term in $\partial_t^r\mathcal R_{J,t}(y)$ by
$C_{J,r}n|z_t(y)|^{J+1}$.
Now
\[
 n|z_t(y)|^{J+1}
 =n\frac{|y|^{J+1}}{\mathfrak v(t)^{(J+1)/2}}
 \lesssim n^{-(J-1)/2}|y|^{J+1},
\]
where \eqref{prof:canred:eq:deleted-var-comp} was invoked in the last inequality.  Setting $J=J_{\mathrm{Tay}}$
proves \eqref{prof:canred:eq:taylor-remainder-bound}.

Let us isolate the cubic and higher-order terms by writing
\[
 \mathcal C_t(y):=
 \sum_{j=3}^{J_{\mathrm{Tay}}}\frac{i^j}{j!}a_j(t)y^j
 +\mathcal R_{J_{\mathrm{Tay}},t}(y),
\]
so that $\log\psi_t(y)=-y^2/2+\mathcal C_t(y)$.  Since
$b_n=o(\sqrt n)$, \eqref{prof:canred:eq:standardized-cumulants} and
\eqref{prof:canred:eq:taylor-remainder-bound} imply, uniformly for
$0\le r\le M+1$,
\begin{equation}
 |\partial_t^r\mathcal C_t(y)|
 \le C_r n^{-1/2}|y|^3,
 \qquad |y|\le b_n.
 \label{prof:canred:eq:central-exponent-correction}
\end{equation}
Indeed, every term of degree $j\ge4$ is bounded by the cubic scale because
$(|y|/\sqrt n)^{j-3}\le1$ on the central interval, and the same comparison
applies to the Taylor remainder. 
As consequences of \eqref{prof:canred:eq:central-exponent-correction} we have
\begin{align}
 \sup_{\substack{|t|\le t_0\\0<|y|\le b_n}}
 \frac{|\mathcal C_t(y)|}{y^2}
 &\le C\frac{b_n}{\sqrt n}\longrightarrow0,
 \label{prof:canred:eq:central-exponent-relative}
 \end{align}
 and
 \begin{align}
 \Re\log\psi_t(y)&\le-cy^2,
 \qquad |t|\le t_0,\quad |y|\le b_n.
 \label{prof:canred:eq:central-log-realpart}
\end{align}
for all sufficiently large $n$.

We compare $\psi_t(y)$ with the standard Gaussian
characteristic function.  Since
\[
 \psi_t(y)-e^{-y^2/2}
 =e^{-y^2/2}\left(e^{\mathcal C_t(y)}-1\right),
\]
the elementary inequality $|e^z-1|\le |z|e^{|z|}$ reduces the problem to
controlling $\mathcal C_t(y)$.  By
\eqref{prof:canred:eq:central-exponent-relative}, for all sufficiently large
$n$ we have $|\mathcal C_t(y)|\le y^2/4$ uniformly on $|y|\le b_n$.
Together with the case $r=0$ of
\eqref{prof:canred:eq:central-exponent-correction}, this yields
\[
 |\psi_t(y)-e^{-y^2/2}|
 \le e^{-y^2/2}|\mathcal C_t(y)|e^{|\mathcal C_t(y)|}
 \le Cn^{-1/2}|y|^3e^{-y^2/4}.
\]
In particular,
\begin{align}
 |\psi_t(y)-e^{-y^2/2}|
 &\le Cn^{-1/2}(1+|y|^3)e^{-c y^2},
 \label{prof:canred:eq:central-zero}
\end{align}
uniformly on $|y|\le b_n$.

We next estimate the positive $t$-derivatives.  Because the Gaussian term
$-y^2/2$ is independent of $t$, the one-variable Fa\`a di Bruno formula gives
\[
 \partial_t^r \psi_t(y)
 =\psi_t(y)
  \sum_{\mathfrak p\in\operatorname{Part}([r])}
  \prod_{B\in\mathfrak p}\partial_t^{|B|}\mathcal C_t(y),
 \qquad r\ge1.
\]
Fix a partition $\mathfrak p$ with $m=|\mathfrak p|$ blocks.  Applying
\eqref{prof:canred:eq:central-exponent-correction} to each block gives
\[
 \prod_{B\in\mathfrak p}
 \left|\partial_t^{|B|}\mathcal C_t(y)\right|
 \le C_r n^{-m/2}|y|^{3m}
 \le C_r n^{-1/2}(1+|y|^{3r}),
\]
since $1\le m\le r$.  On the other hand,
\eqref{prof:canred:eq:central-log-realpart} gives
$|\psi_t(y)|\le e^{-cy^2}$.  Summing over the finitely many partitions of
$[r]$ then yields
\begin{align}
 |\partial_t^r \psi_t(y)|
 &\le C_rn^{-1/2}(1+|y|^{3r})e^{-c y^2},
 \qquad 1\le r\le M+1.
 \label{prof:canred:eq:central-derivatives}
\end{align}

Since a polynomial times $e^{-cy^2}$ is integrable on $\mathbb R$,
\eqref{prof:canred:eq:central-zero} gives
\[
 \int_{|y|\le b_n}
 \bigl|\psi_t(y)-e^{-y^2/2}\bigr|\,\dd y
 \le Cn^{-1/2}.
\]
Likewise, for each fixed $1\le r\le M+1$,
\eqref{prof:canred:eq:central-derivatives} gives
\[
 \int_{|y|\le b_n}
 \bigl|\partial_t^r\psi_t(y)\bigr|\,\dd y
 \le C_r n^{-1/2}.
\]
These bounds control, respectively, the zeroth-order error in the first
integral of \eqref{prof:canred:eq:central-tail-decomposition} and its positive
$t$-derivatives.

\localheading{Complementary Fourier region: Gaussian-type decay.}
We next control the second and third integrals in
\eqref{prof:canred:eq:central-tail-decomposition}.  For
$p\in[\underline p_{\mathrm{tilt}},1-\underline p_{\mathrm{tilt}}]$ and $|\theta|\le\pi$,
\begin{align}
 |1-p+pe^{i\theta}|^2
 &=1-4p(1-p)\sin^2(\theta/2)
 \le e^{-c\theta^2},
 \label{prof:canred:eq:one-site-fourier-decay}
\end{align}
Applying \eqref{prof:canred:eq:one-site-fourier-decay} to the factors in
\eqref{prof:eq:product-psi_t}, with
$\theta=y/\sqrt{\mathfrak v(t)}$, and using
\eqref{prof:canred:eq:deleted-var-comp}, gives
\begin{equation}
 |\psi_t(y)|\le e^{-cy^2},\qquad |y|\le y_{\max}(t).
 \label{prof:canred:eq:product-fourier-decay}
\end{equation}

For the derivative estimates, abbreviate
$f_x(t,\theta):=1-p_x(t)+p_x(t)e^{i\theta}$ and differentiate the product in
\eqref{prof:eq:product-psi_t}.  For fixed $1\le r\le M+1$, each
Leibniz/Fa\`a di Bruno term in $\partial_t^r\psi_t(y)$ differentiates at most
$r$ site factors.  The remaining $n-O_r(1)$ factors are untouched, so
\eqref{prof:canred:eq:one-site-fourier-decay}, together with
\eqref{prof:canred:eq:deleted-var-comp}, still contributes an
$e^{-cy^2}$ factor.  The choices of differentiated sites, derivatives of the centering phase
$e^{-iy\mathfrak m(t)/\sqrt{\mathfrak v(t)}}$, derivatives of the argument
$y/\sqrt{\mathfrak v(t)}$, and derivatives of the finitely many affected
$f_x$ contribute only a fixed polynomial in $n$ and $1+|y|$.  Thus, for each
fixed $1\le r\le M+1$, there is an integer $K_r\ge0$ such that
\[
 |\partial_t^r\psi_t(y)|
 \le C_r n^{K_r}(1+|y|)^{K_r}e^{-cy^2},
 \qquad
 b_n\le|y|\le y_{\max}(t).
\]
For $y\ge b_n$, the polynomial factor can be absorbed into a weaker
Gaussian: for each fixed $K_r$,
$(1+y)^{K_r}e^{-cy^2}\le C_r e^{-(c/2)y^2}$.  Hence the complementary integral
is at most $C_r n^{K_r}e^{-(c/2)b_n^2}$.  Because only the finite range
$0\le r\le M+1$ is needed, we may choose
$C_{\mathrm{tail}}=C_{\mathrm{tail}}(M,\underline p)$ once so that this bound is at
most $n^{-2}$ for every such $r$.  Together with
\eqref{prof:canred:eq:product-fourier-decay}, this gives
\begin{align}
 \sup_{|t|\le t_0}
 \int_{b_n\le|y|\le y_{\max}(t)}
 |\partial_t^r\psi_t(y)|\,\dd y
 &\le n^{-2}
 \qquad(0\le r\le M+1),  \label{prof:canred:eq:complementary-tail-bounds}\\
 \int_{|y|\ge b_n}e^{-y^2/2}\,\dd y
 &\le n^{-2}. \label{prof:eq:tailpiece}
\end{align}

\localheading{Moving endpoints.}
It remains to justify differentiating the second integral in
\eqref{prof:canred:eq:central-tail-decomposition}, whose endpoints are
$\pm y_{\max}(t)$.  Since \eqref{prof:canred:eq:deleted-var-comp} gives
$\mathfrak v(t)\asymp n$ and the fixed derivatives satisfy
$\mathfrak v^{(\ell)}(t)=O_\ell(n)$, Fa\`a di Bruno applied to
$x\mapsto x^{1/2}$ yields
$y_{\max}^{(\ell)}(t)=O_\ell(n^{1/2})$ for every fixed $\ell\ge1$.
For the first derivative, Leibniz' rule gives explicitly
\[
\begin{aligned}
 \frac{\dd}{\dd t}
 \left(
  \int_{b_n}^{y_{\max}(t)}\psi_t(y)\,\dd y
  +\int_{-y_{\max}(t)}^{-b_n}\psi_t(y)\,\dd y
 \right)
 &=
 \int_{b_n<|y|\le y_{\max}(t)}\partial_t\psi_t(y)\,\dd y \\
 &\quad
 +y_{\max}'(t)\bigl[\psi_t(y_{\max}(t))+\psi_t(-y_{\max}(t))\bigr].
\end{aligned}
\]
Repeating this differentiation shows that, for fixed $r\ge1$, one obtains
the integral term
\[
 \int_{b_n<|y|\le y_{\max}(t)}\partial_t^r\psi_t(y)\,\dd y
\]
plus a finite sum of endpoint terms.  Each consists of a product of derivatives
of $y_{\max}(t)$ multiplying
one mixed derivative $\partial_t^a\partial_y^b\psi_t(\pm y_{\max}(t))$ with
$a+b\le r-1$.  
On the one hand, \eqref{prof:canred:eq:complementary-tail-bounds} states that the integral
term is $O_r(n^{-2})$.
On the other hand, for the endpoint terms, \eqref{prof:canred:eq:fourier-scales} gives
$ \frac{\pm y_{\max}(t)}{\sqrt{\mathfrak v(t)}}=\pm\pi$.
Hence by the product formula \eqref{prof:eq:product-psi_t}, every
undifferentiated site factor at $y=\pm y_{\max}(t)$ satisfies
\[
 |f_x(t,\pm\pi)|=|1-2p_x(t)|\le1-2\underline p_{\mathrm{tilt}}<1.
\]
In any fixed mixed derivative only $O_r(1)$ site factors are differentiated,
so the remaining $n-O_r(1)$ factors contribute $e^{-cn}$.  Derivatives of
$y_{\max}(t)$, of the exponential phase
$e^{-iy\mathfrak m(t)/\sqrt{\mathfrak v(t)}}$, and of the finitely many
differentiated site factors contribute at most polynomial factors in $n$.
Therefore every endpoint term is
\begin{equation}
 O(n^Ce^{-cn})=O(n^{-2}).
 \label{prof:canred:eq:moving-endpoint-bound}
\end{equation}

\localheading{Conclusion.}
For $r=0$, \eqref{prof:canred:eq:central-tail-decomposition},
\eqref{prof:canred:eq:central-zero},
\eqref{prof:canred:eq:complementary-tail-bounds}, and
\eqref{prof:eq:tailpiece} together prove
\eqref{prof:canred:eq:central-factor-close}.  For $1\le r\le M+1$, the
central derivative bound \eqref{prof:canred:eq:central-derivatives}, the
complementary estimate \eqref{prof:canred:eq:complementary-tail-bounds}, and
the endpoint bound \eqref{prof:canred:eq:moving-endpoint-bound} give the
stronger estimate $\mathfrak h^{(r)}(t)=O_r(n^{-1/2})$, and hence
\eqref{prof:canred:eq:central-factor-derivative}.  Finally,
\eqref{prof:canred:eq:central-factor-close} gives
$\mathfrak h(t)\in[\frac12,\frac32]$ for all sufficiently large $n$; applying Fa\`a di Bruno to
$\log\mathfrak h(t)$ and using
\eqref{prof:canred:eq:central-factor-derivative} proves
\eqref{prof:canred:eq:log-central-factor-derivative}.
\end{proof}

The next lemma derives finite-difference bounds for $\log q_A$; its $a=0$
case supplies the coefficients $\Delta^m\log q_A(0)$ used in the Gibbs
representation of \cref{prof:canred:lem:gibbs}.

For a function $f$ on consecutive integers, write
\begin{equation*}
 (\Delta f)(h):=f(h+1)-f(h),
 \qquad
 \Delta^0f:=f,
 \qquad
 \Delta^m:=\Delta\circ\Delta^{m-1}\quad(m\ge1).
\end{equation*}

\begin{lemma}[Sharp finite differences of the deleted local mass]
\label{prof:canred:lem:finite-differences}
There is $n_0=n_0(M,\underline p)$ such that, for $n\ge n_0$, uniformly over marked
sets $A$ with $1\le d=|A|\le M$ and integers $a,m$ with
$1\le m\le d$ and $0\le a\le d-m$,
\begin{equation}
 \abs{\Delta^m\log q_A(a)}
 \le C_{M,\underline p}
 \begin{cases}
 n^{-1},&m=1,\\
 n^{1-m},&2\le m\le d.
 \end{cases}
 \label{prof:canred:eq:fd-target}
\end{equation}
\end{lemma}

\begin{proof}
For real $h\in[0,d]$, let $t_{\mathrm{cen}}(h)$ be the centering parameter from
\cref{prof:canred:lem:saddle}.  At integer $h$, the tilting identity
\eqref{prof:canred:eq:coefficient-tilt-identity}, the centering equation
$\mathfrak m(t_{\mathrm{cen}}(h))=k-h$, and
\eqref{prof:canred:eq:central-factor-prob} give
\begin{equation}
 q_A(h)
 =\exp\{\mathfrak k(t_{\mathrm{cen}}(h))-t_{\mathrm{cen}}(h)(k-h)\}
 \frac{\mathfrak h(t_{\mathrm{cen}}(h))}{\sqrt{2\pi\mathfrak v(t_{\mathrm{cen}}(h))}}.
\end{equation}
Define the centered saddle action by
\[
 \mathfrak A_{\mathrm{cen}}(h):=\mathfrak A_h(t_{\mathrm{cen}}(h))
 =\mathfrak k(t_{\mathrm{cen}}(h))-t_{\mathrm{cen}}(h)(k-h).
\]
Use this to define a smooth interpolation of $\log q_A$ on $[0,d]$:
\begin{align}
 \mathfrak F(h)&:=\mathfrak A_{\mathrm{cen}}(h)-\frac12\log(2\pi\mathfrak v(t_{\mathrm{cen}}(h)))
       +\log\mathfrak h(t_{\mathrm{cen}}(h)).
 \label{prof:canred:eq:F-interpolation}
\end{align}
Then $\mathfrak F(h)=\log q_A(h)$ at every integer $0\le h\le d$.

The derivative bookkeeping is governed by one simple composition rule.  If
$G$ has uniformly bounded fixed $t$-derivatives on $[-t_0,t_0]$, then
\eqref{prof:canred:eq:centering-derivatives} and the fixed-order Fa\`a di Bruno
formula give
\begin{equation}
 \sup_{h\in[0,d]}
 \left|\frac{\dd^m}{\dd h^m}G(t_{\mathrm{cen}}(h))\right|
 \le C_{G,m}n^{-m}.
 \label{prof:canred:eq:centering-composition-rule}
\end{equation}
Indeed, every term contains a product
$t_{\mathrm{cen}}^{(r_1)}\cdots t_{\mathrm{cen}}^{(r_\ell)}$ with
$r_1+\cdots+r_\ell=m$, and hence has size $O(n^{-m})$.

We now apply this rule to the three terms of $\mathfrak F(h)$ in
\eqref{prof:canred:eq:F-interpolation}.  
For the first saddle-action term, differentiation and the centering equation $\mathfrak m(t_{\mathrm{cen}}(h))=k-h$ to produce a clean identity
\[
 \mathfrak A_{\mathrm{cen}}'(h)= \mathfrak k'(t_{\mathrm{cen}}(h)) t_{\mathrm{cen}}'(h) - t_{\mathrm{cen}}'(h)(k-h) + t_{\mathrm{cen}}(h) =t_{\mathrm{cen}}(h).
\]
Therefore
\begin{equation}
 \sup_{h\in[0,d]}|\mathfrak A_{\mathrm{cen}}'(h)|\le Cn^{-1},
 \qquad
 \sup_{h\in[0,d]}|\mathfrak A_{\mathrm{cen}}^{(m)}(h)|
 =\sup_{h\in[0,d]}|t_{\mathrm{cen}}^{(m-1)}(h)|
 \le C_mn^{1-m}
 \quad(m\ge2).
 \label{prof:canred:eq:saddle-action-derivatives}
\end{equation}
For $m\ge2$, this is the scale that will dominate the two local-limit
corrections; for $m=1$, it already has the target size $O(n^{-1})$.

For the second, variance correction, term in \eqref{prof:canred:eq:F-interpolation}, note that
$\mathfrak v=\mathfrak k^{(2)}$.  Thus
\eqref{prof:canred:eq:deleted-logpart-derivatives} gives
$\mathfrak v^{(j)}(t)=O_j(n)$ for every fixed $j\ge0$, while
\eqref{prof:canred:eq:deleted-var-comp} gives $\mathfrak v(t)\asymp n$.
Fa\`a di Bruno applied to $\log\mathfrak v(t)$ therefore shows that every
fixed $t$-derivative of $\log\mathfrak v(t)$ is $O(1)$.  Applying
\eqref{prof:canred:eq:centering-composition-rule} gives
\begin{equation}
 \sup_{h\in[0,d]}
 \left|\frac{\dd^m}{\dd h^m}\log\mathfrak v(t_{\mathrm{cen}}(h))\right|
 \le C_mn^{-m}.
\end{equation}
Likewise, \cref{prof:canred:lem:central} gives uniformly bounded fixed
$t$-derivatives of $\log\mathfrak h(t)$, the last term in \eqref{prof:canred:eq:F-interpolation}, so the same composition rule yields
\begin{equation}
 \sup_{h\in[0,d]}
 \left|\frac{\dd^m}{\dd h^m}\log\mathfrak h(t_{\mathrm{cen}}(h))\right|
 \le C_mn^{-m}.
 \label{prof:canred:eq:log-central-factor-h}
\end{equation}
For $m\ge2$, these two local-limit corrections are one power of $n^{-1}$
smaller than the principal contribution
\eqref{prof:canred:eq:saddle-action-derivatives}; for $m=1$, all three
contributions are $O(n^{-1})$.  Combining
\eqref{prof:canred:eq:F-interpolation} with
\eqref{prof:canred:eq:saddle-action-derivatives} through
\eqref{prof:canred:eq:log-central-factor-h} gives
\begin{equation}
 \sup_{h\in[0,d]}|\mathfrak F'(h)|\le Cn^{-1},
 \qquad
 \sup_{h\in[0,d]}|\mathfrak F^{(m)}(h)|\le C_mn^{1-m}
 \quad(m\ge2).
 \label{prof:canred:eq:F-derivatives}
\end{equation}

It remains to pass from the derivative bounds for the smooth interpolation
$\mathfrak F$ to the required finite-difference bounds.  The fundamental
theorem of calculus gives
\[
 \Delta\mathfrak F(a)
 =\mathfrak F(a+1)-\mathfrak F(a)
 =\int_0^1 \mathfrak F'(a+s_1)\,\dd s_1.
\]
Applying this identity successively yields, after $m$ iterations,
\[
 \Delta^m\mathfrak F(a)
 =\int_{[0,1]^m}
 \mathfrak F^{(m)}(a+s_1+\cdots+s_m)
 \,\dd s_1\cdots\dd s_m.
\]
If $0\le a\le d-m$, then
$a+s_1+\cdots+s_m\in[0,d]$ throughout the integration region.  Hence
\eqref{prof:canred:eq:F-derivatives} gives
\[
 |\Delta^m\mathfrak F(a)|
 \le \sup_{h\in[0,d]}|\mathfrak F^{(m)}(h)|
 \le C_{M,\underline p}
 \begin{cases}
  n^{-1},&m=1,\\
  n^{1-m},&m\ge2.
 \end{cases}
\]
By construction $\mathfrak F(j)=\log q_A(j)$ at every integer
$0\le j\le d$.  Hence, for integer $a$,
$\Delta^m\mathfrak F(a)=\Delta^m\log q_A(a)$, and the displayed estimate
proves \eqref{prof:canred:eq:fd-target}.
\end{proof}

The estimate \eqref{prof:canred:eq:fd-target} is the bridge from the scalar
local limit calculation to dependence among the marked sites.  We now use
these finite-difference bounds to obtain an exact Gibbs representation of the
marked conditional law.

\begin{lemma}[Exact marked Gibbs representation]
\label{prof:canred:lem:gibbs}
Let $A\subseteq\mathcal I$ be a marked set of size $d\le M$, and write
$\nu_A:=\Pp(\mathsf B_A\in\cdot\mid \mathsf N_{\mathcal I}=k)$ for its conditional marked law, \emph{cf.\@} \eqref{prof:canred:eq:marked-marginal}.  
For $1\le m\le d$, set
\[
c_m:=\Delta^m\log q_A(0).
\]
For $B\subseteq A$ and $\boldsymbol\xi=(\xi_x)_{x\in A}\in\{0,1\}^A$, write
$\xi_B:=\prod_{x\in B}\xi_x$; since $\xi_x\in\{0,1\}$, this is the indicator
that every site of $B$ is occupied in the marked configuration $\boldsymbol\xi$.
Then for $n\ge n_0(M,\underline p)$,
\begin{equation}
 \nu_A(\dd\boldsymbol\xi)
 =\frac{
 \exp\bigg\{
 \sum_{B\subseteq A: |B|\ge2}
 c_{|B|}\,\xi_B
 \bigg\}}{
 \mathbb E_{\mu_A}\!\left[\exp\bigg\{
 \sum_{B\subseteq A: |B|\ge2}c_{|B|}\,\xi_B
 \bigg\}\right]}\,\mu_A(\dd\boldsymbol\xi),
 \label{prof:canred:eq:gibbs-rep}
\end{equation}
where, after increasing $n_0$ if necessary, $\mu_A$ is a product Bernoulli
measure on $\{0,1\}^A$ with parameters in
$[\underline p/2,1-\underline p/2]$.  
Moreover,
\begin{equation}
 |c_{|B|}|
 \le C_{M,\underline p}n^{-(|B|-1)}
 \qquad(|B|\ge2).
 \label{prof:canred:eq:interaction-bound}
\end{equation}
\end{lemma}

\begin{proof}
The finite differences $c_m=\Delta^m\log q_A(0)$, $m\ge1$, together with
the constant value $\log q_A(0)$ reconstruct $\log q_A$ exactly on
$\{0,1,\ldots,d\}$.  To see this, repeatedly use
$g(r+1)=g(r)+(\Delta g)(r)$, first for $g=\log q_A$ and then for its
successive finite differences.  After $h$ steps, Pascal's relation gives
\begin{equation}
 \log q_A(h)
 =\log q_A(0)+\sum_{m=1}^{h}c_m\binom hm,
 \qquad 0\le h\le d.
 \label{prof:canred:eq:fd-reconstruction}
\end{equation}

Let $\boldsymbol\xi=(\xi_x)_{x\in A}\in \{0,1\}^A$. 
If $|\boldsymbol\xi|=h$, then exactly $h$ sites of $A$ are occupied.  By the
definition of $\xi_B$,
$\sum_{B\subseteq A:|B|=m}\xi_B$ counts the $m$-element subsets of the
$h$ occupied sites, and hence
\begin{equation}
 \binom{|\boldsymbol\xi|}{m}
 =\sum_{B\subseteq A: |B|=m}\xi_B.
 \label{prof:canred:eq:binomial-monomials}
\end{equation}
Substituting \eqref{prof:canred:eq:fd-reconstruction} and
\eqref{prof:canred:eq:binomial-monomials} into the exact marked marginal
\eqref{prof:canred:eq:marked-marginal} therefore gives
\[
 \nu_A(\boldsymbol\xi)\propto
 \left(\prod_{x\in A}p_x^{\xi_x}(1-p_x)^{1-\xi_x}\right)
 \exp\left\{
 c_1\sum_{x\in A}\xi_x
 +\sum_{|B|\ge2}c_{|B|}\xi_B
 \right\}.
\]
The one-body coefficient $c_1$ only changes the product reference measure, so
we absorb it into the Bernoulli odds by setting
\[
 \frac{p_x^*}{1-p_x^*}:=e^{c_1}\frac{p_x}{1-p_x}.
\]
By \cref{prof:canred:lem:finite-differences}, $|c_1|\le Cn^{-1}$; after
increasing $n_0$, the shifted parameters $p_x^*$ remain in
$[\frac{\underline p}{2},1-\frac{\underline p}{2}]$.  Taking $\mu_A$ to be this product Bernoulli law yields
\eqref{prof:canred:eq:gibbs-rep}.  For $|B|\ge2$, the coefficient of the
$|B|$-body interaction is exactly $c_{|B|}$, and
\eqref{prof:canred:eq:fd-target} with $a=0$ gives
\eqref{prof:canred:eq:interaction-bound}.
\end{proof}

The Gibbs representation \eqref{prof:canred:eq:gibbs-rep} reduces the
estimate of the conditional cumulant $\Cum_k$ under $\Pp(\,\cdot\mid \mathsf N_{\mathcal I}=k)$ to a connectivity statement.  
Relative to the
product measure $\mu_A$, all dependence among the marked coordinates is
carried by the interaction terms
$
 c_{|B|}\,\xi_B$, $|B|\ge2$,
where the coefficients $c_{|B|}$ satisfy \eqref{prof:canred:eq:interaction-bound}.  
Under $\mu_A$, random variables
supported on disjoint coordinate sets are independent, so a joint cumulant
vanishes whenever its variables split into two nonempty groups supported on
disjoint sets of sites.  Thus, when the cumulant is expanded in the
interaction coefficients, only collections of interacting subsets $B$ that
connect all $d$ marked sites can contribute.  Viewing each such $B$ as a
hyperedge, connectivity forces
\[
 \sum_B (|B|-1)\ge d-1,
\]
where the inserted hyperedges are counted with multiplicity.  Their product
of interaction coefficients is therefore $O(n^{1-d})$.  The next lemma makes
this connectivity argument quantitative.

\begin{lemma}[Connectivity bound for conditional cumulants]
\label{prof:canred:lem:hypergraph}
Fix integers $M\ge1$ and $1\le j\le M$.  Under
\eqref{prof:canred:eq:standing}, let $x_1,\ldots,x_j$ involve exactly $d$ distinct
sites.  Then
\[
 \abs{\Cum_k(\mathsf B_{x_1},\ldots,\mathsf B_{x_j})}
 \le C_{M,\underline p}n^{1-d}.
\]
The bound is uniform in the multiplicity pattern of the repeated sites.
\end{lemma}

\begin{proof}
For $d=1$, the claim follows because a fixed-order cumulant is a finite
linear combination of moments, and all moments of variables in $[0,1]$ are
bounded by one.  Assume $d\ge2$, and let $A$ be the set of
distinct sites among $x_1,\ldots,x_j$.  Set
\[
 \mathscr H_A:=\{B\subseteq A:|B|\ge2\},
 \qquad h_d:=|\mathscr H_A|.
\]
Here $B\in\mathscr H_A$ denotes a specific interaction subset of the marked
set $A$.  By \cref{prof:canred:lem:gibbs}, for $n\ge n_0$
the conditional marked law is \eqref{prof:canred:eq:gibbs-rep}.  Its
coefficient for the interaction subset $B$ depends only on the cardinality
$|B|$, so for $B\in\mathscr H_A$ we write
$
 c_B:=c_{|B|}$.
Then \eqref{prof:canred:eq:interaction-bound} gives
\[
 |c_B|\le C_{M,\underline p}n^{-(|B|-1)},\qquad B\in\mathscr H_A.
\]

\localheading{Generating the desired cumulant.}
Let $\boldsymbol\alpha=(\alpha_x)_{x\in A}$ record the multiplicities of the marked sites among
$x_1,\ldots,x_j$, so $\alpha_x\ge1$ and $\sum_{x\in A}\alpha_x=j$.  We use two sets
of auxiliary variables for two different purposes.  The source variables
$\boldsymbol s=(s_x)_{x\in A}$ generate the marked-site cumulant we want to estimate;
the interaction variables $\boldsymbol z=(z_B)_{B\in\mathscr H_A}$ interpolate between
the product reference law ($\boldsymbol z=\mathbf 0$) and the Gibbs law
($\boldsymbol z=\boldsymbol c$, where $\boldsymbol c:=(c_B)_{B\in\mathscr H_A}$).  Define
\begin{equation}
 \mathscr L_A(\boldsymbol s,\boldsymbol z)
 :=\log\E_{\mu_A}
 \exp\left\{
 \sum_{x\in A}s_x\xi_x+
 \sum_{B\in\mathscr H_A}z_B\xi_B
 \right\}.
 \label{prof:canred:eq:HAsz}
\end{equation}
Writing
$\partial_{\boldsymbol s}^{\boldsymbol\alpha}:=\prod_{x\in A}\partial_{s_x}^{\alpha_x}$,
the Gibbs representation gives
\[
 \partial_{\boldsymbol s}^{\boldsymbol\alpha}\mathscr L_A(\mathbf 0,\boldsymbol c)
 =\Cum_k(\mathsf B_{x_1},\ldots,\mathsf B_{x_j}).
\]

Choose $\delta_d>0$ so small that
$e^{(d+h_d)\delta_d}-1<\frac{1}{2}$.  If $|s_x|,|z_B|\le\delta_d$, then the
random exponent in \eqref{prof:canred:eq:HAsz} has absolute value at most
$(d+h_d)\delta_d$, and therefore the expectation of its exponential lies
within distance $\frac12$ of $1$.  The principal logarithm in
\eqref{prof:canred:eq:HAsz} is thus analytic throughout this polydisk,
uniformly in the product probabilities.  For a
hyperedge multi-index
$\boldsymbol r=(r_B)_{B\in\mathscr H_A}\in\mathbb N_0^{\mathscr H_A}$,
write
\[
 |\boldsymbol r|:=\sum_{B\in\mathscr H_A}r_B,\qquad
 \boldsymbol r!:=\prod_{B\in\mathscr H_A}r_B!,\qquad
 \boldsymbol c^{\boldsymbol r}:=\prod_{B\in\mathscr H_A}c_B^{r_B},\qquad
 \partial_{\boldsymbol z}^{\boldsymbol r}:=\prod_{B\in\mathscr H_A}\partial_{z_B}^{r_B}.
\]
For $n$ large enough, $|c_B|<\frac{\delta_d}{2}$, so Taylor expansion from the
product point $\boldsymbol z=\mathbf 0$ to the Gibbs point $\boldsymbol z=\boldsymbol c$ is absolutely convergent:
\begin{equation}
 \partial_{\boldsymbol s}^{\boldsymbol\alpha}\mathscr L_A(\mathbf 0,\boldsymbol c)
 =\sum_{\boldsymbol r\in\mathbb N_0^{\mathscr H_A}}
 \frac{\boldsymbol c^{\boldsymbol r}}{\boldsymbol r!}
 \partial_{\boldsymbol z}^{\boldsymbol r}\partial_{\boldsymbol s}^{\boldsymbol\alpha}
 \mathscr L_A(\mathbf 0,\mathbf 0).
 \label{prof:canred:eq:z-taylor}
\end{equation}

\localheading{Connectivity and the interaction cost.}
For such a multi-index $\boldsymbol r$, define
\[
 R:=|\boldsymbol r|,
 \qquad
 W(\boldsymbol r):=\sum_{B\in\mathscr H_A}r_B(|B|-1).
\]
These two quantities have separate meanings.  The integer $R$ is the number
of interaction insertions, counted with multiplicity.  The quantity $W$ is
the total power of $n^{-1}$ supplied by those interactions, because
\eqref{prof:canred:eq:interaction-bound} gives
\begin{equation}
 |\boldsymbol c^{\boldsymbol r}|
 \le C_{M,\underline p}^R n^{-W(\boldsymbol r)}.
 \label{prof:canred:eq:c-r-weight}
\end{equation}

At $\boldsymbol z=\mathbf 0$, the derivative in \eqref{prof:canred:eq:z-taylor} is a joint
cumulant under the product law $\mu_A$ of
$\alpha_x$ copies of $\xi_x$ for each $x\in A$ and
$r_B$ copies of $\xi_B$ for each interaction hyperedge $B$.  If the
multihypergraph formed by the inserted hyperedges is disconnected, its
connected components use disjoint coordinate sets.  Because every marked
site appears among the source variables, the listed random variables split
into two nonempty independent families under $\mu_A$.  Their joint cumulant
is hence zero.  
Thus only connected multihypergraphs contribute to
\eqref{prof:canred:eq:z-taylor}.

Connectivity forces exactly the power needed for the theorem:
\begin{equation}
 W(\boldsymbol r)\ge d-1.
 \label{prof:canred:eq:connected-weight}
\end{equation}
Indeed, begin with $d$ singleton components.  Adding a hyperedge $B$ can
reduce the number of components by at most $|B|-1$, so reaching one component
requires total reduction at least $d-1$.  Since every hyperedge has size at
least two, we also have
\begin{equation}
 W(\boldsymbol r)\ge R.
 \label{prof:canred:eq:weight-R}
\end{equation}

\localheading{Summing the connected expansion.}
It remains to verify absolute summability using the connectivity power
counts \eqref{prof:canred:eq:connected-weight} and \eqref{prof:canred:eq:weight-R}.
We first record the crude fixed-order cumulant bound used for this purpose.
The moment--cumulant identity \eqref{prof:eq:moment-cumulant-identity}, derived
above from the partition form of the multivariate Fa\`a di Bruno formula, is
triangular in the number of variables:
\[
 \mathbb E\!\left[\prod_{i=1}^L Y_i\right]
 =\Cum(Y_1,\ldots,Y_L)
  +\sum_{\substack{\mathfrak p\in\operatorname{Part}([L])\\|\mathfrak p|\ge2}}
    \prod_{\mathfrak b\in\mathfrak p}
    \Cum\bigl((Y_i)_{i\in\mathfrak b}\bigr).
\]
Every block \(\mathfrak b\) in the second sum has \(|\mathfrak b|<L\), so the
order-$L$ cumulant is determined recursively by the $L$th moment and
lower-order cumulants.  Equivalently, M\"obius inversion on the lattice of
set partitions solves this triangular recursion in closed form:
\[
 \Cum(Y_1,\ldots,Y_L)
 =\sum_{\mathfrak p\in\operatorname{Part}([L])}
   (-1)^{|\mathfrak p|-1}(|\mathfrak p|-1)!
   \prod_{\mathfrak b\in\mathfrak p}
   \mathbb E\!\left[\prod_{i\in\mathfrak b}Y_i\right].
\]
If $|Y_i|\le1$, every block moment in this sum has absolute value at most
one.  Writing $S(L,k)$ for the number of partitions of $[L]$ into $k$ blocks,
we therefore obtain
\[
 |\Cum(Y_1,\ldots,Y_L)|
 \le \sum_{k=1}^L S(L,k)(k-1)!.
\]
Moreover, $k!S(L,k)$ is the number of surjections from $[L]$ onto $[k]$,
so $k!S(L,k)\le k^L$.  Hence
\[
 \sum_{k=1}^L S(L,k)(k-1)!
 \le \sum_{k=1}^L k^{L-1}
 \le L^L
 \le e^L L!,
\]
where the last inequality follows from $L!\ge(L/e)^L$.  Thus, for a universal
constant $C$,
\[
 |\Cum(Y_1,\ldots,Y_L)|\le C^L L!.
\]
In \eqref{prof:canred:eq:z-taylor}, the derivative at $\boldsymbol z=\mathbf 0$ is a cumulant of
$L=j+R$ variables, each equal to some $\xi_x$ or $\xi_B$ and hence bounded by
one.  Therefore
\begin{equation}
 \left|
 \partial_{\boldsymbol z}^{\boldsymbol r}\partial_{\boldsymbol s}^{\boldsymbol\alpha}
 \mathscr L_A(\mathbf 0,\mathbf 0)
 \right|
 \le C^{j+R}(j+R)!.
 \label{prof:canred:eq:derivative-cumulant-bound}
\end{equation}
For fixed $R$, the multinomial theorem applied to
$(1+\cdots+1)^R$ with $h_d$ summands gives
\begin{equation}
 \sum_{\boldsymbol r\in\mathbb N_0^{\mathscr H_A}:|\boldsymbol r|=R}
 \frac1{\boldsymbol r!}
 =\frac{h_d^R}{R!}.
 \label{prof:canred:eq:multiindex-count}
\end{equation}

For $R\le d-1$, there are only finitely many possible values of $R$, and
every connected term already carries $n^{-W}\le n^{1-d}$ by
\eqref{prof:canred:eq:connected-weight}.  Equations
\eqref{prof:canred:eq:derivative-cumulant-bound} and
\eqref{prof:canred:eq:multiindex-count} therefore show that all such terms
contribute $O_{M,\underline p}(n^{1-d})$.

For $R\ge d$, use the stronger bound $W\ge R$ from
\eqref{prof:canred:eq:weight-R}.  Combining
\eqref{prof:canred:eq:c-r-weight} through
\eqref{prof:canred:eq:multiindex-count}, and using
$h_d\le2^d\le2^M$ together with
$(j+R)!/R!\le C_MR^M$ for $j\le M$, yields
\[
 \sum_{\substack{|\boldsymbol r|=R\\\text{connected}}}
 \frac{|\boldsymbol c^{\boldsymbol r}|}{\boldsymbol r!}
 \left|
 \partial_{\boldsymbol z}^{\boldsymbol r}\partial_{\boldsymbol s}^{\boldsymbol\alpha}
 \mathscr L_A(\mathbf 0,\mathbf 0)
 \right|
 \le
 C_M R^M\left(\frac{C_M}{n}\right)^R.
\]
Since $R\ge d$,
$n^{-R}=n^{1-d}n^{-(R-d+1)}$; for $n\ge2C_M$, summing over $R\ge d$
again gives $O_{M,\underline p}(n^{1-d})$.  Together with the finite-$R$ part, this
proves
\[
 \abs{\Cum_k(\mathsf B_{x_1},\ldots,\mathsf B_{x_j})}
 \le C_{M,\underline p}n^{1-d}
\]
for all sufficiently large $n$.

For the finitely many smaller $n$, a fixed-order cumulant of variables in
$[0,1]$ is bounded by a constant depending only on $M$.  Enlarging
$C_{M,\underline p}$ therefore extends the same estimate to every $n$.  Repeated
occurrences of a marked coordinate only change the fixed multiplicities
$\alpha_x$ and do not affect the connectivity argument.
\end{proof}

\subsection{From conditional Bernoullis to canonical tilts}
\label{prof:app:canonical-cumulants}

In this subsection we identify the canonical tilt with a product Bernoulli
law conditioned on total population $k_N$.  Once this identification is made,
\cref{prof:canred:lem:hypergraph} gives the required support-sensitive
canonical cumulant bound.

To make this identification, extend the notation $\nu_{N,u}$ from mean-zero fields
to arbitrary $u\in\mathbb R^{V_N}$ by setting
\[
 \nu_{N,u}(\eta)
 :=\frac{\exp\{\sum_xu_x\eta_x\}}
 {\sum_{\eta'\in\Omega_{N,k_N}}\exp\{\sum_xu_x\eta'_x\}},
 \qquad \eta\in\Omega_{N,k_N}.
\]
Adding a constant to $u$ multiplies numerator and denominator by the same
factor $e^{ck_N}$, so this law depends only on $u$ modulo constants, and agrees
with the earlier definition in \eqref{prof:eq:tilt} after replacing $u$ by
its mean-zero representative.  For such a field $u$, the scalar map
$
 \zeta\mapsto\sum_{x\in V_N}\frac{e^{\zeta+u_x}}{1+e^{\zeta+u_x}}
$
is continuous and strictly increasing from $0$ to $n_N$.  Since $0<k_N<n_N$,
there is therefore a unique population centering parameter $\zeta(u)$ satisfying
\[
 \sum_{x\in V_N}\frac{e^{\zeta(u)+u_x}}{1+e^{\zeta(u)+u_x}}=k_N.
\]  Under the density window \eqref{prof:eq:density-window}, there exist
$c=c(\rho_0)>0$ and $\underline p=\underline p(\rho_0)\in(0,\frac12)$ such that
\begin{equation}
 \begin{aligned}
 p_x(u)&=\frac{e^{\zeta(u)+u_x}}{1+e^{\zeta(u)+u_x}},
 &\qquad \sum_xp_x(u)&=k_N,\\
 \|u\|_\infty\le c&\quad\Longrightarrow\quad
 \underline p\le p_x(u)\le1-\underline p,
 & &x\in V_N,
 \end{aligned}
 \label{prof:eq:standing-saddle-probabilities}
\end{equation}
uniformly in $N$.  To verify the last line, fix any constant $c>0$ (for
instance $c=1$), write
$\operatorname{logistic}(r):=e^r/(1+e^r)$, recall that
$\rho_N=k_N/n_N$, and suppose $\|u\|_\infty\le c$.  Monotonicity of the
logistic function gives
\[
 \operatorname{logistic}(\zeta(u)-c)
 \le \rho_N\le
 \operatorname{logistic}(\zeta(u)+c),
\]
because the average of the occupation numbers
$\frac{1}{n_N}\sum_{x\in V_N} \operatorname{logistic}(\zeta(u)+u_x)=\rho_N$.
Applying the inverse function
$\operatorname{logistic}^{-1}(r)=\log(r/(1-r))$ yields
\[
 \operatorname{logistic}^{-1}(\rho_N)-c
 \le \zeta(u)\le
 \operatorname{logistic}^{-1}(\rho_N)+c.
\]
By \eqref{prof:eq:density-window},
$\operatorname{logistic}^{-1}(\rho_N)$ ranges over a compact interval
depending only on $\rho_0$; hence $\zeta(u)+u_x$ remains in a fixed compact
interval.  Applying the logistic function to that interval gives
$\underline p=\underline p(\rho_0)\in(0,\frac12)$ as claimed.

\begin{lemma}[Canonical support-sensitive cumulants]
\label{prof:lem:canonical-support-cumulants}
Fix $m<\infty$.  Suppose the Bernoulli parameters $p_x(u)$ in
\eqref{prof:eq:standing-saddle-probabilities} lie in $[\underline p,1-\underline p]$.  If
$\boldsymbol\alpha=(\alpha_x)_{x\in V_N}$ is a coordinate multi-index of total order at
most $m$, and $d:=|\{x:\alpha_x>0\}|\ge1$, write
\[
 \partial^{\boldsymbol\alpha}:=\prod_{x\in V_N}\partial_{u_x}^{\alpha_x}.
\]
Then
\begin{equation}
 \left|
 \partial^{\boldsymbol\alpha}
 \log [z^{k_N}]\prod_{x\in V_N}(1+ze^{u_x})
 \right|
 \le C_{m,\underline p}n_N^{1-d}.
 \label{prof:eq:canonical-support-derivative}
\end{equation}
Equivalently, if independent Bernoulli variables with parameters $p_x(u)$
are conditioned on their sum being $k_N$, every conditional joint cumulant
of fixed total order whose arguments involve $d$ distinct sites is
$O_{m,\underline p}(n_N^{1-d})$.
\end{lemma}

\begin{proof}
Let $(\mathsf B_x)_{x\in V_N}$ be independent Bernoulli variables with parameters
$p_x(u)$ from \eqref{prof:eq:standing-saddle-probabilities}.  For
$b\in\{0,1\}^{V_N}$, the product probability is proportional to
\[
 \prod_{x\in V_N}
 \left(\frac{p_x(u)}{1-p_x(u)}\right)^{b_x}
 =\exp\left\{\zeta(u)\sum_x b_x+\sum_xu_xb_x\right\},
\]
because $p_x(u)/(1-p_x(u))=e^{\zeta(u)+u_x}$.  On the event
$\sum_xb_x=k_N$, the first term in the exponent is the constant
$\zeta(u)k_N$.  Hence
\begin{align}
 \mathrm{Law}\left((\mathsf B_x)_{x\in V_N}\,\middle|\,\sum_x\mathsf B_x=k_N\right)
 =\nu_{N,u}.
 \label{prof:eq:law-conditional}
\end{align}

Let
\[
 \mathcal Z_N^{\mathrm{can}}(u):=[z^{k_N}]\prod_{x\in V_N}(1+ze^{u_x})
 =\sum_{\substack{\eta\in\{0,1\}^{V_N}\\ \sum_x\eta_x=k_N}}
   e^{\sum_xu_x\eta_x}.
\]
This is the partition function of $\nu_{N,u}$.  Derivatives of its logarithm
are the corresponding canonical joint cumulants, so
\begin{equation}
 \partial^{\boldsymbol\alpha}\log \mathcal Z_N^{\mathrm{can}}(u)
 =\operatorname{Cum}_{\nu_{N,u}}
   (\underbrace{\eta_x,\ldots,\eta_x}_{\alpha_x\text{ copies}},
    \ x\in\operatorname{supp}\boldsymbol\alpha).
\end{equation}
The cumulant on the right has total order at most $m$ and involves exactly
$d=|\operatorname{supp}\boldsymbol\alpha|$ distinct sites, with repeated
coordinates encoded by the multiplicities $\alpha_x$.
Applying \cref{prof:canred:lem:hypergraph} to the conditional law \eqref{prof:eq:law-conditional}, with
$n=n_N$, $\mathcal I=V_N$, and $k=k_N$, proves
\eqref{prof:eq:canonical-support-derivative}.
\end{proof}

\subsection{From sitewise cumulants to response estimates}
\label{prof:app:canonical-response-consequences}

\begin{proof}[Proof of \cref{prof:lem:canonical-response}]
Fix $c,\underline p$ as in \eqref{prof:eq:standing-saddle-probabilities}.  The
equivalent cumulant statement in
\cref{prof:lem:canonical-support-cumulants} gives
\eqref{prof:eq:canonical-distinct-cumulant-general} directly for any
$x_1,\ldots,x_j$ involving $d$ distinct sites.  In particular,
\begin{equation}
\begin{aligned}
 |\operatorname{Cum}_u(\eta_x,\eta_y,\eta_z)|&\le Cn_N^{-2},
 &&x,y,z\ \text{distinct},\\
 |\operatorname{Cum}_u(\eta_x,\eta_x,\eta_y)|&\le Cn_N^{-1},
 &&x\ne y, \\
 |\operatorname{Cum}_u(\eta_x,\eta_x,\eta_x)|&\le C.
\end{aligned}
\label{prof:eq:cumulant-est}
\end{equation}

To prove \eqref{prof:eq:canonical-multilinear-response}, we first expand the Fr\'echet--cumulant identity
\eqref{prof:eq:Frechet-cumulant-identity} in coordinates:
\begin{equation*}
 \mathrm d^j\Lambda_N(u)[a_1,\ldots,a_j]
 =\sum_{x_1,\ldots,x_j}
  \operatorname{Cum}_u(\eta_{x_1},\ldots,\eta_{x_j})
  \prod_{r=1}^j a_r(x_r).
\end{equation*}
Let $\mathfrak p(x_1,\ldots,x_j)\in\operatorname{Part}([j])$ be the coincidence partition:
indices $r,r'$ lie in the same block precisely when $x_r=x_{r'}$.
For a fixed partition $\mathfrak p$, write $y_{\mathfrak b}$ for the common
site attached to each block $\mathfrak b\in\mathfrak p$.  Taking absolute
values, applying \eqref{prof:eq:canonical-distinct-cumulant-general} with
$d=|\mathfrak p|$, and then dropping the requirement that different blocks
use different sites gives
\begin{equation}
\begin{aligned}
 \bigl|\mathrm d^j\Lambda_N(u)[a_1,\ldots,a_j]\bigr|
 &\le
 \sum_{\mathfrak p\in\operatorname{Part}([j])}
 \sum_{\substack{x_1,\ldots,x_j:\\
          \mathfrak p(x_1,\ldots,x_j)=\mathfrak p}}
 \bigl|\operatorname{Cum}_u(\eta_{x_1},\ldots,\eta_{x_j})\bigr|
 \prod_{r=1}^j |a_r(x_r)| \\
 &\le C_{j,\rho_0}
 \sum_{\mathfrak p\in\operatorname{Part}([j])} n_N^{1-|\mathfrak p|}
 \sum_{\substack{(y_{\mathfrak b})_{\mathfrak b\in\mathfrak p}:\\
          y_{\mathfrak b}\ne y_{\mathfrak b'}
          \text{ for }\mathfrak b\ne\mathfrak b'}}
 \prod_{\mathfrak b\in\mathfrak p}
 \prod_{r\in\mathfrak b}|a_r(y_{\mathfrak b})| \\
 &\le C_{j,\rho_0}
 \sum_{\mathfrak p\in\operatorname{Part}([j])} n_N^{1-|\mathfrak p|}
 \prod_{\mathfrak b\in\mathfrak p}
 \left(\sum_{x\in V_N}\prod_{r\in\mathfrak b}|a_r(x)|\right).
\end{aligned}
\label{prof:eq:coincidence-partition-factorized-bound}
\end{equation}
Thus the problem is reduced to estimating the factor $\sum_{x\in V_N}\prod_{r\in\mathfrak b}|a_r(x)|$ for each $\mathfrak b\in \mathfrak p$.  

Now fix $i\in \{1,\ldots, j\}$.
Given a coincidence partition $\mathfrak p$, let $\mathfrak b_i$ denote the block
containing $i$.  We first treat the blocks other than $\mathfrak b_i$.
For a singleton block
$\mathfrak b=\{r\}$, Cauchy--Schwarz gives
\begin{align}
 \sum_x|a_r(x)|\le n_N^{1/2}\|a_r\|_{2,\mathrm{cnt}}.
  \label{prof:eq:singleton-block-l2-bound}
\end{align}
For a block $\mathfrak b$ with $|\mathfrak b|\ge2$, generalized H\"older and
monotonicity of counting measure $\ell^p$ norms give
\begin{equation}
 \sum_x\prod_{r\in\mathfrak b}|a_r(x)|
 \le\prod_{r\in\mathfrak b}\|a_r\|_{\ell^{|\mathfrak b|}(V_N)}
 \le\prod_{r\in\mathfrak b}\|a_r\|_{2,\mathrm{cnt}}.
 \label{prof:eq:coincidence-block-l2-bound}
\end{equation}
For the distinguished block $\mathfrak b_i$, bound the factor $a_i$ by
$\|a_i\|_\infty$, and use counting measure $\ell^2$ bounds for the remaining
factors.  This gives
\begin{equation}
 \sum_x\prod_{r\in\mathfrak b_i}|a_r(x)|
 \le \|a_i\|_\infty
 \begin{cases}
  n_N, & |\mathfrak b_i|=1,\\
  n_N^{1/2}\|a_r\|_{2,\mathrm{cnt}},
    & \mathfrak b_i=\{i,r\},\\
  \displaystyle\prod_{r\in\mathfrak b_i\setminus\{i\}}
     \|a_r\|_{2,\mathrm{cnt}}, & |\mathfrak b_i|\ge3.
 \end{cases}
 \label{prof:eq:distinguished-coincidence-block-bound}
\end{equation}
Indeed, the second case is \eqref{prof:eq:singleton-block-l2-bound} applied to
the single remaining factor.  In the third case, generalized H\"older applied
directly to the $|\mathfrak b_i|-1\ge2$ remaining factors gives the
product of counting measure $\ell^2$ norms.

For each given partition $\mathfrak p$, let $s$ be the number of singleton blocks
other than $\mathfrak b_i$, and let $q$ be the number of nonsingleton blocks
other than $\mathfrak b_i$.  Thus
$
 |\mathfrak p|=1+s+q
 $.
In \eqref{prof:eq:coincidence-partition-factorized-bound}, the cumulant prefactor
therefore contributes $n_N^{-s-q}$.  The $s$ singleton blocks other than $\mathfrak b_i$ contribute
$n_N^{s/2}$ by \eqref{prof:eq:singleton-block-l2-bound}, whereas the
nonsingleton blocks other than $\mathfrak b_i$ contribute no positive power of $n_N$ by
\eqref{prof:eq:coincidence-block-l2-bound}.  Hence, before the distinguished
block $\mathfrak b_i$ is included, the net power of $n_N$ is
$
 n_N^{-s-q}n_N^{s/2}=n_N^{-q-s/2}$.
It remains only to include the factor supplied by
\eqref{prof:eq:distinguished-coincidence-block-bound}.  Because $j\ge3$, if $|\mathfrak b_i|=1$ then the remaining $j-1\ge2$
indices occupy blocks other than $\mathfrak b_i$, so either some such block is nonsingleton
($q\ge1$) or at least two of them are singletons ($s\ge2$).  If
$|\mathfrak b_i|=2$, the remaining $j-2\ge1$ indices occupy at least one block other than $\mathfrak b_i$, so $q\ge1$ or
$s\ge1$.  For $|\mathfrak b_i|\ge3$, no
additional restriction is needed.  Accordingly,
\[
\begin{array}{c|c|c}
 |\mathfrak b_i| & \text{net exponent of }n_N & \text{why the exponent is }\le0 \\
 \hline
 1 & 1-q-s/2
   & q\ge1\ \text{or}\ s\ge2, \\
 2 & \tfrac12-q-s/2
   & q\ge1\ \text{or}\ s\ge1, \\
 \ge3 & -q-s/2
   & q,s\ge0.
\end{array}
\]
Therefore the contribution of each fixed partition $\mathfrak p$ to the
right-hand side of \eqref{prof:eq:coincidence-partition-factorized-bound} is bounded by
the right-hand side of the fixed-$i$ estimate below.  Since $\operatorname{Part}([j])$ is finite for fixed $j$,
summing over $\mathfrak p\in\operatorname{Part}([j])$ and enlarging $C_{j,\rho_0}$ if necessary
gives
\begin{equation*}
 \bigl|\mathrm d^j\Lambda_N(u)[a_1,\ldots,a_j]\bigr|
 \le C_{j,\rho_0}\|a_i\|_\infty
      \prod_{r\ne i}\|a_r\|_{2,\mathrm{cnt}}.
\end{equation*}
Minimizing the last display over $i$ gives
\eqref{prof:eq:canonical-multilinear-response}.

Finally we prove \eqref{prof:eq:canonical-covariance-response}. At $u=0$, the exact first chaos covariance identity
\eqref{prof:eq:first-chaos-isometry} gives
$\mathrm d^2\Lambda_N(0)=\beta_NI$ on $\mathbb R_0^{V_N}$. 
By the fundamental theorem of calculus,
\begin{equation}
 \mathrm d^2\Lambda_N(u)-\mathrm d^2\Lambda_N(0)
 =\int_0^1\mathrm d^3\Lambda_N(\lambda u)[\,\cdot\,,\,\cdot\,,u] \,\dd\lambda.
 \label{prof:eq:canonical-Hessian-integral}
\end{equation}
The \(2\to2\) estimate follows directly from
\eqref{prof:eq:canonical-multilinear-response}.  For the
\(\infty\to\infty\) estimate, the coordinate entries of the integrand in
\eqref{prof:eq:canonical-Hessian-integral} are
\[
 \bigl[\mathrm d^3\Lambda_N(\lambda u)[\,\cdot\,,\,\cdot\,,u]\bigr](x,y)
 =\sum_z u(z)\operatorname{Cum}_{\lambda u}(\eta_x,\eta_y,\eta_z).
\]
Hence its \(\infty\to\infty\) norm is bounded by the maximal absolute row sum:
\begin{align*}
 \bigl\|\mathrm d^3\Lambda_N(\lambda u)[\,\cdot\,,\,\cdot\,,u]\bigr\|_{\infty\to\infty}
 &\le \sup_x\sum_y
 \left|\sum_z u(z)\operatorname{Cum}_{\lambda u}(\eta_x,\eta_y,\eta_z)\right| \\
 &\le \|u\|_\infty\sup_x
   \sum_{y,z}|\operatorname{Cum}_{\lambda u}(\eta_x,\eta_y,\eta_z)|.
\end{align*}
For fixed $x$, split the pairs $(y,z)$ according as $x,y,z$ involve three,
two, or one distinct sites.  Using the three estimates in 
\eqref{prof:eq:cumulant-est} we find
\[
 \sup_x\sum_{y,z}|\operatorname{Cum}_{\lambda u}(\eta_x,\eta_y,\eta_z)|
 \le C\left[n_N^2n_N^{-2}+n_Nn_N^{-1}+1\right]
 \le C.
\]
Therefore
\[
 \bigl\|\mathrm d^3\Lambda_N(\lambda u)[\,\cdot\,,\,\cdot\,,u]\bigr\|_{\infty\to\infty}
 \le C\|u\|_\infty
\]
uniformly in $\lambda\in[0,1]$.  Integrating this bound in
\eqref{prof:eq:canonical-Hessian-integral} proves the
\(\infty\to\infty\) estimate and hence
\eqref{prof:eq:canonical-covariance-response}.
\end{proof}

\subsection*{Acknowledgement of AI use}
The author acknowledges the use of ChatGPT--5.6 Sol for identifying several technical ingredients used in the proofs, and for assistance with refining the exposition.
The author thoroughly verified all mathematical derivations, proofs, and the final text, and takes full responsibility for the contents of this work.

\end{document}